\documentclass{amsart}
\usepackage[utf8x]{inputenc}
\usepackage{amsmath, amsthm, amssymb}
\usepackage[usenames,dvipsnames,svgnames,table]{xcolor}
\usepackage[margin=1.3in]{geometry}
\usepackage{hyperref}
\usepackage{cleveref}
\usepackage[noadjust]{cite}
\usepackage{overpic}
\usepackage{xfrac}

\newtheorem{theorem}{Theorem}[section]
\newtheorem{proposition}[theorem]{Proposition}
\newtheorem{lemma}[theorem]{Lemma}
\newtheorem{definition}[theorem]{Definition}
\newtheorem{remark}[theorem]{Remark}
\newtheorem{corollary}[theorem]{Corollary}
\newcommand{\N}{\mathbb N}
\newcommand{\R}{\mathbb R}
\newcommand{\Z}{\mathbb Z}
\newcommand{\T}{\mathbb T}
\newcommand{\eps}{\varepsilon}

\newcommand{\dd}{\, \mathrm{d}}
\newcommand{\vv}{\langle v\rangle}

\DeclareMathOperator{\supp}{supp}

\newcommand{\TMeps}{{\T^3_{M_\eps}}}
\newcommand{\TM}{{\T^3_M}}

\newcommand{\Qs}{Q_{\rm s}}
\newcommand{\Qns}{Q_{\rm ns}}
\newcommand{\dom}{\R^3\times\R^3}

\newcommand{\vvo}{\langle v_0\rangle}
\newcommand{\vvone}{\langle v_1\rangle}

\newcommand{\vvp}{\langle v'\rangle}
\newcommand{\vvt}{\langle v_\rmcr\rangle}

\newcommand{\rmcr}{{\rm cr}}

\newcommand{\cK}{\mathcal{K}}
\newcommand{\cL}{\mathcal{L}}

\newcommand{\be}{\begin{equation}}
\newcommand{\ee}{\end{equation}}

\numberwithin{equation}{section}
\numberwithin{theorem}{section}

\title[Classical solutions of the Boltzmann equation]{Classical solutions of the Boltzmann equation with irregular initial data}

\author{Christopher Henderson}
\address{Department of Mathematics, University of Arizona, Tucson, AZ 85721}
\email{ckhenderson@math.arizona.edu}

\author{Stanley Snelson}
\address{Department of Mathematical Sciences, Florida Institute of Technology, Melbourne, FL 32901}
\email{ssnelson@fit.edu}

\author{Andrei Tarfulea}
\address{Department of Mathematics, Louisiana State University, Baton Rouge, LA 70803}
\email{tarfulea@lsu.edu}

\thanks{CH was supported by NSF grants DMS-2003110 and DMS-2204615 and acknowledges support of the Institut Henri Poincaré (UAR 839 CNRS-Sorbonne Université) and LabEx CARMIN (ANR-10-LABX-59-01). SS was supported by a Simons Foundation Collaboration Grant, Award \#855061 and NSF grant DMS-2213407. AT was supported by NSF grants DMS-2012333 and DMS-2108209.}

\begin{document}

\maketitle

\begin{abstract}
This article considers the spatially inhomogeneous, non-cutoff Boltzmann equation. We construct a large-data classical solution given bounded, measurable initial data with uniform polynomial decay of mild order in the velocity variable. Our result requires no assumption of strict positivity for the initial data, except locally in some small ball in phase space. We also obtain existence results for weak solutions when our decay and positivity assumptions for the initial data are relaxed.

Because the regularity of our solutions may degenerate as $t$ tends to $0$, uniqueness is a challenging issue. We establish weak-strong uniqueness under the additional assumption that the initial data possesses no vacuum regions and is H\"older continuous.

As an application of our short-time existence theorem, we prove global existence near equilibrium for bounded, measurable initial data that decays at a finite polynomial rate in velocity.\\
\\
\noindent{}\textsc{Titre. Solutions classiques de l'\'equation de Boltzmann avec donn\'ee initiale irr\'eguli\`ere}
\\
\noindent{}\textsc{R\'esum\'e}. Cette article \'etudie l'\'equation de Boltzmann inhomog\`ene en espace sans troncature angulaire. En supposant la donn\'ee initiale mesurable et born\'ee \`a d\'ecroissance  polynomiale d'ordre limit\'e en la variable de vitesse, on construit une solution classique. Aucune hypoth\`ese de positivit\'e stricte de la donn\'ee initiale n'est n\'ec\'essaire, mais le r\'esultat repose sur une hypoth\`ese locale de positivit\'e stricte sur une petite boule dans l'espace des phases. On obtient l'existence de solutions faibles en rel\^achant les hypoth\`eses de d\'ecroissance et de positivit\'e. 

La question d'unicit\'e est rendue difficile car la r\'egularit\'e des solutions peut d\'eg\'en\'erer quand $t$ tend vers $0$. On \'etablit  l'unicit\'e faible-forte sous l'hypoth\`ese suppl\'ementaire pour la donn\'ee initiale: absence de r\'egion vide et continuit\'e de H\"older.

En application du r\'esultat d'existence en temps court, on prouve l'existence globale pr\`es d'un \'equilibre pour une donn\'ee initiale mesurable qui d\'ecroit polynomialement en la vitesse.
\end{abstract}

\tableofcontents

\section{Introduction}
We consider the Boltzmann equation, a fundamental kinetic integro-differential equation from statistical physics \cite{cercignani1969kinetic, truesdell-muncaster, cercignani1988book,  chapmancowling, villani2002review}. The unknown function $f(t,x,v)\geq 0$ models the particle density of a diffuse gas in phase space at time $t \geq 0$, location $x \in \R^3$, and velocity $v\in \R^3$. 
The equation reads
\begin{equation}\label{e.boltzmann}
 (\partial_t + v\cdot \nabla_x) f = Q(f,f),
\end{equation}
where the left-hand side is a transport term, and $Q(f,g)$ is the Boltzmann collision operator with {\it non-cutoff} collision kernel, which we describe in detail below.

The purpose of this article is to develop a well-posedness theory for \eqref{e.boltzmann} on a time interval $[0,T]$, making minimal assumptions on the initial data $f_{\rm in}(x,v)\geq 0$. 
In particular, we would like our local existence theory to properly encapsulate the {\it regularizing effect} of the non-cutoff Boltzmann equation. This effect comes from the nonlocal diffusion produced by $Q(f,f)$ in the velocity variable, and has been studied extensively, as we survey below in Section \ref{s:related}. In light of this regularizing effect, it is natural and desirable to construct a solution $f$ with initial data in a low-regularity (ideally zeroth-order) space, such that $f$ has at least enough regularity for positive times to evaluate the equation in a pointwise sense. However, so far this has only been achieved in the close-to-equilibrium \cite{alonso2020giorgi,silvestre2022nearequilibrium} and space homogeneous (i.e $x$-independent) \cite{desvillettes2004homogeneous} regimes. For the general case, essentially all of the local existence results for classical solutions in the literature \cite{amuxy2010regularizing, amuxy2011bounded, amuxy2013mild, morimoto2015polynomial, HST2020boltzmann, henderson2021existence} require $f_{\rm in}$ to lie in a weighted Sobolev space of order at least $4$. The current article fills this gap by constructing a solution with initial data in a weighted $L^\infty$-space. 

Another goal of our analysis is to optimize the requirement on the decay of $f_{\rm in}$ for large velocities.  Because of the nonlocality of $Q$, decay of solutions is intimately tied to regularity, and since we work in the physical regime $\gamma \leq 0$ (see \eqref{e.B-def}), the decay of $f$ for positive times is limited by the decay of $f_{\rm in}$. 
In our main existence result, we require $f_{\rm in}$ to have pointwise polynomial decay of order $3 + 2s$, where $2s\in (0,2)$ is the order of the diffusion (see \eqref{e.b-bounds}). In particular, the energy density $\int_{\R^3} |v|^2 f(t,x,v) \dd v$ of our solutions may be infinite, which places them outside the regime where the conditional regularity estimates of Imbert-Silvestre \cite{imbert2020review} may be applied out of the box.

The possible presence of vacuum regions in the initial data is a key source of difficulty. The regularization coming from $Q(f,f)$ relies on positivity properties of $f$, in a complex way that reflects the nonlocality of $Q$. 
In the space homogeneous setting, conservation of mass provides sufficient positivity of $f$ for free.  
The close-to-equilibrium assumption would also ensure that $f$ has regions of strict positivity at all times. By contrast, in the case of general initial data, any lower bounds for $f$ must degenerate at a severe rate as $t\searrow 0$, which impacts the regularity of the solution for small times and causes complications for the well-posedness theory. Our main existence theorem requires a weak positivity assumption, namely that $f_{\rm in}$ is uniformly positive in some small ball in phase space.

We consider solutions posed on the spatial domain $\R^3$, with no assumption that the solution or the initial data decay for large $|x|$. This regime includes the physically important example of a localized disturbance away from a Maxwellian equilibrium $M(v) = c_1 e^{-c_2|v|^2}$; that is $f = M(v) + g(t,x,v)$, where $g(t,x,v) \to 0$ as $|x|\to\infty$ but $g$ is not necessarily small. Our regime also includes spatially periodic solutions as a special case. The lack of integrability in $x$ is a nontrivial source of difficulty and in particular makes energy methods much less convenient. Also, the total mass, energy, and entropy of the solution could be infinite, so we do not have access to the usual bounds coming from conservation of mass and energy and monotonicity of entropy.

In Section \ref{s:prior-existence}, we give a more complete bibliography of well-posedness results for \eqref{e.boltzmann}.

\subsection{The collision operator}

Boltzmann's collision operator is a bilinear integro-differential operator defined for functions $f,g:\R^3\to\R$ by
\begin{equation}\label{e.Q}
 Q(f,g) := \int_{\R^3} \int_{\mathbb S^2} B(v-v_*,\sigma) [f(v_*')g(v') - f(v_*) g(v)] \dd \sigma \dd v_*.
 \end{equation}
 Because collisions are assumed to be elastic, the pre- and post-collisional velocities all lie on a sphere of diameter $|v-v_*|$ parameterized by $\sigma\in\mathbb S^2$, and are related by the formulas
\[ 
v' = \frac{v+v_*} 2 + \frac{|v-v_*|} 2 \sigma, \quad v_*' = \frac{v+v_*} 2 - \frac{|v-v_*|} 2 \sigma.\]
We take the standard \emph{non-cutoff} collision kernel $B$ defined by
\begin{equation}\label{e.B-def}
B(v-v_*,\sigma) = b(\cos\theta) |v-v_*|^\gamma, \quad \cos\theta = \sigma \cdot \frac{v-v_*}{|v-v_*|},
\end{equation}
for some $\gamma > -3$. The angular cross-section $b$ is singular as $\theta$ (the angle between pre- and post-collisional velocities) approaches $0$ and satisfies the bounds
\begin{equation}\label{e.b-bounds}
c_b \theta^{-1-2s} \leq b(\cos\theta) \sin \theta \leq \frac 1 {c_b} \theta^{-1-2s},
\end{equation}
for some $c_b>0$ and $s\in (0,1)$. This implies $b$ has the asymptotics 
$b(\cos\theta) \approx \theta^{-2-2s}$
as $\theta \to 0$. The parameters $\gamma$ and $s$ reflect the modeling choices made in defining $Q(f,g)$. When electrostatic interactions between particles are governed by an inverse-square-law potential of the form $\phi(x) = c|x|^{1-p}$ for some $p>2$, then one has $\gamma = (p-5)/(p-1)$ and $s = 1/(p-1)$. As is common in the literature, we consider arbitrary pairs $(\gamma,s)$ and disregard the parameter $p$.

For our main results, we assume 
\[
	\gamma < 0,
\]
but otherwise, we do not place any restriction on $\gamma$ and $s$. 
The integral in \eqref{e.Q} has two singularities: as $\theta\to 0$, and as $v_*\to v$. The non-integrable singularity at $\theta \approx 0$ (grazing collisions), which is related to the long-range interactions taken into account by the physical model, is the source of the regularizing properties of the operator $Q$.

\subsection{Main results}

For $1\leq p\leq \infty$, define the velocity-weighted $L^p$ norms 
\[
	\|f\|_{L^p_q(\R^3)} = \|\vv^q f\|_{L^p(\R^3)},
		\quad
	\|f\|_{L^p_q(\Omega\times\R^3)} = \|\vv^q f\|_{L^p(\Omega\times\R^3_v)}
\]
where 
$\Omega$ is any subset of $\R^3_x$ or $[0,\infty)\times\R^3_x$.

Our results involve kinetic H\"older spaces $C^{\beta}_{\rm \ell}$ that are defined precisely in Section \ref{s:holder-spaces} below. These spaces are based on a distance $d_\ell$ that is adapted to the scaling and translation symmetries of the Boltzmann equation. Roughly speaking, a function in $C^\beta_\ell$ for some $\beta>0$ is $C^\beta$ in $v$, $C^{\beta/(1+2s)}$ in $x$, and $C^{\beta/(2s)}$ in $t$. 

Note that the subscript $q$ in $L^p_q$ refers to a decay exponent, while the subscript $\ell$ in $C^\beta_\ell$ refers to the distance $d_\ell$. 

In the statement of our main results, for brevity's sake, we make the convention that all constants may depend on the parameters $\gamma$, $s$, and $c_b$ from the collision kernel, even if they are not specifically mentioned.

Our first main result is about the existence of classical solutions:

\begin{theorem}\label{t:existence}
Let $\gamma\in (-3,0)$ and $s\in (0,1)$. 
Assume that $f_{\rm in} \geq 0$ lies in $L^\infty_q(\R^6)$ with $q > 3 + 2s$, and that for some $x_m,v_m \in \R^3$ and $\delta, r>0$,
\begin{equation}\label{e.mass-core}
 f_{\rm in}(x,v) \geq \delta, \quad \text{ for } (x,v) \in B_r(x_m,v_m).
 \end{equation}
Then there exists $T>0$ depending on $q$ and $\|f_{\rm in}\|_{L^\infty_q}$ and a solution $f(t,x,v) \geq 0$ to the Boltzmann equation \eqref{e.boltzmann} in $L^\infty_q([0,T]\times \R^6)$. This solution is locally of class $C^{2s}_\ell$. More precisely, for each compact $\Omega\subset (0,T]\times \R^6$, there exist $C, \alpha>0$ depending on $q$, $\Omega$, $x_m$, $v_m$, $r$, $\delta$, and $\|f_{\rm in}\|_{L^\infty_q(\R^6)}$, such that
\[ \|f\|_{C^{2s+\alpha}_\ell(\Omega)} \leq C.\]

Furthermore, for any $m\geq 0$ and partial derivative $D^k f$, where $k$ is a multi-index in $(t,x,v)$ variables, there exists $q(k,m)>0$ such that for any compact $\Omega\subset (0,T]\times\R^3$, 
\be\label{e.c070502}
	f_{\rm in} \in L^\infty_{q(k,m)}(\R^6) \quad \Rightarrow \quad \|D^k f\|_{L^\infty_{m}(\Omega\times\R^3)} \leq C,
\ee
with $C$ depending on $k$, $m$, $\Omega$, and the initial data.  
 If $f_{\rm in}$ decays faster than any polynomial, i.e. $f_{\rm in} \in L^\infty_q(\R^6)$ for all $q>0$, then the solution $f$ is $C^\infty$ in all three variables for positive times, and $D^k f\in L^\infty_{m}(\Omega\times\R^3)$ for all $m\geq 0$, multi-indices $k$, and compact sets $\Omega$. 

At $t=0$, the solution agrees with $f_{\rm in}$ in the following weak sense: for all $\varphi \in C^1_{t,x} C^2_v$ with compact support in $[0,T)\times \R^6$, 
\begin{equation}\label{e.weak-initial-data}
 \int_{\R^6} f_{\rm in}(x,v) \varphi(0,x,v) \dd v \dd x = \int_0^T \int_{\R^6} [f(\partial_t + v\cdot \nabla_x)\varphi + Q(f,f) \varphi] \dd v \dd x \dd t.
 \end{equation}
\end{theorem}

Some comments on the theorem statement are in order:

\begin{itemize}
	
	\item The local regularity of $f$ of order $2s+\alpha$, where $\alpha>0$ is uniform on compact sets, is enough to make pointwise sense of $Q(f,f)$, as we prove in Lemma \ref{l:Q-makes-sense}. The norm $C^{2s+\alpha}_\ell$ also controls the material derivative $(\partial_t + v\cdot \nabla_x)f$ (see \cite[Lemma 2.7]{imbert2018schauder}). Therefore, although $\partial_t f$ and $\nabla_x f$ do not necessarily exist classically, the two sides of equation \eqref{e.boltzmann} have pointwise values and are equal at every $(t,x,v)\in (0,T]\times \R^6$.
 
	\item In general, our solutions may have a discontinuity at $t=0$. If we make the additional assumption that $f_{\rm in}$ is continuous, then $f$ is continuous as $t\to 0$ and agrees with $f_{\rm in}$ pointwise. This is proven in Proposition \ref{p:cont-match}.
	
	\item It is not \emph{a priori} obvious that the time integral on the right in \eqref{e.weak-initial-data} converges, since the regularity required to make pointwise sense of $Q(f,f)$ degenerates as $t\to 0$. Using the weak formulation of the collision operator (see \eqref{e.weak-collision} below) one can bound $\int_{\R^3} Q(f,f) \varphi \dd v$ from above using only bounds for $\varphi$ in $W_v^{2,\infty}$ and $f$ in $L^\infty\cap L^1_{\gamma+2s}$. This implies that the formula \eqref{e.weak-initial-data} is well-defined.

	\item Our results only depend on dimension through the restrictions on the velocity weight, which is related to integrability.  In particular, \Cref{t:existence} likely holds in $\R^d\times \R^d$ for any $d\geq 2$ after replacing the restriction $q > 3 + 2s$ with $q > d+2s$; likewise for the results below.  We state our results only in $\R^3\times \R^3$ because (i) this is the physically relevant case and (ii) we rely on many previous results stated only for $\R^3\times \R^3$ (which, again, are likely hold in any dimension).
	
\end{itemize}

One can deduce existence of a classical solution in the $\gamma=0$ case fairly easily  via \Cref{t:existence}.  Indeed, the $C^{2s+\alpha}_\ell$ estimates are uniform as $\gamma\nearrow 0$, so one can use local convergence of solutions $f^\gamma$ to obtain $f^0$ that solves~\eqref{e.boltzmann}, after performing a suitable convergence analysis for $Q(f,f)$ as $\gamma \nearrow 0$.  On the other hand, the higher regularity estimates depend on $\gamma$, so we cannot deduce the smoothing~\eqref{e.c070502}.  We arrive at the following, whose proof we omit as it follows exactly as outlined above:
\begin{corollary}\label{c.gamma=0}
Let $\gamma=0$ and $s\in (0,1)$. 
Assume that $f_{\rm in} \geq 0$ lies in $L^\infty_q(\R^6)$ with $q > 3 + 2s$, and that for some $x_m,v_m \in \R^3$ and $\delta, r>0$ the assumption~\eqref{e.mass-core} holds. 
Then there exists $T>0$ depending on $q$ and $\|f_{\rm in}\|_{L^\infty_q}$ and a solution $f(t,x,v) \geq 0$ to the Boltzmann equation \eqref{e.boltzmann} in $L^\infty_q([0,T]\times \R^6)$. This solution is locally of class $C^{2s}_\ell$. More precisely, for each compact $\Omega\subset (0,T]\times \R^6$, there exist $C, \alpha>0$ depending on $q$, $\Omega$, $x_m$, $v_m$, $r$, $\delta$, and $\|f_{\rm in}\|_{L^\infty_q(\R^6)}$, such that
\[ \|f\|_{C^{2s+\alpha}_\ell(\Omega)} \leq C.\]
At $t=0$, the solution agrees with $f_{\rm in}$ in the weak sense given by~\eqref{e.weak-initial-data}.
\end{corollary}

Although our main interest is in constructing classical solutions, our approach is robust enough to prove the existence of weak solutions when the decay and positivity conditions on $f_{\rm in}$ are relaxed. In particular, we obtain a well-defined notion of weak solution for any $f_{\rm in}\in L^\infty_q(\R^6)$, $q>3+\gamma+2s$, without any quantitative lower bound assumptions on $f_{\rm in}$. 
 These weak solutions do not have enough regularity to evaluate $Q(f,f)$ pointwise, so we define the weak formulation of the collision operator as follows:
\begin{equation}\label{e.weak-collision}
W(g,h,\varphi) :=  \frac 1 2 \int_{\R^3} \int_{\mathbb S^2} B(v-v_*,\sigma) g(v) h(v_*) [\varphi(v_*') + \varphi(v') - \varphi(v_*) - \varphi(v)] \dd \sigma \dd v_* .
\end{equation}
This weak form of $Q(f,f)$ is very classical and goes back to James Clerk Maxwell's 1867 work on the theory of gases \cite{maxwell1867}. 
When $f$ is sufficiently smooth and rapidly decaying, the identity $\int_{\R^3} \varphi Q(f,f) \dd v = \int_{\R^3} W(f,f,\varphi) \dd v$ follows from the pre-post-collisional change of variables and symmetrization, see, e.g. \cite[Chapter 1, Section 2.3]{villani2002review}.  

Our result on weak solutions is as follows:

\begin{theorem}\label{t:weak-solutions}
Let $\gamma$ and $s$ be as in Theorem \ref{t:existence}. Assume that $f_{\rm in}\geq 0$ lies in $L^\infty_{q}(\R^6)$ for some $q>3+\gamma+2s$. Then there exists $T>0$ depending on $q$ and $\|f_{\rm in}\|_{L^\infty_{q}(\R^6)}$ and $f: [0,T]\times \R^6\to [0,\infty)$ such that, for any $\varphi \in C^1_{t,x} C^2_v$ with compact support in $[0,T)\times\R^6$, there holds
\begin{equation}\label{e.weak-solution-def}
\int_{\R^6} f_{\rm in}(x,v) \varphi(0,x,v) \dd v \dd x = \int_0^T \int_{\R^6}[ f (\partial_t + v\cdot \nabla_x)\varphi +W(f,f,\varphi)] \dd v \dd x \dd t.
\end{equation}
If, in addition, there exist $\delta, r>0$ and $x_m,v_m\in \R^3$ with
\[ f_{\rm in}(x,v) \geq \delta, \quad \text{ for } (x,v) \in B_r(x_m)\times B_r(v_m),\]
then $f$ is locally H\"older continuous: for any compact $\Omega\subset (0,T]\times\R^6$, there exist $C,\alpha>0$ depending on $\Omega$, $x_m$, $v_m$, $r$, $\delta$, and $\|f\|_{L^\infty_q(\R^6)}$, with 
\[ \|f\|_{C^\alpha_\ell(\Omega)}\leq C.\]
\end{theorem}
This definition of weak solution is similar to one used by Alexandre \cite{alexandre2001solutions} who worked under stricter hypotheses on the initial data.

	The careful reader might be surprised by the H\"older regularity in \Cref{t:weak-solutions} in view of the De Giorgi theory developed by Imbert and Silvestre under the conditional assumptions on the mass, energy, and entropy densities. 
	Notice that \Cref{t:weak-solutions} is compatible with solutions not having finite mass or energy densities.  The results of Imbert and Silvestre, however, only rely on quantities (roughly) like
	\[
		\int f(v+w) |w|^{\gamma + 2s} \dd w.
	\]
	This is bounded above by the $L^\infty_q$-norm of $f$ and below by the mass spreading results of~\cite{HST2020lowerbounds}.  Hence, we are able to access many of their results, see \Cref{s.local_reg,s.collision_operator}.

Next, we present our main result on uniqueness. This is a challenging issue because of the generality of our existence theorem. We discuss some of the specific difficulties in Section \ref{s:intro-uniqueness} below. 
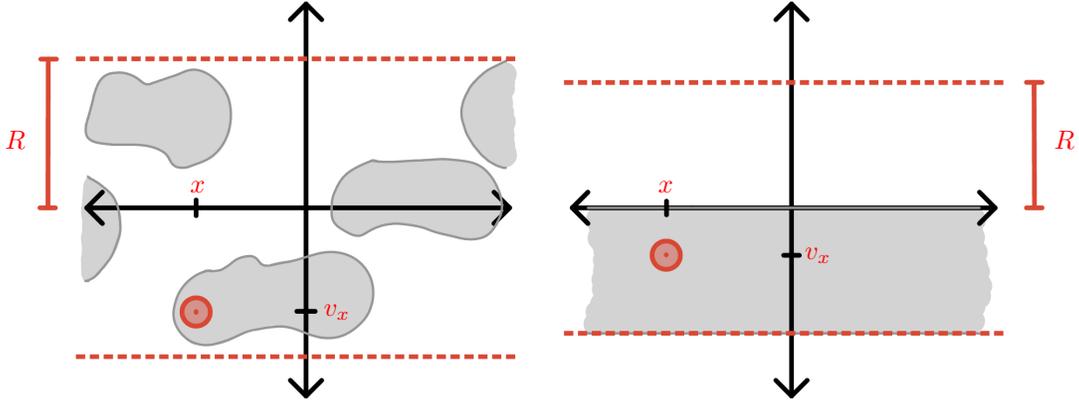
\begin{figure}
	
	\begin{overpic}[scale=.175]
		{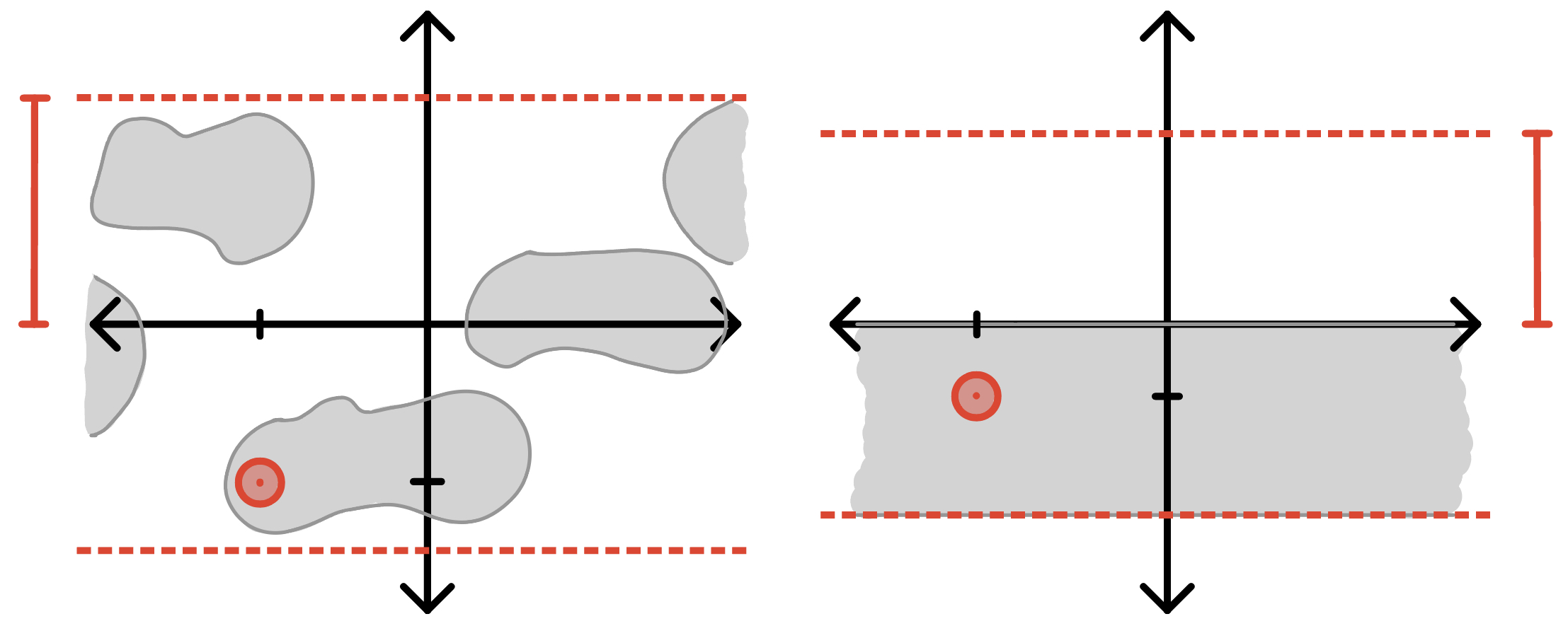}
		\put(-2,25){\color{red} $R$}
		\put(100,25){\color{red} $R$}
		\put(16,20.75){\color{red}$x$}
		\put(29,8.75){\color{red} $v_x$}
		\put(61.5,20.75){\color{red} $x$}
		\put(75.75,14.25){\color{red} $v_x$}
	\end{overpic}
	\caption{A cartoon depicting two $f_{\rm in}$ satisfying the condition~\eqref{e.uniform-lower}.  The set where $f_{\rm in}\geq \delta$ is depicted by the gray shading.  Notice that, for every $x$, there is a $v_x$ between the dashed red lines such that a translation of the red ball ($B_r$) centered at $(x,v_x)$ is within the shaded region.  This would not be the case if, e.g., $f_{\rm in} \equiv 0$ for all $x \in B_1(0)$.}
	\label{f.uniform-lower}
\end{figure}

For this result, we make additional assumptions on $f_{\rm in}$: there are $\delta$, $r$, $R>0$ so that
\begin{equation}\label{e.uniform-lower}
	\text{for each } x\in \R^3,
		\ \text{ there is } \ v_x\in B_R(0)
		\ \text{ such that }\  f_{\rm in} \geq \delta 1_{B_r(x,v_x)}.
\end{equation}
(See Figure \ref{f.uniform-lower}.) This condition is stronger than \eqref{e.mass-core} and rules out vacuum regions in the spatial domain at $t=0$.  We also need to assume that $f_{\rm in}$ is H\"older continuous. Our uniqueness theorem is as follows:

\begin{theorem}\label{t:uniqueness}
Let $\gamma \in (-3,0)$ and $s\in(0,1)$.  For any $\alpha>0$ and for $q>0$ sufficiently large, depending only on $\alpha$, $\gamma$, $s$, and $c_b$, assume that $f_{\rm in}\in L^\infty_q(\R^6)\cap C^\alpha_\ell(\R^6)$, 
and that $f_{\rm in}$ satisfies the lower bound assumption \eqref{e.uniform-lower}. Let $f$ be the classical solution on $[0,T]\times\R^6$ constructed in Theorem \ref{t:existence} with initial data $f_{\rm in}$. 

Then there exists $T_U>0$ depending on $\delta$, $r$, $R$, $\alpha$, $\|f_{\rm in}\|_{L^\infty_q(\R^6)}$, and $\|f_{\rm in}\|_{C^\alpha_\ell(\R^6)}$, such that for any weak solution $g$ in the sense of Theorem \ref{t:weak-solutions} with initial data $f_{\rm in}$, and such that
\[
g\in L^1([0,T_U],L^\infty_q(\R^6)),
\]
the equality $f(t,x,v)=g(t,x,v)$ holds everywhere in $[0,\min(T,T_U)]\times\R^6$.
\end{theorem}
Let us make the following comments on the statement of this theorem:
\begin{itemize}

	\item This uniqueness result holds up to a time $T_U$ that depends on the $C^\alpha$ bound for $f_{\rm in}$, and may be smaller than $T$, the time of existence granted by Theorem \ref{t:existence}. This makes sense because our proof of uniqueness breaks down as $\alpha$ is sent to 0.

	\item The admissible values of the parameter $q$ in Theorem \ref{t:uniqueness} are explicitly computable from our proof.

\end{itemize}

The uniqueness or non-uniqueness of the solutions constructed in Theorem \ref{t:existence}, without any regularity assumption for $f_{\rm in}$, remains an interesting open question.  
For other examples of nonlinear evolution equations where uniqueness is not understood in the same generality as existence, even though the system regularizes instantaneously, we refer to  \cite{kiselev2008blowup, HST2020landau, anceschi2021fokkerplanck}.

\subsection{Application: global existence near equilibrium}

In our main results, we do not assume that our initial data is close to equilibrium. In the case that $f_{\rm in}$ is sufficiently close to a Maxwellian equilibrium state, solutions are known to exist globally in time and converge to equilibrium as $t\to\infty$, and there is a large literature about this regime (see Section \ref{s:prior-existence} below). Although in general the near-equilibrium and far-from-equilibrium regimes seem very different mathematically, we are nevertheless able to prove a new result in the close-to-equilibrium regime as an application of our Theorem \ref{t:existence}:

\begin{corollary}\label{c:global}
Let $\gamma$ and $s$ be as in Theorem \ref{t:existence}, let $M(x,v) = (2\pi)^{-3/2} e^{-|v|^2/2}$, and let $q_0>5$ be fixed. There exists $q_1> q_0$ depending on $q_0$, $\gamma$, $s$, and $c_b$, 
such that for any $f_{\rm in}:\T^3\times\R^3\to [0,\infty)$ such that
\[ 
f_{\rm in} \in L^\infty_{q_1}(\T^3\times\R^3), 
\]
and for any $\eps\in (0,\frac 1 2)$, there exists a $\delta>0$, depending on $\eps$ and $\|f_{\rm in}\|_{L^\infty_{q_1}(\T^3\times\R^3)}$, such that if
\[ 
\|f_{\rm in} - M\|_{L^\infty_{q_0}(\T^3\times\R^3)} < \delta,
\]
then there exists a global classical solution $f:[0,\infty) \times\T^3\times\R^3 \to [0,\infty)$ to the Boltzmann equation \eqref{e.boltzmann}, with 
\[ 
\|f(t) - M\|_{L^\infty_{q_0}(\T^3\times\R^3)} < \eps,
\]
for all $t\geq 0$.
\end{corollary}
By the results of \cite{desvillettes2005global}, this solution converges to $M$ as $t\to\infty$ faster than $t^{-p}$ for any $p>0$ for $q_1$ sufficiently large, depending on $p$.

The proof of Corollary \ref{c:global} is based on a strategy developed by the second named author, jointly with Silvestre \cite{silvestre2022nearequilibrium}. The idea is to combine a short-time existence theorem, the conditional regularity estimates of Imbert-Silvestre \cite{imbert2020smooth}, and the global trend-to-equilibrium result of Desvillettes-Villani \cite{desvillettes2005global}. The result in \cite{silvestre2022nearequilibrium} worked under the assumption $\gamma+2s\geq0$.  
Taking advantage of the generality of the short-time existence theorem in the current paper, Corollary \ref{c:global} improves on \cite{silvestre2022nearequilibrium} in two ways: by including the case $\gamma+2s<0$, and by working with initial data that decays at a finite polynomial rate, rather than at a rate faster than any polynomial. For the regime $\gamma+2s<0$, this seems to be the first result proving global existence near equilibrium for initial data in a zeroth-order space.

\subsection{Related work}

\subsubsection{Prior well-posedness results and comparison}\label{s:prior-existence}

Existence results for the non-cutoff Boltzmann equation fall into the following categories:

\begin{itemize}

\item {\it Spatially homogeneous solutions.} Local well-posedness and smoothing are well understood for the space homogeneous equation, and classical solutions are known to exist globally when $\gamma+2s\geq 0$. We refer to \cite{ukai1984gevrey, desvillettes2004homogeneous, fournier2008uniqueness, desvillettes2009stability, morimoto2009homogeneous, chen2011smoothing,  he2012homogeneous-boltzmann, glangetas2016sharp} and the references therein, as well as \cite{villani1998weak} for so-called $H$-solutions and \cite{lu2012measure, morimoto2016measure} for measure-valued solutions. 

\item {\it Close-to-equilibrium solutions.} 
For close-to-equilibrium solutions, we refer to \cite{gressman2011boltzmann, amuxy2011global, alexandre2012global,  alonso2018polynomial, alonso2020giorgi, herau2020regularization, zhang2020global, duan2021global, silvestre2022nearequilibrium, cao2022moments} and the references therein. The first result for near-equilibrium initial data in a weighted $L^\infty_{x,v}$ space was \cite{alonso2020giorgi}, which applied to $\gamma>0$. This was extended to the case $\gamma+2s\in [0,2]$ in \cite{silvestre2022nearequilibrium}, and we extend it to $\gamma+2s<0$ in our Corollary \ref{c:global}. We also refer to \cite{duan2021global, morimoto2016global, zhang2020global} for other results in low-regularity spaces.  Results that work with polynomially decaying initial data include \cite{herau2020regularization, alonso2018polynomial, alonso2020giorgi, cao2022moments}.

\item {\it Close-to-vacuum solutions.} Recently, global solutions that are close to the vacuum state $f\equiv 0$ have been constructed by Chaturvedi \cite{chaturvedi2021vacuum}, with initial data in a tenth-order Sobolev space with Gaussian weight.

\item {\it Weak solutions.} 
A generalized notion of solution for the non-cutoff Boltzmann equation called {\it renormalized solution with defect measure} was constructed by Alexandre-Villani \cite{alexandre2002longrange}. The uniqueness and regularity of these solutions are not understood, but they exist globally in time, for any initial data $f_{\rm in}$ such that 
\[
	\int_{\R^6} f_{\rm in}(x,v)(1+|v|^2 + |x|^2 + \log f_{\rm in}(x,v)) \dd x \dd v < \infty.
\]
This assumption is weaker than ours in terms of local integrability and $v$-decay, but stronger in terms of $x$-decay. If $f_{\rm in}$ satisfies the assumptions of both \cite{alexandre2002longrange} and our Theorem \ref{t:weak-solutions}, then our weak solutions are, in particular, renormalized solutions with defect measure, as can be seen from the stability theorem \cite[Theorem 2]{alexandre2002longrange} and the fact that our weak solutions are obtained as a limit of classical solutions of \eqref{e.boltzmann}.

\item {\it Short-time solutions.} Early results on local existence in the non-cutoff case were due to the AMUXY group (Alexandre, Morimoto, Ukai, Xu, and Yang) \cite{amuxy2010regularizing, amuxy2011bounded} and required initial data to lie in Sobolev spaces of order 4 with Gaussian velocity weights. 
Later results relaxed the decay assumption by treating initial data with finite polynomial decay, at the cost of increasing the regularity requirement on $f_{\rm in}$. The first result in this direction was from Morimoto-Yang \cite{morimoto2015polynomial}, who worked with $s\in (0,\frac 1 2)$ and $\gamma \in (-\frac 3 2, 0]$ and took $\vv^q f_{\rm in}\in H^6(\T^3\times\R^3)$ with $q> 13$. Next,  the work \cite{HST2020boltzmann} by the current authors assumed $\max\{-3,-\frac 3 2 - 2s\} < \gamma < 0$ and required $\vv^q f_{\rm in} \in H^5(\T^3\times\R^3)$ for some non-explicit $q$. See also \cite{amuxy2011uniqueness} for an earlier uniqueness result in a similar regime as \cite{HST2020boltzmann}. Most recently, Henderson-Wang \cite{henderson2021existence} extended the result of \cite{HST2020boltzmann} to the case $\gamma+2s<-\frac 3 2$. The only prior results that require fewer than 4 derivatives for $f_{\rm in}$ are restricted to the case $s\in (0,\frac 1 2)$: see \cite{amuxy2013mild}, which requires at least two Sobolev derivatives as well as spatial localization, and \cite[Theorem 1.2]{henderson2021existence}, which requires only $\vv^q f_{\rm in}\in C^1(\T^3\times\R^3)$ and $q$ sufficiently large, but uses a specific argument that cannot be generalized to $s>\frac 1 2$. 
Our Theorem \ref{t:existence} represents a significant improvement, in terms of the decay and regularity assumptions on $f_{\rm in}$, and applies for any $\gamma< 0$ and $s\in (0,1)$.

\end{itemize}

\subsubsection{Regularizing effect}\label{s:related}

The regularizing effect of the non-cutoff Boltzmann equation is a major theme and motivation of this work. The first rigorous understanding of this effect came in the 1990s 
with Desvillettes' work on the two-dimensional homogeneous setting \cite{desvillettes1995kac, desvillettes1996homogeneous, desvillettes1997nonradial}, as well as functional estimates for $Q(f,g)$ in Sobolev spaces by various authors \cite{lions1998compactness, alexandre1998entropy, villani1999entropy}, culminating in the sharp entropy dissipation estimate of Alexandre-Desvillettes-Villani-Wennberg  \cite{alexandre2000entropy}. (Much earlier, the idea that $Q(f,g)$ behaves like a fractional differentiation operator in $v$ was understood on a heuristic level by Cercignani \cite{cercignani1969kinetic}.) The key property for many of these estimates is the following functional identity for the collision operator:
\begin{equation}\label{e.entropy-dissipation}
\int_{\R^3} Q(f,f) \log f \dd v = -\frac 1 4\int_{\R^6}\int_{\mathbb S^2} B(v-v_*,\sigma)\left( f' f_*' - f f_*\right) \log \frac{f' f_*'}{ff_*} \dd \sigma \dd v_*\dd v \leq 0.
\end{equation}
This identity implies the entropy $\int_{\R^6} f\log f \dd v \dd x$ is nonincreasing for solutions of \eqref{e.boltzmann}, but even more, it implies a smoothing effect in the $v$ variable, because the quantity on the right---called the {\it entropy dissipation}---turns out to control $\|\sqrt f\|_{H^s_v}$, up to a lower-order correction term. 
In the context of the homogeneous Boltzmann equation, this fractional smoothing effect can be iterated to show solutions are $C^\infty$. 

For the full inhomogeneous equation, the matter is more difficult because the diffusion acts only in velocity, and the smoothing effect of \eqref{e.boltzmann} is therefore hypoelliptic rather than parabolic. Results such as \cite{amuxy2010regularizing, chen2012smoothing} proved that any solution that lies in $H^5_{x,v}(\R^6)$ uniformly in time, decays faster than any polynomial, and satisfies a lower bound on the mass density, is in fact $C^\infty$. More recently, the breakthrough result of Imbert-Silvestre \cite{imbert2020smooth}, which finished off a long program of the two authors and Mouhot \cite{silvestre2016boltzmann, imbert2016weak, imbert2018schauder, imbert2018decay, imbert2020lowerbounds} (see also the survey articles \cite{mouhot2018review,imbert2020review, silvestre2022review}), established  $C^\infty$ estimates for solutions of \eqref{e.boltzmann} that depend only on bounds for the mass, energy, and entropy densities of the solution, as well as (when $\gamma<0$) the polynomial decay rates of the initial data. See also \cite{loher2022quantitative} for a quantitative version of the H\"older estimate of \cite{imbert2016weak}. 
These results do not use entropy dissipation estimates, relying instead on understanding the ellipticity of $Q(f,g)$ as an integro-differential operator. The current article adapts some techniques from  the Imbert-Silvestre program.

\subsubsection{Comparison with Landau equation}\label{s:landau}

The Landau equation is a kinetic model that can be derived from the Boltzmann equation \eqref{e.boltzmann} in the limit as grazing collisions (collisions with $\theta \approx 0$ in \eqref{e.b-bounds}) predominate (see e.g. \cite{desvillettes1992grazing, alexandre2004landau}). This equation reads
\[ 
\partial_t f + v\cdot \nabla_x f = Q_L(f,f),
\]
where the Landau collision operator takes the form
\begin{equation}\label{e.QL}
Q_L(f,g) = \nabla_v\cdot(\bar a^f \nabla_v g) + \bar b^f \cdot \nabla_v g + \bar c^f g,
\end{equation}
and $\bar a^f$, $\bar b^f$, and $\bar c^f$ are defined in terms of velocity integrals of $f$. This is an important model in plasma physics and has also attracted a great deal of interest as an equation with similar mathematical properties to the Boltzmann equation. 

In \cite{HST2020landau}, we proved existence and uniqueness results for the Landau equation that are in a similar spirit to Theorem \ref{t:existence} and Theorem \ref{t:uniqueness} above. While \cite{HST2020landau} provides a helpful outline for the current study, the Boltzmann case turns out to be much more challenging. This is partly because $Q_L(f,g)$ defined in \eqref{e.QL} is a second-order differential operator which is {\it local} in $g$, unlike $Q(f,g)$, which is nonlocal in both $f$ and $g$.   
The local structure of the Landau equation is more amenable to barrier arguments because, letting $f$ be a solution and $g$ be an upper (resp. lower) barrier, one can derive good upper (resp. lower) bounds for $Q_L(f,g)$ at a crossing point between $f$ and $g$, using information about $g$ only at the crossing point. In the Boltzmann case, bounding $Q(f,g)$ is more subtle because one has to take into account the values of both $f$ and $g$ in the entire velocity domain $\R^3$. Since we use barrier arguments extensively in this work, the ``double nonlocality'' of $Q$ is a significant source of difficulty. 

Compared to \cite{HST2020landau}, the current study makes a much less stringent positivity assumption on the initial data: the main result of \cite{HST2020landau} required that no $x$ location in $\R^3$ be too far from a region in which $f_{\rm in}$ is uniformly positive, whereas our Theorem \ref{t:existence} only requires that $f_{\rm in}$ is uniformly positive in one single region. This is due to improvements in our method, rather than differences between the two equations.

We also refer to the well-posedness theorem of \cite{anceschi2021fokkerplanck} for a nonlinear Fokker-Planck equation (studied earlier in 
\cite{imbert2021nonlinear, liao2018fokkerplanck}) that shares some properties with the Boltzmann and Landau equations. With a similar approach as~\cite{HST2020landau}, the authors construct a solution with initial data in $L^\infty_{x,v}(\R^6)$.  
As in \cite{HST2020landau} and the current article, an extra assumption of H\"older continuity is needed to prove uniqueness. 
We note that the authors also prove an interesting estimate on the diffusion asymptotics of the solution.

\subsection{Difficulties and proof ideas}

\subsubsection{Existence}

The prior large-data existence results for \eqref{e.boltzmann} cited above are based on the energy method. To demonstrate some disadvantages of this method, let us integrate \eqref{e.boltzmann} against $\varphi(x) f$ for a compactly supported cutoff $\varphi$, which is needed because $f$ and its derivatives may not decay as $|x|\to \infty$. This gives
\be\label{e.energy}
	\frac 1 2 \frac d {dt} \int_{\R^6} \varphi f^2 \dd x \dd v
		\leq \int_{\R^6}\left[ \varphi f Q(f,f) -\frac 1 2 f^2 v\cdot \nabla_x \varphi \right] \dd v \dd x.
\ee
One would like to bound the right-hand side in terms of $\int \varphi f^2 \dd v \dd x$, but this is not possible with either term. 
First, the collision operator $Q(f,f)$ cannot be controlled using only $L^2_v$-based norms of $f$ due to the kinetic factor $|v-v_*|^\gamma$  (recall $\gamma \in (-3,0)$). 
Instead, an $L^p_v$ bound is needed, with $p>2$ depending on $\gamma$ and $s$. Therefore, to continue with $L^2$-based energy estimates, one must seek bounds on higher derivatives of $f$ in order to use an embedding theorem.  Second, the $Q(f,f)$ integral involves three $f$ terms, so an $L^2$-estimate will not close.  (One might hope to use the fact that, in some sense, $Q(f,\cdot)$ involves an average over $v$ in order to close the estimate.  Unfortunately, $Q$ has no such average in $x$, meaning $L^\infty_x$-regularity of $f$ is required.)  One cannot sidestep this by working in an $L^p$ space with $p \in (2,\infty)$: the analogous integral will involve $p+1$ copies of $f$ and, hence, will not close.
For these reasons, the energy method seems incompatible with working in a zeroth-order space. 


Similarly, the growth of $v\cdot \nabla_x \varphi$ for large $v$ means the second term on the right cannot be controlled by $\int \varphi f^2 \dd v \dd x$. (When $\gamma+2s>0$, there are also terms coming from $Q(f,f)$ that grow as $|v|\to\infty$, leading to a similar issue, even in the spatially periodic case where $\varphi$ is not needed.) One standard way to overcome this issue \cite{amuxy2010regularizing, amuxy2011bounded, amuxy2013mild} is to divide $f$ by a time-dependent Gaussian $e^{(\rho-\kappa t) \vv^2}$, which adds a term proportional to $\vv^2 f$ to the equation, with the correct sign to absorb the terms with growing velocity dependence. However, this requires $f_{\rm in}$ to have velocity decay proportional to a Gaussian. More intricate methods, based on the coercivity properties of $Q(f,f)$, have been found to deal with this velocity growth \cite{morimoto2015polynomial, HST2020boltzmann, henderson2021existence}, but these also require working with polynomial decay of relatively high degree.

Instead of the energy method, we use a barrier argument to propagate decay estimates in $L^\infty_q$ from $t=0$ forward in time, using barriers of the form $g = N e^{\beta t} \vv^{-q}$ with $N, \beta, q>0$. The function $g$ is a valid barrier if $q>3+\gamma+2s$ and if $f$ also decays at a rate proportional to $\vv^{-q}$, which we show via a detailed analysis of $Q(f, g)$ in Lemma \ref{l:Q-polynomial}. This argument gives a closed estimate in the space $L^\infty_q([0,T_q]\times\R^6)$ for some $T_q>0$ depending on $\|f_{\rm in}\|_{L^\infty_q(\R^6)}$.

To understand the regularity of our solutions for positive times, we need to propagate higher decay estimates for $f$, because each step of the regularity bootstrap uses up a certain number of velocity moments. This brings up a subtle difficulty: since our time of existence should depend only on some fixed $L^\infty_{q_0}$ norm of $f_{\rm in}$ with $q_0$ small, we need to propagate higher $L^\infty_q$ norms to a common time interval $[0,T]$ depending only on the norm of $f_{\rm in}$ in $L^\infty_{q_0}$. We note that this is one place where the double nonlocality (see Section \ref{s:landau}) causes issues.  To overcome this, we return to our barrier argument, proceeding more carefully in order to extract a small gain in the exponent $q$, which can then be iterated to bound any $L^\infty_q$ norm on $[0,T]$, provided $\|f_{\rm in}\|_{L^\infty_q}$ is finite (see Lemma \ref{p:upper-bounds}).



Once we have propagated sufficient decay forward in time, 
we would like to apply the global regularity estimates of \cite{imbert2020smooth}. These estimates are an important tool in our study, but applying them to the problem we consider is not straightforward for several reasons. First, the authors of \cite{imbert2020smooth} work under the assumption $\gamma+2s\in [0,2]$, while we treat any $\gamma\in (-3,0)$ and $s\in (0,1)$. Therefore, we need to extend the analysis of \cite{imbert2020smooth} to the case $\gamma+2s< 0$, with suitably modified hypotheses (see Section \ref{s:global-reg}). The change of variables developed in \cite{imbert2020smooth} to pass from local to global regularity estimates is defined in a way that does not generalize well to the case $\gamma+2s< 0$, and the main novelty of our work in Section \ref{s:global-reg} is defining a suitable change of variables for this case. Second, the estimates of \cite{imbert2020smooth} require a uniform-in-$x$ positive lower bound on the mass density $\int_{\R^3} f(t,x,v)\dd v$, but it  does not seem possible to propagate such a bound forward from time zero with current techniques (unlike in the space homogeneous case). Instead, we work with initial data that is pointwise positive in a small ball, and spread this positivity to the whole domain via our result in \cite{HST2020lowerbounds}. This means we need to re-work the regularity estimates of \cite{imbert2020smooth} to depend quantitatively on pointwise lower bounds for $f$ rather than a lower mass density bound. Finally, we need to understand how the regularity of $f$ degenerates as $t\searrow 0$ in the case of irregular initial data, which requires us to revisit some of the arguments in  \cite{imbert2020smooth} to track the dependence on $t$.

Our extension of the global regularity estimates and change of variables of \cite{imbert2020smooth} to the case $\gamma+2s< 0$ may be of independent interest.

\subsubsection{Uniqueness}\label{s:intro-uniqueness}

In this section, we discuss some of the difficulties in proving uniqueness. Given two solutions $f$, $g$, the bilinearity of $Q$ implies, with $h=f-g$,
\begin{equation}\label{e.h-eqn}
 \partial_t h + v\cdot \nabla_x h =  Q(h,f) + Q(g,h).
 \end{equation}
Any standard strategy for proving uniqueness would involve bounding $h$ in some norm, using this equation or its equivalent. For the sake of discussion, we set aside the (nontrivial) difficulties related to velocity growth on the right-hand side of \eqref{e.h-eqn}, as well as the potential lack of decay for large $x$, to focus on a more serious difficulty: that some regularity of $f$ in the $v$ variable is needed to control the term $Q(h,f)$, either of order $2s$ for a pointwise bound or order $s$ for integrals like $\int g Q(h,f) \dd v$. In the context of irregular initial data, regularity estimates for $f$ must degenerate as $t\searrow 0$, but one may still get a good bound for $h$ if this degeneration is slow enough for a Gr\"onwall-style argument. Let us distinguish between two very different regimes:


\medskip

{\it If vacuum regions are present in the initial data}, then the known lower bounds for $f$, which are expected to be sharp, degenerate very quickly in such regions, at a rate like $e^{-c/t}$ (see \cite{HST2020lowerbounds}).  The available regularization mechanisms, such as entropy dissipation or linear De Giorgi estimates, rely on lower bounds for the mass density of $f$, and are therefore useless as $t\searrow 0$. For this reason, uniqueness of solutions in this regime is expected to be very difficult and require completely new ideas, if it even holds.

\medskip

{\it If there are no vacuum regions in the initial data}, the situation appears more hopeful, because we can use our earlier result \cite{HST2020lowerbounds} to obtain positive lower bounds for $f$ and $g$ that are uniform for small times. Because $f$ and $g$ satisfy good lower and upper bounds on some time interval, they enjoy the regularity provided by entropy dissipation on that interval, and one might try to exploit this regularity to prove uniqueness. Let us make a brief digression to explain why this approach does not work: As described in Section \ref{s:related}, one has an {\it a priori} bound on  $\sqrt f$ in $L^2_{t,x}H^s_v([0,T]\times\R^6)$ via the formal identity \eqref{e.entropy-dissipation}, which can be improved to a bound for $f$ itself in the same space, using our $L^\infty_q$ estimates for $f$. The same bounds apply to $g$ and (by the triangle inequality) $h$. Integrating \eqref{e.h-eqn} against $h$ and using coercivity and trilinear estimates for $Q$ that are standard in the literature, one would obtain an estimate of the following form (recall that we are ignoring velocity weights and the possible lack of decay for large $x$): 
\[
	\frac d {dt} \|h(t)\|_{L^2_{x,v}}^2
		\lesssim \int \|h(t,x)\|_{L^2_v} \|f(t,x)\|_{H^s_v} \|h(t,x)\|_{H^s_v} \dd x.
\]
To close this estimate, one would need to bound the right-hand side by a constant times $\|h(t)\|_{L^2_{x,v}}^2$. An $L^\infty_x$-bound on $\|f(t,x)\|_{H^s_v}$ and a bound like $\|h(t,x)\|_{H^s_v} \lesssim  \|h(t,x)\|_{L^2_v}$ would be sufficient to do this, but unfortunately, we only have bounds for $f$ and $h$ in $L^2_x H^s_v$, so it is not at all clear how to close the above argument.



This gap between an $L^2_x$ estimate arising from the formal structure of the equation, and a desired estimate in $L^\infty_x$, is reminiscent of the current state of the global well-posedness problem for \eqref{e.boltzmann}: $L^\infty_x$ bounds for the mass, energy, and entropy densities would be sufficient to extend large-data solutions globally in time \cite{imbert2020smooth}, but the natural conservation laws of the equation only provide bounds in $L^1_x$ for these densities. Bridging this gap is widely considered to be out of reach with current techniques. Based on this apparent similarity, we believe that our assumption of H\"older continuity for $f_{\rm in}$ in Theorem \ref{t:uniqueness} is more than a technicality, and that removing it may be a difficult problem.

Instead of entropy dissipation, one may try to apply the global H\"older estimates of De Giorgi and Schauder type from \cite{imbert2020smooth} to obtain enough regularity to bound $Q(h,f)$ pointwise.  Although these estimates on $[\tau,T]\times\R^6$ are uniform in $x$, they also must degenerate as $\tau\to 0$ since they include the case of irregular initial data. In Proposition \ref{p:nonlin-schauder}, we determine the explicit dependence on $\tau$ when Schauder estimates are applied to $f$: ignoring velocity weights, one has
\begin{equation}\label{e.fake-schauder}
	\|f\|_{C^{2s+\alpha'}_\ell([\tau,T]\times\R^6)}
		\leq C \tau^{-1 + \frac{\alpha-\alpha'}{2s}}
			\|f\|_{C^\alpha_\ell([\tau/2,T]\times\R^6)}^{1+(\alpha+2s)/\alpha'},
\end{equation}
with $\alpha' = \alpha \frac {2s}{1+2s}$. This exponent of $\tau$ is consistent with a gain of regularity of order $2s+\alpha' - \alpha$ on a kinetic cylinder of width $\sim\tau^{\frac 1 {2s}}$ in the time variable (see \eqref{e.cylinder}). By a similar heuristic, the global $C^\alpha_\ell$ estimate (Theorem \ref{t:degiorgi}) on $[\tau/2,T]\times\R^6$ should have a constant proportional to $\tau^{-\frac \alpha {2s}}$. Combining this with \eqref{e.fake-schauder}, a bound for the $C^{2s+\alpha'}_\ell$ norm in terms of $\|f\|_{L^\infty}$ would give an overall $\tau$ dependence of 
\be\label{e.c070501}
\tau^{-1 + \frac{\alpha - \alpha'}{2s} - \frac \alpha {2s}(1+\frac{\alpha+2s}{\alpha'})},
\ee
which is not integrable as $\tau\to 0$. Therefore, this line of argument does not seem feasible without any additional regularity assumptions for $f_{\rm in}$.

Another tempting approach is to work directly with the time shifts of $f$
\[
	g(t) = \frac{f(t+h) - f(t)}{h^{\sfrac\alpha2}}
		\qquad\text{(for fixed $h>0$)}.
\]
Unfortunately, this approach also fails.  As above, we cannot use ``smoothing'' effects of the equation due to time factors like~\eqref{e.c070501}. If we try to naively propagate forward a bound on $g$ by, e.g., barrier arguments, we would first need a (uniform in $h$) bound on the initial data $g(0) = (f(h) - f_{\rm in}) h^{-\sfrac\alpha2}$. These obstacles are the main reason why proving uniqueness is much harder, and why an additional assumption on the initial data seems necessary.

On the other hand, if $f$ were bounded in $C^\alpha _\ell$ uniformly for small times, an estimate of the form \eqref{e.fake-schauder} would be sufficient to derive a time-integrable bound on $Q(h,f)$ in \eqref{e.h-eqn}.  This motivates our extra assumption that $f_{\rm in}$ is H\"older continuous, and the following step-by-step strategy for proving uniqueness:

\begin{enumerate}

\item Prove that the H\"older modulus of $f_{\rm in}$ in $(x,v)$ variables is propagated forward to positive times. To do this, we study the function $g$ defined for $(t,x,v,\chi,\nu) \in \R^{13}$ and $m>0$ by
\begin{equation*}
g(t,x,v,\chi,\nu) = \frac{f(t,x+\chi,v+\nu) - f(t,x,v)}{(|\chi|^2+|\nu|^2)^{\alpha/2}} \vv^m.
\end{equation*}
Bounding $g$ in $L^\infty$ on a short time interval is equivalent to controlling the weighted $\alpha$-H\"older seminorm of $f$. Note that this is the H\"older seminorm with respect to the Euclidean scaling on $\R^6$, not the kinetic scaling that one might expect. This choice is imposed on us by the proof. 

Using \eqref{e.boltzmann}, we derive an equation satisfied by $g$ and use Gr\"onwall's inequality to bound $g$ on a short time interval. This step requires a detailed analysis of the quantity $Q(f,f)(t,x+\chi,v+\nu) - Q(f,f)(t,x,v)$, the repeated use of annular decompositions of the velocity integrals defining $Q$,  and an estimate of the form~\eqref{e.fake-schauder} coming from a carefully scaled version of the Schauder estimates.

This approach to propagating H\"older continuity is inspired by \cite{constantin2015SQG}. 

\item Show that the H\"older regularity for $f$ in $(x,v)$ from the previous step implies H\"older regularity in $t$ as well. This property is clearly false for general functions on $\R^7$, so we must exploit the equation \eqref{e.boltzmann}. The proof is surprisingly intricate and is based on controlling a finite difference in $t$ of $f$ via well-chosen barriers.

\item Using the regularity from the prior two steps, apply Schauder estimates to conclude $C^{2s+\alpha'}_\ell$ regularity for $f$, for some $\alpha'>0$. 

\item Armed with this regularity for $f$, return to \eqref{e.h-eqn} and use (for the only time in this paper) the energy method to bound $h= f-g$ in a weighted $L^2$-norm and establish weak-strong uniqueness. The energy method is chosen because of its compatibility with our notion of weak solution, but one must contend with the lack of decay for large $x$. This step of the proof combines the strategy for $L^2$-estimates developed in \cite{HST2020boltzmann, henderson2021existence} with a spatial localization method that is compatible with the transport term. The particular form of our localizing cutoff function (which depends on both $x$ and $v$) leads to extra difficulties because we cannot deal with the $x$ and $v$ integrations separately.
\end{enumerate}

We should note that this strategy requires working with regularity in all three variables because of the application of Schauder estimates, even though the important ingredient for proving uniqueness is the regularity in $v$.

\subsection{Open problems}

\subsubsection{Relaxing the H\"older continuity assumption in Theorem \ref{t:uniqueness}}
As discussed in Section \ref{s:intro-uniqueness}, proving uniqueness without any regularity assumptions on $f_{\rm in}$ may be a difficult problem. In the Landau case, a recent result  \cite{henderson2022schauder} by the first named author, jointly with W.~Wang, derived a uniqueness theorem that requires $f_{\rm in}$ to be H\"older continuous in $x$ but only logarithmically H\"older in $v$, via Schauder estimates with time-irregular coefficients. See \cite{biagi2022schauder} for a similar Schauder estimate. 
It is likely that an analogous improvement is available for the Boltzmann equation via a refinement of the Schauder estimates in \cite{imbert2018schauder}, though this would be nontrivial to prove. Even with such an improvement, H\"older regularity in $x$ would be needed for the initial data.

\subsubsection{The case $\gamma > 0$} 

In the case $\gamma > 0$, the analysis of the Boltzmann equation is somewhat different because the kinetic term $|v-v_*|^\gamma$ in $Q(f,g)$ becomes a growing weight instead of a singularity.  Our argument in this paper for local existence uses the assumption $\gamma\leq 0$ crucially, and a different argument would be required for $\gamma > 0$. We have proven some of our intermediate results in this paper without the restriction $\gamma\leq 0$, with a mind to eventually filling this gap.

\subsubsection{Classical solutions without a locally uniform lower bound}

Our construction of classical solutions requires a locally uniform positive lower bound at time zero (condition \eqref{e.mass-core} in Theorem \ref{t:existence}). This is automatically true if the initial data is continuous and not identically zero, but our initial data may be discontinuous, so \eqref{e.mass-core} is an extra assumption we have to make. In either limits $\delta\searrow 0$ or $r\searrow 0$, we lose all quantitative control on the pointwise regularity of our solutions, and we can only recover a weak solution in the sense of Theorem \ref{t:weak-solutions}. On the other hand, if $f_{\rm in}$ is identically 0, then the solution is also identically zero for positive times, and is therefore perfectly smooth. This leaves open the question of regularity for solutions with initial data $f_{\rm in}\in L^\infty_q(\R^6)$ that is not identically zero but nowhere uniformly positive. 

\subsubsection{Decay estimates and continuation}

Continuation criteria for \eqref{e.boltzmann} are highly relevant because they represent partial progress toward the outstanding open problem of global existence of non-perturbative solutions. As with any short-time existence result, our Theorem \ref{t:existence} implies a continuation criterion: solutions can be extended past any time $T$ such that $\|f(t)\|_{L^\infty_q(\R^6)}$ remains finite for $t\in[0,T]$, for some $q>3+2s$. (By \cite{HST2020lowerbounds}, the lower bound condition \eqref{e.mass-core} is automatically satisfied for any positive time $T$, with constants depending on $T$, as long as it holds at time zero. Note that the time of existence granted by Theorem \ref{t:existence} does not depend quantitatively on $\delta$, $r$, or $v_m$.)

On the other hand, the continuation criterion of \cite{HST2020lowerbounds} (which combined the lower bounds of \cite{HST2020lowerbounds} with the continuation criterion of \cite{imbert2020smooth}) states that solutions can be continued as long as $\|f(t)\|_{L^\infty_x (L^1_2)_v(\R^6)}$ remains finite. The continuation criterion of \cite{HST2020lowerbounds} only applies to solutions that are smooth, rapidly decaying, and spatially periodic, and applies only when $\gamma+2s\in [0,2]$. Ideas related to the decay analysis in the current paper could likely strengthen \cite{HST2020lowerbounds} by enlarging the class of solutions and ranges of $(\gamma,s)$ that can be handled, and possibly by replacing  $L^\infty_x (L^1_2)_v$ with a weaker $L^\infty_x (L^1_q)_v$ norm. 
We plan to explore this question in a future article.

\subsection{Notation}\label{s.notation}
For any $\lambda\in \R$, we write $\lambda_+ = \max\{\lambda,0\}$ and $\lambda_- = \max\{-\lambda,0\}$.  We call a constant \emph{universal} if it depends only on $\gamma$, $s$, and the constant $c_b$ in \eqref{e.b-bounds}. Inside of proofs, to keep the notation clean, we often write $A \lesssim B$ to mean $A\leq CB$ for a constant $C>0$ depending on $\gamma$, $s$, $c_b$, and the quantities in the statement of the lemma or theorem being proven. We also write $A\approx B$ when $A\lesssim B$ and $B\lesssim A$.  Occasionally, we use the notation $z= (t,x,v)$.  When $z$ is accompanied by a subscript, so are the corresponding $(t,x,v)$; e.g., $z_0 = (t_0,x_0,v_0)$.

Throughout the manuscript, it is always assumed that $\gamma < 0$ unless otherwise indicated (some results apply to the case $\gamma \in [0,1]$ as well).

We say that a solution to~\eqref{e.boltzmann} is classical if $(\partial_t + v\cdot\nabla_x) f$ and $Q(f,f)$ are continuous and~\eqref{e.boltzmann} holds pointwise.

Throughout the paper, we work with various integro-differential operators.  Given a kernel $K$, we denote
\be\label{e.c062801}
	(\cL_K g)(v)
		= \int_{\R^3} (g(v') - g(v)) K(v,v') \dd v'.
\ee
When the kernel $K$ is obvious from context, we omit it notationally and simply write $\cL$.

\subsection{Outline of the paper} In Section \ref{s:prelim}, we recall and slightly extend some results from the literature that are needed for our study. Section \ref{s:global-reg} extends the change of variables and global regularity estimates of \cite{imbert2020smooth} to the case $\gamma+2s<0$. Section \ref{s:existence} is devoted to the proof of existence. Section \ref{s:time-reg} addresses the extension of H\"older regularity from $(x,v)$ variables to the $t$ variable. Section \ref{s:holder-xv}  propagates a H\"older modulus from $t=0$ to positive times, and Section \ref{s:uniqueness} finishes the proof of uniqueness. Section \ref{s:global} proves existence of global solutions near equilibrium. Appendix \ref{s:cov-appendix} proves the key properties of the change of variables defined in Section \ref{s:global-reg}, and Appendix \ref{s:lemmas} contains some technical lemmas.

\section{Preliminaries}\label{s:prelim}

\subsection{Kinetic H\"older spaces}\label{s:holder-spaces}

To study the regularity properties of the Boltzmann equation, we use the kinetic H\"older spaces from \cite{imbert2020smooth, imbert2018schauder}, which we briefly recall now. 

First, let us recall two transformations that are well-adapted to the symmetries of linear kinetic equations with velocity diffusion of order $2s$. For $z_1 = (t_1,x_1,v_1)$ and $z = (t,x,v)$ points of $\R^7$, define the Lie product
\[ z_1 \circ z = ( t_1 + t, x_1 + x + tv_1, v_1 + v).\]
and the dilation
\be\label{e.dilation-def}
	\delta_r(z) = (r^{2s}t, r^{1+2s}x, rv), \quad r>0.
\ee
Next, define the distance
\begin{equation}\label{e.dl}
d_\ell(z_1,z_2) := \min_{w\in \R^3} \max\{ |t_1-t_2|^{\frac 1 {2s}}, |x_1 - x_2 - (t_1-t_2) w|^{\frac 1 {1+2s}}, |v_1 - w|, |v_2 - w|\}.
\end{equation}
In fact, $d_\ell$ does not satisfy the triangle inequality if $s< 1/2$; see \cite{imbert2018schauder}. This fact causes no issues in our analysis, and we refer to $d_\ell$ as a distance regardless. 
This distance is invariant under left translations (hence the $\ell$ subscript, which stands for left-invariant) and dilations: for any $z_1, z_2, \xi\in \R^7$ and $r>0$,
\begin{align}
d_\ell(\xi\circ z_1, \xi\circ z_2) &= d_\ell(z_1,z_2).\label{e.left}\\
d_\ell(\delta_r(z_1), \delta_r(z_2)) &= rd_\ell(z_1,z_2),\label{e.dilation}
\end{align}
The distance $d_\ell$ is not invariant under right translations. However, for right translations in the velocity variable, one has the useful property
\begin{equation}\label{e.right-translation}
\begin{split}
d_\ell(z_1\circ(0,0,w), z_2\circ(0,0,w)) &\leq d_\ell(z_1,z_2) + |t_1-t_2|^{1/(1+2s)} |w|^{1/(1+2s)}\\
&\leq d_\ell(z_1,z_2) + d_\ell(z_1,z_2)^{2s/(1+2s)} |w|^{1/(1+2s)}
\end{split}
\end{equation}

We define the kinetic cylinders in a way that respects the transformations \eqref{e.left}, \eqref{e.dilation}:
\begin{equation}\label{e.cylinder}
	Q_r(z_0)
		= \{z=(t,x,v) \in \R^7 :t_0-r^{2s}< t\leq t_0, |x-x_0-(t-t_0)v_0|<r^{1+2s}, |v-v_0| < r\}.
\end{equation}
We often write $Q_r = Q_r(0)$. Note that $Q_r = \delta_r(Q_1)$, and $Q_r(z_0) = z_0 \circ Q_r$.

The kinetic H\"older spaces are defined in terms of approximation by polynomials. For any monomial $m$ in the variables $t,x,v$ of the form
\[ m(t,x,v) = c t^{\alpha_0} x_1^{\alpha_1} x_2^{\alpha_2} x_3^{\alpha_3} v_1^{\alpha_4} v_2^{\alpha_5} v_3^{\alpha_6},\]
with $c\neq 0$, we define the kinetic degree as 
\[\mbox{deg}_k m = 2s\alpha_0 + (1+2s)(\alpha_1 + \alpha_2 + \alpha_3) + \alpha_4 + \alpha_5 + \alpha_6.\]
This definition is compatible with the scaling $(t,x,v) \mapsto \delta_r(t,x,v)$. For any nonzero polynomial $p(t,x,v)$, we define its kinetic degree as the maximum of $\mbox{deg}_k$ over all monomial terms in $p$.

Now we are ready to define the kinetic H\"older spaces:

\begin{definition}
Given any $\alpha>0$ and any open set $D\subset \R^7$, a continuous function $f:D\to \R$ is \emph{$\alpha$-H\"older continuous} at $z_0\in D$ if there exists a polynomial $p(t,x,v)$ with $\mbox{deg}_k(p)< \alpha$, and 
\begin{equation}\label{e.f-holder}
 |f(z) - p(z)| \leq C d_\ell(z,z_0)^\alpha, \quad z\in D.
 \end{equation}
We say $f\in C^\alpha_\ell(D)$ if the inequality \eqref{e.f-holder} holds at all points of $D$. The semi-norm $[f]_{C^\alpha_\ell(D)}$ is the smallest value of the constant $C$ such that \eqref{e.f-holder} holds for all $z,z_0\in D$ (with the polynomial $p$ depending on $z_0$).

The norm $\|f\|_{C^\alpha_\ell(D)}$ is defined as $\|f\|_{L^\infty(D)} + [f]_{C^\alpha_\ell(D)}$.
\end{definition}

For functions defined on open subsets $D\subset \R^6$, the seminorm $[f]_{C^\alpha_{\ell,x,v}(D)}$ can be defined similarly as the smallest constant $C>0$ such that for every $(x_0,v_0)\in D$, there is a polynomial $p(x,v)$ with $\mbox{deg}_k(p)< \alpha$, such that
\[ 
|f(x,v) - p(x,v)|\leq Cd_\ell((0,x_0,v_0), (0,x,v))^\alpha.
\]

We also define the global kinetic H\"older spaces with polynomial weights:
\begin{definition}\label{d:C-alpha-q}
Given $\alpha, q>0$ and $0<\tau<T$, we define the weighted semi-norm
\[ [f]_{C^\alpha_{\ell,q}([\tau,T]\times\T^6)} := \sup\left\{ (1+|v|)^q [f]_{C^\alpha_\ell(Q_r(z))} : r\in (0,1] \text{ and } Q_r(z) \subset [\tau,T]\times\R^6\right\}.\]
We say $f\in C^\alpha_{\ell,q}([\tau,T]\times\R^6)$ if the norm
\[ \|f\|_{C^\alpha_{\ell,q}([\tau,T]\times\R^6)} = \|f\|_{L^\infty_q([\tau,T]\times\R^6)} + [f]_{C^\alpha_{\ell,q}([\tau,T]\times\R^6)}\]
is finite.
\end{definition}

\subsection{Well-posedness for regular initial data}

As part of our existence proof, we need to construct solutions corresponding to smooth, rapidly decaying approximations of our initial data. For this, we use the following proposition, which combines 
two short-time existence results from the literature. 
We state here a non-sharp result with assumptions that are uniform in $\gamma$ and $s$ for the sake of brevity (and because we do not need the sharp version).

\begin{proposition}\label{p:prior-existence}
Let $\gamma \in (-3,0)$ and $s\in (0,1)$. 
Let $\T_M^3$ be the 3-dimensional torus of side length $M>0$.

For any $k\geq 6$, there exists $n_0,p_0>0$ depending on universal constants and $k$, such that for any initial data $f_{\rm in}\geq 0$ defined for $(x,v)\in \T_M^3\times \R^3$ with $f_{\rm in} \in H^k_{n}\cap L^\infty_{p}(\T^3_M\times \R^3)$ with $n\geq n_0$ and $p\geq p_0$, there exists a unique solution $f\geq 0$ to \eqref{e.boltzmann} in $C^0([0,T], H^k_{n}\cap L^\infty_{p}( \T_M^3\times \R^3))$ for some $T>0$ depending on $\|f_{\rm in} \|_{H^k_{n}} + \|f_{\rm in}\|_{L^\infty_{p}}$, with $f(0,x,v) = f_{\rm in}(x,v)$.
\end{proposition}
The proofs for the case $M=1$ can found in the following works: for any $s\in (0,1)$ and $\max\{-3,-3/2-2s\} < \gamma < 0$, see \cite{HST2020boltzmann}. For $s\in (0,1)$ and $\gamma \in (-3, -2s)$, see \cite{henderson2021existence}. 

To extend this result to the case of general $M>0$, we rescale to the torus $\T^3$ of side length 1 by defining
\[
\tilde f_{\rm in} := M^{\gamma+3} f_{\rm in}^\eps (M x, M v), \quad x\in \mathbb T^3, v\in \R^3.
\]
The result for the $M=1$ case gives us a solution $\tilde f$ on $[0,T]\times \T^3\times\R^3$, and to scale back to the torus of size $M$, we define
\[
f^\eps(t,x,v) := \frac 1 {M^{\gamma+3}}\tilde f^\eps\left( t, \frac x {M}, \frac v {M}\right), \quad t\in [0,T_\eps], x\in \TM, v\in \R^3.
\]
By a direct calculation, $f$ solves the Boltzmann equation \eqref{e.boltzmann}, with initial data $f_{\rm in}$. The function $f$ lies in the same regularity spaces as $\tilde f$.

\subsection{Carleman representation}\label{s:carleman}

The collision operator $Q(f,g)$ defined in \eqref{e.Q} can be written as a sum of two terms $Q = Q_{\rm s} + Q_{\rm ns}$, where the first (``\textbf{\textit{s}}ingular'') term $Q_{\rm s}$ acts as a nonlocal diffusion operator of order $2s$. The second (``\textbf{\textit{n}}on\textbf{\textit{s}}ingular'') term $Q_{\rm ns}$ is a lower-order convolution term. By adding and subtracting $f(v_*')g(v)$ inside the integral in \eqref{e.Q}, one has
\begin{equation}\label{e.c062901}
\begin{split}
	&Q_{\rm s}(f,g) = \int_{\R^3} \int_{{\mathbb S}^2} (g(v')-g(v)) f(v_*') B(|v-v_*|,\sigma)\dd \sigma \dd v_*
		\qquad\text{and}\\
	&Q_{\rm ns}(f,g) = g(v) \int_{\R^3} \int_{{\mathbb S}^2} (f(v_*')-f(v_*)) B(|v-v_*|,\sigma) \dd \sigma \dd v_*.
\end{split}
\end{equation}

It can be shown \cite{alexandre2000noncutoff, silvestre2016boltzmann} that $Q_{\rm s}(f,\cdot)$ is equal to an integro-differential operator with kernel depending on $f$:
\begin{lemma}{\cite[Section 4]{silvestre2016boltzmann}}\label{l:Q1}
The term $Q_{\rm s}(f,g)$ can be written
\begin{equation}\label{e.Q1-new}
Q_{\rm s}(f,g) = \int_{\R^3} (g(v')-g(v)) K_f (v, v') \dd v',
\end{equation}
with kernel  
\begin{equation}\label{e.kernel}
K_f(v,v') = \frac {1} {|v'-v|^{3+2s} }\int_{(v'-v)^\perp} f(v+w) |w|^{\gamma+2s+1} \tilde b(\cos\theta) \dd w,
\end{equation}
where $\tilde b$ is uniformly positive and bounded.
\end{lemma}

Above, we have used the shorthand $(v-v')^\perp$ to mean $\{w: w\cdot(v-v') = 0 \}$.

For the term $Q_{\rm ns}$, we have the following formula, which is related to the Cancellation Lemma of \cite{alexandre2000entropy}:

\begin{lemma}\label{l:Q2}
The term $Q_{\rm ns}(f,g)$ can be written
\[ 
\Qns(f,g) = Cg(v)\int_{\R^3} f(v+w) |w|^\gamma \dd w,
\]
for a constant $C>0$ depending only on the bounds \eqref{e.b-bounds} for the collision cross-section $b$.
\end{lemma}

\subsection{Self-generating lower bounds}

The main result of \cite{HST2020lowerbounds} states that if $f_{\rm in}$ is uniformly positive in some ball in $(x,v)$ space, this positivity is spread instantly to the entire domain:   
\begin{theorem}{\cite[Theorem 1.2]{HST2020lowerbounds}}\label{t:lower-bounds}
Let $\gamma \in (-3,1)$ and $s\in (0,1)$. Suppose that $f$ is a classical solution ($C^1$ in $t,x$ and $C^2$ in $v$) of \eqref{e.boltzmann} on $[0,T]\times \R^6$, with initial data $f(0,x,v)\geq 0$ satisfying the lower bound \eqref{e.mass-core}, i.e.
\begin{equation*}
 f(0,x,v) \geq \delta, \quad (x,v) \in B_r(x_m,v_m),
 \end{equation*}
for some $x_m,v_m\in \R^3$ and $\delta, r>0$. 
Assume that $f$ satisfies
\begin{equation}\label{e.hydro-general}
\begin{split}
&\sup_{t\in [0,T],x\in \R^3} \int_{\R^3} \vv^{(\gamma+2s)_+} f(t,x,v) \dd v \leq K_0, \qquad \text{ and}\\
&\sup_{t\in [0,T],x\in \R^3} \| f(t,x,\cdot)\|_{L^p(\R^3)} \leq P_0
	\quad \text{ for some } p>\frac{3}{3+\gamma+2s} \quad (\mbox{only if } \gamma + 2s < 0).
\end{split}
\end{equation}
We emphasize that the first inequality in~\eqref{e.hydro-general} is assumed in all cases, while the second one is additionally assumed only in the case $\gamma + 2s<0$. Then 
\[ f(t,x,v) \geq \mu(t,x) e^{-\eta(t,x)|v|^2}, \quad (t,x,v) \in (0,T]\times\R^6,\]
where $\mu(t,x)$ and $\eta(t,x)$ are uniformly positive and bounded on any compact subset of $(0,T]\times \R^3$, and depend only on $t$, $T$, $|x-x_m|$, $\delta$, $r$, $v_m$, $K_0$, and $P_0$.

Furthermore, near the point $x_m$, the lower bounds are uniform up to time zero:
\begin{equation}\label{e.small-t-lower}
f(t,x,v) \geq \mu, \quad (t,x,v) \in (0,T]\times B_{r/2}(x_m,v_m),
\end{equation}
for $\mu>0$ depending on $\delta$, $r$, $v_m$, $K_0$, and $P_0$.
\end{theorem}

 As stated in \cite{HST2020lowerbounds}, this theorem requires an upper bound on the energy density $\int_{\R^3} |v|^2 f(t,x,v) \dd v$. However, it is clear from the proof that a bound on the $\gamma+2s$ moment is sufficient. More specifically, the only purpose of the energy density bound is to estimate $Q_{\rm s}$ from above via Lemma \ref{l:C2Linfty} below, 
and a bound for $\int_{\R^3} \vv^{\gamma+2s} f\dd v$ suffices to estimate the convolution $f\ast|v|^{\gamma+2s}$ in Lemma \ref{l:C2Linfty}.

We should also note that \eqref{e.small-t-lower} is not stated as part of the main result of \cite{HST2020lowerbounds}, but follows immediately from \cite[Lemma 3.1 and Proposition 3.3]{HST2020lowerbounds}.

The following lemma gives a cone of nondegeneracy for the collision kernel $K_f$. When combined with the previous theorem, it provides coercivity estimates for $Q_s(f,\cdot)$ that depend only on the initial data and the quantities in \eqref{e.hydro-general}.

\begin{lemma}{\cite[Lemma 4.1]{HST2020lowerbounds}}\label{l:cone}
	Let $f:\R^3\to \R$ be a nonnegative function with $f(v) \geq \delta 1_{B_r(v_m)}$ for some $\delta, r > 0$ and $v_m\in \R^3$. There exist constants $\lambda, \mu, C > 0$ (depending on $\delta$, $r$, and $|v_m|$) such that for each $v\in \R^3$, there is a symmetric subset of the unit sphere $A(v)\subset \mathbb S^2$ such that:
	\begin{itemize}
		\item $|A(v)|_{\mathcal H^2}\geq \mu (1+|v|)^{-1}$.	
		where $|\cdot|_{\mathcal H^2}$ is the $2$-dimensional Hausdorff measure.
		\item For all $\sigma \in A(v)$, $|\sigma\cdot v|\leq C$.
		\item Whenever $(v-v')/|v-v'| \in A(v)$, 
			\[
				K_f(v,v') \geq \lambda (1+|v|)^{1+\gamma+2s} |v'-v|^{-3-2s}.
			\]
	\end{itemize}
\end{lemma}

\subsection{Local regularity estimates}\label{s.local_reg}

We recall the local regularity estimates of \cite{imbert2016weak} and \cite{imbert2018schauder} for linear kinetic equations of the following type:
\begin{equation}\label{e.linear-kinetic}
\partial_t f + v\cdot \nabla_x f = \int_{\R^3} (f(t,x,v') - f(t,x,v))K(t,x,v,v') \dd v' + h,
\end{equation}
where the kernel $K$ satisfies suitable ellipticity assumptions. First, we have a De Giorgi-type estimate that gives H\"older continuity of solutions:

\begin{theorem}[\cite{imbert2016weak}]\label{t:degiorgi}
Let $K: (-1,0]\times B_1 \times B_2 \times \R^3 \to \R_+$ be a kernel satisfying the following ellipticity conditions, uniformly in $t$ and $x$, for some $\lambda, \Lambda > 0$:
\begin{flalign}
&\text{For all } v\in B_2, r>0,\quad \inf_{|e|=1}\int_{B_r(v)} ((v'-v)\cdot e)^2_+ K(t,x,v,v') \dd v' \geq \lambda r^{2-2s},\quad \text{(if $s< 1/2$)},\label{e.coercivity1}
\end{flalign}
\begin{flalign}
&\text{For any } f(v) \text{ supported in } B_2,  \label{e.coercivity2}&\\
&\phantom{ \text{For any } f(v) } \iint_{B_2\times\R^3} f(v) (f(v) - f(v')) K(t,x,v,v') \dd v' \dd v \geq \lambda \|f\|_{\dot H^s(\R^3)}^2 - \Lambda \|f\|_{L^2(\R^3)}^2,\nonumber
\end{flalign}
\begin{flalign}
&\text{For all }   v\in B_2 , r>0, \quad \int_{\R^3\setminus B_r(v)} K(t,x,v,v') \dd v' \leq \Lambda r^{-2s},&\label{e.upper1}
\end{flalign}
\begin{flalign}
&\text{For all }   v'\in B_2 , r>0, \quad \int_{\R^3\setminus B_r(v')} K(t,x,v,v') \dd v \leq \Lambda r^{-2s},&\label{e.upper2}
\end{flalign}
\begin{flalign}
&\text{For all } v \in B_{7/4}, \quad \left| \mbox{p.v.} \int_{B_{1/4}(v)}(K(t,x,v,v') - K(t,x,v',v)) \dd v'\right| \leq \Lambda,&\label{e.cancellation1}
\end{flalign}
\begin{flalign}
&\text{For all } r\in [0,1/4] \text{ and } v\in B_{7/4}, \label{e.cancellation2}&\\
&\phantom{\text{For }}\left| \mbox{p.v.} \int_{B_r(v)} (K(t,x,v,v') - K(t,x,v',v)) (v'-v) \dd v'\right| \leq \Lambda (1+r^{1-2s}), \quad \text{(if $s\geq 1/2$)}.\nonumber
\end{flalign}
Let $f:(-1,0]\times B_1\times \R^3\to\R$ be a bounded function that is a solution of \eqref{e.linear-kinetic} in $Q_1$, for some bounded function $h$. Then $f$ is H\"older continuous in $Q_{1/2}$, and 
\[ \|f\|_{C^\alpha_\ell(Q_{1/2})} \leq C\left( \|f\|_{L^\infty((-1,0]\times B_1\times \R^3)} + \|h\|_{L^\infty(Q_1)}\right).\]
The constants $C>0$ and $\alpha \in (0,1)$ depend only on $\lambda$ and $\Lambda$.
\end{theorem}

Next, we recall Schauder-type estimates for linear kinetic integro-differential equations of the form \eqref{e.linear-kinetic}. 
As in \cite{imbert2018schauder}, the kernel is assumed to be elliptic in the sense of the following definition:
\begin{definition}[Ellipticity class]\label{d:ellipticity}
Given $s\in (0,1)$ and $0< \lambda < \Lambda$, a kernel $K:\R^3\setminus\{0\}\to \R_+$ lies in the ellipticity class of order $2s$ if
\begin{itemize}
\item $K(w) = K(-w)$.
\item For all $r>0$,
\begin{equation}\label{e.upper-bound-schauder}
 \int_{B_r} |w|^2 K(w) \dd w \leq \Lambda r^{2-2s}.
 \end{equation}
\item For any $R>0$ and $\varphi \in C^2(B_R)$,
\begin{equation}\label{e.coercivity3}
 \iint_{B_R\times B_R} |\varphi(v) - \varphi(v')|^2 K(v'-v) \dd v' \dd v \geq \lambda \iint_{B_{R/2}\times B_{R/2}} |\varphi(v) - \varphi(v')|^2 |v'- v|^{-3-2s} \dd v' \dd v.
 \end{equation}
\item If $s<1/2$, assume in addition that for each $r>0$,
\begin{equation}\label{e.coercivity4}
 \inf_{|e|=1} \int_{B_r} (w\cdot e)^2_+ K(w) \dd w \geq \lambda r^{2-2s}.
 \end{equation}
\end{itemize}
\end{definition}
For technical convenience, we quote the scaled form of the Schauder estimate on cylinders $Q_{2r}$ with $r>0$, as in \cite[Theorem 4.5]{imbert2020smooth}:
\begin{theorem}[\cite{imbert2018schauder, imbert2020smooth}]\label{t:schauder}
Let $0 < \alpha < \min(1,2s)$, and let $\alpha' = \frac {2s}{1+2s}\alpha$. Let $f:(-(2r)^{2s},0]\times B_{(2r)^{1+2s}} \times \R^3\to \R$ be a solution of the linear equation \eqref{e.linear-kinetic} in $Q_{2r}$ for some bounded function $h$ and some integral kernel $K_z(w) = K(t,x,v,v+w): (-(2r)^{2s}, 0]\times B_{(2r)^{1+2s}}\times \R^3\times\R^3\to [0,\infty)$ satisfying, for each $t$, $x$, and $v$, the ellipticity assumptions of Definition \ref{d:ellipticity} for uniform constants $0< \lambda< \Lambda$, as well as the H\"older continuity assumption
\begin{equation}\label{e.kernel-holder}
 \int_{B_\rho} |K_{z_1}(w) - K_{z_2}(w)| |w|^2 \dd w \leq A_0 \rho^{2-2s} d_\ell(z_1,z_2)^{\alpha'}, \quad z_1,z_2\in Q_{2r}, \rho>0,
\end{equation}
for some $\alpha>0$. If $f\in C^\alpha_\ell((-(2r)^{2s},0]\times B_{(2r)^{1+2s}} \times \R^3)$ and $h\in C^\alpha_\ell(Q_{2r})$, then 
\[\begin{split}
 \|f\|_{C^{2s+\alpha'}_\ell(Q_{r})} &\leq C\left(\max\left(r^{-2s-\alpha'+\alpha}, A_0^{(2s+\alpha'-\alpha)/\alpha'} \right)[f]_{C^\alpha_\ell([-(2r)^{2s},0]\times B_{(2r)^{1+2s}}\times\R^3)}\right.\\
&\qquad \left. + [h]_{C^{\alpha'}_\ell(Q_{2r})} + \max(r^{-\alpha'}, A_0)\|h\|_{L^\infty(Q_{2r})}\right).
\end{split}\]
The constant $C$ depends on $s$, $\lambda$, and $\Lambda$.
\end{theorem}

\begin{remark}\label{r:ellipticity}
The local estimates of Theorems \ref{t:degiorgi} and \ref{t:schauder} impose a number of conditions on the integral kernel $K$. When the kernel is defined in terms of a function $f$ according to the formula for $K_f$ from \eqref{e.kernel}, one must place appropriate conditions on $f$ so that the kernel $K_f$ satisfies all the hypotheses of these two theorems.

Regarding the coercivity conditions \eqref{e.coercivity1}, \eqref{e.coercivity2}, \eqref{e.coercivity3}, and \eqref{e.coercivity4}, it is understood in the literature (see \cite{imbert2016weak, chaker2020coercivity, imbert2020smooth}) that all of these conditions follow from the existence of a cone of nondegeneracy as in Lemma \ref{l:cone}. 

The upper bound conditions \eqref{e.upper1} and \eqref{e.upper2}  from Theorem \ref{t:degiorgi} hold for $K_f$ (locally in $v$) whenever the convolution $f\ast|\cdot|^{\gamma+2s} = \int_{\R^3} f(v+w)|w|^{\gamma+2s} \dd w$ is bounded. This is shown in \cite[Lemmas 3.4 and 3.5]{imbert2016weak}.  The cancellation conditions \eqref{e.cancellation1} and \eqref{e.cancellation2} hold whenever the convolutions $f\ast |\cdot|^\gamma$ and $f\ast |\cdot|^{\gamma+1}$ are bounded, from \cite[Lemmas 3.6 and 3.7]{imbert2016weak}. In particular, these conditions all hold whenever $f\in L^\infty_q(\R^3)$ for $q>3+\gamma+2s$. We emphasize that these four lemmas from \cite{imbert2016weak} are proven for any $\gamma$ and $s$ such that $\gamma+2s\leq 2$, including in the case $\gamma+2s<0$.

From Lemma \ref{l:K-upper-bound-2}, we see that the upper bound \eqref{e.upper-bound-schauder} is also satisfied whenever the convolution $f\ast |\cdot|^{3+\gamma+2s}$ is bounded. 

To sum up, this discussion shows that the kernel $K_f$ defined by \eqref{e.kernel} satisfies the hypotheses of Theorems \ref{t:degiorgi} and \ref{t:schauder} whenever $f\in L^\infty_q(\R^3)$ with $q>3+\gamma+2s$ and $f$ satisfies a pointwise lower bound condition as in Lemma \ref{l:cone}, with constants depending on $|v|$, $\|f\|_{L^\infty_q(\R^3)}$, $q$, $\delta$, $r$, and $|v_m|$.

In general, these estimates for $K_f$ degenerate as $|v|\to \infty$, which means $K_f$ is uniformly elliptic on any fixed bounded domain in velocity space, but not uniformly elliptic globally. The change of variables $\mathcal T_0$, described in \cite{imbert2020smooth} and Section \ref{s:global-reg} below, addresses this difficulty.
\end{remark}

\subsection{Estimates for the collision operator}\label{s.collision_operator}

First, we have an integral estimate on annuli for the kernel $K_f$ defined in \eqref{e.kernel}:

\begin{lemma}{\cite[Lemma 4.3]{silvestre2016boltzmann}}\label{l:K-upper-bound}
For any $r>0$, 
\[
	\int_{B_{2r}\setminus B_r} K_f(v,v+w) \dd w \leq C \left( \int_{\R^3} f(v+w)|w|^{\gamma+2s} \dd w\right) r^{-2s}.
\]
\end{lemma}
The following two closely related estimates can be proven by writing the integral over $B_r$ (respectively, $B_r^c$) as a sum of integrals over $B_{r2^{-n}}\setminus B_{r2^{-n-1}}$ for $n=0,1,2,\ldots$ (respectively $n=-1,-2,\ldots$) and applying Lemma \ref{l:K-upper-bound} for each $n$:
\begin{lemma}\label{l:K-upper-bound-2}
For any $r>0$, 
\[
	\int_{B_{r}} K_f(v,v+w) |w|^2 \dd w 
		\leq C \left( \int_{\R^3} f(v+w)|w|^{\gamma+2s} \dd w\right) r^{2-2s}.
\]

\[
	\int_{B_r^c} K_f(v,v+w) \dd w
		\leq C \left( \int_{\R^3} f(v+w)|w|^{\gamma+2s} \dd w\right) r^{-2s}.
\]
\end{lemma}

\begin{lemma}{\cite[Lemma 2.3]{imbert2020lowerbounds}}\label{l:C2Linfty}
For any bounded, $C^2$ function $\varphi$ on $\R^3$, the following inequality holds:
\[
	|Q_s(f, \varphi)|
		\leq C\left(\int_{\R^3} f(v+w) |w|^{\gamma+2s} \dd w\right)
			\|\varphi\|_{L^\infty(\R^3)}^{1-s}
			\|D^2 \varphi\|_{L^\infty(\R^3)}^s.\]
\end{lemma}

The next lemma, which appears to be new, is related to \cite[Proposition 3.1(v)]{HST2020boltzmann}, but the statement here is sharper in terms of the decay exponent. The small gain in the exponent provided by this lemma will be crucial in propagating higher polynomial decay estimates forward in time.

\begin{lemma}\label{l:Q-polynomial}
For any $q_0> 3 + \gamma + 2s$ let $f\in L^\infty_{q_0}(\R^3)$ be a nonnegative function, and choose $q\in [q_0, q_0-\gamma]$. (Recall that $\gamma < 0$.) Then there holds
\[ Q(f,\langle \cdot\rangle^{-q})(v) \leq C\|f\|_{L^\infty_{q_0}(\R^3)}\vv^{-q}.\]
The constant $C$ depends on universal constants and $q$.
\end{lemma}

It is easy to see from the computations below that a sharper estimate can be given with $\vv^{-q}$ replaced by a factor with faster decay.  We do not, however, see an application of the sharper estimate, and its proof will require a bit more care.  Hence, for brevity, we simply state and prove the less sharp version above.

\begin{proof}
Writing $Q = Q_{\rm s} + Q_{\rm ns}$, from Lemma \ref{l:Q2} we have
\[ Q_{\rm ns} (f, \langle \cdot\rangle^{-q})(v) \approx \vv^{-q} \int_{\R^3} f(v-w) |w|^{\gamma} \dd w \lesssim \vv^{-q} \|f\|_{L^\infty_{q_0}(\R^3)},\]
by the convolution estimate Lemma \ref{l:convolution}, since $q_0> 3+\gamma+2s > 3+\gamma$.

For the singular term, Lemma \ref{l:Q1} gives
\[ Q_{\rm s} (f,\langle \cdot \rangle^{-q})(v)= \int_{\R^3} K_f(v,v') [\langle v'\rangle^{-q} - \vv^{-q}] \dd v'.\]
If $|v|\leq 2$, then Lemma \ref{l:C2Linfty} implies
\[\begin{split}
 Q_{\rm s}(f,\langle \cdot \rangle^{-q})(v) &\lesssim \left(\int_{\R^3} f(v+w) |w|^{\gamma+2s} \dd w \right) \|\vv^{-q} \|_{L^\infty(\R^3)}^{1-s} \|D^2\vv^{-q}\|_{L^\infty(\R^3)}^s\\
 &\lesssim \|f\|_{L^\infty_{q_0}(\R^3)}
 \lesssim \|f\|_{L^\infty_{q_0}(\R^3)} \vv^{-q},
\end{split} \]
since $v$ lives on a bounded domain and $q_0 > 3 + \gamma + 2s$.

When $|v|\geq 2$, we write the integral over $\R^3$ as an infinite sum by defining, for each integer $k$, the annulus $A_k(v) = B_{2^k|v|}(v) \setminus B_{2^{k-1}|v|}(v)$. For the terms with $k\leq -1$, for $v' \in A_k(v) \subset B_{|v|/2}(v)$ we Taylor expand $g(v) := \vv^{-q}$ to obtain
\[ g(v') - g(v) = (v'-v) \cdot \nabla g(v) + \frac 1 2 (v'-v) \cdot (D^2g (z) (v'-v)), \quad \text{for some } z\in B_{|v|/2}(v).\]
The symmetry of the kernel $K_f$ implies $\int_{A_k(v)} (v'-v) \cdot \nabla g(v) K_f(v,v') \dd v' = 0$. We then have, using Lemma \ref{l:K-upper-bound} and that $q_0 > 3 + \gamma + 2s$,
\[\begin{split}
\Big|\sum_{k\leq -1} \int_{A_k(v)} &K_f(v,v') [\langle v'\rangle^{-q} - \vv^{-q}] \dd v'\Big| 
	\lesssim \|D^2 g\|_{L^\infty(B_{|v|/2}(v))} \sum_{k\leq -1} \int_{A_k(v)}|v'-v|^2 K_f(v,v') \dd v'\\
	&\lesssim \vv^{-q-2}\left(\int_{\R^3} f(v+w) |w|^{\gamma+2s}\dd w\right)  \sum_{k\leq -1} (2^{k-1} |v|)^{2-2s}\\
	&\lesssim \vv^{-q-2} \|f\|_{L^\infty_{q_0}(\R^3)} \vv^{(\gamma+2s)_+} |v|^{2-2s}
	\lesssim \vv^{-q}\|f\|_{L^\infty_{q_0}(\R^3)}.
\end{split}\]

  For the terms with $k\geq 0$, we further divide $A_k(v)$ into $A_k(v) \cap B_{|v|/2}(0)$ and $A_k(v) \setminus B_{|v|/2}(0)$. (Note that $A_k(v) \cap B_{|v|/2}(0)$ is empty unless $k = 0$ or $k=1$.) In $A_k(v) \setminus B_{|v|/2}(0)$, we use $\langle v'\rangle^{-q} \lesssim \vv^{-q}$ and Lemma \ref{l:K-upper-bound} to write
\[\begin{split}
	\sum_{k\geq 2}\int_{A_k(v)\setminus B_{|v|/2}} K_f(v,v')
		[\langle v'\rangle^{-q} - \vv^{-q}] \dd v'
	&\lesssim \vv^{-q} \left(\int_{\R^3} f(v+w) |w|^{\gamma+2s} \dd w\right)\sum_{k\geq 2} (2^{k-1}|v|)^{-2s}\\
	 &\lesssim \|f\|_{L^\infty_{q_0}(\R^3)} \vv^{-q} ,
\end{split}
\]
where we again used that $q_0> 3+\gamma+2s$.

It only remains to bound the integral over $v' \in A_k(v) \cap B_{|v|/2}(0)$. From  \cite[Lemma 2.4]{henderson2021existence}, we have
\[
	K_f(v,v') \lesssim \frac 1{|v-v'|^{3+2s}} \|f\|_{L^\infty_{q_0}} \vv^{3+\gamma+2s-q_0}
		\qquad \text{if } |v'|
		\leq |v|/2.
\]
This implies 
\[
\begin{split}
	\int_{A_k(v) \cap B_{|v|/2}} K_f(v,v') &[\langle v'\rangle^{-q} - \vv^{-q}] \dd v'
		\lesssim \|f\|_{L^\infty_{q_0}} \int_{A_k(v) \cap B_{|v|/2}}\frac{ \vv^{3+\gamma+2s-q_0}}{|v-v'|^{3+2s}} \langle v'\rangle^{-q} \dd v'\\
		&\approx \|f\|_{L^\infty_{q_0}} \vv^{-q_0+\gamma} \int_{B_{|v|/2}} \langle v'\rangle^{-q}\dd v',
\end{split}
\]
where we used the nonnegativity of $K_f$ to discard the term $-\vv^{-q}$ inside the integral and we also used that $|v'-v|\geq |v|/2$.

When $q >3$, the final integral is finite and we clearly find
\be
	\int_{B_{|v|/2}} K_f(v,v') [\langle v'\rangle^{-q} - \vv^{-q}] \dd v'
		\lesssim \|f\|_{L^\infty_{q_0}} \vv^{-q_0 + \gamma}
		\leq \|f\|_{L^\infty_{q_0}}\vv^{- q},
\ee
due to the condition $q \leq q_0 - \gamma$.  When $q < 3$, the above becomes
\be
	\begin{split}
	\int_{B_{|v|/2}} K_f(v,v') [\langle v'\rangle^{-q} - \vv^{-q}] \dd v'
		&\lesssim \|f\|_{L^\infty_{q_0}} \vv^{-q_0 + \gamma + 3 - q}
		\\&\leq \|f\|_{L^\infty_{q_0}} \vv^{-(3+\gamma + 2s) + \gamma + 3 - q}
		=  \|f\|_{L^\infty_{q_0}} \vv^{- 2s - q}
	\end{split}
\ee
since $q_0 > 3 + \gamma + 2s$.  The case $q=3$ is  the same up to an additional $\log \vv$ factor, which is controlled by $\vv^{-2s}$.  Hence, in all cases, we find
\be
	\int_{B_{|v|/2}} K_f(v,v') [\langle v'\rangle^{-q} - \vv^{-q}] \dd v'
		\lesssim \|f\|_{L^\infty_{q_0}} \vv^{- q}.
\ee
This completes the proof. 
\end{proof}

\section{Change of variables and global regularity estimates}\label{s:global-reg}

The regularity estimates and continuation criterion of Imbert-Silvestre \cite{imbert2020smooth} apply to the case $\gamma + 2s\in [0,2]$. In this section, we discuss the extension of these results to $\gamma + 2s < 0$, with suitably modified hypotheses. 

As mentioned above, our current study requires these regularity estimates 
to establish the smoothness of our solutions for positive times. 
 More generally, the global estimates and the change of variables used to prove them 
are important tools in the study of the non-cutoff Boltzmann equation, so extending these tools to the case $\gamma+2s<0$ may be of independent interest.

The key obstacle in passing from local estimates (Theorems \ref{t:degiorgi} and \ref{t:schauder}) to global estimates on $[0,T]\times\R^3\times\R^3$ is the degeneration of the upper and lower ellipticity bounds for the collision kernel $K_f(t,x,v,v')$ as $|v|\to \infty$. To overcome this, the authors of \cite{imbert2020smooth} developed a change of variables that ``straightens out'' the anisotropic ellipticity of the kernel and allows to precisely track the behavior of the estimates for large $|v|$. Their change of variables is defined as follows:  
for a fixed reference point $z_0 = (t_0,x_0,v_0) \in [0,\infty)\times \R^6$ with $|v_0|>2$ and with $\gamma + 2s\geq 0$, define the linear transformation $T_0:\R^3\to \R^3$ by
\[
	T_0 (av_0 + w)
		:= \frac a {|v_0|}v_0 + w
		\quad \text{ where } a\in \R, \, w\cdot v_0 = 0, \quad \text{(if $\gamma+2s\geq 0$).}
\]
Next, for $(t,x,v) \in Q_1(z_0)$ and when $\gamma + 2s \geq 0$, define
\begin{equation}\label{e.cov+}
\begin{split}
 	(\bar t, \bar x, \bar v)
		= \mathcal T_0(t,x,v)
		&:= \left(
			t_0 + \frac t {|v_0|^{\gamma+2s}},
			x_0 + \frac {T_0x + tv_0}{|v_0|^{\gamma+2s}},
			v_0 + T_0v\right) \\
		&= z_0 \circ \left(
			\frac t {|v_0|^{\gamma+2s}},
			\frac {T_0 x} {|v_0|^{\gamma+2s}},
			T_0v\right).
 \end{split}
 \end{equation}
When $|z_0| \leq 2$, let $\mathcal T_0(t,x,v) = z_0 \circ z$.
 
The definition \eqref{e.cov+} applies when $\gamma + 2s \in [0,2]$, and does not generalize well to the case $\gamma + 2s< 0$.  Indeed, the solution $f$ of~\eqref{e.boltzmann} is not defined at the point $(\bar t, \bar x, \bar v)$ if $\bar t < 0$, which occurs if, e.g., $\gamma + 2s < 0$ and $v_0$ is sufficiently large. 
Thus, for the case $\gamma+2s < 0$, we introduce a new definition: first, extend the definition of $T_0$ as follows:
\begin{equation}\label{e.T0-def}
	T_0 (av_0 + w)
		:= |v_0|^\frac{\gamma + 2s}{2s} \Big(\frac a {|v_0|}v_0 + w\Big)
		\quad \text{ where } a\in \R, \, w\cdot v_0 = 0, \quad \text{(if $\gamma+2s<0$).}
\end{equation}
Then, let
\begin{equation}\label{e.cov-}
\begin{split}
 (\bar t, \bar x, \bar v)
 	= \mathcal T_0(t,x,v)
	&:= \left( t_0 + t, x_0 + T_0x + tv_0, v_0 + T_0v\right) \\
 &= z_0 \circ \left(t , T_0 x,  T_0v\right).
 \end{split}
 \end{equation}
 As above, when $|v_0|<2$, we take $\mathcal T_0z = z_0\circ z$.  It is interesting to note that sending $s\to 1$ in \eqref{e.cov-} recovers the change of variables that was applied to the Landau equation in \cite{henderson2017smoothing} in the very soft potentials regime.

For $r>0$ and $z_0 = (t_0,x_0,v_0)\in \R^7$, define 
\[
	\mathcal E_r(z_0)
		:= \mathcal T_0(Q_r), \quad E_r(v_0)
		:=
		v_0 + T_0(B_r),
\]
and
\[
	\mathcal E_r^{t,x}(z_0)
		:= \begin{cases}
			\Big\{\Big(t_0 +\frac{t}{|v_0|^{\gamma+2s}}, x_0 + \frac{1}{|v_0|^{\gamma+2s}}(T_0 x + tv_0)\Big) : t\in [-r^{2s},0], x\in B_{r^{1+2s}}\Big\}, &\gamma+2s\geq 0,
			\smallskip \\
			\big\{\big(t_0 +t, x_0 + T_0x +   tv_0\big) :t\in [-r^{2s},0], x\in B_{r^{1+2s}}\big\}, &\gamma+2s<0,
\end{cases} \]
 so that
 \be
 	\mathcal E_r(z_0)
		= \mathcal E_r^{t,x}(z_0)
			\times E_r(v_0).
\ee
 
Now, for a solution $f$  of the Boltzmann equation \eqref{e.boltzmann} and $z_0$ such that $t_0\geq 1$, let us define
\[ \bar f(t,x,v) = f(\bar t, \bar x, \bar v).\]
By direct computation, $\bar f$ satisfies
\[ \partial_t \bar f + v\cdot \nabla_x \bar f = \mathcal L_{\bar K_f} \bar f + \bar h, \quad \text{ in } Q_1,\]
where
\begin{equation}\label{e.barKf}
	\bar K_f(t,x,v,v')
		:= \begin{cases}
		\frac{1}{|v_0|^{1+\gamma+2s}} \, K_f(\bar t, \bar x, \bar v, v_0 + T_0v'), &\qquad \gamma+2s\geq 0,
		\smallskip\\
		|v_0|^{2+\frac{3\gamma}{2s}} \, K_f(\bar t, \bar x, \bar v, v_0 +T_0 v'), &\qquad\gamma+2s<0, \end{cases}
\end{equation}
and
\[ \bar h(t,x,v) = c_b |v_0|^{-(\gamma+2s)_+} f(\bar t, \bar x, \bar v) [f\ast |\cdot|^\gamma](\bar t, \bar x, \bar v).\]
For use in the convolution defining $K_f$, we also define
\[
	\bar v' := v_0 + T_0 v'.
\]

In order to apply the local regularity estimates (Theorems \ref{t:degiorgi} and \ref{t:schauder} above) to $\bar f$, one obviously needs to check that the kernel $\bar K_f$ satisfies the hypotheses of these two theorems. In the case $\gamma+2s\in [0,2]$, this was done in \cite[Section~5]{imbert2020smooth}. As stated, the results in \cite[Section~5]{imbert2020smooth} require a uniform upper bound on the energy density $\int_{\R^3} |v|^2 f(t,x,v) \dd v$, but it is clear from the proof that this can be replaced by a bound on the $2s$ moment $\int_{\R^3} |v|^{2s}f(t,x,v) \dd v$.

\begin{proposition}{\cite[Theorems 5.1 and 5.4]{imbert2020smooth}}\label{p:moderately-soft-kernel}
Assume $\gamma + 2s \in [0,2]$. Let $z_0 = (t_0,x_0,v_0)$ be arbitrary, and let $\mathcal E_1(z_0)$ be defined as above. If $f$ satisfies 
\[
\begin{aligned}
 	&0<m_0\leq \int_{\R^3}f(t,x,v) \leq M_0,
		&\qquad&\qquad&
  	&\int_{\R^3}|v|^{2s}f(t,x,v) \dd v \leq \tilde E_0, \\
	 & \int_{\R^3} f(t,x,v)\log f(t,x,v) \dd v \leq H_0,
	 	&\quad&\text{and }\quad&
	&\sup_{v\in\R^3}\int_{\R^3}f(t,x,v+u)|u|^\gamma \dd u \leq C_\gamma,
\end{aligned}
\]
for all $(t,x) \in \mathcal E_1^{t,x}(z_0)$, then the kernel $\bar K_f(t,x,v,v')$ defined in \eqref{e.barKf} ($\gamma+2s\geq 0$ case) satisfies the conditions \eqref{e.coercivity1}---\eqref{e.cancellation2} of Theorem \ref{t:degiorgi}, with constants depending only on $\gamma$, $s$, $m_0$, $M_0$, $\tilde E_0$, $H_0$, and $C_\gamma$. In particular, all constants are independent of $v_0$.

Furthermore, for each $z=(t,x,v) \in \mathcal E_1$, the kernel $\bar K_z(w) = \bar K_f(t,x,v,v+w)$ lies in the ellipticity class defined in Definition \ref{d:ellipticity}, with constants as in the previous paragraph. 
\end{proposition}

For the case $\gamma +2s<0$, the result corresponding to Proposition \ref{p:moderately-soft-kernel} is contained in the following proposition, which we prove in Appendix \ref{s:cov-appendix}:
\begin{proposition}\label{p:very-soft-kernel}
Asume $\gamma+2s <0$. Let $z_0$ and $\mathcal E_1(z_0)$ be as in Proposition \ref{p:moderately-soft-kernel}, and assume that
\[ \|f(t,x,\cdot)\|_{L^\infty_q(\R^3)} \leq L_0, \qquad \text{for all } (t,x) \in \mathcal E_1^{t,x}(z_0),\]
for some $q>3 + 2s$, and
\[ f(t,x,v) \geq \delta, \quad \text{for all } (t,x,v) \in \mathcal E_1^{t,x}(z_0)\times B_r(v_m),\]
for some $v_m\in \R^3$, and $\delta, r>0$. Then the kernel $\bar K_f(t,x,v,v')$ defined in \eqref{e.barKf} ($\gamma+2s<0$ case) satisfies the conditions \eqref{e.coercivity1}---\eqref{e.cancellation2} of Theorem \ref{t:degiorgi}, with constants depending only on $L_0$, $\delta$, $r$, and $v_m$.

Furthermore, for each $z\in \mathcal E_1$, the kernel $\bar K_z(w) = \bar K_f(t,x,v,v+w)$ lies in the ellipticity class defined in Definition \ref{d:ellipticity}, with constants as in the previous paragraph.  
\end{proposition}

We note that one should be able to obtain a more optimized version of \Cref{p:very-soft-kernel} by working in a weighted $L^p_v$-based space with $p < \infty$; however, this is not necessary for our work here, so we state and prove the simpler version above. It is also interesting to note that our definition \eqref{e.cov-} would not work well in the case $\gamma+2s\geq 0$, so the separate definitions seem to be unavoidable.

Let us give a more concise, and less sharp, restatement of the previous two propositions, that is sufficient for our purposes. When $q>3 + 2s$, the norm $\|f\|_{L^\infty_q(\R^3)}$ controls the mass and entropy densities, as well as the $2s$-moment and the constant $C_\gamma$, so we have the following: 
\begin{proposition}\label{p:concise-cov}
Let $z_0$ and $\mathcal E_1(z_0)$ be as in Proposition \ref{p:moderately-soft-kernel}, and assume that 
\[ \|f(t,x,\cdot)\|_{L^\infty_q(\R^3)} \leq L_0, \qquad \text{for all } (t,x) \in \mathcal E_1^{t,x}(z_0),\]
for some $q>3 + 2s$. 
Assume further that
\[
	f(t,x,v) \geq \delta,
		\qquad \text{for all } (t,x,v) \in \mathcal E_1^{t,x}(z_0)\times B_r(v_m),
\]
for some $v_m\in \R^3$, and $\delta, r>0$.

Then the kernel $\bar K_f(t,x,v,v')$ defined in \eqref{e.barKf} satisfies all the conditions of Theorem \ref{t:degiorgi} and belongs to the ellipticity class defined in Definition \ref{d:ellipticity}, with constants depending only on $L_0$, $\delta$, $r$, and $v_m$.
\end{proposition}

The H\"older regularity of the kernel is also required for applying the Schauder estimate of Theorem \ref{t:schauder}. In the $\gamma+2s\geq 0$ case, this regularity is given by \cite[Lemma 5.20]{imbert2020smooth}. For $\gamma+2s<0$, we prove it in Lemma \ref{l:cov-holder} below. We summarize these two results here:
\begin{lemma}\label{l:concise-holder}
For any $f:[0,T]\times \R^6 \to [0,\infty)$ such that $f\in C^\alpha_{\ell,q_1}([0,T]\times\R^6)$ with 
\[
	q_1>
		\begin{cases}
			5+\frac{\alpha}{1+2s}
				&\quad \text{ if }\gamma+2s\geq 0,
			\\
			3 + \frac{\alpha}{1+2s}
				&\quad \text{ if }\gamma+2s< 0,
		\end{cases}
\]
and for any $|v_0|>2$ and $r\in (0,1]$, let
\[
	\bar K_{f,z}(w)
		:= \bar K_f(t,x,v,v+w)
		\qquad \text{ for } z = (t,x,v) \in Q_{2r}.
\]
Then we have
\[
	\int_{B_\rho} |\bar K_{f,z_1}(w) - \bar K_{f,z_2}(w)| |w|^2 \dd w
		\leq \bar A_0 \rho^{2-2s} d_\ell(z_1,z_2)^{\alpha'}, \quad \rho > 0, z_1,z_2\in Q_{2r},
\]
with $\alpha' = \alpha \frac {2s} {1+2s}$ and
\be\label{e.A_0}
	\bar A_0 \leq C |v_0|^{P(\alpha,\gamma,s)} \|f\|_{C^\alpha_{\ell, q_1}([0,T]\times\R^6)},
\ee
where 
\[P(\gamma,s,\alpha) = \begin{cases} \frac \alpha {1+2s}(1-2s-\gamma)_+, &\gamma+2s\geq 0,\\
2+\frac \alpha {1+2s}, & \gamma+2s<0,\end{cases}\]
and the constant $C$ depends on universal quantities, $\alpha$, and $q_1$, but is independent of $|v_0|$.
\end{lemma}

Next, we discuss global regularity estimates for solutions of the Boltzmann equation. 
The following is a global (in $v$) H\"older estimate that combines the local De Giorgi/Nash/Moser-type estimate of Theorem \ref{t:degiorgi} with the change of variables $\mathcal T_0$. It extends \cite[Corollary 7.4]{imbert2020smooth} to the case where $\gamma+2s$ may be negative.

\begin{theorem}\label{t:global-degiorgi}
Let $f$ be a solution of the Boltzmann equation \eqref{e.boltzmann} in $[0,T]\times \R^6$. For some domain $\Omega \subseteq \R^3$ and $\tau \in (0,T)$, assume there are $\delta, r , R>0$ such that for each $x\in \Omega$, there exists $v_x\in B_R$ with 
\[ 
f(t,x,v) \geq \delta, \quad \text{in }  [\tau/2,T]\times(B_r(x,v_x) \cap\Omega\times \R^3),
\]
and assume that $\|f\|_{L^\infty_{q_0}([0,T]\times\R^6)} \leq L_0$ for some $q_0> 3 + 2s$. Define
\begin{equation}\label{e.barc}
	\bar c
		= \begin{cases} 	1
			&\qquad \text{ if }\gamma+2s\geq 0,\\
			- \frac{\gamma}{2s}
			&\qquad \text{ if }\gamma+2s<0. 
		\end{cases} 
 \end{equation}
 
Then, for any $q>3$ such that $f\in L^\infty_{q}([0,T]\times\R^6)$, and any $\Omega'$ compactly contained in $\Omega$, there exists $\alpha_0>0$ such that for any $\alpha \in (0,\alpha_0]$, one has $f\in C^\alpha_{\ell,q-\bar c\alpha}$, with
 \begin{equation}\label{e.dg-est}
  \|f\|_{C^\alpha_{\ell,q-\bar c\alpha}([\tau,T]\times\Omega'\times\R^3)} \leq C\|f\|_{L^\infty_{q}([0,T]\times \R^6)}.
  \end{equation}
The constants $C$ and $\alpha_0$ depend on $L_0$, $q$, $\gamma$, $s$, $\delta$, $r$, $v_m$, $\tau$, $\Omega$, and $\Omega'$. 

If $\Omega = \R^3$, then we can replace $\Omega'$ with $\R^3$ in \eqref{e.dg-est}.
\end{theorem}

\begin{proof}

Choose a point $z_0= (t_0,x_0,v_0)$ and $r\in (0,1)$. We prove an interior estimate in the cylinder $Q_r(z_0)$, which implies the statement of the Theorem via a standard covering argument. 

We claim that for any $\bar z_1, \bar z_2 \in \mathcal E_{r/2}(z_0)$, 
\begin{equation}\label{e.holder-local}
 |f(\bar z_1) - f(\bar z_2)| \leq C (1+|v_0|)^{-q} d_\ell(z_1,z_2)^\alpha,
 \end{equation}
with $C, \alpha>0$ as in the statement of the theorem. 

If $|v_0|\leq 2$, then \eqref{e.holder-local} follows from Theorem \ref{t:degiorgi}. The dependence of the constants $C$ and $\alpha$ in this case is made clear from the discussion in Remark \ref{r:ellipticity}. Note that $3 + 2s\geq 3+\gamma+2s$.

If $|v_0|>2$, the proof of \eqref{e.holder-local} proceeds exactly as in \cite[Proposition 7.1]{imbert2020smooth}
and uses the change of variables \eqref{e.cov+} or \eqref{e.cov-}. 
Let us remark that the proof of Proposition 7.1 in \cite{imbert2020smooth} makes use of Lemma \ref{l:convolution} below, as well as \cite[Lemma 6.3]{imbert2020smooth}. Although \cite{imbert2020smooth} works under the global assumption $\gamma+2s\in [0,2]$, it is clear from the proof that \cite[Lemma 6.3]{imbert2020smooth} applies to both cases $\gamma+2s\geq 0$ and $\gamma+2s<0$, with the same statement. 
The key point is that we can apply Theorem \ref{t:degiorgi} with the constants $\lambda$ and $\Lambda$ depending on $L_0$, $\delta$, $r$, and $v_m$, and independent of $v_0$, by Proposition \ref{p:concise-cov}. We omit the details of the proof of \eqref{e.holder-local} since they are the same as in \cite{imbert2020smooth}. 

Estimate \eqref{e.holder-local} is equivalent to
\[
	[\bar f]_{C^\alpha_\ell(Q_{r/2}(z_0))}
		\lesssim (1+|v_0|)^{-q}.
\]
Using Lemma \ref{l:holder-cov} or \cite[Lemma 5.19]{imbert2020smooth} to translate from $\bar f$ to $f$, we obtain
\[
	\|f\|_{C^\alpha_\ell(\mathcal E_{r/2}(z_0))}
		\lesssim (1+|v_0|)^{\bar c\alpha-q} + \|f\|_{L^\infty(\mathcal E_{r/2}(z_0))},
\]
with $\bar c$ as in the statement of the theorem.

In order to estimate the $C^{\alpha}_{\ell,q}$ seminorm of $f$ (See Definition \ref{d:C-alpha-q}), we need to work with local seminorms of $f$ on cylinders $Q_{r/2}(z_0)$ rather than on the twisted cylinders $\mathcal E_{r/2}(z_0)$.  From the definitions \eqref{e.cov+} and \eqref{e.cov-} of the $\mathcal T_0$ change of variables, we see that $\mathcal E_{r/2}(z_0)\supset Q_{(r/2)|v_0|^{-\bar c}}(z_0)$. Using \cite[Lemma 3.5]{imbert2020smooth}, we extend the upper bound to the larger set $Q_{r/2}(z_0)$:
\[
	\begin{split}
		[f]_{C^\alpha_\ell(Q_{r/2}(z_0))} 
			&\lesssim \left((1+|v_0|)^{\bar c\alpha - q} + (|v_0|^{-\bar c})^{-\alpha}\|f\|_{L^\infty(Q_{r/2}(z_0))}\right)\\
			& \lesssim (1+|v_0|)^{\bar c\alpha - q} \|f\|_{L^\infty_q([0,T]\times\R^6)}.
\end{split}
\]
This estimate implies the desired upper bound on $[f]_{C^{\alpha}_{\ell,q-\bar c \alpha}([0,T]\times\R^6)}$, which concludes the proof.
\end{proof}

Next, we have a global Schauder estimate that improves $C^\alpha_\ell$-regularity as in Theorem \ref{t:global-degiorgi} to $C^{2s+\alpha}_\ell$-regularity. In order to apply the estimate to derivatives of $f$, this estimate is stated for the linear Boltzmann equation
\begin{equation}\label{e.linear-boltzmann}
\partial_t g + v\cdot \nabla_x g = Q_{\rm s}(f,g) + h.
\end{equation}
Choosing $g=f$ and $h = Q_{\rm ns}(f,f)$ would recover the original Boltzmann equation \eqref{e.boltzmann}.

Unlike in the previous theorem, we work out the explicit dependence on $\tau$ of the estimate in Theorem \ref{t:global-schauder}. The dependence on $\tau$ is needed in Proposition \ref{prop:holder_propagation} which is one ingredient of our proof of uniqueness.  

\begin{theorem}\label{t:global-schauder}
Let $f:[0,T]\times\R^6\to [0,\infty)$, and assume that for some domain $\Omega\subseteq \R^3$, there exist $\delta, r, R>0$ such that, if $x \in \Omega$, there is $v_x\in \R^3$ with
\[
	f(t,x,v) \geq \delta, \quad \text{ in } [0,T]\times
				\left(B_r(x,v_x) \cap (\Omega \times \R^3)\right).
\]
Assume that $\|f\|_{L^\infty_{q_0}([0,T]\times\R^6)} \leq L_0$ for some $q_0>3 + 2s$.  Furthermore, assume $f\in C^\alpha_{\ell,q_1}([0,T]\times\R^6)$ for some $\alpha\in (0,\min(1,2s))$ and $q_1$ as in Lemma \ref{l:concise-holder}.

Let $g$ be a solution of \eqref{e.linear-boltzmann} with $h\in C^{\alpha'}_{\ell,q_1}([0,T]\times\R^6)$, where $\alpha' = \frac {2s}{1+2s}\alpha$ and $2s+\alpha'\not\in\{1,2\}$. Then, for any $\tau\in (0,T)$ and $\Omega'$ compactly contained in $\Omega$, one has the estimate
\begin{equation}\label{e.2salpha}
\begin{split}
	&\|g\|_{C^{2s+\alpha'}_{\ell,q_1-\kappa}([\tau,T]\times\Omega'\times\R^3)} \\
		&\quad\leq C\left(1 + \tau^{-1 + \frac{\alpha-\alpha'}{2s}}\right)
			\|f\|_{C^\alpha_{\ell, q_1}([0,T]\times\R^6)}^\frac{\alpha -\alpha'+ 2s}{\alpha'}
			\left( \|g\|_{C^\alpha_{\ell,q_1}([0,T]\times\R^6)} + \|h\|_{C^{\alpha'}_{\ell, q_1}([0,T]\times \R^6)}\right),
\end{split}
\end{equation}
where the moment loss for higher order regularity of $g$ is
\[\begin{split}
 \kappa& := \begin{cases} (1-2s-\gamma)_+ \left(1 + \frac \alpha {1+2s} - \frac \alpha {2s}\right) + 2s+\alpha', &\gamma+2s\geq 0\\
(2 + \frac \alpha {1+2s})(\frac{1+2s}\alpha-  \frac 1 {2s}) - \gamma \left(1+\frac \alpha{1+2s}\right), &\gamma+2s <  0,\end{cases}
\end{split}\]
and the constant $C$ depends on $\gamma$, $s$, $\alpha$, $q_0$, $q_1$, $L_0$, $\delta$, $r$, $R$, $\Omega$, and $\Omega'$.

If $\Omega = \R^3$, then we can replace $\Omega'$ with $\R^3$ in \eqref{e.2salpha}.
\end{theorem}
\begin{proof}
We follow the proof of \cite[Proposition 7.5]{imbert2020smooth} suitably modifying the argument in order to allow the case $\gamma + 2s < 0$ and to determine the explicit dependence on $\tau$.

In this proof, to keep the notation clean, any norm or seminorm given without a domain, such as $\|f\|_{C^\alpha_{\ell,q}}$, is understood to be over $[0,T]\times\R^6$.

Let $z_0 = (t_0,x_0,v_0)\in (\tau,T]\times\Omega'\times\R^3$ be fixed, and define
\[
	\rho = \min\big(1, (t_0/2)^\frac{1}{2s}\big).
\]
Since we do not track the explicit dependence of this estimate on $\Omega'$ and $\Omega$, we assume without loss of generality that $\mathcal E_\rho(z_0)\subset [0,T]\times\Omega\times\R^3$.

If $|v_0|>2$, then let $\varphi\in C_c^\infty(\R^3)$ be a cutoff function supported in $B_{|v_0|/8}$ and identically 1 in $B_{|v_0|/9}$. Define $\bar g := [(1-\varphi)g]\circ \mathcal T_0$. 
The function $\bar g$ is defined on $Q_\rho$ and satisfies 
\[ \partial_t \bar g + v\cdot \nabla_x \bar g = \int_{\R^3} (\bar g' - \bar g) \bar K_f(t,x,v,v') \dd v' + \bar h + \bar h_2,\]
where $\bar g' = \bar g(t,x,v')$, and $\bar K_f$ is defined in \eqref{e.barKf}. The source terms are defined by
\[ \begin{split}
	\bar h
		&= |v_0|^{-(\gamma+2s)_+} h \circ \mathcal T_0,\\
	\bar h_2
		&= |v_0|^{-(\gamma+2s)_+} \int_{\R^3} \left(\varphi(v') - \varphi(\bar v)\right) g(\bar t, \bar x, v') K_f(\bar t, \bar x, \bar v, v') \dd v'
		\\&
		= |v_0|^{-(\gamma+2s)_+} \int_{\R^3} \varphi(v') g(\bar t, \bar x, v') K_f(\bar t, \bar x, \bar v, v') \dd v'.
\end{split}
\]
We note that the last equality follows from the fact that $\varphi(\bar v) = 0$.  Indeed, recall that $\supp \varphi \subset B_{|v_0|/8}$ while a straightforward computation using the definition of $T_0$ (see~\eqref{e.cov-} and above) shows that $|\bar v| \geq |v_0|/2 > |v_0|/8$.

By Proposition \ref{p:concise-cov}, the kernel $\bar K_f$ lies in the ellipticity class of Definition \ref{d:ellipticity}, with constants depending on $L_0$, $\delta$, $r$, and $R$. Applying the local Schauder estimate Theorem \ref{t:schauder} to $\bar g$, we obtain
\begin{equation}\label{e.local-schauder}
\begin{split}
	[\bar g]_{C^{2s+\alpha'}_\ell(Q_{\rho/2})} 
	&\lesssim \max\left( \rho^{-2s-\alpha' + \alpha}, \bar A_0^{(2s+\alpha' - \alpha)/\alpha'}\right)
		[\bar g]_{C^\alpha_{\ell}((-\rho^{2s},0]\times B_{\rho^{1+2s}} \times\R^3)}\\
 	&\quad  + [\bar h + \bar h_2]_{C_\ell^{\alpha'}(Q_\rho)}
	 + \max(\rho^{-\alpha'},\bar A_0)\|\bar h + \bar h_2\|_{L^\infty(Q_\rho)},
\end{split}
\end{equation}
where we recall the definition of $\bar A_0$ from~\eqref{e.A_0}.
Let us estimate the terms in this right-hand side one by one. 
 To estimate $\bar A_0$, we use Lemma \ref{l:concise-holder}:
\[ \begin{split}
\bar A_0 &\lesssim |v_0|^{P(\gamma,s,\alpha)} \|f\|_{C^\alpha_{\ell,q_1}},
\end{split} \]
with $q_1$ and $P(\gamma, s, \alpha)$ as in Lemma \ref{l:concise-holder}. 
 Next, Lemma \ref{l:holder-cov} below (if $\gamma+2s<0$) or \cite[Lemma 5.19]{imbert2020smooth} (if $\gamma+2s\geq 0$) imply
\[
	[\bar g]_{C^\alpha_\ell((-\rho^{2s},0]\times B_{\rho^{1+2s}}\times\R^3)}
		\leq \|(1-\varphi) g\|_{C^\alpha_\ell}
		\lesssim |v_0|^{-q_1}\|g\|_{C^\alpha_{\ell,q_1}},
\]
with $q_1$ as above.  The last inequality follows because $1-\varphi$ is supported for $|v|\geq |v_0|/9$.

For the terms involving $\bar h$ and $\bar h_2$, note that
\[\begin{split}
\|\bar h\|_{L^\infty(Q_\rho)} &= |v_0|^{-(\gamma+2s)_+} \|h\|_{L^\infty(\mathcal E_\rho(z_0))} \leq |v_0|^{-(\gamma+2s)_+ - q_1}\|h\|_{L^\infty_{q_1}},\\
[\bar h]_{C^{\alpha'}_\ell(Q_\rho)} &\leq |v_0|^{-(\gamma+2s)_+} [h]_{C^{\alpha'}_\ell(\mathcal E_\rho(z_0))} \leq |v_0|^{-(\gamma+2s)_+-q_1} [h]_{C^{\alpha'}_{\ell,q_1}},
\end{split}\]
where the second line used Lemma \ref{l:holder-cov} or \cite[Lemma 5.19]{imbert2020smooth}. For $\bar h_2$, we use \cite[Lemma 6.3 with $q=q_1$ and Corollary 6.7 with $q = q_1 - \alpha/(1+2s)$]{imbert2020smooth}
to write
\[ \begin{split}
\|\bar h_2\|_{L^\infty(Q_\rho)} &\lesssim |v_0|^{-q_1 +\gamma - (\gamma+2s)_+}\|f\|_{L^\infty_{q_1}}\|g\|_{L^\infty_{q_1}} ,\\
[\bar h_2]_{C^{\alpha'}_\ell(Q_\rho)} &\lesssim |v_0|^{-q_1 + \gamma -(\gamma+2s)_+ + 2\alpha/(1+2s)} \|f\|_{C^\alpha_{\ell,q_1}} \|g\|_{C^\alpha_{\ell,q_1}}.
\end{split} \]
Combining all of these inequalities with \eqref{e.local-schauder}, we have
\[ \begin{split}
	& [\bar g]_{C^{2s+\alpha'}_\ell(Q_{\rho/2})}\\
		 &\lesssim \max\Big( t_0^\frac{-2s+ (\alpha -\alpha')}{2s}, |v_0|^{P(\gamma,s,\alpha)\frac{2s- \alpha + \alpha'}{\alpha'}}\|f\|_{C^\alpha_{\ell,q_1}}^\frac{2s-\alpha +\alpha'}{\alpha'}\Big)
		 	|v_0|^{-q_1} \|g\|_{C^\alpha_{\ell,q_1}}
		\\&\quad  
		 	+|v_0|^{-(\gamma+2s)_+ - q_1}[h]_{C^{\alpha'}_{\ell,q_1}}
			+ |v_0|^{-q_1+\gamma-(\gamma+2s)_+ + \frac{2\alpha}{1+2s}}\|f\|_{C^\alpha_{\ell,q_1}} \|g\|_{C^\alpha_{\ell,q_1}}\\
		 &\quad  + \max\Big(t_0^{-\frac{\alpha'}{2s}},|v_0|^{P(\gamma,s,\alpha)}\|f\|_{C^\alpha_{\ell,q_1}}\Big)
		 	\Big(|v_0|^{-(\gamma+2s)_+-q_1}\| h\|_{L^\infty} + |v_0|^{-q_1+\gamma-(\gamma+2s)_+}\|g\|_{L^\infty_{q_1}} \|f\|_{L^\infty_{q_1}}\Big).
\end{split} \]
Keeping only the largest powers of $t_0^{-1}$, $|v_0|$, and $\|f\|_{C^\alpha_{\ell,q_1}}$, we have (recall that $\alpha<2s$)
\[
	[\bar  g ]_{C^{2s+\alpha'}_\ell(Q_{\rho/2})}
			\lesssim \mathcal A,
\]
where we have introduced the shorthand
\be
	\mathcal A := \Big(1 + t_0^{-1+\frac{\alpha-\alpha'}{2s}}\Big)
			|v_0|^{-q_1 + P(\gamma,s,\alpha)\big(1+\frac{2s-\alpha}{\alpha'}\big)}
			\|f\|_{C^\alpha_{\ell,q_1}}^{1+\frac{2s-\alpha}{\alpha'}}
			\left(
				\|g\|_{C^\alpha_{\ell,q_1}}
				+ \|h\|_{C^{\alpha'}_{\ell,q_1}}
			\right).
\ee
Now, we apply Lemma \ref{l:holder-cov} or \cite[Lemma 5.19]{imbert2020smooth} to translate from $\bar g$ back to $g$:
\[
	[g]_{C^{2s+\alpha'}_\ell(\mathcal E_{\rho/2}(z_0))}
		\lesssim |v_0|^{\bar c (2s+\alpha')}
			\|\bar g\|_{C^{2s+\alpha'}_\ell(Q_{\rho/2}(z_0))} 
		\lesssim |v_0|^{\bar c(2s+\alpha')}
			\left( \mathcal A + \|g\|_{L^\infty(Q_{\rho/2}(z_0))}\right),
 \]
 with $\bar c$ as in \eqref{e.barc}.

 As in the proof of Theorem \ref{t:global-degiorgi}, we need to pass from twisted cylinders $\mathcal E_\rho(z_0)$ to kinetic cylinders $Q_\rho(z_0)$. Since $\mathcal E_\rho(z_0)\supset Q_{\rho|v_0|^{-\bar c}}(z_0)$, we use \cite[Lemma 3.5]{imbert2020smooth} to obtain
\[
	[g]_{C^{2s+\alpha'}_\ell(Q_\rho(z_0))}
		\lesssim |v_0|^{\bar c(2s+\alpha')} \mathcal A
			+ |v_0|^{\bar c(2s+\alpha')} \|g\|_{L^\infty(Q_\rho(z_0))}.
\] 
We finally have
\begin{equation}\label{e.schauder-cylinder}
	[g]_{C^{2s+\alpha'}_\ell(Q_\rho(z_0))}
		\lesssim  \Big(1 + t_0^{-1+\frac{\alpha-\alpha'}{2s}}\Big) (1+|v_0|)^{-q_1 + \kappa}
			 \|f\|_{C^\alpha_{\ell,q_1}}^{1+\frac{2s-\alpha}{\alpha'}}
			 \left(
			 	\|g\|_{C^\alpha_{\ell,q_1}}
				+ \|h\|_{C^{\alpha'}_{\ell,q_1}}
			\right),
 \end{equation}
with 
\[ \kappa := P(\gamma,s,\alpha)\Big(1+\frac{2s-\alpha}{\alpha'}\Big)+\bar c(2s+\alpha'),\]
as in the statement of the theorem. 

We have derived \eqref{e.schauder-cylinder} under the assumption $|v_0|>2$. If $|v_0|\leq 2$, then we may apply Theorem \ref{t:schauder} directly to $g$, without using the change of variables. Proceeding as above, we obtain \eqref{e.schauder-cylinder} in this case as well, using $1\lesssim ( 1+ |v_0|)^{-q_1+\kappa}$. This completes the proof, since $t_0\gtrsim \tau$.
\end{proof}

When we apply the linear estimate of Theorem \ref{t:global-schauder} to solutions of the Boltzmann equation, we obtain the following time-weighted Schauder estimate:  

\begin{proposition}\label{p:nonlin-schauder}
Let $f:[0,T]\times\R^6\to [0,\infty)$ be a classical solution to \eqref{e.boltzmann} satisfying the assumptions for $f$ in \Cref{t:global-schauder}. For any $t_0\in (0,T)$, $\alpha>0$, and $q_1$ as in Lemma \ref{l:concise-holder}, the estimate 
\[
	\|f\|_{C^{2s+\alpha'}_{\ell,q_1-\kappa}([t_0/2,t_0]\times\R^6)}
		\leq C t_0^{-1+(\alpha-\alpha')/(2s)}
			\|f\|_{C^\alpha_{\ell,q_1+{[\alpha/(1+2s)+\gamma]_+}}([0,t_0]\times\R^6)}^{1+(\alpha+2s)/\alpha'},
\]
holds whenever the right-hand side is finite, where $\alpha' = \alpha\frac{2s}{1+2s}$, $C>0$ is a constant depending on $\gamma$, $s$, $\alpha$, $q_0$, $q_1$, $L_0$, $\delta$, $r$, $R$, $\Omega$, and $\Omega'$, and $\kappa$ is the constant from Theorem \ref{t:global-schauder}.
\end{proposition}
\begin{proof}
Theorem \ref{t:global-schauder} with with $g=f$ and $h=Q_{\rm ns}(f,f)$ implies
\begin{equation*}
\begin{split}
 &\|f\|_{C^{2s+\alpha'}_{\ell,q_1-\kappa}([t_0/2,t_0]\times\R^6)}\\
	 & \qquad\leq C t_0^{-1+(\alpha-\alpha')/(2s)}
	 	\|f\|_{C^\alpha_{\ell,q_1}([0,t_0]\times\R^6)}^{(\alpha-\alpha'+2s)/\alpha' }
		\left(\|f\|_{C^\alpha_{\ell,q_1}([0,t_0]\times\R^6)} + \|Q_{\rm ns}(f,f)\|_{C^{\alpha'}_{\ell,q_1}([0,t_0]\times\R^6)}\right)\\
	 &\qquad\leq C t_0^{-1+(\alpha-\alpha')/(2s)}
	 	\|f\|_{C^\alpha_{\ell,q_1+{[\alpha/(1+2s)+\gamma]_+}}([0,t_0]\times\R^6)}^{1+(\alpha+2s)/\alpha'} ,
 \end{split}
 \end{equation*}
 using Lemma \ref{l:Q2holder} to bound $Q_{\rm ns}(f,f)$. 
\end{proof}

Next, we discuss higher regularity estimates for the solution $f$. The following proposition is in some sense a restatement of the main theorem of \cite{imbert2020smooth}, using hypotheses that are convenient for our purposes ($L^\infty_q$ bounds and pointwise lower bounds for $f$, rather than the mass, energy, and entropy density bounds used in \cite{imbert2020smooth}). This result extends the higher regularity estimates to the case $\gamma+2s<0$, although we should point out that the hypotheses here are stronger than in \cite{imbert2020smooth} and it is not currently known how to prove global regularity estimates {\it depending only on mass, energy, and entropy bounds} in the case $\gamma+2s<0$.

\begin{proposition}[Higher regularity]\label{p:higher-reg}
Let $f$ be a classical solution to \eqref{e.boltzmann} on $[0,T]\times\R^6$, and let $\tau \in (0,T)$. Assume that 
for some $\Omega \subseteq \R^3$, there exist $\delta, r, R>0$ such that for any $x\in \Omega$, there is a $v_x\in B_R$ such that $f(t,x,v) \geq \delta$ whenever $(x,v)\in B_r(x,v_x)\cap(\Omega\times\R^3)$. 
Fix any $m,k \geq 0$.   
Then there exists $q(k,m)$ such that, if $f \in L^\infty_{q(k,m)}([0,T]\times \R^6)$ and $j$ is a multi-index in $(t,x,v)$ with $|j| \leq k$, then
\be
	\label{e.higher-reg-est} 
	\|D^j f\|_{L^\infty_m([\tau,T]\times\Omega'\times\R^3)} \leq C.
\ee
The constant $q(k,m)$ depends  on $k$, $m$, $\delta$, $r$, $R$,  $\tau$, $T$, $\Omega'$, and $\Omega$. 
  The constant $C$ depends on the same quantities as well as
 $\|f\|_{L^\infty_{q(k,m)}}$. Furthermore, $q(k,m)$ and $C$ are nonincreasing functions of $\tau$.

If $\Omega = \R^3$, then we can replace $\Omega'$ with $\R^3$ in \eqref{e.higher-reg-est}.
\end{proposition}

If we were working on a periodic spatial domain and only considering the case $\gamma+2s\geq 0$, we could remove the dependence on higher $L^\infty_q$-norms of $f$, by using decay estimates as in \cite{imbert2018decay}. However, we need to apply these estimates on the whole space $\R^3_x$ and also $\gamma+2s<0$, so we state it in the form above. We intend to apply Proposition \ref{p:higher-reg} in situations where higher $L^\infty_q$-norms of $f$ are bounded in terms of the initial data and weaker norms of $f$.

The proof of Proposition \ref{p:higher-reg} is the same as the proof of \cite[Theorem 1.2]{imbert2020smooth} and consists of the following ingredients:
\begin{itemize}
\item The change of variables and global estimates described in this section.

\item The bootstrapping procedure explained in Section 9 of \cite{imbert2020smooth}, which consists of applying Theorems \ref{t:global-degiorgi} and \ref{t:global-schauder} to partial derivatives and increments of $f$. This makes use of certain facts about increments that are proven in Section 8 of \cite{imbert2020smooth}. The analysis in Sections 8 and 9 of \cite{imbert2020smooth} does not use the sign of $\gamma+2s$ in any way. 

At each step of the bootstrapping, a certain (non-explicit) number of velocity moments are used up.  The result of \cite{imbert2020smooth} proves estimates for all partial derivatives of $f$, but one can obviously stop the process after a finite number of iterations, which gives rise to the condition $f \in L^\infty_{q(k,m)}([0,T]\times\R^6)$ in our statement.

\item Bilinear estimates for the operators $Q_{\rm s}$ and $Q_{\rm ns}$ in H\"older spaces, which are also needed in the course of bootstrapping. These lemmas are also essentially independent of the sign of $\gamma+2s$, but we record them in Appendix \ref{s:lemmas} for the sake of completeness: see Lemmas \ref{l:Q1holder} and \ref{l:Q2holder}.
\end{itemize}

Finally, we have a continuation criterion for smooth solutions, which will be used in our proof of Theorem \ref{t:existence}. It is intended mainly as an internal result and works with solutions defined on a torus of general side length.  It is certainly not sharp.

\begin{proposition}[Continuation criterion]\label{p:continuation}
Let $\T_M^3$ be the periodic torus of side length $M>0$. 
Let $f$ be a classical solution to \eqref{e.boltzmann} in $[0,T)\times \T_M^3\times \R^3$ with $T>0$. 
Suppose that initial data $f_{\rm in}$ is smooth, is rapidly decaying in $v$, and that there is $\delta, r>0$ $x_m \in \T_M^3$, and $v_m\in \R^3$ such that 
\[
	f_{\rm in}(x,v)\geq \delta, \quad (x,v)\in  B_r(x_m, v_m).
\]

Then there exists $q_{\rm cont}$ such that, if 
\be\label{e.c62701}
	\|f\|_{L^\infty_{q_{\rm cont}}([0,T)\times \T^3_M\times\R^3)} < \infty,
\ee
then $f$ can be extended to a classical solution $[0,T+\eps]\times \T_M^3\times \R^3$ for some $\eps>0$.

The decay rate $q_{\rm cont}$ depends on $M$, $\gamma$, $s$, $T$, $\delta$, $r$, and $|v_m|$. 
The constant $\eps$ depends on the same quantities as well as  $\|f(t)\|_{L^\infty_{q_{\rm cont}}([0,T)\times \T^3_M\times\R^3)}$. Furthermore, $q_{\rm cont}$ is a nonincreasing function of $T$, and $\eps$ is a nondecreasing function of $T$.
\end{proposition}
\begin{proof}
First, by scaling, we may consider the case $M=1$.  Indeed, defining $f^M(t,x,v) := M^{\gamma+3}f(t,Mx,Mv)$, it is clear that: (i) $f^M$ solves~\eqref{e.boltzmann} on $[0,T)\times\T^3\times \R^3$, (ii) $f^M$ exists on $[0,T+\eps]\times\T^3\times\R^3$ if and only if $f$ exists on $[0,T+\eps]\times\T_M^3\times\R^3$, and (iii) inequality \eqref{e.c62701} holds if and only if $\|f^M\|_{L^\infty_{q_{\rm cont}}([0,T)\times\T^3\times\R^3)} < \infty$.  

Fix $k= 6$, and let $n$ and $p$ be the corresponding constants from \Cref{p:prior-existence}. Define $q_{\rm cont} = \max\{p, q(6,n+2)\}$, with $q(6,n+2)$ as in \Cref{p:higher-reg}, and note that $q_{\rm cont}$ depends on the quantities claimed in the statement, via the dependencies inherited by \Cref{p:higher-reg}.

We can apply \Cref{p:higher-reg} for all multi-indices $j$ of order at most $6$ to find 
\be\label{e.c62901}
	\sup_{|j| \leq 6} \|D^j f\|_{L^\infty_{n+2}([T/2,T]\times \T^3\times \R^3)}
		\leq C_0.
\ee
Note that the lower bound condition required to apply \Cref{p:higher-reg} follows from \Cref{t:lower-bounds} and the compactness of $\T^3$.  The constant $C_0$ in~\eqref{e.c62901} depends on $T$, $n$, $p$, $\delta$, $r$, $|v_m|$,  
and $\|f(t)\|_{L^\infty_{q_{\rm cont}}([0,T)\times \T^3\times\R^3)}$, 
again as a result of \Cref{p:higher-reg}.

We claim that if $\|f\|_{L^\infty_{q_{\rm cont}}}([0,T)\times\T^3\times\R^3)< \infty$, then $f$ can be extended past time $T$. Indeed, for any $t \in [T/2,T)$, the estimate \eqref{e.c62901} 
provides a uniform bound for
\be
	f(t) \in H^6_{n} \cap L^\infty_{q_{\rm cont}} (\T^3 \times \R^3)
\ee
depending only on the constants above.  Since $q_{\rm cont}\geq p$, and by our choice of $k$, $n$, and $p$, we may apply \Cref{p:prior-existence} at any $t\in [T/2,T)$ to obtain a solution $\tilde f$ in $C^0([t, t+\eps'), H^{k}_{n} \cap L^\infty_{q_{\rm cont}}(\T^3\times \R^3))$ with $\eps'$ depending on the constant $C_0$ in \eqref{e.c62901}. From Proposition \ref{p:higher-reg}, $C_0$ is a nonincreasing function of $\tau = T/2$, which implies $\eps'$ is nondecreasing in $T$.  Since $k= 6$, Sobolev embedding implies $\tilde f$ is a classical solution (in particular, it is twice differentiable in $x$ and $v$, and, via the equation \eqref{e.boltzmann}, once differentiable in $t$). The proof is then finished by choosing $\eps = \eps'/2$ and $t \in (T-\eps'/2,T)$ and concatenating $f$ and $\tilde f$.
\end{proof}

\section{Existence of solutions}\label{s:existence}

This section is devoted to the proof of Theorems \ref{t:existence} and \ref{t:weak-solutions}.

\subsection{Decay estimates}

We begin with novel decay estimates that are needed for our construction. These estimates are stated for any suitable solution $f$ of the Boltzmann equation \eqref{e.boltzmann} on a periodic spatial domain $\T_M^3$, i.e. the torus of side length $M>0$. In Section \ref{s:construct} below, we apply these estimates to our approximating sequence. Throughout this subsection, we assume the initial data corresponding to $f$ satisfies
\begin{equation}\label{e.initial-data-good}
f_{\rm in} \in C^\infty(\T^3_M\times \R^3) \cap L^\infty_{q'}(\T_M^3\times\R^3) \quad \text{for all } q'\geq 0.
\end{equation}
However, we do not assume {\it a priori} that $f$ satisfies polynomial decay of all orders for positive times.

First, using a barrier argument, we show that polynomial upper bounds of order larger than $3+\gamma+2s$ are propagated forward in time: 
\begin{lemma}\label{l:simple-bound}
Let $q_0 > 3+\gamma+2s$ and $q\in [q_0,q_0-\gamma]$ be fixed. Let $f$ be a solution of \eqref{e.boltzmann} on $[0,T]\times\T_M^3\times\R^3$ for some $M>0$, and assume $f\in L^\infty_{q'}([0,T]\times\T_M^3\times\R^3)$  for some $q'> q$ and that $f_{\rm in}$ satisfies \eqref{e.initial-data-good}. Then
\[
	\|f(t)\|_{L^\infty_{q}(\T^3_{M}\times \R^3)}
		\leq  \|f_{\rm in}\|_{L^\infty_{q}(\T_M^3\times \R^3)} \exp(C_0 \|f\|_{L^\infty_{q_0}([0,T]\times\T_M^3\times\R^3)} t), \quad 0\leq t\leq T,\]
for a constant $C_0>0$ depending only on universal quantities and $q$ and $q_0$. In particular, $C_0$ is independent of $f_{\rm in}$ and the norm of $f$ in $L^\infty_{q'}$.
\end{lemma}

\begin{proof}
Define the barrier function $g(t,x,v) =N e^{\beta t} \vv^{-q}$, with $N,\beta>0$ to be chosen later. By taking $N> \|f_{\rm in}\|_{L^\infty_{q}(\T^3_{M_\eps}\times \R^3)}$, we ensure $f(0,x,v) < g(0,x,v)$ for all $x$ and $v$. We would like to show
\begin{equation}\label{e.claim}
f(t,x,v) < g(t,x,v), \quad (t,x,v) \in [0,T]\times \T_M^3\times \R^3.
\end{equation}
If this bound fails, then because $f$ decays at a rate strictly faster than $\vv^{-q}$, and $f$ is periodic in the $x$ variable, there must be a first time $t_\rmcr\in (0,T]$ and location $(x_\rmcr,v_\rmcr)$ at which $f$ and $g$ cross. At the first crossing time, we have the following equalities and inequalities:
\[ 
\begin{split}
\partial_t f(t_\rmcr,x_\rmcr,v_\rmcr) &\geq \partial_t g(t_\rmcr,x_\rmcr,v_\rmcr),\\
 \nabla_x f(t_\rmcr,x_\rmcr,v_\rmcr) &= 0,\\
 f(t_\rmcr,x,v) &\leq g(t_\rmcr,x,v), \quad x\in \TM, v\in \R^3.
 \end{split}
 \]
Combining this with the Boltzmann equation \eqref{e.boltzmann}, we have
\begin{equation}\label{e.crossing}
\partial_t g(t_\rmcr,x_\rmcr,v_\rmcr) \leq Q(f,f) (t_\rmcr,x_\rmcr,v_\rmcr) \leq Q(f,g)(t_\rmcr,x_\rmcr,v_\rmcr).
\end{equation}
To justify the last inequality, write $Q(f,f) = Q_{\rm s}(f,f) + Q_{\rm ns}(f,f)$ and use the nonnegativity of the kernel $K_{f}$ to  obtain
\[ \begin{split}
	Q_{\rm s} (f, f)(t_\rmcr,x_\rmcr,v_\rmcr)
		&= \int_{\R^3} K_{f}(v,v') [f(t_\rmcr,x_\rmcr,v') - f(t_\rmcr,x_\rmcr,v_\rmcr)] \dd v'\\
		&\leq \int_{\R^3} K_{f}(v,v') [g(t_\rmcr,x_\rmcr,v') - g(t_\rmcr,x_\rmcr,v_\rmcr)] \dd v',
 \end{split}
 \]
since $f(t_\rmcr,x_\rmcr,v) \leq g(t_\rmcr,x_\rmcr,v)$ and $f(t_\rmcr,x_\rmcr,v_\rmcr) = g(t_\rmcr,x_\rmcr,v_\rmcr)$.
Next, realize that $Q_{\rm ns}(f,f)(v) = f(t_\rmcr,x_\rmcr,v) [f\ast |\cdot|^\gamma](t_\rmcr,x_\rmcr,v) \leq g(t_\rmcr,x_\rmcr,v)[f\ast |\cdot|^\gamma](t_\rmcr,x_\rmcr,v)$. We have established the last inequality in \eqref{e.crossing}. 

From Lemma \ref{l:Q-polynomial} we have, for some $C_0>0$ as in the statement of the lemma,
\[ 
Q(f,g)(t_\rmcr,x_\rmcr,v_\rmcr) = N e^{\beta t_\rmcr} Q(f,\langle \cdot \rangle^{-q})(t_\rmcr,x_\rmcr,v_\rmcr) \leq C_0 N e^{\beta t_\rmcr} \|f(t_\rmcr)\|_{L^\infty_{q_0}(\T_M^3 \times \R^3)} \langle v_\rmcr\rangle^{-q}.
\]
Together with \eqref{e.crossing} and $\partial_t g = N\beta e^{\beta t} \vv^{-q}$, this implies
\[ N \beta e^{\beta t_\rmcr} \langle v_\rmcr \rangle^{-q} \leq C_0 N e^{\beta t_\rmcr} \|f\|_{L^\infty_{q_0}([0,T]\times\TM\times\R^3)} \langle v_\rmcr\rangle^{-q},\]
which is a contradiction if we choose $\beta = 2C_0\|f\|_{L^\infty_{q_0}([0,T]\times\TM\times\R^3)}$. Hence,~\eqref{e.claim} must hold.

After choosing 
\[
	N = \|f_{\rm in}\|_{L^\infty_{q}([0,T]\times\TM\times\R^3)} + \nu
\] 
in \eqref{e.claim} and sending $\nu \to 0$, the proof is complete.
\end{proof}

Consider the $q= q_0$ case of Lemma \ref{l:simple-bound}. Intuitively, we would like to iterate this estimate to obtain a uniform \emph{a priori} bound on $\|f(t)\|_{L^\infty_{q_0}}$ up to some positive time. To do this precisely, we use the following technical lemma, which encodes the result of such an iteration.
\begin{lemma}{\cite[Lemma 2.4]{HST2020landau}}\label{l:annoying}
If $H:[0,T]\to \R_+$ is a continuous increasing function, and $H(t) \leq A e^{Bt H(t)}$ for all $t\in [0,T]$ and some constants $A, B > 0$, then
\[ H(t) \leq e A \quad \text{ for } \quad 0\leq t\leq  \min\left( T, \frac 1 {eAB}\right).\]
\end{lemma}

Next, we prove the key result of this subsection, which allows us to bound higher $L^\infty_q$ norms of $f(t)$ in terms of lower decay norms and the initial data:

\begin{lemma}\label{p:upper-bounds}
Let $q_{\rm base}>3+\gamma+2s$ be fixed, and let $f\in L^\infty_{q_{\rm base}+\eps}([0,T]\times\TM\times\R^3)$ be a solution to \eqref{e.boltzmann}, with $f_{\rm in}$ satisfying \eqref{e.initial-data-good} and $\eps>0$. 
There exists $C>0$ depending on universal quantities and $q_{\rm base}$, such that for $T_f$ given by 
\be\label{e.c080401}
	T_f = \frac C {\|f_{\rm in}\|_{L^\infty_{q_{\rm base}}(\TM\times \R^3)}},
\ee
the following hold: 
\begin{enumerate}
\item[(a)] The solutions $f$ satisfy
\[ 
\|f(t)\|_{L^\infty_{q_{\rm base}}(\TM\times\R^3)} \leq C\|f_{\rm in}\|_{L^\infty_{q_{\rm base}}(\TM\times \R^3)}, \quad 0\leq t\leq \min(T_f,T).
\]
\item[(b)] If $q> q_{\rm base}$ and  $f\in L^\infty_{q+\eps}([0,T]\times\TM\times\R^3)$, there holds
\begin{equation}\label{e.epsilon-bound}
	 \|f(t)\|_{L^\infty_q(\TM\times\R^3)}
	 	\leq \|f_{\rm in}\|_{L^\infty_q(\TM\times\R^3)} \exp\left[M_q\left(t,\|f_{\rm in}\|_{L^\infty_q(\TM\times\R^3)}\right) \right], \quad 0\leq t\leq \min(T_f,T),
 \end{equation}
for some increasing function $M_q:\R_+\times \R_+\to \R_+$ depending on universal constants, $q$, $q_{\rm base}$, and $\|f_{\rm in}\|_{L^\infty_{q_{\rm base}}(\T_M^3\times\R^3)}$. 
\end{enumerate}
\end{lemma}
\begin{proof}
To prove (a), for any $t\in (0,\min(T_f,T)]$, we define 
\[
H(t):= \|f\|_{L^\infty_{q_{\rm base}}([0,t]\times \TM\times\R^3)}.
\]
Applying Lemma \ref{l:simple-bound} with $T = t$, we obtain $H(t) \leq \|f_{\rm in}\|_{L^\infty_{q_{\rm base}}(\TM\times\R^3)} \exp( C_0 H(t) t)$.   Lemma \ref{l:annoying} with $A = \|f_{\rm in}\|_{L^\infty_{q_{\rm base}}(\TM\times\R^3)}$ and $B = C_0$ implies 
\begin{equation}\label{e.propagation}
	\|f(t)\|_{L^\infty_{q_{\rm base}}(\TM\times\R^3)}
		\leq C \|f_{\rm in}\|_{L^\infty_{q_{\rm base}}(\TM\times\R^3)},
		\quad 0\leq t\leq \min(T_f,T),
 \end{equation}
 with $T_f$ as in the statement of the proposition.  This establishes (a). 

For (b), given $q>q_{\rm base}$, let $N\in \Z_{\geq 0}$ and $\theta\in [0,1)$ be such that $q = q_{\rm base} + (N+\theta)|\gamma|$. Applying Lemma \ref{l:simple-bound}, followed by \eqref{e.propagation}, we have 
\[
\begin{split}
	\|f(t)\|_{L^\infty_{q_{\rm base}+|\gamma|}(\TM\times\R^3)}
		&\leq \|f_{\rm in}\|_{L^\infty_{q_{\rm base}+|\gamma|}(\TM\times\R^3)}\exp(C_0 \|f\|_{L^\infty_{q_{\rm base}}([0,T]\times\TM\times\R^3)} t)
		\\&
		\leq \|f_{\rm in}\|_{L^\infty_{q_{\rm base}+|\gamma|}(\TM\times\R^3)}\exp(C C_0 \|f_{\rm in}\|_{L^\infty_{q_{\rm base}}(\TM\times\R^3)} t),
 \end{split}
 \]
 for $0\leq t\leq \min(T_f,T)$, where $C$ is the constant from \eqref{e.propagation}  and $C_0$ is the constant from \Cref{l:simple-bound}. 
We have shown that, up to time $\min(T_f,T)$, the bound \eqref{e.epsilon-bound} holds in the case $q= q_{\rm base} + |\gamma|$, with $M_{q_{\rm base}+|\gamma|}(t,z) = C t \|f\|_{L^\infty_{q_{\rm base}}}$.

We now iterate this argument $N$ times to obtain
 \[
 \begin{split}
	&\|f(t)\|_{L^\infty_{q_{\rm base}+N|\gamma|}(\TM\times\R^3)}\\
	  &~~\leq \|f_{\rm in}\|_{L^\infty_{q_{\rm base}+N|\gamma|}(\TM\times\R^3)}\exp(C \|f\|_{L^\infty_{q_{\rm base}+(N-1)|\gamma|}([0,T]\times\TM\times\R^3)} t)\\
	 &~~\leq \|f_{\rm in}\|_{L^\infty_{q_{\rm base}+N|\gamma|}(\TM\times\R^3)}\exp\left(C \|f_{\rm in}\|_{L^\infty_{q_{\rm base}+(N-1)|\gamma|}} \exp\left[ M_{q_{\rm base}+(N-1)|\gamma|}\left(t,\|f_{\rm in}\|_{L^\infty_{q_{\rm base}+(N-1)|\gamma|}}\right)\right] t\right)
	\\&~~
		= \|f_{\rm in}\|_{L^\infty_{q_{\rm base}+N|\gamma|}(\TM\times\R^3)}\exp\left(M_{q_{\rm base} + N |\gamma|}(t, \|f_{\rm in}\|_{L^\infty_{q_{\rm base}+N|\gamma|}})\right),
 \end{split}
 \]
where we use the recursive definition $M_{q_{\rm base}+N|\gamma|}(t,z) = C t z \exp(M_{q_{\rm base}+(N-1)|\gamma|}(t,z))$. 
Finally, for any small $\nu>0$, we apply Lemma \ref{l:simple-bound} again, with exponents $q-\nu =q_{\rm base} + (N+\theta)|\gamma| - \nu$ and $q_{\rm base}+N|\gamma|$,  and argue similarly to obtain
\begin{equation}\label{e.propagation-q}
 \|f(t)\|_{L^\infty_{q-\nu}(\TM\times\R^3)} \leq \|f_{\rm in}\|_{L^\infty_q(\TM\times\R^3)} \exp\left[ M_q\left(t,\|f_{\rm in}\|_{L^\infty_q(\TM\times\R^3)}\right)\right], \quad 0\leq t\leq \min(T_f,T),
 \end{equation}
 with $M_q(t,z) = C_0 z t  \exp(M_{q_{\rm base}+N|\gamma|}(t,z))$. 
 
 The functions $M_q(t,z)$ depend on $\|f\|_{L^\infty_{q_{\rm base}}([0,t]\times \T_{M}^3\times \R^3)}$, but this quantity is bounded in terms of $\|f_{\rm in}\|_{L^\infty_{q_{\rm base}(\T_M^3\times\R^3)}}$, by (a).
 
 Since $q-\nu< q$, our applications of Lemma \ref{l:simple-bound} are justified. The right-hand side is independent of $\nu$, so we can send $\nu\to 0$ and conclude (b). 
 \end{proof}

\subsection{Construction of approximate solutions}\label{s:construct}

Consider initial data $f_{\rm in}\in L^\infty_{q_{\rm base}}(\T_M^3\times\R^3)$ with $q_{\rm base}>3+\gamma+2s$. Our first step is to approximate $f_{\rm in}$ by smoothing, cutting off large values of $x$ and $v$, extending by $x$-periodicity, and adding a region of uniform positivity. This will give rise to a sequence of solutions $f^\eps$ that solve \eqref{e.boltzmann} with initial data $f_{\rm in}$ in the limit as $\eps\to 0$. The same construction of $f^\eps$ will be used to build classical solutions and weak solutions.

In more detail, for any $r\in (0,1)$, define the following functions.  First, let $\psi:\R^6\to \R_+$ be a standard smooth mollifier supported in $B_1(0)$ with $\int_{B_1(0)} \psi \dd x = 1$, and then denote
\be
	\psi_\eps(x,v) = \eps^{-6}\psi(x/\eps,v/\eps).
\ee
Next, for any $r>0$, let $\zeta_r:\R^3\to \R_+$ be a smooth cutoff such that $\sup_{\R^3} |\nabla \zeta_r|\lesssim 1$ and 
\be
	\zeta_r(\xi) = \begin{cases}
				1 \quad \text{ for } \xi \in B_{1/r}\\
				0 \quad \text{ for } \xi \in B_{1/r+1}^c,
			\end{cases}
\ee
Then, for any $\eps>0$, define
\begin{equation}\label{e.f0eps}
 f_{\rm in}^\eps(x,v) := \zeta_\eps(x)\zeta_\eps(v) [f_{\rm in}\ast \psi_\eps](x,v)+ \eps \psi(x,v).
\end{equation}
Letting $\TMeps$ be the three-dimensional torus of side length $M_\eps := 2(1/\eps+2)$ centered at $(0,0,0)$, we extend $f_{\rm in}^\eps$ by $x$-periodicity to obtain a smooth function on $\TMeps\times\R^3$, or equivalently, a smooth function on $\R^6$ that is $M_\eps$-periodic in the $x$ variable.
  
The following construction of an approximate solution $f^\eps$ is more intricate than one might expect. First, the time of existence provided to us by Proposition \ref{p:prior-existence} depends on the $H^k_n \cap L^\infty_p$ space one chooses, so we need an extra argument to obtain a smooth, rapidly decaying solution on a uniform time interval, even though $f_{\rm in}^\eps$ is smooth and rapidly decaying. Second, the exponent $q_{\rm cont}$ in the continuation criterion of Proposition \ref{p:continuation} may degenerate to $+\infty$ as $T\searrow 0$ (notice the dependence on $T$ of $q_{\rm cont}$ in \Cref{p:continuation}), so for small times we must perform the continuation ``by hand'' using Proposition \ref{p:prior-existence} and our decay estimates above.


For each $\eps>0$, by Proposition \ref{p:prior-existence}, there is a time $T_\eps>0$ and a solution $f^\eps(t,x,v) \geq 0$ defined on $[0,T_\eps]\times \T_{M_\eps}^3\times \R^3$, continuous down to $t=0$, with $f^\eps(0,x,v) =f_{\rm in}^\eps(x,v)$. We assume $T_\eps \leq T_f$ from Lemma \ref{p:upper-bounds}, since our goal is to show $f^\eps$ exists up to time $T_f$.  Noting that $f_{\rm in}^\eps$ is smooth in $(x,v)$ and rapidly decaying in $v$, we choose some large (fixed) values of $k$, $n$, and $p$ when we apply Proposition \ref{p:prior-existence}. We then have
\begin{equation}\label{e.epsilon-smooth}
f^\eps(t) \in H^k_{n}\cap L^\infty_{p}(\T_{M_\eps}^3\times\R^3) \quad\text{for } t\in [0,T_\eps].
\end{equation}
Choosing $T_\eps$ smaller if necessary, we also have
\begin{equation}\label{e.twice}
\|f^\eps(t)\|_{H^k_n(\T_{M_\eps}^3\times\R^3)} \leq 2\left(\|f_{\rm in}^\eps\|_{H^k_n(\T_{M_\eps}^3\times\R^3)} + \|f_{\rm in}^\eps\|_{L^\infty_p(\T_{M_\eps}^3\times\R^3)}\right), \quad t\in [0,T_\eps].
\end{equation}
Now, let $q>p$ be arbitrary. Applying Proposition \ref{p:prior-existence} a second time in the space $H^k_n \cap L^\infty_{q}$, we see there is some $T_{\eps,q}\in (0,T_\eps]$ such that $\|f^\eps(t)\|_{L^\infty_{q}(\T_{M_\eps}^3\times\R^3)} < \infty$ when $t\in [0,T_{\eps,q}]$. Lemma \ref{p:upper-bounds}(b) then implies estimate \eqref{e.epsilon-bound} holds up to time $T_{\eps,q}$. Since the right-hand side of \eqref{e.epsilon-bound} is bounded uniformly in $t\in [0,T_{\eps}]$, 
we can combine this with \eqref{e.twice} to bound the norm of $f^\eps(t)$ in the space $H^k_n \cap L^\infty_q$ by a constant depending only on $T_\eps$ and the initial data $f_{\rm in}^\eps$. We apply Proposition \ref{p:prior-existence} again to conclude $f^\eps$ lies in $L^\infty_q([0,T_{\eps,q} + T_{\eps,q}']\times\T_{M_\eps}^3\times\R^3)$ for some $T_{\eps,q}'$ depending only on the upper bound in \eqref{e.epsilon-bound}. Lemma \ref{p:upper-bounds} then implies the estimate \eqref{e.epsilon-bound} can be extended to the time interval $[0,T_{\eps,q} + T_{\eps,q}']$. Combining this with \eqref{e.twice}, the process can be iterated finitely many times until $f\in L^\infty_q([0,T_\eps]\times\T_{M_\eps}^3\times\R^3)$, with estimate \eqref{e.epsilon-bound} valid up to time $T_\eps$. Since $q>p$ was arbitrary, we conclude that all $L^\infty_q(\T_{M_\eps}^3\times\R^3)$ norms of $f^\eps(t)$ are finite, with the estimate \eqref{e.epsilon-bound} valid, up to time $T_\eps$.

Note that $f_{\rm in}^\eps$ satisfies the lower bound condition \eqref{e.mass-core} with $\delta =\eps$, $r =1$, and $(x_m, v_m) = (0,0)$. 
Additionally, we can apply \Cref{p:upper-bounds} to control the $L^\infty_q$-norm of $f^\eps$ on $[0,T_\eps]\times\T_{M_\eps}^3\times\R^3$ for any $q > q_{\rm base}$, which, since the $L^\infty_q$-norms are increasing in $q$, yields a bound on the $L^\infty_{q_{\rm cont}}$-norm of $f$ on $[0,T_\eps]\times\T_{M_\eps}^3\times\R^3$.
Combining these two ingredients, 
we can apply the continuation criterion of Proposition \ref{p:continuation} to extend the solution $f^\eps$ to a time interval $T_\eps + \eta$, with $\eta$ depending on $T_\eps$, the initial data, and the $L^\infty_{q_{\rm cont}}$-norm of $f^\eps$ on $[0,T_\eps]\times\T_{M_\eps}^3\times\R^3$. 
Since Proposition \ref{p:continuation} implies $f^\eps(t) \in L^\infty_{q_{\rm cont}}$ for $t< T_\eps + \eta$, Lemma \ref{p:upper-bounds} tells us that the bound \eqref{e.epsilon-bound} holds up to time $T_\eps+\eta$, and we can repeat this argument finitely many times until $f^\eps$ exists up to time $T_f$ and lies in $L^\infty_{q_{\rm cont}}([0,T_f]\times\T_{M_\eps}^3\times\R^3)$.


 Since Proposition \ref{p:continuation} implies $f^\eps(t) \in L^\infty_{q_{\rm cont}}$ for $t< T_\eps + \eta$, Lemma \ref{p:upper-bounds} tells us that the bound \eqref{e.epsilon-bound} holds up to time $T_\eps+\eta$, and we can repeat this argument finitely many times until $f^\eps$ exists up to time $T_f$ and lies in $L^\infty_{q_{\rm cont}}([0,T_f]\times\T_{M_\eps}^3\times\R^3)$. 

To extend higher $L^\infty_q$ estimates up to time $T_f$, we combine Proposition \ref{p:prior-existence} and Lemma \ref{p:upper-bounds} in the same way as above. We omit the details of this step. We now have a solution $\tilde  f^\eps \in L^\infty_q([0,T_f]\times\T_{M_\eps}^3\times\R^3)$ for every $q \geq q_{\rm base}$. By \Cref{p:higher-reg}, $f^\eps$ is also smooth, with regularity estimates depending on $\eps$. 

Let us summarize the results of the last two subsections:

\begin{proposition}\label{p:upper-bounds2}
Let $q_{\rm base}> 3+\gamma+2s$, and let  $T_f = C \|f_{\rm in}\|_{L^\infty_{q_{\rm base}}}^{-1}$ as in Lemma \ref{p:upper-bounds}. For any $\eps>0$, with $f_{\rm in}^\eps$ defined as in \eqref{e.f0eps}, there exist smooth, rapidly decaying solutions $f^\eps\geq 0$ to \eqref{e.boltzmann} on $[0,T_f]\times\T_{M_\eps}^3\times\R^3$ with initial data $f_{\rm in}^\eps$. These $f^\eps$ satisfy the estimates
\begin{equation}\label{e.base-est}
	\|f^\eps(t)\|_{L^\infty_{q_{\rm base}}(\TMeps\times\R^3)} 
		\leq C\|f_{\rm in}^\eps\|_{L^\infty_{q_{\rm base}}(\TMeps\times \R^3)},
			\quad 0\leq t\leq T_f.
\end{equation}
and for all $q \geq q_{\rm base}$,
\begin{equation}\label{e.higher-decay-est}
	 \|f^\eps(t)\|_{L^\infty_q(\TMeps\times\R^3)}
	 	\leq \|f_{\rm in}^\eps\|_{L^\infty_q(\TMeps\times\R^3)}
			\exp\left[M_q\left(t,\|f_{\rm in}^\eps\|_{L^\infty_q(\TMeps\times\R^3)}\right) \right],
				\quad 0\leq t\leq T_f,
\end{equation}
with $M_q$ as in Lemma \ref{p:upper-bounds}.
\end{proposition}

Recall that $\|f_{\rm in}\|_{L^\infty_{q_{\rm base}}(\R^6)}< \infty$ by assumption. For $\eps \in (0,1)$, we clearly have
$$
	\|f_{\rm in}^\eps\|_{L^\infty_{q_{\rm base}
 }(\TMeps\times\R^3)}
 		\leq C_{q_{\rm base}}\left(\|f_{\rm in}\|_{L^\infty_{q_{\rm base}}(\R^6)} + \eps\right).
$$
Therefore, the right-hand side of \eqref{e.base-est} is independent of $\eps \in (0,1)$.  This, in turn, implies that the left-hand side of~\eqref{e.base-est} is independent of $\eps$ as well.

On the other hand, the right-hand side of \eqref{e.higher-decay-est} is independent of $\eps$ if and only if $f_{\rm in} \in L^\infty_q(\R^6)$. Even if this upper bound is not uniform in $\eps$, the quantities are still finite up to time $T_f$ (which is independent of $\eps$). 
This fact allows us to apply the continuation criterion of Proposition \ref{p:continuation} even if $f_{\rm in} \not \in L^\infty_{q_{\rm cont}}(\R^6)$.

In the next subsection, we derive regularity estimates that are uniform in $\eps$.

\subsection{Regularity of $f^\eps$ for positive times}

In the next two subsections, we prove the existence of classical solutions (Theorem \ref{t:existence}). Therefore, we work under the assumption that $f_{\rm in}$ satisfies the quantitative lower bound \eqref{e.mass-core}, and that $f_{\rm in}\in L^\infty_{q_0}(\R^6)$ with $q_0>3+2s$.

We no longer need to use the compactness of our spatial domain. From now on, we consider $f^\eps$ to be defined on $[0,T_f]\times\R^6$, periodic in $x$ with period $M_\eps$. Recall that $M_\eps\to \infty$ as $\eps\to 0$ and that $T_f$ is defined in~\eqref{e.c080401}. The estimates in this subsection apply on domains that are bounded in the $x$ variable, so for any fixed such domain, the $x$-periodicity is irrelevant for $\eps$ small enough. For brevity, we implicitly assume throughout this subsection that $\eps$ is small enough for any statement we make about a bounded $x$ domain.

In order to apply regularity estimates in an $\eps$-independent way, we first need suitable lower bounds for the solutions $f^\eps$ for positive times. 
Importantly, the bound on $\|f^\eps\|_{L^\infty_{q_0}([0,T_f]\times \R^6)}$ in~\eqref{e.base-est} is $\eps$-independent, and this is the crucial bound on which all others depend. 
 For $\eps$ sufficiently small depending on $\delta$, $|x_m|$, and $|v_m|$, the hypothesis \eqref{e.mass-core} implies
\be\label{e.c62704}
	f_{\rm in}^\eps(x,v) \geq \frac \delta 2, \quad (x,v) \in B_r(x_m,v_m).
\ee
Applying Theorem \ref{t:lower-bounds} to the smooth solutions $f^\eps$, we have
\begin{equation}\label{e.uniform-lower-b}
 f^\eps(t,x,v) \geq \mu(t,x) e^{-\eta(t,x) |v|^2},
 \end{equation}
with $\mu$ and $\eta$ uniformly positive and bounded on any compact subset of $(0,T_f]\times\R^3$, and depending only on $\delta$, $r$, $t$, $|x-x_m|$, $v_m$, and $\|f^\eps\|_{L^\infty_{q_0}([0,T_f]\times\R^6)}$. 
Here we used that $\|f^\eps\|_{L^\infty_{q_0}([0,T_f]\times \R^6)}$ controls the integral quantities~\eqref{e.hydro-general} in the conditions of \Cref{t:lower-bounds}. 
 Because of Proposition \ref{p:upper-bounds2}, the norm $\|f^\eps\|_{L^\infty_{q_0}([0,T_f]\times\R^6)}$ is bounded above by a constant times $\|f_{\rm in}\|_{L^\infty_{q_0}(\R^6)}$, and therefore $\mu$ and $\eta$ can be chosen independently of $\eps$.

Now we apply regularity estimates. Let $z_0 = (t_0,x_0,v_0)\in (0,T_f]\times\R^6$,  and let $r_0 \leq  \min(1,(t_0/2)^{1/(2s)})$.  With $\bar c$ as in Theorem \ref{t:global-degiorgi}, choose $\alpha>0$ small enough that $q_1 := q_0 - \bar c \alpha > 3+2s$. Since $q_0>3+2s>3$ and \eqref{e.uniform-lower-b} holds, Theorem \ref{t:global-degiorgi} gives
\begin{equation}\label{e.C-alpha-ell}
  \|f^\eps\|_{C^\alpha_{\ell, q_1}(Q_{r_0}^{t,x}(z_0)\times\R^3)} \leq C_{t_0} \|f^\eps\|_{L^\infty_{q_0}([0,T]\times\R^6)},
  \end{equation}
  where $Q_{r_0}^{t,x}(z_0) := (t_0-r_0^{2s},t_0]\times \{x: |x-x_0 - (t-t_0)v_0| < r_0^{1+2s}\}$.
  
The next step is to apply Schauder. Since we only assume decay of order $q_0>3+2s$ for $f_{\rm in}$, we cannot afford to use the global (in $v$) Schauder estimate of Theorem \ref{t:global-schauder}, so we proceed with the local Schauder estimate of Theorem \ref{t:schauder} instead. For  any $z=(t,x,v)\in Q_{r_0}(z_0)$, we define $K_{f^\eps,z}(w) = K_{f^\eps}(t,x,v,v+w)$. We need to check that the kernel satisfies the H\"older hypothesis in Theorem \ref{t:schauder}:

\begin{lemma}\label{l:f-eps-kernel}
With $q_1$, $\alpha$, $z_0$, $r_0$, and $f^\eps$ as above but under the additional condition that $(q_1 - 3 - \gamma-2s)(1+2s) > \alpha$, the kernel $K_{f^\eps,z}(w)$ satisfies the H\"older continuity condition \eqref{e.kernel-holder}, with constant $A_0$ depending on universal constants, $q_1$, $\alpha$, $t_0$, $x_0$, $r_0$, $\delta$, $|v_m|$, $|x_0-x_m|$, and $\|f_{\rm in}\|_{L^\infty_{q_0}}$.  It has no dependence on $\eps$ as long as $\eps$ is small enough that~\eqref{e.c62704} holds.
\end{lemma}
\begin{proof}
For $z_1,z_2 \in Q_{r_0}(z_0)$, we have
\[
\begin{split}
 K_{f^\eps,z_1}(w) - K_{f^\eps,z_2}(w) &= |w|^{-3-2s} \int_{\{h\cdot w = 0\}} |h|^{\gamma+2s+1} [f^\eps(z_1\circ(0,0,h)) - f^\eps(z_2\circ(0,0,h))] \tilde b(w,h) \dd h\\
 & = K_{g,z_1}(w),
 \end{split}
 \]
where $g(z) = f^\eps(z) - f^\eps((z_2\circ z_1^{-1}) \circ z)$. With Lemma \ref{l:K-upper-bound-2}, this implies, for $\rho>0$,
\begin{equation}\label{e.Brho}
\begin{split}
\int_{B_\rho}&|K_{f^\eps,z_1}(w) - K_{f^\eps,z_2}(w)||w|^2 \dd w = \int_{B_\rho} |K_{g,z_1}(w)| |w|^2\dd w\\
&\leq \left(\int_{\R^3}|w|^{\gamma+2s} |g(t_1,x_1,v_1+w) |\dd w \right) \rho^{2-2s}\\
&= \left(\int_{\R^3}|w|^{\gamma+2s} |f^\eps(z_1\circ(0,0,w)) - f^\eps(z_2\circ(0,0,w))| \dd w \right) \rho^{2-2s}.
\end{split}
\end{equation}
Next, we estimate $|f^\eps(z_1\circ(0,0,w)) - f^\eps(z_2\circ(0,0,w))|$. Note that for $w\in \R^3$, one has $z_i\circ(0,0,w) \in [t_0/4, T_f]\times\R^6$ for $i=1,2$. Once again, we recall that $T_f$ is defined in~\eqref{e.c080401} We claim that, for any $w\in \R^3$,
%
\[
\begin{split}
	 &|f^\eps(z_1\circ(0,0,w)) - f^\eps(z_2\circ(0,0,w))| 
		\\&\qquad\lesssim \|f^\eps\|_{C^\alpha_{\ell,q_1}(Q_{r_0}^{t,x}(z_0)\times\R^3)} \langle v_1 + w\rangle^{-q_1} d_\ell(z_1\circ(0,0,w),z_2\circ(0,0,w))^\alpha,
 \end{split}
 \]
 where $q_1 = q_0 -\bar c\alpha$ as above. Indeed, this formula follows by using the seminorm $[f^\eps]_{C^\alpha_{\ell,q_1}(Q_{r_0}^{t,x}(z_0)\times\R^3)}$ when $d_\ell(z_1\circ(0,0,w),z_2\circ(0,0,w))< 1$ and the norm $\|f^\eps\|_{L^\infty_{q_1}(Q_{r_0}^{t,x}(z_0)\times\R^3)}$ when $d_\ell(z_1\circ(0,0,w),z_2\circ(0,0,w))\geq 1$. We have also used $\langle v_1 + w\rangle \approx \langle v_2 + w\rangle$, since $v_1, v_2 \in B_{r_0}(v_0)$.
 
 Now, using \eqref{e.right-translation}, we have
 \[ 
\begin{split}
	|f^\eps(z_1\circ(0,0,w)) &- f^\eps(z_2\circ(0,0,w))| \\
		& \lesssim \|f^\eps\|_{C^\alpha_{\ell,q_1}(Q_{r_0}^{t,x}(z_0)\times\R^3} \langle v_1 + w\rangle^{-q_1} 
			\Big(d_\ell(z_1,z_2) + d_\ell(z_1,z_2)^\frac{2s}{1+2s} |w|^\frac{1}{1+2s}\Big)^{\alpha},
\end{split}
\]
Returning to \eqref{e.Brho}, we have
\begin{equation}
\begin{split}
	&\int_{B_\rho}|K_{f^\eps,z_1}(w) - K_{f^\eps,z_2}(w)||w|^2 \dd w\\
		 &\leq \rho^{2-2s}\|f^\eps\|_{C^\alpha_{\ell,q_1}(Q_{r_0}^{t,x}(z_0)\times\R^3)}
		 	\int_{\R^3} |w|^{\gamma+2s} \langle v_1 + w\rangle^{-q_1} 
				\Big(d_\ell(z_1,z_2) + d_\ell(z_1,z_2)^\frac{2s}{1+2s} |w|^\frac{1}{1+2s}\Big)^{\alpha} \dd w\\
		 &\leq \rho^{2-2s}\|f^\eps\|_{C^\alpha_{\ell,q_1}(Q_{r_0}^{t,x}(z_0)\times\R^3)} \langle v_0\rangle^{\gamma+2s+\frac{\alpha}{1+2s}} d_\ell(z_1,z_2)^{\alpha'},
\end{split}
\end{equation}
with $\alpha' = \alpha \frac {2s}{1+2s}$. We have used $q_1>3+2s > 3+\gamma+2s +\alpha/(1+2s)$. 
Applying \eqref{e.C-alpha-ell}, we see that $\|f^\eps\|_{C^\alpha_{\ell,q_1}(Q_{r_0}^{t,x}(z_0)\times\R^3)}$ is bounded independently of $\eps$. This implies \eqref{e.kernel-holder} holds, as in the statement of the lemma. 
\end{proof}

Because of Lemma \ref{l:f-eps-kernel}, we may apply Theorem \ref{t:schauder} to $f^\eps$ and obtain
\[ 
\|f^\eps\|_{C^{2s+\alpha'}_\ell(Q_{r_0/2}(z_0))} \leq C_0,
\]
with $C_0$ depending on $t_0$, $|v_0|$, and the initial data $f_{\rm in}$, but independent of $\eps$.

\subsection{Convergence as $\eps \to 0$ and the conclusion of \Cref{t:existence}}

For each compact subset $\Omega\subset (0,T_f]\times  \R^6$, our work above implies that $f^\eps$ is bounded in $C^{2s+\alpha'}_\ell(\Omega)$ for some $\alpha'$ depending on $\Omega$. (Note that the dependence of $\alpha'$ on $\Omega$ follows from the dependencies of $\alpha_0$ in \Cref{t:global-degiorgi}).  This implies the sequence $f^\eps$ is precompact in $C^{2s+\alpha''}_\ell(\Omega)$ for any $\alpha''\in (0,\alpha')$, and some subsequence of $f^\eps$ converges in $C^{2s+\alpha''}_\ell(\Omega)$ to a function $f$. Since $\Omega$ was arbitrary, $f$ can be defined as an element of $C^{2s}_{\ell, \rm loc}([0,T]\times \R^6)$, and for any compact $\Omega$, there is an $\alpha''$ with $f\in C^{2s+\alpha''}_\ell(\Omega)$. Since $f^\eps \to f$ pointwise, $f$ also lies in $L^\infty_{q_0}([0,T_f]\times \R^6)$, by Proposition \ref{p:upper-bounds2}.

From \cite[Lemma 2.7]{imbert2018schauder}, the norm $C^{2s+\alpha''}_\ell$ controls the material derivative $(\partial_t + v\cdot \nabla _x)f^\eps$ (but not the separate terms $\partial_t f^\eps$ and $\nabla_x f^\eps$). In particular, for each compact $\Omega$,
\[ \|(\partial_t + v\cdot \nabla_x) f^\eps\|_{C^{\alpha''}_\ell(\Omega)} \leq C\|f^\eps\|_{C^{2s+\alpha''}_\ell(\Omega)},\]
and the convergence of $f^\eps$ in $C^{2s+\alpha''}_\ell$ implies $(\partial_t+v\cdot \nabla_x)f$ is a locally H\"older continuous function.

To analyze the convergence of $Q(f^\eps,f^\eps)$ as $\eps\to 0$, we use the following lemma:  

\begin{lemma}\label{l:Q-makes-sense}
Let $g, h\in L^\infty_{q}(\R^3)$ with $q> \gamma +2s +3$. For some $v_0\in \R^3$ and $\alpha>0$, assume $h\in C^{2s+\alpha}(B_1(v_0))$. Then
\begin{equation}\label{e:Qgh}
	|Q(g,h)(v)|
		\leq C\|g\|_{L^\infty_{q}(\R^3)}( \|h\|_{C^{2s+\alpha}(B_1(v_0))}+\|h\|_{L^\infty(\R^3)}), 
		 	\quad v\in B_1(v_0).
 \end{equation}
\end{lemma}
\begin{proof}
Writing $Q = Q_{\rm s} + Q_{\rm ns}$ as usual, the singular term is handled by \cite[Lemma 4.6]{imbert2020smooth}, which implies
\[ |Q_{\rm s}(g,h)(v)| \leq C\left(\int_{\R^3} g(w)|v+w|^{\gamma+2s} \dd w\right) \|h\|_{L^\infty(\R^3)}^{\frac \alpha {2s+\alpha}} [h]_{C^{2s+\alpha}(v)}^{\frac{2s}{2s+\alpha}},\]
where $[h]_{C^{2s+\alpha}(v)}$ denotes the smallest constant $N>0$ such that there exists a polynomial $p$ of degree less than $2s+\alpha$ with $|h(v+w) - p(w)|\leq N |w|^{2s+\alpha}$ for all $w\in \R^3$. Using $a^{\frac{\alpha}{2s+\alpha}}b^{\frac{2s}{2s+\alpha}}\lesssim a + b$, and noting that
\[ [h]_{C^{2s+\alpha}(v)} \leq [h]_{C^{2s+\alpha}(B_1(v))} + \|h\|_{L^\infty(B_1(v)^c)},\]
we see that $Q_{\rm s}(g,h)(v)$ is bounded by the right-hand side of \eqref{e:Qgh}, using the convolution estimate of \Cref{l:convolution} and $\vv\approx \langle v_0\rangle$.

For $Q_{\rm ns}(g,h)(v) = c_b h(v) [g\ast |\cdot|^\gamma](v)$, another application of Lemma \ref{l:convolution} implies 
\[
 |Q_{\rm ns}(g,h)(v)| \leq C  \|h\|_{L^\infty(\R^3)} \|g\|_{L^\infty_q(\R^3)}\vv^\gamma,
 \]
 since $q>3$. The conclusion of the lemma follows.
\end{proof}
 Using bilinearity, we write $Q(f^\eps, f^\eps) - Q(f,f) = Q(f^\eps, f^\eps-f) + Q(f^\eps-f, f)$. Let $q$ equal the average of $q_0$ and $3+\gamma+2s$. Then, since $f^\eps$ and $f$ share a common uniform bound in $L^\infty_{q_0}([0,T_f]\times\R^6)$ with $q< q_0$, and $f^\eps\to f$ uniformly on compact sets, we in fact have $f^\eps \to f$ strongly in $L^\infty_{q}([\tau,T_f]\times\Omega_x\times\R^3)$ for any $\tau\in (0,T_f)$ and $\Omega_x\subset \R^3$. Together with the convergence in $C^{2s+\alpha''}_\ell(\Omega)$ for compact $\Omega$, this is enough to apply Lemma \ref{l:Q-makes-sense} and conclude $Q(f^\eps,f^\eps) \to Q(f,f)$ locally uniformly. In particular, $Q(f,f)$ is well-defined. We have shown that $f$ satisfies the Boltzmann equation \eqref{e.boltzmann} in the pointwise sense.

To address the initial data, we multiply the equation \eqref{e.boltzmann} satisfied by $f^\eps$ by some $\varphi\in C^1_{t,x} C^2_v$ with compact support in $[0,T_f) \times \R^6$, and integrate by parts:
\begin{equation}\label{e.phi-0}
\begin{split}
 \iint_{\R^3\times\R^3} \varphi(0,x,v) f_{\rm in}^\eps(x,v) \dd v \dd x &= \int_0^{T_f} \iint_{\R^3\times\R^3} f^\eps (\partial_t \varphi + v\cdot \nabla_x \varphi)\dd v \dd x \dd t \\
 &\quad +\int_0^{T_f} \iint_{\R^3\times\R^3} \varphi Q(f^\eps, f^\eps) \dd x \dd v \dd t.
\end{split}
\end{equation}
The left-hand side converges to $\iint_{\R^3\times\R^3} \varphi(0,x,v) f_{\rm in}(x,v) \dd x \dd v$ by the convergence of $f_{\rm in}^\eps$ to $f_{\rm in}$ in $L^1(\supp(\varphi(0,\cdot,\cdot))$ (recall the definition of $f_{\rm in}^\eps$~\eqref{e.f0eps}). The convergence of the first integral on the right in \eqref{e.phi-0} is also straightforward, by the uniform upper bounds for $f^\eps$ in $L^\infty_{q_0}$ and the pointwise convergence of $f^\eps$ to $f$. 
For the second integral on the right, we need to proceed more carefully. The continuity properties needed to apply Lemma \ref{l:Q-makes-sense} and control $Q(f^\eps, f^\eps)$ pointwise may degenerate as $t\to 0$ at a potentially severe rate. Therefore, we use the weak formulation of the collision operator to bound this integral. This is made precise in the following lemma:

\begin{lemma}\label{l:weak-estimate}
For any $\varphi\in C^2(\R^3)$, and $v, v_*\in \R^3$, there holds 
\begin{equation}\label{e.sphere-est}
\left|\int_{\mathbb S^2} B(v-v_*,\sigma) [\varphi(v_*') + \varphi(v') - \varphi(v_*) - \varphi(v)] \dd \sigma\right| \leq C \|\varphi\|_{C^2(\R^3)}|v-v_*|^\gamma (1+|v-v_*|^{2s}),
\end{equation}
for a universal constant $C$. In particular, for any functions $g, h$ on $\R^3$ such that the right-hand side is finite, one has
\begin{equation*}
 \left|\int_{\R^3}W(g,h,\varphi) \dd v\right| \leq C \|\varphi\|_{C^{2}(\R^3)} \iint_{\R^3\times \R^3} g(v_*) h(v) |v-v_*|^\gamma (1+|v-v_*|^{2s}) \dd v_* \dd v,
\end{equation*}
where $W(g,h,\varphi)$ is defined as in \eqref{e.weak-collision}.
\end{lemma}
Estimates of this general type are common in the Boltzmann literature, see e.g. \cite[Chapter 2, Formula (112)]{villani2002review}. However, we could not find a reference with the 
asymptotics $|v-v_*|^{\gamma+2s}$ inside the integral. The sharp asymptotics will be important in the proof of Theorem \ref{t:weak-solutions} below, where we only assume enough velocity decay to control the $L^1$ moment of order $\gamma+2s$ of our weak solutions.
\begin{proof}
Let us recall some facts about the geometry of elastic collisions:
\begin{align}
|v_*' - v_*| &= |v' - v|\label{e.pre-post}\\
v_*' + v' &= v_*+v\label{e.momentum}\\
|v' - v| &\approx \theta |v-v_*|,\label{e.cosine-law}
\end{align}
where we recall $\theta$ from~\eqref{e.B-def}. 
The second fact \eqref{e.momentum} corresponds to conservation of momentum. Let us also introduce the standard abbreviations $F' = F(v')$, $F_*' = F(v_*')$, $F_* = F(v_*)$, and $F = F(v)$ for any function $F$.

Recalling that $B(v-v_*,\sigma) = |v-v_*|^\gamma \theta^{-2-2s} \tilde b(\theta)$, with $\tilde b (\theta) \approx 1$, we divide the integral over $\mathbb S^2$ in \eqref{e.sphere-est} into two domains:
\[ D_1 := \{ \sigma : |\theta| \leq |v-v_*|^{-1}\}, \quad D_2 := \mathbb S^2  \setminus D_1,\]
where $D_2$ is empty if $|v-v_*|\leq 1/\pi$. In $D_1$, following a common method for controlling the angular singularity, we Taylor expand $\varphi$ and use the identities \eqref{e.momentum} and \eqref{e.pre-post}. We obtain 
\[\begin{split}
 \varphi' + \varphi_*' - \varphi_* - \varphi &= \nabla \varphi_*\cdot (v_*' - v_*) + \nabla \varphi\cdot (v' - v) + O(\|D^2\varphi\|_{L^\infty} |v'-v|^2)\\
&= (\nabla \varphi - \nabla \varphi_*)\cdot (v' - v) + O(\|D^2\varphi\|_{L^\infty}|v' - v|^2).
\end{split}
\]
From \eqref{e.cosine-law}, the second term on the right in the last expression is proportional to $\|D^2\varphi\|_{L^\infty} \theta^2|v-v_*|^2$. To handle the first term on the right, we parameterize $\mathbb S^2$ with spherical coordinates $\sigma = (\theta, \eta) \in [0,\pi]\times [0,2\pi]$, where $\theta=0$ corresponds to $v = v'$. A simple geometric argument shows $\left| \int_0^{2\pi} (v' - v) \dd \eta\right| \lesssim |v-v_*|\theta^2$. Therefore, we have
\[
 \left|\int_0^{2\pi} [ \varphi' + \varphi_*' - \varphi_* - \varphi] \dd \eta\right|\lesssim \|D^2\varphi\|_{L^\infty} \theta^2 |v-v_*|^2,
 \]
which implies
\[ 
\begin{split}
\Big|\int_{D_1} &B(v-v_*,\sigma)[\varphi' + \varphi_*' - \varphi_* - \varphi] \dd \sigma \dd v_* \dd v \Big|\\
&\lesssim  |v-v_*|^{\gamma+2} \int_0^{|v-v_*|^{-1}} \theta^{-2-2s} \|D^2\varphi\|_{L^\infty} \theta^2 \sin \theta \dd \theta
\lesssim \|D^2\varphi\|_{L^\infty} |v-v_*|^{\gamma+2s} .
\end{split}
\]
For the integral over $D_2$, since $|\theta|\geq |v-v_*|^{-1}$, we have 
\[ \begin{split}
\Big| \int_{D_2} &|v-v_*|^\gamma \theta^{-2-2s} \tilde b(\theta)[\varphi' + \varphi_*' - \varphi_* - \varphi] \dd \sigma \Big|\\
&\lesssim \|D^2\varphi\|_{L^\infty} |v-v_*|^{\gamma} \int_{|v-v_*|^{-1}}^\pi \int_0^{2\pi} \theta^{-2-2s} \sin\theta \dd \eta \dd \theta
\lesssim \|D^2\varphi\|_{L^\infty}  |v-v_*|^{\gamma}(1+|v-v_*|^{2s}), 
\end{split}
\]
which establishes \eqref{e.sphere-est}.

Next, recalling the weak formulation \eqref{e.weak-collision}  we have
\begin{equation*}
 \int_{\R^3} W(g,h,\varphi) \dd v = \frac 1 2 \int_{\R^3} \iint_{\R^3\times \mathbb S^2}   B(v-v_*,\sigma) g h_* [\varphi' + \varphi_*'- \varphi_* - \varphi] \dd \sigma \dd v_* \dd v,
 \end{equation*}
 and the last conclusion of the lemma follows directly from \eqref{e.sphere-est}.
\end{proof}

Returning to \eqref{e.phi-0}, 
for each $t\in (0,T_f]$, the locally uniform convergence of $Q(f^\eps,f^\eps)$ to $Q(f,f)$ implies
\[\iint_{\R^3\times\R^3} \varphi Q(f^\eps,f^\eps) \dd x \dd v  \to \iint_{\R^3\times\R^3} \varphi Q(f,f) \dd x \dd v.\]
By Lemma \ref{l:weak-estimate} and our uniform upper bound on $\|f^\eps\|_{L^\infty_{q_0}(\R^6)}$ with $q_0>3+2s$, we may apply the Dominated Convergence Theorem to the time integral of $\varphi Q(f^\eps, f^\eps)$ and conclude that $f$ agrees with the initial data $f_{\rm in}$ in the sense of \eqref{e.weak-initial-data}.

 Finally, we consider the higher regularity of $f$. The approximate solutions $f^\eps$ are smooth and rapidly decaying, so for any compact $\Omega\subset (0,T_f]\times\R^3$, partial derivative $D^k$, and $\alpha, m>0$, Proposition \ref{p:higher-reg} provides a $q(k,m)>0$ such that $\|D^k f^\eps\|_{C^\alpha_{\ell,m}(\Omega\times\R^3_v)}$ is bounded for positive times in terms of $\|f_{\rm in}^\eps\|_{L^\infty_{q(k,m)}(\R^6)}$. From Proposition \ref{p:upper-bounds2}, this bound is independent of $\eps$ if the initial data $f_{\rm in}$ is bounded in $L^\infty_{q(k,m)}(\R^6)$.  Hence, applying a standard compactness argument, the $L^\infty_m$ estimates for $D^k f^\eps$ imply $L^\infty_{m}$ estimates for $D^k f$ in the limit as $\eps\to 0$.

This concludes the proof of Theorem \ref{t:existence}.

\subsection{\Cref{t:weak-solutions}: the existence of weak solutions}

In this subsection, we prove Theorem \ref{t:weak-solutions}. The proof is based on the same approximating sequence $f^\eps$ from the proof of Theorem \ref{t:existence}. The relaxed conditions on $f_{\rm in}$ result in weaker uniform regularity for $f^\eps$, and correspondingly, a different notion of convergence as $\eps\to 0$.

In more detail, assume that $f_{\rm in}\geq 0$ lies in $L^\infty_{q_0}(\R^6)$ with $q_0>3+\gamma+2s$. This initial data may not necessarily satisfy any uniform positivity condition. Let $f^\eps_{\rm in}$ be defined as in \eqref{e.f0eps}, and as above, let $f^\eps$ be smooth solutions to \eqref{e.boltzmann} with initial data $f^\eps_{\rm in}$. By Proposition \ref{p:upper-bounds2}, these solutions exist on a uniform time interval $[0,T_f]$, and are uniformly bounded in $L^\infty_{q_0}([0,T_f]\times\R^6)$. Since $L^\infty_{q_0}([0,T_f]\times\R^6)$ is the dual of $L^1_{-q_0}([0,T_f]\times\R^6)$, some sequence of $f^\eps$ converges in the weak-$\ast$ $L^\infty_{q_0}$ sense to a function $f\in L^\infty_{q_0}([0,T_f]\times\R^6)$.

To show $f$ is a weak solution of \eqref{e.boltzmann}, note that for each $\eps>0$ and $\varphi \in C^1_{t,x} C^2_v$ with compact support in $[0,T_f)\times\R^6$, integrating by parts implies the weak formulation \eqref{e.phi-0} holds for $f^\eps$. 
The left-hand side of \eqref{e.phi-0} converges by the $L^1$-convergence of $f^\eps_{\rm in}$ to $f_{\rm in}$ on $\supp(\varphi(0,\cdot,\cdot))$, exactly as above. For the right-hand side, we have
\[ 
\int_0^{T_f} \iint_{\R^3\times\R^3} (f^\eps - f) (\partial_t\varphi+v\cdot\nabla_x\varphi) \dd v\dd x \dd t \to 0,
\]
by the weak-$\ast$ convergence of $f^\eps$ to $f$, since $(\partial_t\varphi + v\cdot\nabla_x \varphi) \in L^1_{-q_0}([0,T_f]\times\R^6)$. For the collision term, since $f^\eps$ is smooth and rapidly decaying, we may apply the identity $\int_{\R^3} \varphi Q(f^\eps,f^\eps) \dd v = \int_{\R^3} W(f^\eps,f^\eps,\varphi) \dd v$. Using bilinearity, we have, for each $t,x$,
\[
\begin{split}
\int_{\R^3}W(f^\eps,&f^\eps,\varphi) \dd v - \int_{\R^3} W(f,f,\varphi) \dd v = \int_{\R^3} W(f^\eps,f^\eps-f,\varphi) \dd v + \int_{\R^3} W(f,f^\eps-f, \varphi) \dd v.
\end{split}
\]
The first term on the right is equal to 
\begin{equation}\label{e.first-term}
\begin{split}
\frac 1 2  \int_{\R^3}f^\eps\int_{\R^3}\int_{\mathbb S^2} (f^\eps_* - f_*)  B(v-v_*,\sigma)[\varphi' + \varphi_*' - \varphi_* - \varphi] \dd \sigma \dd v_* \dd v. 
 \end{split}
\end{equation}
From \eqref{e.sphere-est} in Lemma \ref{l:weak-estimate}, we see that for fixed $v\in \R^3$,
\[ \int_{\mathbb S^2} B(v-v_*,\sigma) [\varphi' + \varphi_*' - \varphi_* - \varphi] \dd \sigma \in L^1_{-q_0}(\R^3_{v_*}),\]
since $q_0> 3+\gamma+2s$. This implies \eqref{e.first-term} converges to 0 as $\eps\to 0$. The same argument, after exchanging the $v$ and $v_*$ integrals, implies $\int_{\R^3} W(f,f^\eps- f,\varphi) \dd v \to 0$. We conclude $f^\eps$ is a weak solution to \eqref{e.boltzmann} in the sense of Theorem \ref{t:weak-solutions}.

Next, consider the additional assumption that $f_{\rm in}(x,v)\geq \delta$ for all $(x,v) \in B_r(x_m,v_m)$. 
As above, this implies lower bounds of the form \eqref{e.uniform-lower-b} for $f^\eps$ that are independent of $\eps$. Together with our uniform bound on $\|f^\eps\|_{L^\infty_{q_0}([0,T_f]\times\R^6)}$ with $q_0>3+\gamma+2s$,  this allows us to apply the local De Giorgi estimate of Theorem \ref{t:degiorgi}. (Note that under such an assumption on $q_0$, we cannot necessarily apply the global De Giorgi estimate of Theorem \ref{t:global-degiorgi}.) We obtain, for any compact $\Omega\subset (0,T_f]\times\R^6$,
\[ \|f^\eps\|_{C^\alpha(\Omega)} \leq C,\]
with $C,\alpha>0$ depending on $\Omega$, $\delta$, $r$, $v_m$, $x_m$, and $\|f_{\rm in}\|_{L^\infty_{q_0}(\R^6)}$. Since this bound is independent of $\eps$, the same conclusion applies to $f$. This concludes the proof of Theorem \ref{t:weak-solutions}.

\subsection{Continuous matching with initial data}

Here we show, under the assumption that $f_{\rm in}$ is continuous, that the solution $f$ is continuous as $t\searrow 0$. We prove this as a consequence of a more general result for linear kinetic integro-differential equations of the following form:
\begin{equation}\label{e.kide}
	\partial_t f + v\cdot \nabla_x f
		= \int_{\R^3} K(t,x,v,v') (f(t,x,v') - f(t,x,v))  \dd v'
			+ c f,
\end{equation}
where $K: [0,T]\times \dom\times\R^3\to [0,\infty)$ and $c: [0,T]\times\dom \to [0,\infty)$ satisfy, for all $(t,x,v)\in [0,T]\times\dom$, $w\in \R^3$, and $r>0$,
\be\label{e.K_c_assumptions}
	\begin{split}
	& K(t,x,v,v+w) = K(t,x,v,v-w),\\
		&\int_{B_{2r}(v)\setminus B_r(v)} 
			K(t,x,v,v')
			\dd v'
			\leq \Lambda \vv^{\gamma+2s} r^{-2s}, \quad \text{ and}
		\\
		&|c(t,x,v)|
			\leq \Lambda \vv^\gamma,
	\end{split}
\ee
for some constants $\Lambda>0$, $\gamma>-3$, and $s\in (0,1)$.  From the Carleman decomposition of $Q(f,f)$ described in Section \ref{s:carleman}, together with Lemma \ref{l:K-upper-bound}, one sees that equation \eqref{e.kide} includes the Boltzmann equation \eqref{e.boltzmann} as a special case.  



\begin{proposition}\label{p:cont-match}
Let $f \in C^2_\ell\cap L^{\infty}([0,T]\times\R^6)$ be a solution to~\eqref{e.kide} with initial data $f_{\rm in}$.  Then, for any $(x_0,v_0)$ and $\eta>0$, there exists $t_\eta, r_\eta>0$ such that, if
\[
	|(x,v) - (x_0,v_0)| < r_\eta
	\quad\text{ and }\quad
	t\in [0,t_\eta),
\]
then
\begin{equation}\label{e.c7146}
	|f(t,x,v) - f_{\rm in}(x_0,v_0)| < \eta.
\end{equation}
The constants $t_\eta$ and $r_\eta$ depend only on $|v_0|$, $\eta$, $\Lambda$, $s$, $\|f\|_{L^\infty}$ and the modulus of continuity of $f_{\rm in}$ at $(x_0,v_0)$.  In particular, the constants do not depend on $\|f\|_{C^2_\ell}$.
\end{proposition}

To apply this proposition to the solution $f$ to \eqref{e.boltzmann} constructed in Theorem \ref{t:existence}, we use the smooth approximating sequence $f^\eps$. Proposition \ref{p:cont-match} applies to the smooth solutions $f^\eps$, with constants depending on $\Lambda \lesssim \|f^\eps\|_{L^\infty_{q_0}(\R^6)}$, which is bounded independently of $\eps$ by Proposition \ref{p:upper-bounds2}. Since $f^\eps \to f$ pointwise in $[0,T]\times\R^6$, the conclusion of Proposition \ref{p:cont-match} also applies to $f$. 

Note that we allow both cases $\gamma < 0$ and $\gamma \geq 0$ in Proposition \ref{p:cont-match}.

\begin{proof}
In the proof, we only obtain the upper bound on $f - f_{\rm in}$ in~\eqref{e.c7146}.   The lower bound can be obtained in an exactly analogous way, so we omit it.

Without loss of generality, we assume that $x_0 = 0$.  Fix $\delta\in(0,1)$ sufficiently small so that
\be\label{e.c7141}
	|f_{\rm in}(x,v) - f_{\rm in}(0,v_0)|
		< \frac{\eta}{2}
		\qquad\text{ if } |x|^2 + |v-v_0|^2 \leq \delta^2. 
\ee

Let
\be\label{e.c_t_epsilon}
	T_\delta
		= \frac{\delta}{4( |v_0| + 2\delta)}.
\ee
Our goal is to construct a supersolution $F$ for $f$ on $[0, T_\delta]\times B_\delta(0,v_0) \subset [0,T_\delta]\times \R^6$.  To begin, we let
\be\label{e.c7142}
	F(t,x,v)
		= e^{2\Lambda \langle v_0\rangle^\gamma t} \left(\|f\|_{L^\infty([0,T]\times\R^6)} \psi\left(\frac{|x-vt|^2 + |v-v_0|^2}{\delta^2} \right) + \frac{\eta}{2} + \rho t + f_{\rm in}(0,v_0)\right),
\ee
where $\psi \in C^2(\R)$ satisfies
\be\label{e.c7147}
	\psi(s)
		= \begin{cases}
			0 \qquad &\text{ if } s \leq 0,\\
			1 \qquad &\text{ if } s \geq 1/2,
		\end{cases}
	\qquad
	0 \leq \psi' \leq 4,
	\qquad\text{ and } \qquad
	|\psi''| \leq 32,
\ee
and
\be\label{e.c7148}
	\rho = A \|f\|_{L^\infty([0,T]\times\R^6)} \frac{\vvo^{\gamma + 2s}}{\delta^{2s}} e^{2\Lambda T_\delta\langle v_0\rangle^\gamma},
\ee
for a large constant $A>0$ to be chosen depending only on $\Lambda$, $\langle v_0\rangle$, $\gamma$, and $s$.  We claim that 
\begin{equation}\label{e.barf-f}
F > f \quad \text{ on } [0,T_\delta)\times \R^6.
\end{equation}
Note that all terms in $F$ except $f_{\rm in}(0,v_0)$ can be made smaller than $\eta$ by choosing $r_\eta>0$ and $t_\eta\in (0,T_\delta)$ sufficiently small.  Therefore, the proof is complete once we establish \eqref{e.barf-f}.

First we note that~\eqref{e.barf-f} holds initially.  Indeed, from~\eqref{e.c7141} and~\eqref{e.c7142}, it is clear that $F> f$ for $(t,x,v) \in \{0\}\times \R^6$.

Next, we show that~\eqref{e.barf-f} holds away from $(0,v_0)$:
\be\label{e.c62702}
	F > f
	\qquad\text{ for } (t,x,v) \in (0,T_\delta]\times B_\delta(0,v_0)^c.
\ee
Fix any $(t,x,v) \in (0,T_\delta]\times B_\delta(0,v_0)^c$.  If $|v-v_0| \geq \delta/\sqrt 2$, then, by the definition of $\psi$~\eqref{e.c7147},
\be\label{e.c62703}
	F(t, x, v)
		> \|f\|_{L^\infty([0,T]\times\R^6)} \psi\left(\frac{|x-vt|^2 + |v-v_0|^2}{\delta^2} \right)
		= \|f\|_{L^\infty([0,T]\times\R^6)}
		\geq f(t,x,v).
\ee
Hence,~\eqref{e.c62702} is established in this case.

Assume now that $|v-v_0| < \delta/\sqrt 2$.  Then
\be
	|x-tv|^2 + |v - v_0|^2
		= |x|^2 + |v-v_0|^2 - 2 t x \cdot v + t^2 |v|^2
		\geq |x|^2 + |v-v_0|^2 - 2 t x \cdot v.
\ee
Next, we use Young's inequality and then that $t< T_\delta$, where $T_\delta$ is defined in~\eqref{e.c_t_epsilon}, and $|v| \leq |v_0| + \delta/\sqrt 2$, to find
\be
	|x-tv|^2 + |v - v_0|^2
		\geq \frac{3|x|^2}{4} + |v-v_0|^2 - 4 t^2 |v|^2
		\geq \frac{3\delta^2}{4} - 4 \frac{\delta^2}{4^2(|v_0|+\delta)^2} (|v_0| + \delta/\sqrt2)^2
		\geq \frac{\delta^2}{2}.
\ee
The argument of~\eqref{e.c62703} then applies to establish~\eqref{e.c62702}.  Hence,~\eqref{e.c62702} holds in both cases.

Due to~\eqref{e.c62702}, if $F \not> f$ on $[0,T_\delta]\times \R^6$, then, defining the crossing time as
\be\label{e.c080402}
	t_\rmcr = \sup\{t : F(t',x,v) \geq f(t',x,v) \quad\text{for all}\quad (t',x,v) \in [0,t]\times \R^6\},
\ee
it follows that $t_\rmcr \in (0,T_\delta]$ by the continuity of $F$ and $f$, and there exists a crossing point $(x_\rmcr, v_\rmcr)\in B_\delta(0,v_0)$ such that
\be\label{e.matching_tch}
	F(t_\rmcr, x_\rmcr, v_\rmcr) = f(t_\rmcr, x_\rmcr, v_\rmcr).
\ee


Since $(t_\rmcr, x_\rmcr, v_\rmcr)$ is a minimum of $F- f$ on $(0,T_\delta)\times B_\delta(0,v_0)$, then, at $(t_\rmcr, x_\rmcr, v_\rmcr)$, we have
\be\label{e.c7143}
	\begin{split}
	&F = f,
		\qquad
	\partial_t F \leq \partial_t f,
		\qquad
	\nabla_x F = \nabla_x f,
		\qquad
		\text{and}
	\\&
	\int_{B_\delta(v_\rmcr)} (F' - F_\rmcr) K(t_\rmcr, x_\rmcr, v_\rmcr,v') \dd v'
		\geq \int_{B_\delta(v_\rmcr)} ( f' - f_\rmcr) K(t_\rmcr, x_\rmcr,v_\rmcr,v') \dd v'.
	\end{split}
\ee
Note that, in the last integral, we have used the notation that, for any $v'$,
\[
	f' = f(t_\rmcr, x_\rmcr, v')
	\quad\text{ and }\quad
	f_\rmcr = f(t_\rmcr, x_\rmcr, v_\rmcr)
\]
and similarly for $F$.  It follows, from~\eqref{e.kide} and~\eqref{e.c7143}, that, at $(t_\rmcr, x_\rmcr, v_\rmcr)$,
\be\label{e.c7144}
	\partial_t F
		+ v\cdot \nabla_x F \leq \partial_t f + v\cdot\nabla_x f = 
		 \mathcal L f
		+ c F.
\ee
 Recall the notation $\cL$ from~\eqref{e.c062801}. 
By a direct calculation with \eqref{e.c7142}, we also have
\[
	\partial_t F
		+ v\cdot \nabla_x F
	= 2\Lambda\langle v_0\rangle^\gamma F
		+ \rho e^{2\Lambda \langle v_0\rangle^\gamma t}.
\]
Note that $c(t,x_\rmcr,v_\rmcr)\leq \Lambda\langle v_\rmcr\rangle^\gamma$ by~\eqref{e.K_c_assumptions}.  Thus, up to decreasing $\delta$, $c(t,x_\rmcr,v_\rmcr)  < 2\Lambda \langle v_0\rangle^\gamma$.  Hence, we only need to show that, 
at $(t_\rmcr, x_\rmcr, v_\rmcr)$,
\be\label{e.c7145}
	\cL f
		\leq  \rho e^{2\Lambda \langle v_0\rangle^\gamma t_\rmcr},
\ee
in order to obtain a contradiction with \eqref{e.c7144} and conclude the proof. 

To prove \eqref{e.c7145}, we start by using~\eqref{e.c7143} to write, at $(t_\rmcr, x_\rmcr, v_\rmcr)$,
\be\label{e.c714_I1_I2}
\begin{split}
	\cL f
		&\leq \int_{B_\delta(v_\rmcr)^c} (f' - f_\rmcr) K(t_\rmcr, x_\rmcr,v_\rmcr, v') \dd v'
			+ \int_{B_\delta(v_\rmcr)} (F' - F_\rmcr) K(t_\rmcr,x_\rmcr,v_\rmcr, v') \dd v'\\
		&=: I_1 + I_2.
\end{split}
\ee
First, we bound $I_1$. 
Using~\eqref{e.K_c_assumptions}, we have, with $A_k = B_{2^{k+1}\delta}(v_\rmcr)\setminus B_{2^{k}\delta}(v_\rmcr)$,
\be\label{e.c714_I1}
	\begin{split}
	I_1
		&\lesssim \|f\|_{L^\infty([0,T]\times\R^6)} \int_{B_\delta(v_\rmcr)^c} K(t_\rmcr, x_\rmcr,v_\rmcr, v') \dd v'\\
		&\lesssim \|f\|_{L^\infty([0,T]\times\R^6)} \sum_{k=0}^\infty \int_{A_k} K(t_\rmcr, x_\rmcr, v_\rmcr, v') \dd v'
		\\&\lesssim \|f\|_{L^\infty([0,T]\times\R^6)} \vvt^{\gamma + 2s} \sum_{k=0}^\infty (\delta 2^k)^{-2s}
		\lesssim \|f\|_{L^\infty([0,T]\times\R^6)} \vvt^{\gamma+2s} \delta^{-2s}.
	\end{split}
\ee

Next, we bound $I_2$. From the integral estimate for $K$ in \eqref{e.K_c_assumptions} applied on a series of annuli, it is easy to show
\begin{equation}\label{e.K-ball}
		\int_{B_r(v)} |v-v'|^2
			K(t,x,v,v')
			\dd v'
			\leq \Lambda \vv^{\gamma+2s} r^{2-2s},
			\end{equation}
			for $r>0$. 
 Using the symmetry of $K$ with respect to $(v'-v_\rmcr)$ as in~\eqref{e.K_c_assumptions}, we see
\[
	\int_{B_\delta(v_\rmcr)} (v' - v_\rmcr) \cdot \nabla_v F(t_\rmcr,x_\rmcr,v_\rmcr) K(t_\rmcr,x_\rmcr,v_\rmcr, v') \dd v' = 0.
\]
Then, using a Taylor expansion and the definition of $F$, we see that
\[
	\begin{split}
	&\left|F' - F_\rmcr - (v' - v_\rmcr) \cdot \nabla_v F(v_\rmcr)\right|\\
	& \lesssim e^{2\Lambda\langle v_0\rangle^\gamma t_\rmcr} \|f\|_{L^\infty([0,T]\times\R^6)} \left(\frac {t_\rmcr^2 + 1}{\delta^2} + \frac{t_\rmcr^2 |x_\rmcr-v_\rmcr t_\rmcr|^2 +  |v_\rmcr-v_0|^2}{\delta^4} \right)|v' - v_\rmcr|^2
	\\&\lesssim e^{2\Lambda \langle v_0\rangle^\gamma t_\rmcr} \|f\|_{L^\infty([0,T]\times\R^6)} \frac{1}{\delta^2}|v' - v_\rmcr|^2,
	\end{split}
\]
using $|x_\rmcr|^2 + |v_\rmcr- v_0|< \delta^2$ and $t_\rmcr < T_\delta$.  Putting the above together and applying~\eqref{e.K-ball}, we find
\be\label{e.c714_I2}
	\begin{split}
	I_2
		&\lesssim e^{2\Lambda\langle v_0\rangle^\gamma t_\rmcr} \|f\|_{L^\infty([0,T]\times\R^6)} \frac{1}{\delta^2} \int_{B_\delta(v_\rmcr)}  |v' - v_\rmcr|^2 K(t_\rmcr,x_\rmcr,v_\rmcr, v')\dd v'\\
		&\lesssim e^{2\Lambda\langle v_0\rangle^\gamma t_\rmcr} \|f\|_{L^\infty([0,T]\times\R^6)}\frac{\vvt^{\gamma + 2s}}{\delta^{2s}}.
	\end{split}
\ee

Combining~\eqref{e.c714_I1_I2},~\eqref{e.c714_I1}, and~\eqref{e.c714_I2}, we find, at $(t_\rmcr, x_\rmcr, v_\rmcr)$,
\[
	\cL  f
		\lesssim \|f\|_{L^\infty([0,T]\times\R^6)}  \frac{\vvt^{\gamma + 2s}}{\delta^{2s}}\left(
			1+ e^{2\Lambda\langle v_0\rangle^\gamma t_\rmcr}\right).
\]
Using the choice of $\rho$ from~\eqref{e.c7148}, the fact that $\vvo \approx \vvt$ (recall that $|v_0 - v_\rmcr| < \delta < 1$), and increasing $A$ as necessary yields the claim~\eqref{e.c7145}.  This concludes the proof.
\end{proof}

\section{Time regularity in kinetic integro-differential equations}\label{s:time-reg}

As part of our proof of uniqueness, we need to show that solutions $f$ to \eqref{e.boltzmann} are H\"older continuous for positive times with uniform bounds as $t\searrow0$ as long as the initial data $f_{\rm in}$ is H\"older continuous. This will be accomplished in Section \ref{s:holder-xv}. To prove this, we first need to understand the following fundamental property: if $f(t,\cdot,\cdot)$ is H\"older in $(x,v)$ for each time value $t$, then they are H\"older in $(t,x,v)$? The corresponding fact for linear parabolic equations (regularity in $x$ implies regularity in $(t,x)$, in the suitable scaling) is classical, but it is a nontrivial task to extend this to kinetic integro-differential equations~\eqref{e.kide} (including Boltzmann). For future potential applications, we prove this property for general linear kinetic equations of the form \eqref{e.kide} given above, with $K$ and $c$ satisfying \eqref{e.K_c_assumptions}.


Recall the local kinetic H\"older seminorm $[f]_{C^\alpha_{\ell,x,v}(D)}$ defined for subsets $D\subset \R^6$, as defined in Section \ref{s:holder-spaces}. Since this section concerns H\"older exponents $\alpha < \min\{1,2s\}$, 
this seminorm can equivalently be defined as 
\[
	[f]_{C^\alpha_{\ell,x,v}(D)}
		= \sup_{(x,v), (x_0,v_0)\in D} \frac{|f(x,v) - f(x_0,v_0)|}{d_\ell((0,x,v),(0,x_0,v_0))^\alpha}.
\]
As usual, we define $\|f\|_{C^\alpha_{\ell,x,v}(D)} = \|f\|_{L^\infty(D)} + [f]_{C^\alpha_{\ell,x,v}(D)}$.

The main result of this section is as follows. We note that this proposition is proven in both cases $\gamma< 0$ and $\gamma\geq 0$.

\begin{proposition}\label{prop:holder_in_t}
Suppose that $f$ solves~\eqref{e.kide}, satisfies ${\vv^{q+\alpha(\gamma+2s)_+/(2s)} f} \in L^\infty_tC^\alpha_{l,x,v}([0,T]\times \R^6)$ for some $\alpha\in (0,\min\{1,2s\})$, and is continuous in all variables.  Then, we have 
\begin{equation*}
	\|f\|_{C^\alpha_{\ell,q}([0,T]\times\R^6)}
		\leq C \sup_{t\in[0,T]} \left\|\vv^{q+\alpha(\gamma+2s)_+/(2s)}f(t)\right\|_{C^\alpha_{\ell,x,v}(\R^6)}.
\end{equation*}
The constant $C$ depends only on universal constants, $\alpha$, and $\Lambda$.
\end{proposition}

The key lemma for proving Proposition \ref{prop:holder_in_t} is the following (recall the definition of the dilation $\delta_r$~\eqref{e.dilation-def}):

\begin{lemma}\label{lem:scaling_holder}
Under the assumptions of \Cref{prop:holder_in_t}, let $z_1 \in [0,T]\times \R^6$, $r\in (0,1]$ be arbitrary, and $z_2 = (t_2,x_2,v_2)$ be such that  $t_2 \in[0, \langle v_1 \rangle^{-(\gamma+2s)_+}]$,  $r^{2s} t_2 + t_1 \in [0,T]$,  and $|x_2|, |v_2| < 1$. Then we have
\begin{equation*}
\begin{split}
	|f(z_1 \circ \delta_r (z_2)) - f(z_1)|
		\lesssim |r|^\alpha \langle v_1 \rangle^{-q} \left(\|f\|_{L^\infty_q([0,T]\times \R^6)} + [\vv^q f(t_1)]_{C^\alpha_{\ell,x,v}(B_1(x_1,v_1))}\right),
\end{split}
\end{equation*}
with the implied constant depending only on universal constants, $\alpha$, and $\Lambda$.
\end{lemma}

\begin{proof}
The proof is based on a barrier argument.  Let $z_1$ be as in the statement of the lemma.  Without loss of generality, we may assume that $t_1 = 0$ and $x_1 = 0$.  Then $z_2 \in [0, \min\{\langle v_1\rangle^{-(2s+\gamma)_+}, r^{-2s}T\}]\times B_1(0) \times B_1(0)$.  

\medskip

\noindent {\em Step 1: An auxiliary function and its equation.} 
Let us set some useful notation.  For any $r\geq 0$ and any function $g$, let 
\[
	g_r(z) = g(z_1 \circ \delta_r(z)).
\]
Then, let
\[
	\tilde F(z) = f_r(z) - f(0, 0, v_1).
\]
It is straightforward to check that
\[
	\partial_t \tilde F + v\cdot \nabla_x \tilde F - \cL f_r
		= r^{2s} c_r f_r,
\]
where we have the defined the nonlocal operator and kernel
\be
	\begin{split}
		&\cL f_r
			= \int_{\R^3} f_r(v') - f_r(v)) \cK (t,x,v,v') \dd v'
		\\&\text{and } \quad
		\cK(t,x,v,v')
			= r^{3+2s} K(r^{2s} t, r^{1+2s} x + r^{2s} v_1 t, rv+v_1, rv' + v_1). 
	\end{split}
\ee
We first notice the following bounds derived from~\eqref{e.K_c_assumptions}: for any $L>0$ and $v$, 
\be\label{e.cK2}
	\begin{split}
		\int_{B_{2L}(v) \setminus B_L(v)} &\cK(v,v') \dd v'
			= r^{3+2s}\int_{B_{2L}(v)\setminus B_L(v)} K(rv + v_1, rv' + v_1) \dd v'
			\\&= r^{2s} \int_{B_{2rL}(rv+v_1)\setminus B_{rL}(rv+v_1)} K(rv +v_1, w) \dd w
			\lesssim L^{-2s}
				\langle rv + v_1\rangle^{\gamma+2s} ,
	\end{split}
\ee
and, applying \eqref{e.cK2} on a decreasing sequence of annuli,
\be\label{e.cK1}
	\begin{split}
		&\int_{B_L(v)} |v'-v|^2 \cK(v,v') \dd v'
			=r^{3+2s}\int_{B_L(v)} |v'-v|^2 K(rv + v_1, rv' + v_1) \dd v'
			\\&\quad= r^{-(2-2s)} \int_{B_{rL}(rv+v_1)} |w - (rv + v_1)|^2 K(rv +v_1, w) \dd w
			\lesssim L^{2-2s}
				\langle rv + v_1\rangle^{\gamma+2s} ,
	\end{split}
\ee
where, for brevity, we have omitted the dependence on $(t,x)$. We also have via \eqref{e.K_c_assumptions} the symmetry property 
\[ \mathcal K(v,v+w) = r^{3+2s} K(r v + v_1, rv + v_1+rw) = r^{3+2s} K(rv+v_1,rv + v_1 - rw) = \mathcal K(v,v-w).\]

Our goal is to obtain a local upper bound on $\tilde F$; that is, a bound at $(t_2, x_2, v_2)$ satisfying the smallness assumption in the statement of the lemma.  Hence, we use a suitable multiplicative cutoff function.  Let $\phi \in C_c^\infty(\R^6)$  be a cut-off function such that, for all $i,j \in \{1,2,3\}$,
\begin{equation}\label{e.cutoff}
	\begin{split}
		&\phi\approx \left( |v|^2 + |x|^2 + 1 \right)^{-\alpha/2}
		\\&|\partial_{x_i} \phi|\lesssim |x|\left(|v|^2+|x|^2+1 \right)^{-\alpha/2-1},
		\\&|\partial_{x_ix_j} \phi| \lesssim |x|^2\left(|v|^2+|x|^2+1 \right)^{-\alpha/2-2}, \quad\text{ and}
	\end{split}
	\qquad
	\begin{split}
		&|\partial_{v_i} \phi|\lesssim |v|\left(|v|^2+|x|^2+1 \right)^{-\alpha/2-1},
		\\&|\partial_{v_iv_j} \phi| \lesssim |v|^2\left(|v|^2+|x|^2+1 \right)^{-\alpha/2-2},
		\\&|\partial_{x_iv_j} \phi| \lesssim |x||v|\left(|v|^2+|x|^2+1 \right)^{-\alpha/2-2}.
	\end{split}
\end{equation}
Define
\be\label{e.psi_twist}
	F = \phi \tilde F.
\ee
Note that this is not the same function $F$ from the proof of Proposition \ref{p:cont-match}. After a straightforward computation, we find
\be\label{e.F_equation}
\begin{split}
	\partial_t F + v\cdot \nabla_x F
		= \frac{v\cdot \nabla_x \phi}{\phi} F + \phi \cL \tilde F
			+ r^{2s}\phi c_r f_r.
\end{split}
\ee
The goal is now to estimate $F$ from above.

\medskip

\noindent{\em Step 2: An upper barrier for $F$.} 
Fix
\be\label{e.R}
	R = \frac {\langle v_1 \rangle} {2r}.
\ee
For $C_0$ to be determined, let
\be\label{e.overline F}
	\begin{split}
		\overline F(t) = 2 e^{t C_0 \langle v_1\rangle^{(\gamma+2s)_+}}
		\Big(
			&\|F(0,\cdot,\cdot)\|_{L^\infty(B_R\times  B_R)}
			+ \sup_{s\in[0,t_2], \max\{|x|, |v|\} = R} F(s,x,v)_+
			\\&+ 
				A r^{2s} t \Lambda \langle v_1\rangle^{\gamma_+}\|f\|_{L^\infty([0,T]\times \R^3\times B_{rR}(v_1)}
		\Big),
	\end{split}
\ee
where $\Lambda$ is the constant from \eqref{e.K_c_assumptions}
and $A>0$ is sufficiently large so that, for any $v$ such that $|v| \leq R$,
\be\label{e.c070301}
	\langle rv + v_1\rangle^{\gamma_+}
		< \frac{A}{\sup \phi} \vvone^{\gamma_+}.
\ee
Our goal is to show that $F \leq \overline F$ on $[0,t_2]\times B_R \times B_R$.  Notice that, by construction,
\[
	\sup_{(x,v)\in \overline B_R \times \overline B_R} F(0,x,v)
		< \overline F(0).
\]
Hence, if $F \leq \overline F$ does not always hold, we can take the first crossing time $t_\rmcr$ that $\sup F(t_\rmcr) = \overline F(t_\rmcr)$.  By the assumed uniform continuity of $f$ in time, we immediately see that $t_\rmcr>0$.

On the $(x,v)$  boundary, that is, when $\max\{|x|,|v|\} = R$, one has $F < \overline F$ by construction.  Hence, any crossing point must occur in the interior, and we can find $(x_\rmcr,v_\rmcr) \in B_R(0) \times  B_R(0)$ such that
\be\label{e.c6241}
	F(t_\rmcr,x_\rmcr,v_\rmcr) = \overline F(t_\rmcr).
\ee
Using that $(t_\rmcr,x_\rmcr,v_\rmcr)$ is the first crossing point, we find the following: 
\be\label{e.c6242}
	\begin{split}
		\partial_t F(t_\rmcr,x_\rmcr,v_\rmcr)
			&\geq \partial_t \overline F(t_\rmcr),\\
		\nabla_{x,v} F(t_\rmcr,x_\rmcr,v_\rmcr) &= 0,\\
		F(t_\rmcr,x_\rmcr,v_\rmcr)
			&\geq F(t_\rmcr,x,v) ~~\text{for all }(x,v)\in\R^3\times B_R,\\
			\cL F(t_\rmcr,x_\rmcr,v_\rmcr) &\leq 0.
	\end{split}
\ee
These facts imply, at the point $(t_\rmcr, x_\rmcr,v_\rmcr)$, 
\[
\begin{split}
	\partial_t F &+ v\cdot \nabla_x F
		\geq \partial_t \overline F
		\\&= C_0 \langle v_1\rangle^{(\gamma+2s)_+} \overline F
			+  e^{t C_0 \vvo^{(\gamma+2s)_+}}
				r^{2s}
				2A \Lambda 
				\langle v_1 \rangle^{\gamma_+}
				\|f\|_{L^\infty([0,T]\times \R^3\times B_{rR}(v_1)}\\
		&\geq 
			C_0 \langle v_1\rangle^{(\gamma+2s)_+} F
			+ 
				r^{2s} \phi c_r
				\|f\|_{L^\infty([0,T]\times \R^3\times B_{rR}(v_1)}.
\end{split}
\]
Above we used~\eqref{e.K_c_assumptions}, the choice of $A$ (recall that $|v_\rmcr|\leq R$), and the smallness assumption on $t_2$.  
This, combined with the equation \eqref{e.F_equation} satisfied by $F$, yields
\[ 
	C_0 \langle v_1\rangle^{(\gamma+2s)_+} F
		+  r^{2s} \phi c_r 
		\|f\|_{L^\infty([0,T]\times\R^6)}
			\leq \frac{v_\rmcr \cdot \nabla_x \phi}{\phi} F
				+ \phi \cL \tilde F + r^{2s} \phi c_r f_r.
\]
By~\eqref{e.cutoff} and Young's inequality, we see that $|(v\cdot\nabla_x \phi) / \phi|$ is bounded uniformly.  
Therefore, up to increasing $C_0$, we will reach a contradiction if we
can show that, at $(t_\rmcr, x_\rmcr, v_\rmcr)$,
\begin{equation}\label{e.holder1}
	\phi \cL \tilde F
		< \frac{C_0}{2}\langle v_1 \rangle^{(\gamma+2s)_+} F.
\end{equation}
Once~\eqref{e.holder1} is established, we can conclude that $F\leq \overline F$ on $[0,t_2]\times B_R\times B_R$.

To keep the proof clean, we adopt the notation $h' = h(t_\rmcr,x_\rmcr,v')$ and $h_\rmcr = h(t_\rmcr, x_\rmcr, v_\rmcr)$ for any function $h$. Recall from~\eqref{e.c6242} that $F' \leq F_\rmcr$. Then, since $F = \phi \tilde F$, we have
\begin{equation}\label{e.max-trick}
\begin{split}
	\phi_\rmcr \cL \tilde F_\rmcr = \int_{\R^3} \phi_\rmcr (\tilde F' - \tilde F_\rmcr) \cK \dd v'
	 &= \int_{\R^3} \phi_\rmcr\left( \frac{F'}{\phi'} - \frac{F_\rmcr}{\phi_\rmcr}\right) \cK \dd v'\\
			&\leq \int_{\R^3} \phi_\rmcr F_\rmcr \left( \frac{1}{\phi'} - \frac{1}{\phi_\rmcr}\right) \cK \dd v' = I.
\end{split}
\end{equation}
Fix $L = \langle v_\rmcr \rangle / 2$ and decompose
\be\label{e.c6243}
	I
		= \int_{B_L(v_\rmcr)} \phi_\rmcr F_\rmcr \left( \frac{1}{\phi'} - \frac{1}{\phi_\rmcr}\right)  \cK \dd v'
			+ \int_{B_L(v_\rmcr)^c} \phi_\rmcr F_\rmcr \left( \frac{1}{\phi'} - \frac{1}{\phi_\rmcr}\right)  \cK \dd v'
		= I_1 + I_2.
\ee
Consider the first term $I_1$.  Expanding $1/\phi'$ to second order in the $v'$ variable, and noting that the first-order term vanishes because of the symmetry of the kernel $\cK$, we have
\be\label{e.c6254}
	\begin{split}
		I_1
		&\leq \frac{1}{2}\phi_\rmcr F_\rmcr \sup_{v' \in B_L(v_\rmcr)} |D_v^2(1/\phi')| \int_{B_L(v_\rmcr)} |v'-v_\rmcr|^2 \cK \dd v'.
	\end{split}
\ee
Now, using the properties ~\eqref{e.cutoff} of $\phi$, as well as the upper bound ~\eqref{e.cK1} for $\cK$, we find
\be\label{e.c6251}
	I_1
		\lesssim  F_\rmcr
			(|v_\rmcr|^2+|x_\rmcr|^2+1)^{-\alpha/2} 
			 \langle rv_\rmcr + v_1\rangle^{\gamma+2s} L^{2-2s}
			\sup_{v' \in B_L(v_\rmcr)} |v'|^2\left( |v'|^2 + |x_\rmcr|^2 + 1 \right)^{\alpha/2-2}.
\ee
First notice that, since $|v_\rmcr|<R = \langle v_1\rangle/(2r)$, we have $\langle rv_\rmcr + v_1\rangle \approx \langle v_1 \rangle$.  Next, by the choice of $L$, we have $\vvp \approx \vvt$. Therefore, \eqref{e.c6251} becomes (up to increasing $C_0$) 
\be\label{e.c6253}
	I_1 \lesssim  \frac{F_\rmcr \langle v_1\rangle^{\gamma + 2s} \langle v_\rmcr\rangle^{2-2s} \langle v_\rmcr\rangle^2}{(|v_\rmcr|^2 + |x_\rmcr|^2 + 1)^2}
		< \frac{C_0}{2} F_\rmcr \langle v_1\rangle^{(\gamma+2s)_+},
\ee
as desired.

We now turn to the second term $I_2$ in~\eqref{e.c6243}.  
Since the $-1/\phi_\rmcr$ term in the integrand has a good sign, we immediately see
\[
	I_2
		\leq \phi_\rmcr F_\rmcr \int_{B_L(v_\rmcr)^c} \frac{1}{\phi'} \cK \dd v'.
\]
Using the asymptotics~\eqref{e.cutoff} of $\phi$ we find
\[
		I_2
		\lesssim \phi_\rmcr F_\rmcr \int_{B_L(v_\rmcr)^c} (|v'|^2+|x_\rmcr|^2+1)^{\alpha/2} \cK \dd v'
		= \phi_\rmcr F_\rmcr \sum_{k=0}^\infty \int_{A_{k,L}} (|v'|^2+|x_\rmcr|^2+1)^{\alpha/2} \cK \dd v',
\]
where we define
\[
	A_{k,L}
		= B_{2^{k+1}L}(v_\rmcr) \setminus B_{2^kL}(v_\rmcr).
\]

On the annulus $A_{k,L}$, we have $|v'| \lesssim | v_\rmcr| + 2^{k+1}L$ so that $(|v'|^2+|x_\rmcr|^2+1)^{\alpha/2} \lesssim \langle v_\rmcr \rangle^\alpha + \langle x_\rmcr \rangle^\alpha + 2^{\alpha k} L^\alpha$. Using the bound~\eqref{e.cK2} for $\cK$ and the fact that $\alpha < 2s$ yields

\be\label{e.c6255}
	\begin{split}
	I_2
		&\lesssim \phi_\rmcr F_\rmcr \sum_{k=0}^\infty \left( \langle v_\rmcr \rangle^\alpha + \langle x_\rmcr \rangle^\alpha + 2^{\alpha k}L^\alpha \right) \int_{A_{k,L}} \cK \dd v' \\
		&\lesssim \phi_\rmcr F_\rmcr \sum_{k=0}^\infty \left( \langle v_\rmcr \rangle^\alpha + \langle x_\rmcr \rangle^\alpha + 2^{\alpha k}L^\alpha \right) \langle r v_\rmcr + v_1 \rangle^{\gamma+2s} 2^{-2s k} L^{-2s}\\
		&\lesssim \frac{\langle v_\rmcr \rangle^\alpha + \langle x_\rmcr \rangle^\alpha + L^{\alpha-2s}}{(|v_\rmcr|^2+|x_\rmcr|^2+1)^{\alpha/2}} F_\rmcr \langle v_1 \rangle^{\gamma+2s} \lesssim F_\rmcr \langle v_1 \rangle^{\gamma + 2s}
	\end{split}
\ee
where in the second-to-last step we used that $r|v_\rmcr| \leq rR \leq |v_1|/2$.

Using~\eqref{e.c6253} and~\eqref{e.c6255} in~\eqref{e.c6243}, we find
\[
	\phi_\rmcr \cL \tilde F
		\leq I_1 + I_2
		< \frac{C_0}{2}  F(t_\rmcr) \langle v_1 \rangle^{(\gamma+2s)_+},
\]
which concludes the proof of~\eqref{e.holder1} and allows us to deduce the upper bound
\be\label{e.c62801}
	F \leq \bar F
		\qquad\text{ on } [0,t_2]\times B_R\times B_R.
\ee

\medskip

\noindent{\em Step 3: Quantitative bounds on $\overline F$.}
We establish here the upper bound on $\overline F$: 
\be\label{e.c6258}
	~~~~~~
	\overline F
		\lesssim r^\alpha \vvone^{-q} \left(\|f\|_{L_q^\infty([0,T]\times \R^6)}
			+ [\vv^q f(0,\cdot,\cdot)]_{C^{\alpha}_{\ell,x,v}(B_2(0,v_1))} \right) ~~\quad \text{ in } [0,t_2]\times B_R\times B_R.
\ee
To this end, fix any $(t,x,v)$ with $t \in [0,t_2]$.  The first step is to notice that the exponential term in the definition~\eqref{e.overline F} of $\overline F$ can be bounded by $e^{C_0}$ since $t\leq t_2 \leq \langle v_1\rangle^{-(\gamma + 2s)_+}$.  Similarly, we have $t\Lambda \langle 2v_1\rangle^{\gamma_+} \lesssim 1$. 
Using these two observation and that $2s > \alpha$ yields
\be\label{e.c62503}
	\begin{split}
	\overline F(t,x,v)
		&\lesssim \|F(0,\cdot,\cdot)\|_{L^\infty(B_R\times B_R)}
			+ \sup_{t\in[0,t_2], \max\{|x'|, |v'|\} = R} F(t',x',v')_+
			\\&
			\qquad + r^\alpha \|f\|_{L^\infty([0,T]\times\R^3\times B_{rR}(v_1))}.
	\end{split}
\ee

We now bound the $F(0,\cdot,\cdot)$ term in~\eqref{e.c62503}.  Fixing any $(x',v') \in B_R\times B_R$, we have
\[ \begin{split}
	F(0,x',v')
		&= \phi(x',v') (f(z_1\circ(\delta_r(0,x',v'))) - f(0,0,v_1))\\
		&= \phi(x',v') (f(0,r^{1+2s}x',rv'+v_1) - f(0,0,v_1)).
		\end{split}
\]
If $r^{1+2s}|x'|, r|v'| \leq 1$, then recalling the asymptotics of $\phi$ from~\eqref{e.cutoff} and the definition  \eqref{e.dl} of $d_\ell$, we have
\be\label{e.c6259}
	\begin{split}
	|F(0,x',v')|
		&\lesssim (|v'|^2+|x'|^2+1)^{-\frac{\alpha}{2}} r^\alpha \left(\max\{|x'|^{\frac{1}{1+2s}}, |v'|\}\right)^\alpha \vvone^{-q}[\vv^q f(0,\cdot,\cdot)]_{C^{\alpha}_{\ell,x,v}(B_1(0,v_1))}
		\\&\leq r^\alpha  \vvone^{-q} [\vv^q f(0,\cdot,\cdot)]_{C^{\alpha}_{\ell,x,v}(B_1(0,v_1))}.
	\end{split}
\ee
On the other hand, if either $r^{1+2s}|x'|$ or $r|v'| \geq 1$, we find
\be\label{e.c62501}
	\begin{split}
	|F(0,x',v')|
		&\lesssim (|v'|^2+|x'|^2+1)^{-\alpha/2} \vvone^{-q}\|f\|_{L_q^\infty([0,T]\times B_{rR}(0)) \times B_{rR}(v_1)}
		\\&
		\leq r^\alpha \vvone^{-q}\|f\|_{L^\infty_q([0,T]\times\R^6)},
	\end{split}
\ee
since $r^{\alpha(1+2s)} \leq r^\alpha$ and $|v'-v_1| \leq rR$ implies that $\vv \approx \vvone$. 
Combining~\eqref{e.c6259} and~\eqref{e.c62501}, we obtain
\be\label{e.c62502}
	|F(0,x',v')|
		\lesssim r^\alpha \vvone^{-q}\left([\vv^q f(0,\cdot,\cdot)]_{C^{\alpha}_{\ell,x,v}(B_1(x_1,v_1))} + \|f\|_{L_q^\infty([0,T]\times\R^6)}\right),
\ee
as desired.

Next, we turn to the middle term on the right hand side of~\eqref{e.c62503}.  Fix any $(t',x',v') \in [0,t_2]\times \overline B_R \times \overline B_R$ where $\max\{|x'|,|v'|\} = R$.  Again, recalling the properties ~\eqref{e.cutoff} and~\eqref{e.psi_twist} of $\phi$, we have
\[
	\begin{split}
	|F(t',x',v')|
		&\lesssim (|v'|^2+|x'|^2+1)^{-\alpha/2} \|f\|_{L^\infty([0,T]\times\R^3\times B_{rR}(v_1))}
	\\&
		\lesssim \min \left( \vvp^{-\alpha}, \langle x' \rangle^{-\alpha} \right) \vvone^{-q}\|f\|_{L_q^\infty([0,T]\times\R^6)}.
	\end{split}
\]
Recall from~\eqref{e.R} that $R =  \langle v_1\rangle/(2r)$.  Then clearly,
\be\label{e.c62504}
	|F(t',x',v')|
		\lesssim r^{\alpha} \vvone^{-q} \|f\|_{L_q^\infty([0,T]\times\R^6)}.
\ee
Combining~\eqref{e.c62503}, \eqref{e.c62502}, and~\eqref{e.c62504}, we find that \eqref{e.c6258} holds true.

\medskip

\noindent{\em Step 4: Conclusion of the proof.} 
Recalling that $|t_2| \leq \vvone^{-(\gamma+2s)_+} \leq 1$ and $|x_2|, |v_2| \leq 1$, we have $\phi(t_2,x_2,v_2) \approx 1$.  From Step 2, we have $F(t_2,x_2,v_2) \leq \overline F(t_2)$.  This yields
\[
	\begin{split}
	f(z_1 \circ \delta_r(z_2)) - f(z_1)
		= \frac{1}{\phi(t_2,x_2,z_2)} F(t_2,x_2,z_2)
		\lesssim \overline F(t_2).
	\end{split}
\]
Combining this with~\eqref{e.c6258}, we find
\[
	f((z_1 \circ \delta_r(z_2)) - f(z_1)
		\lesssim r^\alpha \vvone^{-q} \left([\vv^{-q} f(0,\cdot,\cdot)]_{C^{\alpha}_{\ell,x,v}(B_1(0,v_1))} + \|f\|_{L_q^\infty([0,T]\times\R^6)} \right).
\]
The same proof with $-\overline F$ as a lower barrier of $F$ gives
\[
	f((z_1 \circ \delta_r(z_2)) - f(z_1)
		\gtrsim - r^\alpha \vvone^{-q} \left([\vv^q f(0,\cdot,\cdot)]_{C^{\alpha}_{\ell,x,v}(B_1(0,v_1))} + \|f\|_{L_q^\infty([0,T]\times\R^6)} \right).
\]
We deduce
\[
	|f(z_1\circ\delta_r(z_2)) - f(z_1)|
		\lesssim r^\alpha \vvone^{-q}\left([\vv^q f(0,\cdot,\cdot)]_{C^{\alpha}_{\ell,x,v}(B_1(0,v_1))} + \|f\|_{L_q^\infty([0,T]\times\R^6)} \right),
\]
which concludes the proof.
\end{proof}

Now we are ready to use \Cref{lem:scaling_holder} to prove \Cref{prop:holder_in_t}:

\begin{proof}[Proof of \Cref{prop:holder_in_t}]
We will show
\begin{equation}\label{e.unweighted}
	\begin{split}
		&\|f\|_{C^{\alpha}_\ell(Q_1(z_0)\cap ([0,T]\times \R^6))}
			\\&\qquad
			\lesssim \langle v_0\rangle^{-q + \alpha \frac{(\gamma+2s)_+}{2s}} \Big(\|f\|_{L_q^\infty([0,T]\times \R^6)} +  \sup_{t \in [(t_0-1)_+,t_0]} [\vv^q f(t,\cdot,\cdot)]_{C_{\ell, x,v}^{\alpha}(B_2(x_0,v_0))}\Big).
	\end{split}
\end{equation}
The conclusion of the Proposition follows from \eqref{e.unweighted} in a straightforward way. 

Fix any $z_1$ and $z_2$ in $Q_1(z_0)\cap([0,T]\times\R^6)$, and assume without loss of generality that $t_2\geq t_1$.  
If $d_\ell(z_2, z_1) \geq \frac 1 2 \langle v_1\rangle^{-(\gamma+2s)_+/2s}$, then we simply have
\[
	\frac{|f(z_2) - f(z_1)|}{d_\ell(z_1,z_2)^{\alpha}}
		\lesssim \langle v_0\rangle^{\frac{(\gamma+2s)_+}{2s}\alpha} \vvo^{-q}\| f\|_{L_q^\infty([0,T]\times \R^6)}
\]
since $\langle v_0\rangle \approx \vvone$. Therefore, for the rest of the proof, we assume
\be\label{e.c1231}
	d_\ell(z_1,z_2) < \frac 1 2 \langle v_1 \rangle^{-\frac{(\gamma + 2s)_+}{2s}}.
\ee
We set the notation
\[
	s_2 = \frac{t_2-t_1}{r^{2s}}, 
		\quad
	w_2 = \frac{v_2-v_1}{r},
		\quad\text{ and }\quad
	y_2 = \frac{x_2 - x_1 - r^{2s} s_2 v_1}{r^{1+2s}}.
\]
where $r \in (0,1]$ is to be chosen based on two cases below.  We immediately notice that
\be\label{e.c1233}
	z_2 = z_1 \circ \delta_r(s_2,y_2,w_2).
\ee

The first, simpler case is when
\be\label{e.c1232}
	t_2-t_1 \leq \langle v_1 \rangle^{-(\gamma + 2s)_+} d_\ell(z_1,z_2)^{2s}.
\ee
In this case, we let
\[
	r = 2d_\ell(z_1,z_2).
\]
Let us check that the hypotheses of Lemma \ref{lem:scaling_holder} are satisfied. From \eqref{e.c1231}, we have $r\leq 1$.  Also, \eqref{e.c1232} implies
\[s_2
		= \frac{t_2-t_1}{2^{2s}d_\ell(z_1,z_2)^{2s}} < \langle v_1\rangle^{-(2s+\gamma)_+}.
		\]
From the definition \eqref{e.dl} of $d_\ell$, we have 
\[ |w_2| = \frac{|v_2-v_1|}{2d_\ell(z_1,z_2)}
		\leq 1, \qquad |y_2| = \frac{|x_2 - x_1 - (t_2-t_1)v_1|}{2^{1+2s}d_\ell(z_1,z_2)^{1+2s}}
		\leq 1.\]
Therefore, we can apply~\Cref{lem:scaling_holder} to find
\[
	|f(z_2) - f(z_1)|
		\lesssim r^\alpha \vvo^{-q} \left(\|f\|_{L_q^\infty([0,T]\times \R^6)} + [\vv^q f(t_1,\cdot,\cdot)]_{C^\alpha_{\ell,x,v}(B_2(x_0,v_0))}\right),
\]
using $B_1(x_1,v_1)\subset B_2(x_0,v_0)$. Since $r\approx d_\ell(z_1,z_2)$, this finishes the proof of \eqref{e.unweighted} in this case.

Next, we consider the case where
\be\label{e.c1234}
	t_2-t_1 > \langle v_1 \rangle^{-(\gamma + 2s)_+} d_\ell(z_1,z_2)^{2s}.
\ee
In this case, we let
\[
	r = 2(t_2-t_1)^\frac{1}{2s} \langle v_1 \rangle^\frac{(\gamma+2s)_+}{2s}.
\]
Notice that $r > 2^{1/(2s)} d_\ell(z_1,z_2)$ by~\eqref{e.c1234}. Once again, we would like to apply Lemma \ref{lem:scaling_holder}. From \eqref{e.c1231} and the definition of $d_\ell$, we have
\be
	r
		\leq 2d_\ell(z_1,z_2) \langle v_1\rangle^{(\gamma+2s)_+/(2s)}
		\leq 1.
\ee
We also have $s_2 < \langle v_1\rangle^{-(\gamma+2s)_+}$ by construction. Using $r> 2^{1/(2s)} d_\ell(z_1,z_2)$ and the definition of $d_\ell$, we have
\[  |w_2|
		< \frac{|v_2-v_1|}{2d_\ell(z_1,z_2)}
		\leq 1, \qquad |y_2| = \frac{|x_2 - x_1 - (t_2-t_1)v_1|}{r^{1+2s}} < \frac{|x_2 - x_1 - (t_2-t_1)v_1|}{2^{1+\frac{1}{2s}}d_\ell(z_1,z_2)^{1+2s}}\leq 1.\]
Therefore, we can apply~\Cref{lem:scaling_holder} as above to find
\[
	|f(z_2) - f(z_1)|
		\lesssim r^\alpha \vvone^{-q} \left(\|f\|_{L_q^\infty([0,T]\times \R^6)} + [\vv^{-q} f(t_1,\cdot,\cdot)]_{C^\alpha_{\ell,x,v}(B_2(x_0,v_0))}\right),
\]
which concludes the proof of \eqref{e.unweighted}, since $r\lesssim \langle v_1 \rangle^{(1+\gamma/(2s))_+} d_\ell(z_1,z_2)$ and $\langle v_1 \rangle \approx \langle v_0\rangle$.
\end{proof}

\section{Propagation of H\"older regularity}\label{s:holder-xv}

In this section and the next, we need to place extra assumptions on our initial data as in the statement of Theorem \ref{t:uniqueness}. We recall here the lower bound condition \eqref{e.uniform-lower}: there are $\delta$, $r$, and $R>0$ such that 
\begin{equation*}
\text{for each } x\in \R^3, \, \exists \, v_x\in B_R(0) \text{ such that } f_{\rm in} \geq \delta 1_{B_r(x,v_x)}.
\end{equation*}
Let $f$ be the solution with initial data $f_{\rm in}$ constructed in Theorem \ref{t:existence}, as the limit of the sequence $f^\eps$. Note that the estimate \eqref{e.uniform-lower-b} is independent of $\eps$. In fact, since the initial lower bound \eqref{e.uniform-lower} is now taken to be uniform in $x$, the small-time lower bound \eqref{e.small-t-lower} holds at the $\eps$ level. Since the sequence $f^\eps$ is locally bounded in $C^\alpha$, pointwise convergence implies that $f$ also satisfies the corresponding lower bounds: 
\begin{equation}\label{e.f-lower}
f(t,x,v) \geq \mu, \quad (t,x,v) \in (0,T]\times B_{r/2}(x,v_x),
\end{equation}
for some $\mu>0$ depending on $T$, $\delta$, $r$, $R$, $\|f\|_{L^\infty_q[0,T]\times \R^6)}$, and $q> 3 + \gamma + 2s$. 
The uniformity of these lower bounds as $t\to 0$ will allow us to control the time-dependence of the constants when we apply Schauder estimates.

We also assume the initial data is H\"older continuous.  The main result of this section propagates this H\"older regularity to positive times:

\begin{proposition}\label{prop:holder_propagation}
	Let $f:[0,T]\times\R^6\to \R$ be the solution to \eqref{e.boltzmann} constructed in Theorem \ref{t:existence}.  Suppose that $f_{\rm in}$ satisfies \eqref{e.uniform-lower} for some $\delta$, $r$, and $R$, and that $\langle v\rangle^{m} f_{\rm in} \in C^{\alpha(1+2s)}_\ell(\R^6)$ for some $\alpha\in (0,\min\{1,2s\})$ and $f_{\rm in} \in L^\infty_q(\R^6)$, and that $m$, $q$, and $q-m$ are sufficiently large, depending on $\gamma$, $s$, and $\alpha$.
	
	Then there exists $T_U \in (0,T]$ such that 
	\begin{equation}\label{e.holder-prop-bound}
		\|\langle v\rangle^{m-\alpha(\gamma+2s)_+/(2s)} f\|_{C^\alpha_\ell([0,T_U]\times \R^6)}
			\leq C \|\langle v\rangle^{m} f_{\rm in}\|_{C^{\alpha(1+2s)}_{\ell,x,v}(\R^6)}.
	\end{equation}
	The constants $C$ and $T_U$ depend only on universal constants, $m$, $q$, $\alpha$, $\delta$, $r$, $R$, $[f_{\rm in}]_{C^\alpha_{\ell,q}}$, and $\|f\|_{L^{\infty}_q([0,T]\times\R^6)}$.  
\end{proposition}

To control the H\"{o}lder continuity of $f$, we adapt an idea from a previous work on well-posedness for the Landau equation \cite{HST2020landau}, originally inspired by a method of \cite{constantin2015SQG} to obtain regularity for the SQG equation.
For $(t,x,v,\chi,\nu) \in \R_+ \times \R^3 \times \R^3 \times B_1(0)^2$ and $m \in \N$, define
\begin{equation}\label{e.c010404}
\begin{split}
&\tau f(t,x,v,\chi,\nu) := f(t,x+\chi, v+\nu),\\
&\delta f (t,x,v,\chi,\nu) = \tau f(t,x,v,\chi,\nu) - f(t,x,v),\\
&g(t,x,v,\chi,\nu) = \frac{\delta f (t,x,v,\chi,\nu)}{(|\chi|^{2} + |\nu|^2)^{\alpha/2}} \vv^{m},
\end{split}
\end{equation}
Note that, if $f \in C^\infty$, then $\lim_{(\chi,\nu)\rightarrow (0,0)} g(t,x,v,\chi,\nu)$ exists for every $(t,x,v)$. The function $g$, defined by this limit
on $\R_+ \times \R^3 \times \R^3 \times \lbrace 0 , 0 \rbrace$, is then $C^\infty$.  By symmetry, the maximum of $g$ and the minimum of $g$ are the same magnitude.  Thus, it is equivalent to the $L^\infty_{x,v,\chi,\nu}$ norm of $g$, which is, in turn, equivalent to the weighted $C^\alpha_{x,v}$ semi-norm of $f$, as remarked in this elementary lemma:

\begin{lemma}\label{lem:holder_char}
	Fix any $f : \R^3 \times \R^3 \to \R$ and let $g: \R^3 \times \R^3 \times B_1\times B_1 \to \R$ be defined by
	\[
		g(x,v,\chi,\nu) = \vv^m \frac{\delta f(x,v,\chi,\nu) }{(|\chi|^2+|\nu|^2)^{\alpha/2}}.
	\]
	Then
	\[
		\max_{(x,v,\chi,\nu) \in \R^6\times B_1^2} g
		= \|g\|_{L^\infty(\R^6\times B_1^2)} 
		\approx \|\langle v \rangle^m f\|_{C_{x,v}^\alpha(\R^6)}
		\approx \sup_{(x,v)\in\R^6} \vv^{m} \|f\|_{C_{x,v}^\alpha(B_1(x,v))}
	\]
	where the implied constants depend only on $m$ and $\alpha$.  Here, $C_{x,v}^\alpha$ denotes the standard H\"older space on $\R^6$.
\end{lemma}

We emphasize that $g$ measures the H\"older continuity of $f$ in the Euclidean metric of $\R^6$, rather than the metric $d_\ell$ that matches the scaling of the equation. This choice is imposed on us by the proof: see the term $\chi\cdot \nu/(|\chi|^2 + |\nu|^2) g$ in \eqref{e.c010401} below. This term needs to be uniformly bounded, which would not be the case if the displacements $\chi$ and $\nu$ were given the natural exponents according to the kinetic scaling. It is straightforward to show that the two H\"older norms control each other, although with a loss of exponent: for any suitable function $h$ and domain $\Omega\subset\R^6$, 
\begin{equation}\label{e.holder-compare}
	\|h\|_{C^\frac{\alpha}{1+2s}_{\ell,x,v}(\Omega)}
		\lesssim \|h\|_{C_{x,v}^\frac{\alpha}{1+2s}(\Omega)}
		\lesssim \|h\|_{C^\alpha_{\ell,x,v}(\Omega)},
\end{equation}
where the $C^\alpha_{\ell,x,v}$ norm has been defined in Section \ref{s:holder-spaces}.

Our strategy to prove Proposition \ref{prop:holder_propagation} is to bound $g$ from above using a barrier argument. The defining equation for the barrier $\overline G$ will correspond to the estimates that are available for $g$ at a first crossing point, so that we can derive a contradiction at that point. Therefore, we present the upper bounds for $g$ in the following key lemma, before explaining the barrier argument.

\begin{lemma}\label{lem:holder_prop}
	Let $f$, $\alpha$, $m$, and $q$ be as in Proposition~\ref{prop:holder_propagation}, and let $g$ be defined as in \eqref{e.c010404}. For some $t_\rmcr>0$, if $g$ has a global maximum over $[0,t_\rmcr]\times \R^6\times B_1(0)^2$ at $(t_\rmcr, x_\rmcr,v_\rmcr, \chi_\rmcr, \nu_\rmcr)$, then
	\be \label{e.g-ineq}
		\partial_t g
			\leq  C \left(g + t_{\rmcr}^{\mu(\alpha,s)} g^{\theta(\alpha,s)}\right)
			\qquad\text{ at } (t_\rmcr,x_\rmcr, v_\rmcr, \chi_\rmcr,\nu_\rmcr),
	\ee
	where $\alpha' = \alpha \frac{2s}{1+2s}$, $n\geq 2$ is such that $2s + \alpha'/n < 2$, and,
\[ \begin{split}
	\mu(\alpha,s) &:= \left( - 1 + \frac{\alpha-\alpha'}{2s}\right)\left( \frac{2s+\alpha'/n}{2s+\alpha'}\right) \in (-1,0),\\
	\theta(\alpha,s) &:= 1 +  \left(1+\frac{\alpha+2s}{\alpha'}\right) \left( \frac{2s+\alpha'/n}{2s+\alpha'}\right),
\end{split} \]
and the constant $C>0$ depends on universal constants, $\alpha$, $q$, $m$, $\delta$, $r$, $R$, $n$, and $\|f\|_{L^\infty_q([0,T]\times\R^6)}$.
\end{lemma}

Before proving Lemma~\ref{lem:holder_prop}, we show how to use it to conclude Proposition~\ref{prop:holder_propagation}.
\begin{proof}[Proof of Proposition~\ref{prop:holder_propagation}]
We begin by noting that we can assume, without loss of generality, that
\be\label{e.c12231}
	\|\vv^{q'} D^k f\|_{L^\infty([0,T]\times\R^6)} < \infty
		\qquad\text{ for every $q'$ and multi-index $k$ sufficiently large.}
\ee
Indeed, if not, we may use the approximating sequence $f^\eps$ from the proof of Theorem \ref{t:existence}, which is sufficiently smooth, and the bound \eqref{e.holder-prop-bound}, which does not depend quantitatively on norms of order higher than $\alpha$, is inherited by $f$ in the limit.

First, we claim that, with $f$, $m$, $\alpha$, and $T_U \in (0,T]$ as in the statment of the proposition, 
	\begin{equation}\label{e.holder-prop-bound-xv}
		\|\langle v\rangle^{m} f\|_{L^\infty_t([0,T_U], C^\alpha(\R^6))}
			\leq 2 \|\langle v\rangle^{m} f_{\rm in}\|_{C^\alpha(\R^6)}.
	\end{equation}
To prove \eqref{e.holder-prop-bound-xv}, we use the function $g$ defined in \eqref{e.c010404} and construct a barrier $\overline G$ on a small time interval, that controls $g$ from above. 
With $N > 0$ to be chosen later, define $\overline G$ to be the unique solution to
\begin{equation}\label{e.g_supersoln}
\begin{cases}
	\frac{d}{dt} \overline G(t) = N \left(\overline G + t^{\mu(\alpha,s)}\overline G(t)^{\theta(\alpha,s)}\right),\\
	\overline G(0) = 1 + \|g(0,\cdot)\|_{L^\infty (\R^6 \times B_1(0)^2)} + N \| f \|_{L^\infty_q([0,T] \times \R^6)}.
\end{cases}
\end{equation}
where $\mu(\alpha,s)$ and $\theta(\alpha,s)$ are as in Lemma \ref{lem:holder_prop}. 

This solution $\overline G$ exists on some time interval $[0,T_G]$ with $T_G$ depending only on $\alpha$, $s$, $N$, $\|g(0,\cdot)\|_{L^\infty}$, and $\| f \|_{L^\infty_q}$.  Later, we will choose $N$ depending only on $\| f \|_{L^\infty_q}$. 

Our goal is to show that $g(t,x,v,\chi,\nu) < \overline G(t)$ for all $t\in [0,T_G]$. Let $t_\rmcr$ be the first time that $\|g\|_{L^\infty([0,t_\rmcr]\times \R^3 \times \R^3 \times B_1(0)^2)} = \overline G(t_\rmcr)$. It is clear from \eqref{e.c12231} that $t_\rmcr >0$. We seek a contradiction at $t=t_\rmcr$.

Next, we claim that we may assume existence of a point $(x_\rmcr,v_\rmcr,\chi_\rmcr,\nu_\rmcr) \in \R^3\times\R^3\times \overline B_1(0)^2$ so that
\be\label{e.c12233}
	g(t_\rmcr,x_\rmcr,v_\rmcr, \chi_\rmcr, \nu_\rmcr)
			= \|g(t_\rmcr,\cdot,\cdot,\cdot,\cdot)\|_{L^\infty}
			= \overline G(t_\rmcr).
\ee
Indeed, if not we may take any sequence $(t_\rmcr, x_n, v_n, \chi_n, \nu_n)$ such that 
\be
		g(t_\rmcr,x_n,v_n, \chi_n, \nu_n) \to 
			\|g(t_\rmcr,\cdot,\cdot,\cdot,\cdot)\|_{L^\infty}
			= \overline G(t_\rmcr).
\ee
Then define
\be
	f_n(t,x,v,\chi,\nu)
		= f(t, x+x_n, v, \chi, \nu)
	\quad\text{ and }\quad
	g_n(t,x,v,\chi,\nu)
		= g(t, x+x_n, v, \chi, \nu).
\ee
Due to the fast decay of $f$ as $|v|\to\infty$ and its smoothness, see~\eqref{e.c12231}, it follows that, up to passing to a subsequence, there exist $\tilde f$ and $\tilde g$ such that $f_n \to \tilde f$ and $g_n \to \tilde g$ in $C^{2+\alpha}_{\ell}$ locally uniformly.  

Using again the fast decay and smoothness of $f$, it follows that $|v_n|$ is bounded, in which case, up to taking a subsequence, $v_n \to v_\rmcr$ for some $v_\rmcr\in \R^3$.  Similarly, $(\chi_n, \nu_n) \to (\chi_\rmcr,\nu_\rmcr) \in \overline B_1(0)^2$.  It follows that
\be\label{e.c12232}
	\tilde g(t_\rmcr, x_\rmcr, v_\rmcr, \chi_\rmcr, \nu_\rmcr)
		= \lim_{n\to\infty} g_n(t_\rmcr, x_\rmcr, v_n, \chi_n, \nu_n)
		= \overline G(t_\rmcr).
\ee

Notice that $\tilde f$ inherits all the of the same (global) bounds as $f$ and satisfies the Boltzmann equation~\eqref{e.boltzmann}, by the locally uniform convergence in $C^{2+\alpha}_\ell$. Therefore, without loss of generality, a crossing point exists as in \eqref{e.c12233}. 

Now, we show that, up to increasing $N$ if necessary, $(\chi_\rmcr,\nu_\rmcr) \in B_1(0)^2$; that is, $\chi_\rmcr$ and $\nu_\rmcr$ lie in the interior of $B_1(0)$.  Indeed, if $(\chi_\rmcr,\nu_\rmcr)$ were located on the boundary of $B_1(0)^2$, then a direct calculation using \eqref{e.c010404} and $\chi_\rmcr^2 + \nu_\rmcr^2 \geq 1$ shows
\[
	 g(t_\rmcr, x_\rmcr, v_\rmcr, \chi_\rmcr, \nu_\rmcr)
		\leq \|f\|_{L^\infty_m}
		< \overline G(t_\rmcr),
\]
which contradicts~\eqref{e.c12233}.  It follows that $(\chi_\rmcr,\nu_\rmcr)\in B_1(0)^2$.

From Lemma~\ref{lem:holder_prop}, we find, at $(t_\rmcr, x_\rmcr, v_\rmcr, \chi_\rmcr, \nu_\rmcr)$,
\[
	\partial_t g
		\leq C\left( g + t_0^{\mu(\alpha,s)} g^{\theta(\alpha,s)}\right).
\]
Since $g = \overline G$ at this point, we have
\[
	\partial_t g
		\leq C\left( \overline G + t_0^{\mu(\alpha,s)} \overline G^{\theta(\alpha,s)}\right).
\]
Hence, by increasing $N$ if necessary, we have, due to~\eqref{e.g_supersoln},
\be\label{e.c12234}
	\partial_t g
		< \frac{d\overline G}{dt}.
\ee
On the other hand, since $\overline G - g$ has a minimum at $(t_\rmcr, x_\rmcr, v_\rmcr, \chi_\rmcr, \nu_\rmcr)$, one has
\be\label{e.c12235}
	\partial_t( \overline G - g) \leq 0.
\ee
The inequalities~\eqref{e.c12234} and~\eqref{e.c12235} contradict each other, which implies the time $t_\rmcr$ does not exist.

This establishes that $g< \overline G$ on $[0,T_G]$. The inequality \eqref{e.holder-prop-bound-xv} therefore holds up to a time $T_U = \min\{T, T_G\}$, which implies the existence of $T_U$ as in the statement of the proposition.  

Since the statement we want to prove is in terms of H\"older norms $C^\alpha_\ell$ with kinetic scaling, we apply \eqref{e.holder-compare} to both sides of \eqref{e.holder-prop-bound-xv} and obtain 
\[
	\|\langle v\rangle^{m} f\|_{L^\infty_t([0,T_U], C^\alpha_{\ell,x,v}(\R^6))}  
		\lesssim \|\vv^m f_{\rm in}\|_{C^{\alpha(1+2s)}_{\ell,x,v}(\R^6)}.	
\]
Finally, we apply Proposition \ref{prop:holder_in_t} to conclude \eqref{e.holder-prop-bound}, and the proof is complete.
\end{proof}

We now prove the key estimate of Lemma~\ref{lem:holder_prop}.

\begin{proof}[Proof of Lemma \ref{lem:holder_prop}]
We proceed in several steps.  First, we convert the classical derivative terms in the equation for $f$ to derivatives of $g$.  Next, we obtain bounds on the collision operator using the bounds on $g$ and the Schauder estimates for $f$.  This involves an intricate decomposition of the collision kernel, each portion of which we write as a separate step.

\smallskip

\noindent{}{\it Step 1: The equation for $g$.} If we take a finite difference of the Boltzmann equation \eqref{e.boltzmann}, we get that
\begin{equation}
\label{e.delta}
\partial_t \delta f + v \cdot \nabla_x \delta f + \nu \cdot \nabla_\chi \delta f = \tau Q(f,f) - Q(f,f).
\end{equation}
Multiplying by $\vv^{m} (|\chi|^2 + |\nu|^2)^{-\alpha/2}$ and commuting the derivative operators yields
\be\label{e.c010401}
	\partial_t g
		+ v \cdot \nabla_x g
		+ \nu \cdot \nabla_\chi g
		+ \alpha \frac{\chi \cdot \nu}{|\chi|^2 + |\nu|^2} g
		= \frac{\vv^m}{(|\chi|^2 + |\nu|^2)^\frac{\alpha}{2}} \left( \tau Q(f,f) - Q(f,f)\right).
\ee
We next consider the right hand side of~\eqref{e.c010401}. 
%
Applying the Carleman decomposition $Q = Q_{\rm s} + Q_{\rm ns}$ as usual, we write
\begin{equation}\label{e.320}
	\tau(Q(f,f)) - Q(f,f)
		= Q_\text{s}(\delta f, \tau f)
			+ Q_\text{s}(f, \delta f)
			+ Q_\text{ns}(\delta f, \tau f)
			+ Q_\text{ns}(f, \delta f).
\end{equation}
Since $(t_\rmcr, x_\rmcr,v_\rmcr,\chi_\rmcr,\nu_\rmcr)$ is the location of an interior maximum, we have $\nabla_x g = \nabla_\chi g = 0$ at this point. Hence~\eqref{e.c010401} becomes
\be\label{e.c010402}
	\begin{split}
		\partial_t g
		&+ \alpha \frac{\chi \cdot \nu}{|\chi|^2 + |\nu|^2} g
		= \frac{\vv^m}{(|\chi|^2 + |\nu|^2)^\frac{\alpha}{2}} \left( Q_\text{s}(\delta f, \tau f) + Q_\text{s}(f, \delta f) + Q_\text{ns}(\delta f, \tau f) + Q_\text{ns}(f, \delta f)\right).
	\end{split}
\ee
Since $|\chi\cdot \nu|/(|\chi|^2+|\nu|^2) \leq 1$, the conclusion of the lemma follows if we can find suitable upper bounds for the terms on the right-hand side of \eqref{e.c010402}. We estimate these four terms one by one.

\medskip

\noindent{}{\it Step 2: Bounding the $\Qs(\delta f, \tau f)$ term in~\eqref{e.c010402}.} As usual, since $Q_{\rm s}$ acts only in the velocity variable, we omit the dependence on $(t_\rmcr, x_\rmcr)$ in the following calculation. 

First, we make note of a useful fact often used in the sequel: since $\nu \in B_1(0)$, it follows that
\be\label{e.c010403}
	\tau f (v) \vv^k
		\lesssim \| f \|_{L^\infty_k}
			\qquad\text{ for any $k\geq 0$}.
\ee

Next, we recall that $(t_\rmcr, x_\rmcr, v_\rmcr, \chi_\rmcr, \nu_\rmcr)$ is crucially the location of a maximum of $g$.  
For ease of notation, we drop the `cr' subscript and simply refer to the point as $(t,x,v,\chi,\nu)$. 

We set
\be
	J_1 = \frac{\vv^m}{(|\chi|^2 + |\nu|^2)^{\alpha/2}} \Qs(\delta f, \tau f).
\ee
Using \eqref{e.kernel}, this can be rewritten as:
\[
	\begin{split}
	J_1
			&= \int_{\R^3} (\tau f(v) - \tau f(v')) \frac{\vv^m K_{\delta f}(v,v')}{(|\chi|^2 + |\nu|^2)^{\alpha/2}} \dd v' 
		= \int_{\R^3} (\tau f(v) - \tau f(v')) \tilde{K}_g (v,v') \dd v',
\end{split}
\]
where, recalling the definition~\eqref{e.c010404} of $g$,
\[
	\tilde{K}_g(v,v')
	= \frac{\vv^m K_{\delta f}(v,v')}{(|\chi|^2 + |\nu|^2)^{\alpha/2}}
	\approx |v-v'|^{-3-2s} \int_{w \perp v-v'} g(v+w) \frac{\vv^m}{\langle v+w \rangle^m} |w|^{\gamma+2s+1} \dd w.
\]

Let us record some useful upper bounds for $\tilde K_g$: from Lemma \ref{l:K-upper-bound}, we have, for any $r>0$,
\begin{equation}
\label{e.K_annuli}
\begin{split}
\int_{B_{2r} \setminus B_r} &|\tilde{K}_g(v,v+z)|\dd z \lesssim
\frac{\vv^m}{(|\chi|^2+|\nu|^2)^{\alpha/2}} \left( \int_{\R^3} \delta f(v+w) |w|^{\gamma+2s}\dd w\right) r^{-2s}\\
&\lesssim r^{-2s} \int_{\R^3} \frac{ |g(v+w)| |w|^{\gamma+2s}\vv^m}{ \langle v+w \rangle^m} \dd w
\leq C r^{-2s} \| g \|_{L^\infty}\vv^{m+(\gamma+2s)_+},
\end{split}
\end{equation}
since $m> 3+\gamma+2s$ and by \Cref{l:convolution}. Applying \eqref{e.K_annuli} on an infinite union of annuli $B_{2r}\setminus B_r$, $B_{4r}\setminus B_{2r}$, etc., we obtain
\begin{equation}
\label{e.K_tail}
	 \int_{B_r^c}
		|\tilde{K}_g(v,v+z)| \dd z
			\lesssim r^{-2s} \| g \|_{L^\infty} \vv^{m+(\gamma+2s)_+}.
\end{equation}
Finally, from \cite[Lemma 2.4]{henderson2021existence} (which holds for all ranges of $\gamma+2s$), we obtain the following pointwise upper bound: if $|v'| \leq |v|/3$, then
\begin{equation}
\label{e.K_near_zero}
	|\tilde K_g(v,v')|
		\lesssim \frac{\| g \|_{L^\infty}}{|v-v'|^{3+2s}} \vv^{\gamma + 3 + 2s}
		\lesssim \frac{\| g \|_{L^\infty}}{| v|^{3+2s}} \vv^{\gamma + 3 + 2s}.
\end{equation}
We require this estimate since we work in uniform spaces with weights in $v$, which means we sometimes encounter the quantity
$\vv / \langle v' \rangle$. This is bounded, except when $|v'|$ is small compared to $|v|$, so we need the extra moment decay of \eqref{e.K_near_zero}
to compensate in that case.

The analysis now proceeds in two slightly different ways, based on two cases.

\smallskip

\noindent{}{\it Case 1: $|v| \geq 1$.} Setting $R = 4 \vv / 3$ and $r = |v|/3$, we write
\[
\begin{split}
J_1 &= \int_{B_R^c(v)} (\tau f(v) - \tau f(v')) \tilde K_g(v,v') \dd v' + \int_{B_r(0)} (\tau f(v) - \tau f(v')) \tilde K_g(v,v') \dd v' \\
&\quad\quad\quad\quad+ \int_{B_R(v) \setminus B_r(0)} (\tau f(v) - \tau f(v')) \tilde K_g(v,v') \dd v' =: J_{1,1} + J_{1,2} + J_{1,3}.
\end{split}
\]
Notice that the only term involving the singularity at $v=v'$ is $J_{1,3}$ due to the choice of $R$ and $r$.

For the long-range term, using \eqref{e.K_tail} and the fact that $\langle v'\rangle\gtrsim \vv$ when $v'\in B_R^c$, we have
\[
\begin{split}
J_{1,1} &\leq \int_{B_R^c(v)} \left(|\tau f(v)|\vv^{m+(\gamma+2s)_+} + |\tau f(v')| \langle v'\rangle^{m+(\gamma+2s)_+}\right) \frac{|\tilde K_g(v,v')|}{\vv^{m+(\gamma+2s)_+}} \dd v'\\
&\lesssim \| f \|_{L^\infty_q} \int_{B_R^c(v)} \frac{|\tilde K_g(v,v')|}{\vv^{m+(\gamma+2s)_+}} \dd v'
\lesssim \| f \|_{L^{\infty}_q} \vv^{-2s} \| g \|_{L^\infty}.
\end{split}
\]
In the second inequality, we used that $q > m + (\gamma + 2s)_+$, by assumption.

The near-zero term is bounded using \eqref{e.K_near_zero} as follows:
\[
\begin{split}
J_{1,2} &\lesssim \int_{B_r(0)} \frac{|\tau f (v)| + |\tau f (v')|}{|v|^{3+2s}} \vv^{\gamma + 3 + 2s} \| g \|_{L^\infty} \dd v' \\
&\lesssim \| g \|_{L^\infty} \| f \|_{L^\infty_q}
\int_{B_r(0)}\left( \frac{\vv^{3+\gamma+2s-q} + \vv^{3+\gamma+2s} \langle v'\rangle^{-q}}{ |v|^{3+2s}}\right) \dd v'.
\end{split}
\]
Recalling that $|v| \geq 1$ so $|v|\approx \vv$, we then find
\[
	J_{1,2}
		\lesssim \|g\|_{L^\infty} \|f\|_{L^\infty_q}\int_{B_r(0)} \left(\vv^{\gamma-q} + \vv^\gamma \langle v'\rangle^{-q}\right) \dd v'
		\lesssim \|g\|_{L^\infty} \|f\|_{L^\infty_q} (1+\vv^\gamma),\]
since $q>3$.

The short-range term $J_{1,3}$ is the most difficult to bound because it contains the singularity at $v=v'$. To handle this, we use smoothness of $\tau f$, which follows from the Schauder estimate of \Cref{p:nonlin-schauder}: with $p\geq 0$ to be chosen later,
\begin{equation}\label{e.schauder1}
	\|f\|_{C^{2s+\alpha'}_{\ell,m-p-\kappa}([t_0/2,t_0]\times\R^6)}
		\lesssim t_0^{-1+(\alpha-\alpha')/(2s)} \|f\|_{C^\alpha_{\ell,m-p+[\alpha/(1+2s)+\gamma]_+}([0,t_0]\times\R^6)}^{1+(\alpha+2s)/\alpha'}.
 \end{equation}
Notice that \Cref{p:nonlin-schauder} requires a lower bound on $m-p,$ which holds after further increasing $m$ if necessary.
 
Since we need to bound $J_{1,3}$ in terms of $g$ (which corresponds to the weighted $C^\alpha$ norm of $f$ in $(x,v)$-variables) rather than the $(t,x,v)$-H\"older norm, we combine \eqref{e.schauder1} with Proposition \ref{prop:holder_in_t} to write 
\begin{equation}\label{e.schauder-application}
\begin{split}
\|f\|_{C^{2s+\alpha'}_{\ell,m-p-\kappa}([t_0/2,t_0]\times\R^6)}
	&\leq C t_0^{-1+(\alpha-\alpha')/(2s)}  \sup_{t\in[0,t_0]} \left\|\vv^{m} f(t)\right\|_{C^\alpha_{\ell,x,v}(\R^6)}^{1+(\alpha+2s)/\alpha'}\\
&\leq C t_0^{-1+(\alpha-\alpha')/(2s)} \|g\|_{L^\infty}^{1+(\alpha+2s)/\alpha'},
\end{split}
\end{equation}
using Lemma \ref{lem:holder_char} and \eqref{e.holder-compare}. Here, we have chosen
\[
p :=\left[\frac{\alpha}{1+2s}+\gamma\right]_+ +\frac{\alpha}{2s}(\gamma+2s)_+.
\]
The decay exponent $m-p-\kappa$ on the left in \eqref{e.schauder-application} is too weak for our estimates below. To get around this, we use interpolation to trade regularity for decay: recalling that $\alpha' = \alpha \frac {2s}{1+2s}$, for 
\[
	q \geq m
		+ \frac{2s + \alpha'}{\frac{n-1}{n} \alpha'} \left( (\gamma+2s)_+ + \alpha'\right)
		+ \left(\frac{1}{n-1} + \frac{2s}{\frac{n-1}{n} \alpha'}\right) (p+\kappa),
\]
Lemma \ref{l:moment-interpolation} implies
\begin{equation}\label{e.schauder-application2}
\begin{split}
	\|f\|_{C^{2s+\alpha'/n}_{\ell,m+(\gamma+2s)_+ + \alpha'}([t_0/2,t_0]\times\R^6)} 
		&\lesssim
			\|f\|_{C^{2s+\alpha'}_{\ell,m-p-\kappa}([t_0/2,t_0]\times\R^6)}^{\frac{2s+\alpha'/n}{2s+\alpha'}}
			\|f\|_{L^\infty_{q}}^{\frac{n-1}{n}\frac{\alpha'}{2s+\alpha'}} \\
		&\lesssim
			\left(t_0^{-1+(\alpha-\alpha')/(2s)}
				\|g\|_{L^\infty}^{1+(\alpha+2s)/\alpha'}\right)^{\frac{2s+\alpha'/n}{2s+\alpha'}}
				\|f\|_{L^\infty_{q}}^{\frac{n-1}{n}\frac{\alpha'}{2s+\alpha'}} \\
		&\lesssim t_0^{\mu(\alpha,s)} \|g\|_{L^\infty}^{\eta(\alpha,s)},
\end{split}
\end{equation}
where 
\[ \begin{split}
	\mu(\alpha,s) &:= \left( - 1 + \frac{\alpha-\alpha'}{2s}\right)\left( \frac{2s+\alpha'/n}{2s+\alpha'}\right) \in (-1,0),\\
	\eta(\alpha,s) &:= \left(1+\frac{\alpha+2s}{\alpha'}\right) \left( \frac{2s+\alpha'/n}{2s+\alpha'}\right),
\end{split} \]
and we have absorbed the dependence on $\|f\|_{L^\infty_{q}}$ into the implied constant.

Now, we apply \eqref{e.schauder-application2} to the term $J_{1,3}$.   The argument differs slightly based on whether $2s +\alpha'/2 > 1$ or not.  We consider the former case, as it is more complicated.
Recall, by hypothesis, that $2s + \alpha'/n < 2$.  
For $v'\in B_R(v)\setminus B_r(0)$, since $|v+\nu| \approx |v|$, estimate \eqref{e.schauder-application2} implies
\begin{equation}\label{e.tau-est}
	\begin{split}
	 &\frac{|\tau f (v') - \tau f (v) - (v-v')\cdot \nabla_v (\tau f)(v)|}{|v-v'|^{2s+\alpha'/n}} \vv^{m+(\gamma+2s)_+ +\alpha'}
		\lesssim  t_0^{\mu(\alpha,s)} \|g\|_{L^\infty}^{\eta(\alpha,s)}.
	\end{split}
\end{equation}
Note that we do not require any $(\partial_t + v\cdot\nabla_x)\tau f$ terms in the above Taylor expansion because the shift is only in $v$. 
 Let $\mathcal{A}_j := B_{R/2^j}(v) \setminus B_{R/2^{j+1}}(v)$ with the added condition that
$\mathcal{A}_0 = (B_R(v) \setminus B_{R/2}(v)) \setminus B_r(0)$. 

We treat the cases $j=0$ and $j\geq 1$ separately.  The simpler case is $j=0$, where the regularity of $f$ is not required because $v'$ is bounded away from $v$.  In this case, using that $\vvp \approx \vv$ and the inequality~\eqref{e.K_annuli}, we find
\be
	\begin{split}
	\int_{\mathcal{A}_0} &|\tau f(v) - \tau f(v')| |\tilde K_g(v,v')| dv'
		\lesssim \vv^{-q} \int_{\mathcal{A}_0} |\tilde K_g(v,v')| dv'
		\\&
		\lesssim \vv^{-q} \int_{B_R(v) \setminus B_{R/2}(v)} |\tilde K_g(v,v')| dv'
		\lesssim \vv^{-q - 2s} \vv^{m + (\gamma + 2s)_+} \|g\|_{L^\infty}
		\lesssim \|g\|_{L^\infty}.
	\end{split}
\ee
Next we consider the general $j$ case.  Here we use the symmetry of $\tilde{K}_g(v,v')$ around $v$ (i.e., $\tilde{K}_g(v,v+h) = \tilde{K}_g(v,v-h)$), the regularity estimate~\eqref{e.tau-est}, and the fact that $\vvp \approx \vv$ to obtain
\[
\begin{split}
	\sum_{j \geq 1} &\Big|\int_{\mathcal{A}_j} (\tau f(v') - \tau f(v)) \tilde K_g(v,v') \dd v'\Big| \\
		&= \sum_{j \geq 1} \Big|\int_{\mathcal{A}_j} (\tau f(v') - \tau f(v)  - (v-v')\cdot \nabla_v (\tau f)(v) ) \tilde K_g(v,v') \dd v'\Big| \\
		&\lesssim t_0^{\mu(\alpha,s)}\|g\|_{L^\infty}^{\eta(\alpha,s)}\vv^{-m-(\gamma+2s)_+-\alpha'}\sum_{j \geq 0}  
\int_{\mathcal{A}_j} |v-v'|^{2s+\alpha'/n} |\tilde K_g(v,v')| \dd v' \\
		&\lesssim t_0^{\mu(\alpha,s)}\|g\|_{L^\infty}^{\eta(\alpha,s)} 
 \vv^{-\alpha'} \sum_{j \geq 0} \frac{R^{2s+\alpha'/n}}{2^{j(2s+\alpha'/n)}}\| g \|_{L^\infty} \frac{2^{2s j}}{R^{2s}} 
	\lesssim t_0^{\mu(\alpha,s)}\|g\|_{L^\infty}^{1+\eta(\alpha,s)},
\end{split}
\]
where we again used~\eqref{e.K_annuli} in the second to last inequality.
Putting together both estimates above, we find
\be
	J_{1,3} \lesssim \|g\|_{L^\infty} + t_0^{\mu(\alpha,s)}\|g\|_{L^\infty}^{1 + \eta(\alpha,s)}.
\ee

Putting the estimates of $J_{1,i}$ together, we have that
\begin{equation}
\label{e.J1}
J_1 \lesssim \|g\|_{L^\infty} + t_0^{\mu(\alpha,s)} \|g\|_{L^\infty}^{1+\eta(\alpha,s)}.
\end{equation}

\noindent{}{\it Case 2: $|v| < 1$.} This case is similar to Case 1, but more simple. It is necessary because the estimates
we used for $J_{1,2}$ above relied on $|v|$ being bounded away from zero. 
In this case, however, we 
do not need the near-zero term at all. Instead, with $R = 4\vv/3$ as above, we write
\[
J_1 = \int_{B_R^c(v)} (\tau f(v) - \tau f(v')) \tilde K_g(v,v') \dd v' +
 \int_{B_R(v)} (\tau f(v) - \tau f(v')) \tilde K_g(v,v') \dd v' =: J_{1,1} + J_{1,4},
\]
observing that the first term is indeed the same as $J_{1,1}$ in the previous case (and is bounded in the same way). The new term $J_{1,4}$ 
is bounded in the same way as $J_{1,3}$ in the previous case, now taking $\mathcal{A}_0 = B_R(v) \setminus B_{R/2}(v)$ and observing that, 
since $|v| < 1$, the smoothness estimate \eqref{e.tau-est} for $\tau f$ holds even when $v'$ is close to zero. 
Thus \eqref{e.J1} holds in both cases.

\smallskip

\noindent{}{\it Step 3: Bounding the $\Qs(f, \delta f)$ term in~\eqref{e.c010402}.}  Define
\[
J_2 := \frac {\vv^m}{(|\chi|^2+|\nu|^2)^{\alpha/2}} Q_{\rm s}(f,\delta f) =  \int_{\R^3} \left( g(v') \frac{\vv^m}{\langle v' \rangle^m} - g(v) \right) K_f(v,v') \dd v',
\]
where the second equality used the definition of $g$. 

Since $g(t_\rmcr, x_\rmcr, v_\rmcr,\chi_\rmcr, \nu_\rmcr) = \overline G(t_\rmcr) > 0$ and this is the maximum of $g$ in all variables, we can proceed as in \eqref{e.max-trick} from the proof of Lemma \ref{lem:scaling_holder} to write $g(v') \leq g(v_\rmcr)$, which yields
\[
J_2 \leq  \| g \|_{L^\infty} \int_{\R^3} \left( \frac{\vv^m}{\langle v' \rangle^m} - 1 \right) K_f(v,v') \dd v'.
\]
Defining $\phi(v') := \vv^m / \langle v' \rangle^m - 1$, we have
\[ J_2 \leq \|g\|_{L^\infty} \int_{\R^3} \phi(v') K_f(v,v') \dd v'.\]
Once again, the analysis is slightly different depending on the size of $v$.

\medskip

\noindent{}{\it Case 1: $|v| \geq 1$.}
Define $R = |v|$ and $r = |v|/3$. Note that, for $v' \in B_R^c(0)$, one has  $\phi(v') \leq 0$. Since we seek an upper bound, we can discard the integral over $B_R^c(0)$ from $J_2$. Setting $\mathcal{B} := (B_R(0)\setminus B_r(0)) \setminus B_r(v)$, we then split $J_2$ as
\[
\begin{split}
J_2 &\lesssim \| g \|_{L^\infty} \left( \int_{B_r(v)} \phi(v') K_f(v,v') \dd v' + \int_{B_r(0)} \phi(v') K_f(v,v') \dd v'
+ \int_{\mathcal{B}} \phi(v') K_f (v,v') \dd v' \right)\\
&=: \| g \|_{L^\infty} \left( J_{2,1} + J_{2,2} + J_{2,3} \right).
\end{split}
\]
For $J_{2,1}$, we Taylor expand $\phi$ to second order around $v' = v$ and note that (as in the proof of Lemma \ref{lem:scaling_holder}) the first order term vanishes due to the symmetry of the kernel, since the domain of integration is a ball centered at $v_\rmcr$. This yields %
\[
	J_{2,1} \lesssim \int_{B_r(v)} E(v,v') K_f (v,v') \dd v',
\]
where $|E(v,v')| \lesssim |v-v'|^2 \vv^m \sup_{w \in B_r(v)} \langle w \rangle^{-m-2} \lesssim |v-v'|^2 \vv^{-2}$, since $\langle w\rangle\approx \vv$ for $w \in B_r(v)$. Using Lemma \ref{l:K-upper-bound-2} to bound $K_f$, we have 
\[
\begin{split}
J_{2,1} &\lesssim \vv^{-2} \int_{B_r(v)} |v-v'|^2 K_f(v,v') \dd v'\\
&\lesssim  \vv^{-2} \left( \int_{\R^3} f(v+w)|w|^{\gamma+2s} \dd w \right) r^{2-2s} \lesssim  \| f \|_{L^{\infty}_q} \vv^{(\gamma+2s)_+} \vv^{-2s}\lesssim \|f\|_{L^\infty_q},
\end{split}
\]
since $q> 3+\gamma+2s$.

For $J_{2,2}$, within $B_r(0)$, we have no better estimate than $\phi(v') \leq \vv^m$. However, we can use the pointwise upper bound of \cite[Lemma 2.4]{henderson2021existence}, and the fact that $|v| \geq 1$, to obtain
\[
J_{2,2} \lesssim \int_{B_r(0)} \frac{\vv^m}{|v-v'|^{3+2s}} \| f \|_{L^{\infty}_q} \vv^{3+\gamma+2s-q} \dd v'
\lesssim \vv^{3+m+\gamma-q} \| f \|_{L^{\infty}_q},
\]
since $q\geq m+\gamma+3$.

Lastly, if $v' \in \mathcal{B}$, then $\phi(v')$ is bounded by a constant independent of $v$. Since $\mathcal{B} \subset B_r^c(v)$, we use
the tail estimate for $K_f$ (Lemma \ref{l:K-upper-bound-2}) to obtain
\[
J_{2,3} \lesssim \int_{B_r^c(v)} K_f (v,v') \dd v' \lesssim \| f \|_{L^{\infty}_q} \vv^{-2s}.
\]

Combining the above estimates yields
\begin{equation}
\label{e.J2}
J_2 \lesssim \| g \|_{L^\infty} \| f \|_{L^{\infty}_q}.
\end{equation}

\medskip

\noindent{\it Case 2: $|v| < 1$.} As in Step 2 above, this step is less delicate than the large-$v$ case, but necessary since the estimate used above for $J_{2,2}$ degenerates
as $|v| \rightarrow 0$. In this case, we remark that $\phi$ is bounded uniformly, so we can use the splitting
\begin{equation}\label{e.J2small}
J_2 \lesssim \| g \|_{L^\infty} \left( \int_{B_1(v)} E(v,v') K_f(v,v') \dd v' + \int_{B_1^c(v)} K_f(v,v') \dd v' \right).
\end{equation}
with $E(v,v')$ defined as in Case 1. Since $|E(v,v')| \lesssim |v-v'|^2$ on $B_1(v)$, the same calculation as for $J_{2,1}$ gives that the first term in \eqref{e.J2small} is bounded by $\| f \|_{L^{\infty}_q}$. The second term is also bounded by $\| f \|_{L^{\infty}_q}$ by Lemma \ref{l:K-upper-bound-2}, and we conclude \eqref{e.J2} holds in this case as well.

\medskip

\noindent{}{\it Step 4: Bounding the $\Qns$ terms in~\eqref{e.c010402}.} The last two terms are more straightforward. Recalling that $\tau f(v) \vv^q \lesssim \| f \|_{L^\infty_q}$,
we have
\[ \begin{split}
\frac{\vv^m}{(|\chi|^2+|\nu|^2)^{\alpha/2}} \Qns(\delta f, \tau f) &= c_b \tau f(v_\rmcr) \int_{\R^3} g(v+w) \frac{\vv^m}{\langle v+w\rangle^m} |w|^\gamma \dd w \\
 &\lesssim \| g \|_{L^\infty}\| f \|_{L^\infty_q}\vv^{m-q} \int_{\R^3} \frac{|w|^\gamma}{\langle v + w \rangle^m} \dd w \lesssim \| g \|_{L^\infty} \| f \|_{L^\infty_q},
\end{split}\]
since $m>\gamma+3$ and $q> m+\gamma$. We have used \Cref{l:convolution} to estimate the convolution in the last line.

Finally,
\[ \begin{split}
\frac{\vv^m}{(|\chi|^2+|\nu|^2)^{\alpha/2}} \Qns(f,\delta f)&\approx g(v) \int_{\R^3} f(v+w) |w|^\gamma \dd w\\ 
&\lesssim \| g \|_{L^\infty} \| f \|_{L^\infty_q} \int_{\R^3} \frac{|w|^\gamma}{\langle v+w \rangle^q} \dd w
\lesssim \| g \|_{L^\infty} \| f \|_{L^\infty_q},
\end{split}\]
since $q>\gamma+3$.

Combining our upper bounds for the four terms on the right in \eqref{e.c010402}, the proof of the lemma is complete, setting $\theta(\alpha,s) = 1+\eta(\alpha,s)$.
\end{proof}

\section{Uniqueness}\label{s:uniqueness}

In this section, we complete the proof of Theorem \ref{t:uniqueness}. Letting $f$ be the classical solution guaranteed by Theorem \ref{t:existence} and $g$ a weak solution in the sense of Theorem \ref{t:weak-solutions}, the goal is to establish a Gr\"onwall-type inequality for $h := f-g$ in a space-localized, velocity-weighted, $L^2$-based norm.  

Following \cite{morimoto2015polynomial}, we define our cutoff as follows: let
\[ 
\phi(x,v) := \frac{1}{1+|x|^2+|v|^2}, 
\]
and for $a \in \R^3$, let
\[
 \phi_a(x,v) := \phi(x-a,v).
 \]
Note that
\begin{equation}
\label{e.uniq_kin}
\left| v \cdot \nabla_{x} \phi_a(x,v) \right| = \left| \frac{2 (x-a) \cdot v}{(1+|x-a|^2 + |v|^2)^2} \right| \lesssim \phi_a(x,v).
\end{equation}
For given $n\geq 0$, we define the space-localized, velocity-weighted, $L^2$-based space $X_n$ in terms of the following norm:
\[ 
\| f \|_{X_n}^2 := \sup_{a \in \R^3} \iint_{\R^3\times\R^3} \phi_a(x,v) \vv^{2n} f^2(x,v) \dd v \dd x.
\]


We begin with two auxiliary lemmas. First, we have a modification of Lemma 4.2 from \cite{henderson2021existence}:
\begin{lemma}\label{lem:HW21}
Suppose that $\mu>-3$, $n > 3/2 + \mu$,  and $l > 3/2 + \mu + (3/2 - n)_+$. If $g \in X_n$, then
\begin{equation}
\label{HW21 4.2}
\sup_{a \in \R^3} \iint_{\R^3\times\R^3} \phi_a(x,v) \vv^{-2l} \left( \int_{\R^3} g(x,w) |v-w|^\mu \dd w \right)^2 \dd v \dd x
\lesssim \| g \|_{X_n}^2
\end{equation}
\end{lemma}
\begin{proof}
Without loss of generality, $g \geq 0$. We split the inner-most (convolutional) integral into two regions according to the ball $B_{|v|/10}(v)$. That is,
\[
\begin{split}
&\sup_{a \in \R^3} \iint_{\R^3\times\R^3} \phi_a \vv^{-2l} \left( \int_{\R^3} g(w) |v-w|^\mu \dd w \right)^2 \dd v \dd x \\
&\quad\quad \lesssim\sup_{a \in \R^3} \iint_{\R^3\times\R^3} \phi_a \vv^{-2l} \left(
	\left( \int_{B_{|v|/10}(v)} g(w) |v-w|^\mu \dd w \right)^2 \right.\\
	&\quad \quad \qquad \qquad\qquad \qquad \left.+ \left( \int_{B_{|v|/10}(v)^C} g(w) |v-w|^\mu \dd w \right)^2 \right) \dd v \dd x
	=: I_1 + I_2.
\end{split}
\]
Note that, when $w \in B_{|v|/10}(v)$, we have that $\langle w \rangle \approx \vv$ and $v \in B_{|w|/2}(w)$; in particular, $\sup_a
\phi_a(x,v) / \phi_a(x,w) \approx 1$. Then, Cauchy-Schwarz and
Fubini yield
\[
\begin{split}
I_1 &\lesssim \sup_{a \in \R^3} \iint_{\R^3\times\R^3}\phi_a(x,v) \vv^{-2l} \vv^{3+\mu} \int_{B_{|v|/10}(v)} g(w)^2 |v-w|^\mu \dd w \dd v \dd x \\
&\lesssim \sup_{a \in \R^3} \int_{\R^3} \int_{\R^3} g(w)^2 \int_{B_{|w|/2}(w)} \vv^{-2l+3+\mu} \phi_a(x,v) |v-w|^\mu \dd v \dd w \dd x \\
&\lesssim \sup_{a \in \R^3} \int_{\R^3}\int_{\R^3} \phi_a(x,w) g(w)^2 \langle w \rangle^{-2l+3+\mu} \int_{B_{|w|/2}(w)} |v-w|^\mu \dd v \dd w \dd x \\
&\lesssim \sup_{a \in \R^3}  \int_{\R^3}\int_{\R^3} \phi_a(x,w) g(w)^2 \langle w \rangle^{-2l + 2(3+\mu)} \dd w \dd x \lesssim \| g \|_{X_n}^2,
\end{split}
\]
where we also used that $-l+3+\mu < n$.

For the remaining term, we again use Cauchy-Schwarz, followed by H\"older's inequality in $x$, to obtain
\[
\begin{split}
I_2 &\lesssim \sup_{a \in \R^3} \iint_{\R^3\times\R^3} \phi_a(x,v) \vv^{-2l} \left( \int_{B_{|v|/10}(v)^C} \phi_a(x,w) \langle w \rangle^{2n} g(w)^2 \dd w \right)\\
&\qquad \qquad\times \left( \int_{B_{|v|/10}(v)^C} \frac{|v-w|^{2\mu}}{\langle w \rangle^{2n} \phi_a(x,w)} \dd w \right) \dd v \dd x\\
&\lesssim \| g \|_{X_n}^2 \sup_{a \in \R^3} \sup_{x \in \R^3} \int_{\R^3} \phi_a(x,v) \vv^{-2l}\int_{B_{|v|/10}(v)^C} \frac{|v-w|^{2\mu}}{\langle w \rangle^{2n}}
\left( |x-a|^2 + \langle w \rangle^2 \right) \dd w \dd v \\
&\lesssim \| g \|_{X_n}^2 \sup_{z \in \R^3} \int_{\R^3} \vv^{-2l}
\frac{\vv^{2\mu + (3-2n)_+}|z|^2 + \vv^{2\mu + (5-2n)_+}}{|z|^2 + \vv^2} \dd v \lesssim \| g \|_{X_n}^2.
\end{split}
\]
\end{proof}

We also recall the following result from \cite{HST2020boltzmann}:
\begin{lemma}{\cite[Lemma A.1]{HST2020boltzmann}}\label{l:A1}
For any $\rho > 0$ and $v_0 \in \R^3$ such that $\rho \geq 2|v_0|$, and any $H: \R^3 \rightarrow [0,\infty)$ such that the
right-hand side is finite, we have
\begin{equation}\label{HST A1}
\int_{\partial B_\rho(0)} \int_{(z-v_0)^\perp} H(z+w) \dd w \dd z \lesssim \rho^2 \int_{B_{\rho/2}^C(0)} \frac{H(w)}{|w|} \dd w.
\end{equation}
\end{lemma}

We are now ready to proceed with the proof of uniqueness. With $f$ and $g$ as above, we define $h = f-g$ and observe that
\begin{equation}
\label{e.uniq_dif}
\partial_t h + v \cdot \nabla_x h = Q(h,f) + Q(g,h), 
\end{equation}
in the weak sense. We integrate (in $x$ and $v$) \eqref{e.uniq_dif} against $\phi_a(x,v) h \vv^{2n}$ for some $n\geq \frac 3 2$ to be determined later. Even though this is not an admissible test function for the weak solution $g$, these calculations can be justified by a standard approximation procedure, which we omit. Next, we take a supremum over $a \in \R^3$ to yield
\[
\begin{split}
\frac{1}{2} \frac{d}{dt} \| h \|_{X_n}^2 
	&\leq  \frac 1 2 \sup_{a \in \R^3} \iint_{\R^3\times\R^3} v \cdot \nabla_x \phi_a h^2 \vv^{2n} \dd v \dd x 
	+ \sup_{a \in \R^3}  \iint_{\R^3\times\R^3} \phi_a Q(h,f) h \vv^{2n} \dd v \dd x\\
&\qquad  + \sup_{a \in \R^3}  \iint_{\R^3\times\R^3} \phi_a Q(g,h) h \vv^{2n} \dd v \dd x.
\end{split}
\]
Using \eqref{e.uniq_kin}, we bound the first term on the right by
\[
	\frac 1 2 \sup_{a \in \R^3} \iint_{\R^3\times\R^3}  |v \cdot \nabla_x \phi_a | h^2 \vv^{2n} \dd v \dd x
		\leq \frac 1 2 \sup_{a \in \R^3} \iint_{\R^3\times\R^3}  \phi_a h^2 \vv^{2n} \dd v \dd x
		=  \frac 1 2 \| h \|_{X_n}^2,
\]
which yields
\begin{equation}
\label{e.uinq_gron}
\begin{split}
\frac {d}{dt} \| h \|_{X_n}^2 
	&\leq  \| h \|_{X_n}^2
		+ 2 \sup_{a \in \R^3} \iint_{\R^3\times\R^3} \phi_a \left( Q_\text{s}(h,f) + Q_\text{ns}(h,f) + Q(g,h) \right) h \vv^{2n} \dd v \dd x \\
&=:  \| h \|_{X_n}^2 +  I_1 + I_2 + I_3.
\end{split}
\end{equation}
We bound the terms in this right-hand side one by one. Since $t$ plays no role in these estimates, we prove them for general functions $f, g, h$ defined for $(x,v)\in \R^6$, with the integrals $I_1$, $I_2$, and $I_3$ defined as in \eqref{e.uinq_gron}. As above, norms such as $\|\cdot\|_{L^\infty_q}$ with no specified domain are understood to be over $\R^6$ throughout this section.

For $I_1$, we need the following intermediate lemma:

\begin{lemma}\label{l:Chris_lemma}
For $v\in \R^3$, let
\be\label{e.c062804}
	{\mathcal K}(x,v) := \int_{B_{\vv/2}} K_{|h|}(x,v,v') f(x,v') \dd v',
\ee
where $K_{|h|}(x,v,v')$ is defined in terms of the function $|h|(x,v)$ according to formula \eqref{e.kernel}. Assume that 
$n>\max(\frac 3 2,\gamma+2s+\frac 5 2)$. If $\gamma +2s > -2$, suppose that $q > 3$, and, if $\gamma +2s \leq -2$, suppose that $q > (3 +\gamma)/s$. 
Then there holds
\begin{equation}\label{e.Chris_lemma}
\sup_{a \in \R^3} \int_{\R^3} \int_{B_{10}(0)^C} \phi_a \vv^{2n} {\mathcal K}^2 \dd v \dd x \leq C \| h \|_{X_n}^2 \| f \|_{L^\infty_q}^2,
\end{equation}
for a universal constant $C>0$ that tends to $\infty$ as $\gamma \nearrow 0$. 
\end{lemma}
\begin{proof}
The proof is divided into two cases depending on $\gamma + 2s$.

\medskip

\noindent {\it Case 1: $\gamma + 2s > -2$.} This argument is inspired by Proposition 3.1(i) of \cite{HST2020boltzmann},
but requires some modification due to the presence of
the space-localizing weight $\phi_a$. Let $r = \vv / 2$ and define
\[
H_a(x,z) := h(x,z)^2 \langle z \rangle^{2n} |z| \phi_a(x,z) \quad \text{ and } \quad
\Phi_a(x,v,w) := \frac{|w|^{2\gamma+4s+2}}{\langle v+w \rangle^{2n} |v+w| \phi_a(x,v+w)} .
\]
Using Cauchy-Schwarz twice, we obtain
\begin{equation}\label{uniq: gjiurhub}
\begin{split}
	\phi_a(x,v)^{\frac 1 2} {\mathcal K}(x,v)
		&\leq \| f \|_{L^\infty_q} \int_{B_r(0)} 
			\phi_a(x,v)^{\frac 1 2}
			\langle v' \rangle^{-q} K_{|h|} (x,v,v') \dd v' \\
		&\lesssim \| f \|_{L^\infty_q}
			\int_{B_r(0)} \frac{\langle v' \rangle^{-q}}{|v-v'|^{3+2s}}
			\int_{(v-v')^\perp} \phi_a(x,v)^{\frac 1 2} |h(v+w)||w|^{\gamma+2s+1} \dd w \dd v' \\
		& \lesssim \| f \|_{L^\infty_q} \int_{B_r(0)} \frac{\langle v' \rangle^{-q}}{|v-v'|^{3+2s}}
			\bigg( \int_{(v-v')^\perp} \phi_a(x,v) \Phi_a(x,v,w) \dd w \bigg)^{\frac 1 2}\\
		&\qquad \qquad \qquad \times \bigg( \int_{(v-v')^\perp} H_a(x,v+w) \dd w \bigg)^{\frac 1 2} \dd v'\\
		& \lesssim \| f \|_{L^\infty_q}
			\bigg( \int_{B_r(0)} \frac{\langle v' \rangle^{-q}}{|v-v'|^{6+4s}} \bigg( \int_{(v-v')^\perp} \phi_a(x,v) \Phi_a(x,v,w) \dd w \bigg)\\
		&\qquad \qquad \qquad \times \bigg( \int_{(v-v')^\perp} H_a(x,v+w) \dd w \bigg) \dd v' \bigg)^{\frac 1 2}.
\end{split}
\end{equation}
Note that, in the last inequality we used that $q>3$ to deduce that $\int_{\R^3} \langle v' \rangle^{-q} \dd v' \lesssim 1$. 

At this point we observe a few important algebraic relations between the variables $v$, $v'$, and $w$. Since $|v| \geq 10$, we
have $|v| \approx \vv$,  and since $|v'| < \vv / 2$ (recall~\eqref{e.c062804}), we also have $|v-v'| \approx \vv$. Furthermore, since $w \perp (v-v')$, we have
that 
\be\label{e.c062803}
	|v+w| \approx |v| + |w| \approx \vv + \langle w \rangle
\ee
(that is, a converse to the triangle inequality for those two variables);
see \cite[Lemma~2.4]{henderson2021existence}. With these relations in hand, we have
\[
\begin{split}
&\int_{(v-v')^\perp} \phi_a(x,v) \Phi_a(x,v,w) \dd w
\\&\quad\quad\quad=
\int_{(v-v')^\perp} \frac{|w|^{2\gamma + 4s+2}(1+|x-a|^2+|v+w|^2)}{\langle v + w \rangle^{2n} |v+w| (1+|x-a|^2+|v|^2)} \dd w \\
&\quad\quad\quad \lesssim \int_{(v-v')^\perp} \frac{|w|^{2\gamma + 4s+2}}{\vv^{2n+1} + \langle w \rangle^{2n+1}} \dd w
+ \int_{(v-v')^\perp} \frac{|w|^{2\gamma + 4s + 4}}{(\vv^{2n+1}+\langle w \rangle^{2n+1})\vv^2} \dd w \\
&\quad\quad\quad \lesssim \vv^{2\gamma+4s+3-2n},
\end{split}
\]
where we needed $\gamma + 2s > -2$ so that the singularities at $w=0$ are integrable, and $n > \gamma + 2s + \frac 5 2$ so the tails converge. Since $|v'| < \vv /2$ and $|v| \geq 10$,
we have that $|v-v'| \approx \vv$. Thus, \eqref{uniq: gjiurhub} becomes
\begin{equation}\label{uniq: simplified gjiurhub}
	\phi_a(x,v) {\mathcal K}(x,v)^2
		\lesssim \| f \|_{L^\infty_q}^2 \vv^{2\gamma - 3 -2n} \int_{B_r(0)} \langle v' \rangle^{-q}
\int_{(v-v')^\perp} H_a(x,v+w) \dd w \dd v'.
\end{equation}
This implies, using Lemma \ref{l:A1} on $H_a$, and spherical coordinates for the $v$-integral, that
\[
\begin{split}
&\sup_{a \in \R^3} \int_{\R^3} \int_{B_{10}^C} \phi_a \vv^{2n} {\mathcal K}(x,v)^2 \dd v \dd x
		\approx \sup_{a \in \R^3}\int_{10}^\infty \rho^{2n} \int_{\R^3} \int_{\partial B_\rho(0)} \phi_a(x,z) {\mathcal K}(x,z)^2 \dd z \dd x \dd \rho\\
	&\quad
		\lesssim \| f \|_{L^\infty_q}^2
\sup_{a \in \R^3}\int_{10}^\infty \rho^{2\gamma-3}   \int_{B_r(0)} \langle v' \rangle^{-q} \int_{\R^3} \int_{\partial B_\rho(0)}
\int_{(z-v')^\perp} H_a(x,z+w) \dd w \dd z \dd x \dd v' \dd \rho \\
	&\quad
		\lesssim \| f \|_{L^\infty_q}^2
\sup_{a \in \R^3} \int_{10}^\infty \rho^{2\gamma-1}  \left( \int_{\R^3} \langle v' \rangle^{-q} \dd v' \right)
 \int_{\R^3} \int_{B_{\rho/2}^C(0)} \frac{h(x,w)^2 \langle w \rangle^{2n} |w| \phi_a(x,w)}{|w|} \dd w \dd x \dd \rho \\
 	&\quad
		\lesssim \| f \|_{L^\infty_q}^2
\sup_{a \in \R^3} \int_{10}^\infty \rho^{2\gamma-1}  \int_{\R^3} \int_{B_{\rho/2}^C(0)} h(x,w)^2 \langle w \rangle^{2n} \phi_a(x,w) \dd w \dd x \dd \rho \\
	&\quad \lesssim  \| f \|_{L^\infty_q}^2 \int_{10}^\infty \rho^{2\gamma - 1}
\| h \|_{X_n}^2 \dd \rho \lesssim \| f \|_{L^\infty_q}^2 \| h \|_{X_n}^2,
\end{split}
\]
as desired.  Note that the last inequality uses that $\gamma < 0$ in an essential way and $q>3$ was used in the second-to-last inequality to ensure the integrability of $\vvp^{-q}$.  Additionally, note that no volume element was needed in the change to spherical coordinates in the first line because we are integrating over $\partial B_\rho(0)$, not $\partial B_1(0) = \mathbb S^2$.

\medskip

\noindent {\it Case 2: $\gamma + 2s \leq -2$.} This case requires a detailed analysis of the integral kernel to deal with the more severe singularity. 

As in Case 1, whenever $|v| \geq 10$, $|v'| < \vv/2$ (recall~\eqref{e.c062804}), and $w \perp (v-v')$, we have that
$\vv / \langle v+w \rangle \lesssim 1$; see~\eqref{e.c062803}. Therefore,
\[
\vv^n K_{|h|}(x,v,v') \approx \frac{1}{|v-v'|^{3+2s}} \int_{(v-v')^\perp} \vv^n |h(x,v+w)||w|^{\gamma+2s+1} \dd w
\lesssim K_{\langle \cdot \rangle^n |h|}(x,v,v').
\]
Let us define $H(x,v) = \vv^n |h(x,v)|$.
To properly estimate ${\mathcal K}(x,v)$, we need to take advantage of the decay available from $f$. However, since the integral is
over $B_r(0)$, we cannot exploit this smallness directly. Instead, we split the domain of integration $B_r(0)$ into $B_{r^s}(0)$ and
$B_r(0) \setminus B_{r^s}(0)$ (recall that $r = \vv / 2$ and $|v| \geq 10$, so that the splitting makes sense). Then we have 
\begin{equation}\label{e.uniq-lemma-1}
\begin{split}
\int_{B_r(0) \setminus B_{r^s}(0)} K_H(x,v,v') f(x,v') \dd v' &\lesssim \| f \|_{L^\infty_q} \vv^{-sq}\int_{B_r(0)} K_H(x,v,v') \dd v'\\
&\lesssim \| f \|_{L^\infty_q} \vv^{-sq-2s} \int_{\R^3} H(x,v+w) |w|^{\gamma + 2s} \dd w,
\end{split}
\end{equation}
using $B_r(0)\subset B_{2\vv}(v) \setminus B_{\vv/8}(v)$ and Lemma \ref{l:K-upper-bound}.

Next, we note that $B_{r^s(0)} \subset B_{|v|+r^s}(v) \setminus B_{|v|-r^s}(v)$, so that
\[
\int_{B_{r^s}(0)} K_H(x,v,v') f(x,v') \dd v' \lesssim \| f \|_{L^\infty_q} \int_{B_{|v|+r^s}(v) \setminus B_{|v|-r^s}(v)} K_H(x,v,v') \dd v'.
\]
Expanding into spherical coordinates centered at $v$ (i.e. $v' = v+\rho z$ with $z\in \partial B_1$), we have
\begin{equation}\label{e.uniq-lemma-2}
\begin{split}
	&\int_{B_{r^s}} K_H(x,v,v') f(x,v') \dd v' \lesssim \| f \|_{L^\infty_q} \int_{|v|-r^s}^{|v|+r^s} \frac{\rho^2}{\rho^{3+2s}} \int_{\partial B_1} \int_{z^\perp} H(x,v+w)|w|^{\gamma+2s+1} \dd w \dd z \dd\rho \\
	&\quad\quad\quad \lesssim \| f \|_{L^\infty_q} \vv^{-1 - 2s} r^s 
	\int_{\partial B_1} \int_{z^\perp} H(x,v+w)|w|^{\gamma+2s+1} \dd w \dd z \\
	&\quad\quad\quad \lesssim \| f \|_{L^\infty_q} \vv^{-1-s} \int_{\R^3} H(x,v+w)|w|^{\gamma+2s} \dd w,
\end{split}
\end{equation}
where we used Lemma \ref{l:A1} with $\rho=1$ in the last line. Combining \eqref{e.uniq-lemma-1} and \eqref{e.uniq-lemma-2}, 
we now have
\[
\vv^n {\mathcal K}(v) \lesssim \| f \|_{L^\infty_q} \vv^{-\min\{1+s, qs+2s\}} \int_{\R^3} H(x,v+w)|w|^{\gamma+2s} \dd w.
\]
In order to again use \Cref{lem:HW21} with $\mu = \gamma + 2s \leq -2$, $n= 0 > 3/2 + \mu$, and $l = \min\{1+s, qs+2s\}$, we must verify that $l> 3/2 + \mu + (3/2-n)_+ = 3 + \gamma + 2s$. Since we assume that $q > (3+\gamma)/s$, the inequality holds. Hence, we obtain 
\[
\begin{split}
\sup_{a \in \R^3} \int_{\R^3} \int_{B_{10}^C} &\phi_a \left( \vv^n {\mathcal K}(x,v) \right)^2 \dd v \dd x\\
 &\lesssim
\| f \|_{L^\infty_q}^2 \sup_{a \in \R^3} \iint_{\R^3\times\R^3} \phi_a \vv^{-2(1+s)} \left( \int_{\R^3} H(x,v+w) |w|^{\gamma+2s}\dd w \right)^2 \dd v \dd x \\
& \lesssim \| f \|_{L^\infty_q}^2 \| h \|_{X_n}^2,
\end{split}
\]
using $\|H\|_{X_0} = \|h\|_{X_n}$. Altogether, this yields \eqref{e.Chris_lemma}.
\end{proof}

Now we are ready to bound the singular term in \eqref{e.uinq_gron}:

\begin{lemma}[Bound on $I_1$]\label{l:uniqI1}
Let $q$ and $n$ be as in Lemma \ref{l:Chris_lemma}, and assume in addition that $n> \frac 3 2 +(\gamma+2s)_+$ and $q>  n+3+\gamma+4s$. Assume also that $\vv^m f\in L^\infty_x C^{2s+\alpha}_v$ for some $\alpha\in (0,1)$, and that $m > n + \frac 3 2 + \gamma+2s+\alpha$. With $I_1$ defined as in \eqref{e.uinq_gron}, we have
\[
I_1 \leq  C\| h \|_{X_n}^2 \left( \| f \|_{L^{\infty}_q} + \| \vv^{m} f \|_{L^\infty_xC^{2s+\alpha}_v} \right), 
\]
for any $h$ and $f$ such that the right-hand side is finite. The constant $C>0$ is universal.
\end{lemma}

\begin{proof} 
We use the annular decomposition (defining $A_k(v) :=  B_{2^k|v|}(v) \setminus B_{2^{k-1}|v|}(v)$) to write
\[
Q_\text{s}(h,f) = \sum_{k \in \Z} \int_{A_k(v)} K_h(x,v,v')(f(x,v')-f(x,v)) \dd v'.
\]
Then we have
\[
I_1 \lesssim \sup_{a \in \R^3} \sum_{k \in \Z} \iint_{\R^3\times\R^3} \phi_a \vv^{2n} h \int_{A_k(v)} K_h(x,v,v') (f(x,v')-f(x,v)) \dd v' \dd v \dd x.
\]
The analysis is divided into four cases, based on range and the relative sizes of $v$ and $v'$.

\medskip

\noindent {\it Case 1: $k \leq -1$.} If $2s+\alpha \leq 1$, then
we have
\[
|f(x,v') - f(x,v)| \leq \|f(x,\cdot)\|_{C^{2s+\alpha}_v(B_1(v))} |v'-v|^{2s+\alpha} \leq \|\vv^{m} f(x,\cdot)\|_{C_v^{2s+\alpha}(\R^3_v)} \vv^{-m} (2^k|v|)^{2s+\alpha}.
\]
With Lemma \ref{l:K-upper-bound}, we have for each $x\in \R^3$,
\begin{equation}\label{e.uniq-Holder-annulus}
\begin{split}
&\left| \int_{A_k(v)} K_h(x,v,v') (f(x,v')-f(x,v)) \dd v' \right|\\
& \qquad\qquad\lesssim \vv^{-m} (2^k |v|)^{\alpha} \|\vv^m f(x,\cdot) \|_{C_v^{2s+\alpha}(\R^3)}
\int_{\R^3} |h(x,w)| |v-w|^{\gamma+2s} \dd w.
\end{split}
\end{equation}
On the other hand, if $2s+\alpha \in  (1,2)$ (we may always assume $2s+\alpha< 2$ by taking $\alpha$ smaller if necessary), we have
\[
\begin{split}
|f(x,v') - f(x,v) - \nabla_v f(x,v)\cdot (v'-v)| &\leq \|f(x,\cdot)\|_{C^{2s+\alpha}_v(B_1(v))} |v'-v|^{2s+\alpha}\\
& \lesssim \|\vv^{m} f(x,\cdot)\|_{C_v^{2s+\alpha}(\R^3_v)} \vv^{-m} (2^k|v|)^{2s+\alpha}.
\end{split}
\]
As usual, the symmetry of $K_h(x,v,v')$ implies the first-order term integrates to zero over $A_k(v)$, and we obtain \eqref{e.uniq-Holder-annulus} in this case as well. 

From \eqref{e.uniq-Holder-annulus}, using Lemma \ref{lem:HW21} with $\mu = \gamma+2s$  and $l = m - n - \alpha$, we have
\begin{equation}
\begin{split}
&\sup_{a \in \R^3} \iint_{\R^3\times\R^3} \phi_a(x,v) \vv^{2n} \left( \int_{A_k(v)} K_h(v,v') (f(x,v')-f(x,v)) \dd v' \right)^2 \dd v \dd x \\
&\quad \lesssim \| f \|_{L^\infty_x C^{2s+\alpha}_{m,v}}^2 2^{2k\alpha}
\sup_{a \in \R^3} \iint_{\R^3\times\R^3}\phi_a \vv^{-2(m - n-\alpha)}
\left( \int_{\R^3} |h(x,w)| |v-w|^{\gamma+2s} \dd w \right)^2 \dd v \dd x \\
&\quad \lesssim \| f \|_{L^\infty_xC^{2s+\alpha}_{m,v}}^2 2^{2k\alpha} \| h \|_{X_n}^2.
\end{split}
\end{equation}
The terms for $k \leq 1$ are summable, and we find that
\begin{equation}\label{unique I1 close range}
\begin{split}
&\sup_{a \in \R^3} \sum_{k \leq 1} \iint_{\R^3\times\R^3} \phi_a \vv^{2n} h \int_{A_k(v)} K_h(v,v') (f(x,v')-f(x,v))\dd v' \dd v \dd x \\
&\quad \lesssim
\| h \|_{X_n} \sup_{a \in \R^3} \sum_{k \leq 1} \Big( \iint_{\R^3\times\R^3} \phi_a \vv^{2n} \Big( \int_{A_k(v)} K_h(v,v')(f(x,v')-f(x,v)) \dd v' \Big)^2 \dd v \dd x \Big)^{\frac 1 2}\\
&\quad \lesssim \| f \|_{L^\infty_xC^{2s+\alpha}_{m,v}} \| h \|_{X_n}^2.
\end{split}
\end{equation}

\medskip

\noindent {\it Case 2: $k \geq K_0$ for $K_0 = \log_2(\vv/|v|)$.}
Notice that the choice of $K_0$ yields $|v'| \geq \vv/2$.  
Indeed, by the construction, if $v' \in A_k$, then
\[
	|v'| \geq 2^{k-1} |v|
		\geq 2^{K_0-1} |v|
		= \frac{1}{2} 2^{K_0} |v|
		= \frac{1}{2} \vv.
\]
Thus,
\begin{equation}\label{e.case2}
\begin{split}
&\left| \int_{A_k(v)} K_h(x,v,v')(f(x,v')-f(x,v)) \dd v' \right| \lesssim \vv^{-q} \| f \|_{L^{\infty}_q} \int_{A_k(v)} |K_h(x,v,v')| \dd v' \\
&\quad\quad\quad \lesssim \vv^{-q} \| f \|_{L^{\infty}_q} (2^k \vv)^{-2s} \left( \int_{\R^3} |h(x,w)| |v-w|^{\gamma+2s} \dd w \right).
\end{split}
\end{equation}
Here we used that $\vvp \approx \vv$ 
to obtain the $\vv^{-q}\| f \|_{L^{\infty}_q}$ decay from $f(x,v')$. Therefore,
\begin{equation}\label{unique I1 far large v'}
\begin{split}
&\sup_{a \in \R^3} \sum_{k \geq K_0} \iint_{\R^3\times\R^3} \phi_a \vv^{2n} h \int_{A_k(v)} K_h(v,v')(f(x,v')-f(x,v)) \dd v' \dd v \dd x\\
&\quad\quad\lesssim
\| h \|_{X_n} \| f \|_{L^\infty_q} \sup_{a \in \R^3} \sum_{k > K_0} 2^{-2ks}
\left( \iint_{\R^3\times\R^3} \phi_a \vv^{2(n-q-2s)} \left( \int_{\R^3} |h(x,w)| |v-w|^{\gamma+2s} \dd w \right) \dd v \dd x \right)^{\frac 1 2} \\
&\quad\quad \lesssim \| f \|_{L^\infty_q} \| h \|_{X_n}^2,
\end{split}
\end{equation}
from Lemma \ref{lem:HW21}. We used $q >n+  3 + 2\gamma+2s$.

\medskip

\noindent {\it Case 3: $k \in [0,K_0]$ and $|v| < 10$.}  Notice that $|v'| \leq \vv/2$.  
This case is a formality, and uses exactly the same estimates
as in Case 1.  The only difference is in how we deduce that $\vvp \approx \vv$.  In Case 1, this is because $|v-v'| \leq |v|/2$ from the condition on $k$ and the definition of $A_k$.  Here, it is due to the choice of $K_0$ (the $\lesssim$ direction) and the smallness of $|v|$ (the $\gtrsim$ direction).  Hence, we omit the details.

\medskip

\noindent {\it Case 4: $k \in [0,K_0]$ and $|v| \geq 10$.} This case uses estimates similar to Case 2. With $r = \vv/2$, we decompose as
\be
	\begin{split}
	\sum_{k = 0}^{K_0} &\int_{A_k(v)} K_h(x,v,v')(f(x,v')-f(x,v)) \dd v'\\
		&= \sum_{k = 0}^{K_0} \left(\int_{A_k(v) \cap B_r(0)} + \int_{A_k(v) \cap B_r(0)^c}\right) K_h(x,v,v')(f(x,v')-f(x,v)) \dd v'
		=: J_1 + J_2.
	\end{split}
\ee
We begin with $J_1$:
\[
\begin{split}
J_1
	&= \sum_{k = 0}^{K_0} \int_{A_k(v) \cap B_r(0)} K_h(x,v,v')(f(x,v')-f(x,v)) \dd v' \\
& \lesssim f(v) \int_{B_r(0)} K_{|h|}(x,v,v') \dd v' + \int_{B_r(0)}K_{|h|}(x,v,v') f(x,v') \dd v' =: J_{11} + J_{12}.
\end{split}
\]
  For $J_{11}$, arguing as in \eqref{e.case2}, we note that
\[
\int_{B_r(0)} K_{|h|}(x,v,v') \dd v' \leq \int_{B_{2\vv}(v) \setminus B_{\vv/4}(v)} K_{|h|}(x,v,v') \dd v' \lesssim \vv^{-2s} \int_{\R^3} |h(x,v')| |v-v'|^{\gamma+2s} \dd v'.
\]
Then once more using Lemma \ref{lem:HW21} yields
\begin{equation}\label{uniq: dktluls}
\begin{split}
&\sup_{a \in \R^3} \sum_{k>0} \int_{\R^3} \int_{B_{10}(0)^C} \phi_a \vv^{2n} h \int_{A_k(v) \cap B_r(0)} K_h(x,v,v')f(x,v)\dd v'\dd v \dd x \\
&\quad\quad\quad \lesssim \| h \|_{X_n} \| f \|_{L^\infty_q} \sup_{a \in \R^3}
\left( \int_{\R^3} \int_{\R^3} \vv^{-2q + 2n - 4s} \left( \int_{\R^3} |h(x,w)| |v-w|^{\gamma+2s} \dd w \right)^2 \dd v \dd x \right)^{\frac 1 2} \\
&\quad\quad\quad \lesssim \| f \|_{L^\infty_q} \| h \|_{X_n}^2.
\end{split}
\end{equation}

To estimate $J_{12}$, we recall the notation ${\mathcal K}(x,v) := \int_{B_r(0)} K_{|h|}(x,v,v') f(x,v') \dd v'$. Then
\begin{equation}\label{uniq: brgrt}
\begin{split}
\sup_{a \in \R^3} \sum_{k>0} \int _{\R^3}\int_{B_{10}^C} &\phi_a \vv^{2n} h \int_{A_k(v) \cap B_r(0)} K_{|h|}f(x,v')\dd v' \dd v \dd x\\
&\lesssim \| h \|_{X_n} \left( \sup_{a \in \R^3} \int_{\R^3} \int_{B_{10}^C} \phi_a \vv^{2n} {\mathcal K}^2 \dd v \dd x \right)^{\frac 1 2}
\lesssim \|h\|_{X_n}^2 \|f\|_{L^\infty_q},
\end{split}
\end{equation}
by Lemma \ref{l:Chris_lemma}.

We now consider $J_2$.  Here, we have $|v'| \geq \vv/2$, so the arguments of Case 2 apply verbatim.  Hence, we omit the argument.

This completes the desired estimate for Case 4, which together with the first three cases 
establishes the conclusion of the lemma.
\end{proof}

\begin{lemma}[Bound on $I_2$]\label{l:uniqI2}
With $I_2$ defined as in \eqref{e.uinq_gron}, there holds
\[ 
I_2 \leq C \|h\|_{X_n}^2\|f\|_{L^\infty_q},
\]
whenever $n\geq \frac 3 2$ and $q>n+\frac 3 2 + \gamma$. The constant $C>0$ is universal.
\end{lemma}
\begin{proof}
 With the simple observation that
\[ I_2 \lesssim \| h \|_{X_n} \| Q_{\text{ns}}(h,f) \|_{X_n}
\]
we use the definition of $Q_{\text{ns}}$ and Lemma \ref{lem:HW21} with $\mu = \gamma$ and $l = q-n$, to immediately find
\[
\begin{split}
\| Q_{\text{ns}}(h,f) \|_{X_n}^2 &\lesssim \sup_{a \in \R^3} \iint_{\R^3\times\R^3} \phi_a \vv^{2n} f(x,v)^2 \left( \int_{\R^3} |h(x,z)| |v-z|^\gamma \dd z \right)^2 \dd v \dd x \\
& \lesssim \| f \|_{L^\infty_q}^2 \sup_{a \in \R^3} \iint_{\R^3\times\R^3} \phi_a \vv^{2(n-q)} \left( \int_{\R^3} |h(x,z)| |v-z|^\gamma \dd z \right)^2 \dd v \dd x
\lesssim \| f \|_{L^\infty_q}^2\| h \|_{X_n}^2,
\end{split}
\]
which yields the desired bound for $I_2$.
\end{proof}

\begin{lemma}[Bound on $I_3$]\label{l:uniqI3}
With $I_3$ defined as in \eqref{e.uinq_gron}, there exists a universal $C>0$ such that
\[
I_3 \leq C \|h\|_{X_n}^2 \|g\|_{L^\infty_q},
\]
whenever $n\geq 2$ and $q>2n+\gamma+5$.
\end{lemma}

\begin{proof}
First, define
\[ \Psi_a(x,v) := \phi_a(x,v)^{\frac 1 2} \vv^n \quad \text{ and } \quad J_a(x,v) = \Psi_a(x,v) h(x,v). \]
We begin by splitting $I_3$ into a coercive part and a commutator:
\[
\begin{split}
I_3 &= \sup_{a \in \R^3} \iint_{\R^3\times\R^3} \phi_a \vv^{2n} Q(g,h)h \dd v \dd x \\
&= \sup_{a \in \R^3} \left( \iint_{\R^3\times\R^3} J_a Q(g, J_a) \dd v \dd x + \iint_{\R^3\times\R^3} J_a \left( Q(g,h) \Psi_a - Q(g, J_a) \right) \dd v \dd x \right) \\
&=: \sup_{a \in \R^3} \left( I_{31} + I_{32} \right).
\end{split}
\]
We need to keep the supremum in $a$ on the outside since the "coercive" term $I_{31}$ will contribute a strong negative
component which is needed to control $I_{32}$. Specifically, using a well-known symmetrization technique (see, e.g. \cite[Lemma 4.1]{amuxy2011bounded}), we have
\[
I_{31} = -\frac 1 2 D_a
+ \iint_{\R^3\times\R^3} Q(g, J_a^2) \dd v \dd x,
\]
where
\[
D_a := \iint_{\R^9 \times {\mathbb S}^2} (J_a(x,v')-J_a(x,v))^2 g(x,v_*) B(|v-v_*|,\cos \theta) \dd \sigma \dd v_* \dd v \dd x.
\]
For the second term in $I_{31}$, recalling $Q = Q_{\rm s} + Q_{\rm ns}$, we use a change of variables and the Cancellation Lemma \cite[Lemma 1]{alexandre2000entropy} to write
\[
\begin{split}
\iint_{\R^3\times\R^3} Q_{\rm s}(g, J_a^2) \dd v \dd x &= \iint_{\R^3\times\R^3} J_a^2 \int_{\R^3} (K_g(x,v',v) - K_g(x,v,v')) \dd v' \dd v \dd x\\
&\lesssim \iint_{\R^3\times\R^3} \phi_a \vv^{2n} h^2 \int_{\R^3} g(x,z) |v-z|^\gamma \dd z \dd v \dd x\lesssim \|h\|_{X_n}^2 \|g\|_{L^\infty_q},
\end{split}
\]
since $q>2n+\gamma+5> \gamma+3$. The nonsingular term is handled similarly:
\[
\begin{split}
\iint_{\R^3\times\R^3} Q_{\rm ns}(g, J_a^2) \dd v \dd x &\lesssim \iint_{\R^3\times\R^3} J_a^2 \int_{\R^3} g(x,z) |v-z|^\gamma \dd z \dd v \dd x\lesssim \|h\|_{X_n}^2 \|g\|_{L^\infty_q}.
\end{split}
\]
We conclude
\begin{equation}\label{uinq: coercive I31}
I_{31}+\frac 1 2 D_a \lesssim \| h \|_{X_n}^2 \| g \|_{L^\infty_q}.
\end{equation}

For $I_{32}$, recalling the abbreviations $F = F(x,v)$, $F_* = F(x,v_*)$, $F' = F(x,v')$, and $F_*' = F(x,v_*')$ for any function $F$, and writing $B = B(|v-v_*|,\cos\theta)$, we have
\[
\begin{split}
I_{32} &= \iint_{\R^3\times\R^3} \iint_{\R^3\times\mathbb S^2} B J_a \left[ (g_*' h' - g_*h) \Psi_a - (g_*' J_a' - g_* J_a)\right] \dd \sigma \dd v_* \dd v \dd x\\
&= \iint_{\R^3\times\R^3} \iint_{\R^3\times\mathbb S^2} B J_a g_*' h' (\Psi_a - \Psi_a')\dd \sigma \dd v_* \dd v \dd x,
\end{split}
\]
since $J_a = \Psi_a h$. Next, we apply the pre-post-collisional change of variables: $v\leftrightarrow v'$, $v_* \leftrightarrow v_*'$, $\sigma \mapsto \sigma' := (v-v_*)/|v-v_*|$. 
 This transformation has unit Jacobian and leaves $B(|v-v_*|, \sigma)$ invariant. This gives
\begin{equation}\label{e.I32}
\begin{split}
I_{32} &= \iint_{\R^3\times\R^3} \iint_{\R^3\times\mathbb S^2} B   J_a' g_* h (\Psi_a'-\Psi_a)\dd \sigma \dd v_* \dd v \dd x \\
&= \iint_{\R^3\times\R^3} \iint_{\R^3\times\mathbb S^2} B   J_a g_* h (\Psi_a'-\Psi_a)\dd \sigma \dd v_* \dd v \dd x \\
&\quad + \iint_{\R^3\times\R^3} \iint_{\R^3\times\mathbb S^2} B   (J_a' - J_a) g_* h (\Psi_a'-\Psi_a)\dd \sigma \dd v_* \dd v \dd x.\\
&=: I_{321} + I_{322}.
\end{split}
\end{equation}
We consider the term $I_{322}$ first. We want to extract a ``singular'' piece (in the form of $D_a$, as defined in \eqref{uinq: coercive I31}) which will cancel with
the coercive part of $I_{31}$. Specifically, from the inequality $ab\leq \frac 1 2 (a^2+b^2)$, we have
\[
\begin{split}
I_{322} &\leq \frac 1 2 \iint_{\R^3\times\R^3}  \iint_{\R^3\times{\mathbb S}^2} B (J_a'-J_a)^2 g_*  \dd\sigma \dd v_* \dd v \dd x \\
&\quad\quad\quad + \frac 1 2\iint_{\R^3\times\R^3}  \iint_{\R^3\times{\mathbb S}^2} B h^2 g_* (\Psi_a' - \Psi_a)^2  \dd\sigma \dd w \dd v \dd x,
\end{split}
\]
so that
\[
I_{322} - \frac 1 2 D_a \leq \frac 1 2 \iint_{\R^3\times\R^3}  \iint_{\R^3\times{\mathbb S}^2} B h^2 g_* (\Psi_a' - \Psi_a)^2  \dd\sigma \dd v_* \dd v \dd x.
\]
Next, write $(\Psi'_a- \Psi_a)^2 = 2 \Psi_a (\Psi_a - \Psi_a') + (\Psi_a')^2 - \Psi_a^2$ to obtain
\begin{equation}
\begin{split}
I_{322} - \frac 1 2 D_a &\leq  \iint_{\R^3\times\R^3}  \iint_{\R^3\times{\mathbb S}^2} B h^2 g_* \Psi_a (\Psi_a - \Psi_a') \dd\sigma \dd v_* \dd v \dd x\\
&\quad + \frac 1 2 \iint_{\R^3\times\R^3}  \iint_{\R^3\times{\mathbb S}^2} B h^2 g_*  ((\Psi_a')^2 - \Psi_a^2) \dd\sigma \dd v_* \dd v \dd x.
\end{split}
\end{equation}
Since $J_a = \Psi_a h$, the first term on the right is the negative of $I_{321}$. Returning to \eqref{e.I32}, we now have
\begin{equation}\label{e.uniq3}
I_{32} - \frac 1 2 D_a \leq  \frac 1 2  \iint_{\R^3\times\R^3}  \iint_{\R^3\times{\mathbb S}^2} B h^2 g_*  ((\Psi_a')^2 - \Psi_a^2) \dd\sigma \dd v_* \dd v \dd x.
\end{equation}
It only remains to bound this right-hand side. To do this, we start with the Taylor expansion for $\Psi_a^2 = \phi_a(x,v) \vv^{2n}$ in $v$ (in this proof, subscipts such as $\partial_i$ always denote differentiation in $v$):
\[ \Psi_a^2(x,v')-\Psi_a^2(x,v) = \partial_{i} \Psi_a^2(x,v) (v'-v)_i + \frac 1 2 \partial_{ij}^2 \Psi_a^2(x,\tilde{v}) (v'-v)_i (v'-v)_j, \]
where we sum over repeated indices, and $\tilde{v} = \tau v' + (1-\tau) v$ for some $\tau \in (0,1)$. By a direct calculation, we have 
\[
|\partial_i \Psi_a^2(x,v)| \lesssim \Psi_a^2(x,v) \vv^{-1} \quad \text{ and } \quad
|\partial_{ij}^2 \Psi_a(x,\tilde{v})| \lesssim \Psi_a^2(x,\tilde{v}) \left( \frac{1}{ \langle \tilde{v} \rangle^2} + \phi_a(x,\tilde{v}) \right).
\]
Noting that $v'-v = \frac 1 2 |v-v_*|(\sigma-\sigma' \cos \theta) + \frac 1 2 |v-v_*| (\cos\theta -1) \sigma'$, we have
\begin{equation}\label{e.uniq4}
\begin{split}
&\int_{{\mathbb S}^2} B ((\Psi_a')^2-\Psi_a^2)  \dd \sigma =
\frac{|v-v_*|}{2} \nabla_v \Psi_a^2(x,v) \cdot \int_{{\mathbb S}^2}B (\sigma - (\sigma \cdot \sigma') \sigma')  \dd\sigma\\
&\quad\quad\quad + \frac{|v-v_*|}{2} \nabla_v \Psi_a^2(x,v) \cdot \sigma' \int_{{\mathbb S}^2}B (\cos\theta-1)  \dd\sigma
+\frac 1 2 \int_{{\mathbb S}^2}B \partial_{ij}^2 \Psi_a^2(x,\tilde{v}) (v'-v)_i (v'-v)_j \dd\sigma.
\end{split}
\end{equation}
The first term on the right is zero by symmetry. Since $B \approx \theta^{-2-2s}|v-v_*|^\gamma$, the second term is bounded by $\lesssim \Psi_a^2(x,v) |v-v_*|^{1+\gamma} \vv^{-1}$.
Noting that $|v-v'|^2 = \frac 1 2 |v-v_*|^2 (1-\cos\theta)$ and that $|v|^2 + |v_*|^2 = |v'|^2 + |v_*'|^2$, we bound the third term
by
\[
\Psi_a^2(x,v) |v-v_*|^{2+\gamma} \int_{{\mathbb S}^2} \Theta_a(x,v,\tilde{v})
(1-\cos\theta) |\theta|^{-2-2s} \dd \sigma,
\]
with
\[
\Theta_a(x,v,\tilde v) := \frac{\langle \tilde{v} \rangle^{2n}}{\vv^{2n}}
\frac{\phi_a(x,\tilde{v})}{\phi_a(x,v)} ( \langle \tilde{v} \rangle^{-2} + \phi_a(x,\tilde{v})).
\]
Note that, since $\tilde{v}$ depends on $v'$, it also implicitly depends on $\sigma$. To estimate $\Theta_a$, we split into three cases:

\medskip

\noindent {\it Case 1: $|\tilde{v}| \geq |v|/2$.} In this case, $\phi_a(x,\tilde{v}) / \phi_a(x,v) \lesssim 1$. Using this, as well as $\langle \tilde{v} \rangle^2 \lesssim \vv^2 + \langle v_* \rangle^2$, we obtain
\[ \Theta_a(x,v,\tilde{v}) \lesssim \left(1 + \frac{\langle v_* \rangle^{2n}}{\vv^{2n}} \right) \vv^{-2}. \]

\medskip

\noindent {\it Case 2: $|\tilde{v}| < |v|/2$ and $|x+a| \geq |v|$.} Here again $\phi_a(x,\tilde{v}) / \phi_a(x,v) \lesssim 1$ (due to the size of $|x+a|$) and
also $\phi_a(x,\tilde{v}) \leq \vv^{-2}$, so we have
\[ \Theta_a(x,v,\tilde{v}) \lesssim \vv^{-2}. \]

\medskip

\noindent {\it Case 3 ($|\tilde{v}| < |v|/2$ and $|x+a| < |v|$).} First, note that 
\[
\langle \tilde{v} \rangle^{2n} \phi_a(x,\tilde{v}) (\langle \tilde{v} \rangle^{-2} + \phi_a(x,\tilde{v})) \lesssim \langle \tilde{v} \rangle^{2n-4} \lesssim \vv^{2n-4} + \langle v_*\rangle^{2n-4},\] 
which is always true, but in this case we also have that $\vv^{-2n} \phi_a(x,v)^{-1} \lesssim \vv^{2-2n}$. Thus
\[ 
\Theta_a(x,v,\tilde{v}) \lesssim \frac{\vv^{2n-4} + \langle v_* \rangle^{2n-4}}{\vv^{2n-2}} = \left( 1 + \frac{\langle v_* \rangle^{2n-4}}{\vv^{2n-4}} \right) \vv^{-2}
\lesssim \left( 1 + \frac{\langle v_*\rangle^{2n}}{\vv^{2n}} \right) \vv^{-2}. 
\]
The last inequality followed from $(\frac a b)^{2n-4} \leq 1 + (\frac a b)^{2n}$, since $n\geq 2$.

 Putting all three cases together yields, for all $x$, $v$, $v_*$, and $\sigma$,
\[
 \Theta_a(x,v,\tilde{v}) \lesssim \frac{\vv^{2n} + \langle v_*\rangle^{2n}}{\vv^{2n+2}}. 
\]
With \eqref{e.uniq3} and \eqref{e.uniq4}, we then have 
\begin{equation}\label{uniq: I321}
\begin{split}
I_{32} - \frac 1 2 D_a &\lesssim \iiint_{\R^9} J_a^2 g_* \left( |v-v_*|^{1+\gamma}\vv^{-1} + |v-v_*|^{2+\gamma}
\int_{{\mathbb S}^2} \Theta_a(x,v,\tilde{v}) \frac{1-\cos\theta}{|\theta|^{2+2s}} \dd\sigma \right)\dd v_* \dd v \dd x \\
&\lesssim \| h \|_{X_n}^2 \| g \|_{L^\infty_q} \sup_{v \in \R^3} \left( \int_{\R^3} \frac{|v-v_*|^{1+\gamma}}{\langle v_* \rangle^q \vv} \dd v_*
+ \int_{\R^3} \frac{|v-v_*|^{2+\gamma} (\vv^{2n} + \langle v_* \rangle^{2n})}{\langle v_* \rangle^q \vv^{2n+2}} \dd v_* \right) \\
&\lesssim \| h \|_{X_n}^2 \| g \|_{L^\infty_q}.
\end{split}
\end{equation}
We used $q>\gamma+4$ in the first term, and $q>2n+\gamma+5$ in the second term.

Putting \eqref{uniq: I321} together with \eqref{uinq: coercive I31}
yields the desired bound on $I_3$.
\end{proof}

We are now able to complete the proof of the proof of uniqueness.

\begin{proof}[Proof of \Cref{t:uniqueness}]

First, note that, instead of the assumption $f_{\rm in} \in C^\alpha_{\ell, x, v}$, we may assume that $\vv^{m} f_{\rm in} \in C^\alpha_{\ell,x,v}$, for any fixed $m>0$.  Indeed, up to decreasing $\alpha$ and increasing the exponent $q$ from our hypotheses $f_{\rm in} \in L^\infty_q$, we can interpolate using \Cref{l:moment-interpolation} to trade regularity for velocity decay.

Define $\beta = \frac{\alpha}{1+2s}$ and $\beta' = \beta \frac {2s}{1+2s} = \frac {2s\alpha}{(1+2s)^2}$. Next, choose $m>0$ large enough to satisfy the hypotheses of Lemma \ref{l:uniqI1} and Proposition \ref{prop:holder_propagation}, and define
\[
m' := m-\kappa-{[\beta/(1+2s) + \gamma]_+}-\beta(\gamma+2s)_+/(2s),
\]
where $\kappa>0$ is the constant from Theorem \ref{t:global-schauder}.

Now we apply the Schauder estimate of Proposition \ref{p:nonlin-schauder} (which relies on \eqref{e.f-lower}), followed by the small-time H\"older estimate of Proposition \ref{prop:holder_propagation} to obtain, for $t\in [0,T_U]$,
\begin{equation}\label{e.f-estimate-uniq}
\begin{split}
\| \vv^{m'} f (t)\|_{L^\infty_x C^{2s+\beta'}_{v}} &\lesssim \|f\|_{C^{2s+\beta'}_{\ell,m'}([t/2,t]\times\R^6)}\\
&\lesssim t^{-1+(\beta-\beta')/(2s)}\|f\|_{C^\beta_{\ell,m' + \kappa + [\beta/(1+2s) + \gamma]_+}([0,t]\times\R^6)}^{1+(\beta+2s)/\beta'}\\
&\lesssim t^{-1+(\beta-\beta')/(2s)} \|\vv^{m}f_{\rm in}\|_{C^{\alpha}_{\ell,x,v}(\R^6)}^{1+(\beta+2s)/\beta'},
\end{split}
\end{equation}
since $\alpha = \beta(1+2s)$.

Now, combining Lemma \ref{l:uniqI1} (with $\beta'$ playing the role of $\alpha$), Lemma \ref{l:uniqI2}, and Lemma \ref{l:uniqI3} with inequality \eqref{e.uinq_gron}, we have
\[
\frac{d}{dt} \| h(t) \|_{X_n}^2 - \| h(t) \|_{X_n}^2 \lesssim \| h(t) \|_{X_n}^2 \left( \| g(t) \|_{L^\infty_q} + \| f(t) \|_{L^\infty_q} + \| \vv^m f (t)\|_{L^\infty_x C^{2s+\beta'}_{v}} \right).
\]
Using \eqref{e.f-estimate-uniq} for the last term on the right, and absorbing the norm of $f_{\rm in}$ into the implied constant, we now have
\[
\frac{d}{dt} \| h \|_{X_n}^2 \leq C \left(  \| g(t) \|_{L^\infty_q} + \| f(t) \|_{L^\infty_q}  + t^{- 1 + (\beta-\beta')/2s} \right) \| h \|_{X_n}^2.
\]
By our assumptions that $\|g(t)\|_{L^\infty_q(\R^6)} \in L^1([0,T_U])$ and $\|f(t)\|_{L^\infty_q} \in L^\infty([0,T_U])$, and $\|h(0)\|_{X_n}^2 = 0$, we conclude that $\|h(t)\|_{X_n} \equiv 0$ for all $t\in [0,T_U]$ by Gr\"onwall's inequality.  
After replacing $\alpha(1+2s)$ with $\alpha$, we obtain the statement of Theorem \ref{t:uniqueness}.
\end{proof}

\section{Global existence near equilibrium}\label{s:global}

In this section, we prove Corollary \ref{c:global}. The proof mainly follows the approach of \cite{silvestre2022nearequilibrium}. To pass from the local existence result of Theorem \ref{t:existence} to a global existence result near equilibrium, we must first show that the time of existence depends on the distance of $f_{\rm in}$ to the Maxwellian $M$:

\begin{lemma}\label{l:Max}
Let $M(x,v) = (2\pi)^{-3/2} e^{-|v|^2/2}$, and let $q>3+\gamma+2s$ be fixed. Given $T>0$ and $\eps\in (0,\frac 1 2)$, there exists $\delta>0$ such that if 
\[ \|f_{\rm in} - M\|_{L^\infty_{q}(\T^3\times\R^3)} < \delta,\]
then the solution $f$ to \eqref{e.boltzmann} guaranteed by Theorem \ref{t:existence} exists up to time $T$, and satisfies
\[ \|f(t,\cdot,\cdot) - M \|_{L^\infty_q(\T^3\times\R^3)} < \eps, \quad t\in [0,T].\]
\end{lemma}
\begin{proof}
To begin, we make the restriction $\|f_{\rm in} - M\|_{L^\infty(\T^3\times\R^3)}< \frac 1 2$. From Theorem \ref{t:existence}, the solution $f$ exists on a time interval $[0,T_f]$, with $T_f$ depending only on $\|f_{\rm in}\|_{L^\infty_q(\T^3\times\R^3)} \leq \|M\|_{L^\infty_q(\T^3\times\R^3)} + \frac 1 2$.  In particular, $T_f$ is bounded below by a constant depending only on $q$.

Writing $f = M+\tilde f$, we have the following equation for $\tilde f$:
\[ \partial_t \tilde f + v\cdot \nabla_x \tilde f = Q(M+\tilde f, M+\tilde f) = Q(M+\tilde f, \tilde f) + Q(\tilde f, M), \quad (t,x,v) \in [0,T_f]\times\T^3\times\R^3,\]
since $Q(M,M) = 0$. We will derive an upper bound for $\|\tilde f\|_{L^\infty_q}$ using a barrier argument similar to the proof of Lemma \ref{l:simple-bound}. 

With $T$ and $\eps$ as in the statement of the lemma, let $\delta, \beta>0$ be two constants such that 
\[ \delta e^{\beta T} < \eps.\]
The specific values of $\delta$ and $\beta$ will be chosen later. Defining $g(t,x,v) =\delta e^{\beta t} \vv^{-q}$, and taking $\delta > \|f_{\rm in}-M\|_{L^\infty_{q}(\T^3\times \R^3)}$, we have $|\tilde f(0,x,v)| < g(0,x,v)$ for all $x$ and $v$.  We claim $\tilde f(t,x,v)<g(t,x,v)$ in $[0,\min(T,T_f)]\times\T^3\times\R^3$.  
If not, then by making $q$ slightly smaller, but still larger than $3+\gamma+2s$, we ensure the function $f$ decays in $v$ at a polynomial rate faster than $\vv^{-q}$. Together with the compactness of the spatial domain $\T^3$, this implies there is a first crossing point $(t_\rmcr, x_\rmcr, v_\rmcr)$ with $t_\rmcr>0$, where $\tilde f(t_\rmcr,x_\rmcr,v_\rmcr) = g(t_\rmcr,x_\rmcr,v_\rmcr)$. At this point, one has, as in the proof of Lemma \ref{l:simple-bound},
\begin{equation}\label{e.crossing2}
\partial_t g\leq Q(M+\tilde f,\tilde f) + Q(\tilde f, M) \leq Q(M+\tilde f,g) + Q(\tilde f, M).
\end{equation} 
Lemma \ref{l:Q-polynomial} implies, at $(t_\rmcr, x_\rmcr,v_\rmcr)$,
\begin{equation}
 Q(M+ \tilde f,g) = \delta e^{\beta t_\rmcr} Q(M+\tilde f,\langle \cdot \rangle^{-q}) \leq C \delta e^{\beta t_\rmcr} \|M+\tilde f(t_\rmcr, \cdot,\cdot)\|_{L^\infty_{q}(\T^3 \times \R^3)} \langle v_\rmcr\rangle^{-q}.
 \end{equation}
 Since $\tilde f(t_\rmcr,x_\rmcr,v_\rmcr) = g(t_\rmcr, x_\rmcr,v_\rmcr) = \delta e^{\beta t} \langle v_\rmcr\rangle^{-q} < \frac 1 2 \langle v_\rmcr\rangle^{-q}$, and $(t_\rmcr, x_\rmcr,v_\rmcr)$ is the location of a maximum of $\tilde f$ in $(x,v)$ space, we conclude $\|\tilde f(t_\rmcr,\cdot,\cdot)\|_{L^\infty_q(\T^3\times\R^3)} < \frac 1 2$. This implies  $\|M+\tilde f(t_\rmcr,\cdot,\cdot)\|_{L^\infty_q(\T^3\times\R^3)} \leq C$ for a constant $C$ depending only on $q$. We therefore have
\begin{equation}\label{e.C1}
 Q(M+ \tilde f,g) \leq C \delta e^{\beta t_\rmcr} \langle v_\rmcr\rangle^{-q}.
 \end{equation}

For the term $Q(\tilde f,M)$, which appeared as a result of recentering around $M$, we write $Q(\tilde f,M) = Q_{\rm s}(\tilde f,M) + Q_{\rm ns}(\tilde f,M)$. The singular part is handled by \cite[Lemma 3.8]{silvestre2022nearequilibrium}, whose proof does not depend on the sign of $\gamma+2s$ and is therefore valid under our assumptions. This lemma gives
\begin{equation}\label{e.C2}
 |Q_{\rm s}(\tilde f, M)(t_\rmcr, x_\rmcr,v_\rmcr)|\leq  C \delta e^{\beta t} \langle v_\rmcr\rangle ^{-q+\gamma}, \quad \text{ if } |v_\rmcr|\geq R,
 \end{equation}
for universal constants $C, R>0$. On the other hand, if $|v_\rmcr|< R$, the more crude estimate of Lemma \ref{l:C2Linfty} yields
\begin{equation}\label{e.C3}
\begin{split}
 |Q_{\rm s}(\tilde f,M)(t_\rmcr, x_\rmcr,v_\rmcr)| &\leq C \left(\int_{\R^3} \tilde f(v_\rmcr+w)|w|^{\gamma+2s}\dd w\right) \|M\|_{L^\infty(\R^3)}^{1-s} \|D_v^2 M\|_{L^\infty(\R^3)}^s \\
 &\leq C \|\tilde f\|_{L^\infty_q} \langle R\rangle^{(\gamma+2s)_+} \leq C\delta e^{\beta t_\rmcr}  \langle v_\rmcr\rangle^{-q},
 \end{split}
 \end{equation}
 since $\langle R\rangle^{(\gamma+2s)_+} \approx 1 \approx \langle v_\rmcr\rangle^{-q}$ and $\|\tilde f(t_\rmcr,\cdot,\cdot)\|_{L^\infty_q(\T^3\times\R^3)} \leq \delta e^{\beta t_\rmcr}$.  For the nonsingular part, we have
\begin{equation}\label{e.C4}
\begin{split}
 |Q_{\rm ns}(\tilde f, M)(t_\rmcr,x_\rmcr,v_\rmcr)| &\leq C M(v_\rmcr) \int_{\R^3} \tilde f(t_\rmcr, x_\rmcr, w) |v_\rmcr+w|^\gamma \dd w\\
& \leq C \|\tilde f(t_\rmcr,\cdot,\cdot)\|_{L^\infty_{q}(\T^3\times\R^3)} \langle v_\rmcr\rangle^{-q} \leq C  \delta e^{\beta t_\rmcr}\langle v_\rmcr\rangle^{-q},
 \end{split}
 \end{equation}
since $q>3+\gamma+2s>\gamma+3$ and $M$ decays much faster than $\vv^{-q}$.

Collecting all our inqualities and recalling \eqref{e.crossing2} and $\partial_t g = \delta\beta e^{\beta t} \vv^{-q}$, we now have
\[ \delta \beta e^{\beta t_\rmcr} \langle v_\rmcr \rangle^{-q} \leq C_0 \delta e^{\beta t_\rmcr} \langle v_\rmcr\rangle^{-q},\]
where $C_0$ is the maximum among the constants in \eqref{e.C1}, \eqref{e.C2}, \eqref{e.C3}, and \eqref{e.C4}. This implies a contradiction if we choose $\beta = 2C_0$.

We conclude $f(t,x,v)< g(t,x,v)$ on $[0,\min(T,T_f)]\times\T^3\times\R^3$, as claimed. A similar argument using $-g$ as a lower barrier for $\tilde f$ gives 
\[|\tilde f(t,x,v)|< g(t,x,v) = \delta e^{2C_0 t} \vv^{-q}, \quad 0\leq t \leq \min(T,T_f).\]
Finally, we choose $\delta = \min(\frac 1 4, \eps e^{-2 C_0 T})$, so that $\vv^q \tilde f(t,x,v) < \delta e^{2C_0 T} < \eps$ whenever $t\leq \min(T,T_f)$, and $\|f_{\rm in} - M\|_{L^\infty_q(\T^3\times\R^3)} < \delta < \frac 1 2$.

The above barrier argument does not depend quantitatively on the time of existence $T_f$. If $T> T_f$, then we have shown $\|\tilde f(T_f, \cdot, \cdot)\|_{L^\infty_q(\T^3\times\R^3)} < \eps < \frac 1 2$, and by applying Theorem \ref{t:existence} again with initial data $f(T_f,\cdot,\cdot)$, we can continue the solution to a time interval $[0,2T_f]$. Repeating finitely many times, we continue the solution to $[0,NT_f]\times\T^3\times\R^3$ where $NT_f>T$, with $\|\tilde f\|_{L^\infty_q([0,T]\times\T^3\times\R^3)} < \eps$, as desired.
\end{proof}

Next, we need the main result of \cite{desvillettes2005global}. The result in \cite{desvillettes2005global} is stated for solutions defined on $[0,\infty)\times\T^3\times\R^3$, and gives conditions under which solutions converge to a Maxwellian $M$ as $t\to\infty$. Since the estimates at a fixed time $t$ do not depend on any information about the solution for times greater than $t$, we easily conclude (as in \cite{silvestre2022nearequilibrium}) the following restatement that applies to solutions defined on a finite time interval:

\begin{theorem}\label{t:dv}
Let $f\geq 0$ be a solution to \eqref{e.boltzmann} on $[0,T]\times \T^3\times \R^3$ satisfying, for a family of positive constants $C_{k,q}$, 
\[ \|f\|_{L^\infty([0,T],H^{k}_{q}(\T^3\times \R^3))} \leq C_{k,q} \quad \text{ for all } k,q\geq 0,\]
and also satisfying the pointwise lower bound
\[ f(t,x,v) \geq K_0 e^{-A_0 |v|^2}, \quad \text{ for all } (t,x,v).\]
Then for any $p>0$ and for any $k,q>0$, there exists $C_p>0$ depending on $\gamma$, $s$, $A_0$, $K_0$, the constant $c_b$ in \eqref{e.b-bounds}, and $C_{k',q'}$ for sufficiently large $k'$ and $q'$, such that for all $t \in [0,T]$,
\[   \|f(t,\cdot,\cdot)- M\|_{H^{k}_{q}(\T^3\times \R^3)} \leq C_p t^{-p},\]
where $M$ is the Maxwellian with the same total mass, momentum, and energy as $f$.
\end{theorem}

Along with Lemma \ref{l:Max} and Theorem \ref{t:dv}, the proof of Corollary \ref{c:global} relies on the global regularity estimates of \cite{imbert2020smooth}, which we extended to $\gamma +2s< 0$ in Proposition \ref{p:higher-reg}. 

We are now ready to give the proof:

\begin{proof}[Proof of Corollary \ref{c:global}]

First, let us assume $f_{\rm in}\in L^\infty_q(\T^3\times\R^3)$ for all $q>0$, i.e. $f$ decays pointwise faster than any polynomial. 

For $q_0>5$ fixed, use Lemma \ref{l:Max} to select $\delta_1>0$ such that $f$ exists on $[0,1]\times \T^3\times\R^3$ with $\|f(t)-M\|_{L^\infty_{q_0}(\T^3\times\R^3)} < \eps$ for all $t\in [0,1]$, whenever $\|f_{\rm in}-M\|_{L^\infty_{q_0}(\T^3\times\R^3)} < \delta_1$. By our Theorem \ref{t:existence} and the rapid decay of $f_{\rm in}$, the solution $f$ is $C^\infty$ in all variables. 

Now, if the conclusion of the theorem is false, there is a first time $t_\rmcr >1$ such that 
\begin{equation}\label{e.equilibrium-crossing}
 \|f(t_\rmcr,\cdot,\cdot)-M\|_{L^\infty_{q_0}(\T^3\times\R^3)} = \eps.
 \end{equation}
Since $\|f-M\|_{L^\infty_{q_0}([0,t_\rmcr]\times\T^3\times\R^3)} \leq \eps$, the solution $f$ satisfies uniform lower bounds for $(t,x,v)\in [0,t_\rmcr]\times B_1(0)\times B_1(0)$. These lower bounds, together with the bound on $\|f\|_{L^\infty_{q_0}}$, imply via Proposition \ref{p:higher-reg} that $f$ satisfies uniform estimates in $H^k_q([1,t_\rmcr]\times \T^3\times\R^3)$ for all $k,q>0$, with constants independent of $t$ and depending only on $q_0$ and the norms $\|f_{\rm in}\|_{L^\infty_q}$ of the initial data for $q>0$.

On the time interval $[1,t_\rmcr]$, the function $f$ satisfies the hydrodynamic bounds 
\begin{equation}\label{e.hydro}
0< m_0\leq \int_{\R^3} f(t,x,v) \dd v \leq M_0, \,\int_{\R^3} |v|^2f(t,x,v) \dd v \leq E_0, \, \int_{\R^3} f(t,x,v) \log f(t,x,v) \dd v \leq H_0,
\end{equation}
uniformly in $t$ and $x$, for some constants $m_0, M_0, E_0, H_0$ depending only on $q_0$,  as a result of the inequality $\|f-M\|_{L^\infty_{q_0}(\T^3\times\R^3)}\leq \eps< \frac 1 2$. (This follows from a quick computation, or we may apply \cite[Lemma 2.3]{silvestre2022nearequilibrium}.) From \cite{imbert2020lowerbounds}, this implies 
the lower Gaussian bound $f(t,x,v)\geq c_1 e^{-c_2|v|^2}$, with $c_1,c_2>0$ depending only on $q_0$. 
We note that the work in~\cite{imbert2020lowerbounds} works under the global assumption that $\gamma + 2s \in [0,2]$, but it is clear that, when $\gamma+2s<0$, the lower bounds of \cite{imbert2020lowerbounds} are valid given a bound on the $L^\infty$-norm in addition to the bounds in \eqref{e.hydro}. This $L^\infty$-bound also clearly follows from the inequality $\|f-M\|_{L^\infty_{q_0}} \leq \eps$.

The hypotheses of \cite{desvillettes2005global}, restated above as Theorem \ref{t:dv}, are satisfied, so choosing $p=1$, $k=4$, and using the Sobolev embedding theorem, we have
\[ \|f(t_\rmcr,\cdot,\cdot)\|_{L^\infty_{q_0}(\T^3\times\R^3)} \leq C_1 t_\rmcr^{-1},\]
where $C_1>0$ depends only on $\gamma$, $s$, $c_b$, $q_0$, and the norm of $f_{\rm in}$ in $L^\infty_{q_1}(\T^3\times\R^3)$ for some $q_1$ depending on $q_0$. Combining this with \eqref{e.equilibrium-crossing} gives $t_\rmcr \leq C_1/\eps$.

Letting $T = C_1/\eps+1$, we use Lemma \ref{l:Max} again to select $\delta_2>0$ such that $f$ exists on $[0,T]\times \T^3\times\R^3$, with 
\[\|f(t,\cdot,\cdot)-M\|_{L^\infty_{q_0}(\T^3\times\R^3)}< \eps,\]
for all $t\in [0,T]$. This inequality implies the first crossing time $t_\rmcr> C_1/\eps+1$, a contradiction with $t_\rmcr\leq C_1/\eps$. Therefore, if $\|f_{\rm in}- M\|_{L^\infty_{q_0}(\T^3\times\R^3)} < \delta := \min(\delta_1, \delta_2)$, we conclude there is no crossing time $t_\rmcr$, and $\|f(t,\cdot,\cdot)-M\|_{L^\infty_{q_0}(\T^3\times\R^3)}< \eps$ holds for all $t$ such that the solution $f$ exists. In particular, $\|f(t)\|_{L^\infty_{q_0}(\T^3\times\R^3)}$ is bounded by a constant independent of $t$, and Theorem \ref{t:existence} implies the solution can be extended for all time.

Next, we consider the general case, where $f_{\rm in}$ decays at only a finite polynomial rate. Looking at the proof of the previous case, we see that the choice of $\delta$ depends only on $\eps$, $q_0$, and the size of $f_{\rm in}$ in the $L^\infty_{q_1}$ norm for some $q_1$ depending on $q_0$. Therefore, if $f_{\rm in} \in L^\infty_{q_1}(\T^3\times\R^3)$ and satisfies all the hypotheses of Corollary \ref{c:global}, we can approximate $f_{\rm in}$ by cutting off large velocities, apply the rapid-decay case considered above to obtain global solutions, and take the limit as the cutoff vanishes. We omit the details of this standard approximation procedure.
\end{proof}

\appendix

\section{Change of variables}\label{s:cov-appendix}

This appendix is devoted to the proof of Proposition \ref{p:very-soft-kernel}, which establishes the properties of the integral kernel $\bar K_f$ defined in \eqref{e.barKf} for the case $\gamma + 2s < 0$.  Noting that the case $\gamma + 2s \geq 0$ has received a full treatment in~\cite{imbert2020smooth}, we only consider the regime 
\be
	\gamma+2s<0
\ee
throughout this appendix.

In order to prove Proposition \ref{p:very-soft-kernel}, we need to verify the following for the kernel $\bar K_f(t,x,v,v')$: coercivity, boundedness, cancellation, and H\"older continuity in $(t,x,v)$. The proof strategies and notation broadly follow \cite[Section 5]{imbert2020smooth}, which addressed the case $\gamma+2s\in [0,2]$. However, the details are sufficiently different that it is necessary to provide full proofs. 

When $|v_0|\leq 2$, the change of variables is defined as $\mathcal T_0z = z_0\circ z$, i.e. a simple recentering around the origin. Therefore, $\bar K_f$ inherits the properties of $K_f$, which satisfies suitable ellipticity properties on any bounded velocity domain, see Remark \ref{r:ellipticity}. Therefore, in this appendix we focus only on the case $|v_0|> 2$.

Recalling the definition \eqref{e.T0-def} of the linear transformation $T_0$, we see that in the current regime,
\[ 
T_0 (av_0+w) = |v_0|^{\frac{\gamma+2s}{2s}}\left( \frac a {|v_0|} v_0 + w\right), \quad \text{where } a\in \R, w\cdot v_0 = 0, \gamma+2s< 0.
\]
When we import facts involving this linear transformation from \cite{imbert2020smooth}, we use the notation $T_0^+$ for the transformation $T_0$ as it is defined in the case $\gamma+2s\geq 0$. Then one has
\begin{equation}\label{e.T0-cases}
T_0 v = |v_0|^{\frac{\gamma+2s}{2s}} T_0^+ v, \quad \text{ when } \gamma +2s < 0.
\end{equation}
Note that the definition of $T_0^+$ as a linear transformation does not depend on $\gamma$ or $s$. The $T_0^+$ notation is intended only for use in the current appendix.

In the following lemmas about the change of variables, we omit the dependence of $\bar f$ and $\bar K_f$ on $\bar t$ and $\bar x$, since the conditions all hold uniformly in $t$ and $x$.

\subsection{Coercivity}
 With $A(v) \subset \mathbb S^2$ the subset of the unit sphere given by Lemma \ref{l:cone}, define $\Xi(v)$ to be the corresponding cone in $\R^3$, $\Xi(v) := \{w\in \R^3 : \frac w{|w|} \in A(v)\}$. 
\begin{lemma}[Transformed cone of non-degeneracy]
Let $f$, $\delta$, $r$, and $v_m$ satisfy the assumptions of Lemma \ref{l:cone}. 
 Fix $v_0 \in \R^3$ and $v\in B_2$, and define 
\[ \begin{split}
	\bar A(v)
		&= \{\sigma \in \mathbb S^2 : T_0\sigma/|T_0\sigma| \in A(v_0+T_0v)\},\\
	\bar \Xi(v)
		&= \{w \in \R^3 : T_0 w \in \Xi(v_0+ T_0v)\}.
 \end{split}
 \]
 Then there are constants $\lambda, k>0$, depending only on $\delta$, $r$, and $v_m$ (but not on $v_0$ or $v$), such that  
 \begin{itemize}
	 \item $\bar K_f(v,v+w) \geq \lambda |w|^{-3-2s}$ whenever $w \in \bar \Xi(v)$;
	 \item$\mathcal H^2(\bar A(v)) \geq k$, where $\mathcal H^2$ is the $2$-dimensional Hausdorff measure.
 \end{itemize}
\end{lemma} 
\begin{proof}
For the first bullet point, Lemma \ref{l:cone} and the definition \eqref{e.barKf} of $\bar K_f$ imply that, for $w\in \bar \Xi(v)$,
\begin{equation}
\begin{split}
	\bar K_f(v,v+w)
		&= |v_0|^{2+\frac{3\gamma}{2s}} 
			K_f(v_0+T_0 v, v_0 + T_0 v + T_0 w)\\
		& \geq \lambda  |v_0|^{2+(3\gamma)/(2s)} |\bar v|^{\gamma+2s+1}|T_0 w|^{-3-2s}
 \geq \lambda |w|^{-3-2s},
\end{split}
\end{equation} 
since $|\bar v|\approx |v_0|$ and $|T_0w|\leq |v_0|^{\frac{\gamma+2s}{2s}}|w|$. 

For the second bullet point, use \eqref{e.T0-cases} to write $v_0 + T_0 v = v_0 + T_0^+ \tilde v$ with $\tilde v = |v_0|^{\frac{\gamma+2s}{2s}} v \in B_2$. Next, recall the following fact from \cite[Lemma 5.6]{imbert2020smooth}: for any $\tilde v\in B_2$, 
\begin{equation}\label{e.H2}
\mathcal H^2(\{\sigma \in \mathbb S^2 : T_0^+\sigma/|T_0^+\sigma| \in A(v_0 + T_0^+ \tilde v)\}) \geq k,
\end{equation}
for some $k>0$ depending on the constants of Lemma \ref{l:cone}, and independent of $v_0$. We note that the statement and proof of estimate \eqref{e.H2} do not depend on the values of $\gamma$ and $s$. For any $v\in B_2$, we conclude from \eqref{e.H2}, using $T_0^+ \tilde v = T_0 v$ and  $T_0^+\sigma/|T_0^+\sigma| = T_0\sigma/|T_0\sigma|$, that
\[ 
\mathcal H^2 (\bar A(v)) = \mathcal H^2(\{\sigma \in \mathbb S^2 : T_0\sigma/|T_0\sigma| \in A(v_0 + T_0v)\}) \geq k, 
\]
as desired. 
\end{proof}

\subsection{Boundedness conditions}

Next, we address the upper ellipticity bounds for the kernel $\bar K_f$. The following lemma corresponds to \cite[Lemma 5.10]{imbert2020smooth}, but the proof must be modified to account for the extra powers of $|v_0|$ in the definition \eqref{e.barKf} of $\bar K_f$.

\begin{lemma}\label{l:5-10}
 For $v_0 \in \R^3\setminus B_2$, $v\in B_2$, and $r>0$,
\[
\int_{\R^3\setminus B_r(v)} \bar K_f(v,v') \dd v' \leq \bar \Lambda r^{-2s},
\]
with 
\[
\bar \Lambda := |v_0|^{-\gamma-2s} \int_{\R^3} f(\bar v + w)\left( |v_0|^2 - \left(v_0\cdot \frac w {|w|}\right)^2 + 1\right)^s |w|^{\gamma+2s} \dd w.
\]
\end{lemma}
\begin{proof}
From the definition \eqref{e.barKf} of $\bar K_f$, we have
\[
\begin{split}
	\int_{\R^3\setminus B_r(v)} \bar K_f(v,v') \dd v'
 		&= |v_0|^{2+(3\gamma)/(2s)} \int_{\R^3\setminus B_r(v)} K_f(\bar v, \bar v') \dd v'\\
		 &= \int_{\R^3\setminus T_0(B_{ r})} K_f(\bar v, \bar v + u) \dd u,
\end{split}
\]
from the change of variables $u = \bar v' - \bar v = 
T_0(v' - v)$.   Following \cite{imbert2020smooth}, we use \eqref{e.kernel} to write
\begin{equation}\label{e.Kf-Er}
 \begin{split}
	\int_{\R^3\setminus B_r(v)} \bar K_f(v,v') \dd v'
		&\lesssim \int_{u\in \R^3\setminus T_0(B_{ r})} |u|^{-3-2s} \int_{w\perp u} f(\bar v+w) |w|^{1+\gamma+2s} \dd w \dd u\\
		 &= \int_{w\in \R^3} \left(\int_{u\perp w, u\in \R^3\setminus T_0(B_{ r})} |u|^{-2-2s} \dd u\right) f(\bar v+w)|w|^{\gamma+2s} \dd w,
\end{split} 
\end{equation}
where we used
\[
	\int_u \int_{w\perp u} (\ldots) \dd w \dd u = \int_w \int_{u\perp w} (\ldots) \frac {|u|}{|w|} \dd u \dd w.
\]
%
Recall that $T_0(B_{r})$ is an ellipsoid with radius $\bar r := r |v_0|^\frac{\gamma+2s}{2s}$ in directions orthogonal to $v_0$ and radius $\bar r/|v_0| = r|v_0|^{\gamma/(2s)}$ in the $v_0$ direction. Its intersection with the plane $\{u\perp w\}$ is an ellipse, whose smallest radius is 
\[
	\rho
		:=  \frac {r|v_0|^\frac{\gamma+2s}{2s}}{ \sqrt{|v_0|^2\left(1 - \left(\frac{v_0\cdot w}{|v_0||w|}\right)^2\right) + \left(\frac{v_0\cdot w}{|v_0| |w|}\right)^2}}.
\]
This follows from formula (5.10) in \cite{imbert2020smooth}, with $\bar r = r|v_0|^\frac{\gamma+2s}{2s}$ replacing $r$. We therefore have $\R^3\setminus E_r \subset \R^3\setminus B_\rho$, and
 \[ \begin{split}
  \int_{u\perp w, u\in \R^3\setminus E_r} |u|^{-2-2s} \dd u \lesssim \rho^{-2s}&\leq r^{-2s} |v_0|^{-\gamma-2s} \left( |v_0|^2 \left(1-\left(\frac{v_0\cdot w}{|v_0||w|}\right)^2\right) + \left(\frac{v_0\cdot w}{|v_0| |w|}\right)^2\right)^s\\
  &\leq r^{-2s} |v_0|^{-\gamma-2s} \left( |v_0|^2 - \left( v_0\cdot \frac w {|w|}\right)^2 + 1\right)^s.
  \end{split} \]
  Combining this expression with \eqref{e.Kf-Er}, the conclusion of the lemma follows.
\end{proof}

\begin{lemma}[Boundedness conditions]\label{l:boundedness}
If $\|f\|_{L^\infty_q(\R^3)}< \infty$ for some $q> 3+2s$, then the kernel $\bar K_f$ satisfies the two conditions
\begin{align}
 \int_{\R^3\setminus B_r(v)} \bar K_f(v,v') \dd v' \leq \Lambda r^{-2s}, \quad \text{ for all } v \in B_2 \text{ and } r>0,\label{e.first-boundedness}\\
 \int_{\R^3\setminus B_r(v')} \bar K_f(v,v') \dd v \leq \Lambda r^{-2s}, \quad \text{ for all } v' \in B_2 \text{ and } r>0,\label{e.second-boundedness}
 \end{align}
 for a constant $\Lambda\lesssim \|f\|_{L^\infty_q(\R^3)}$. In particular, $\Lambda$ is independent of the base point $v_0$.
\end{lemma}
\begin{proof}
The proof of \eqref{e.first-boundedness} begins with estimating the expression $\bar \Lambda$ in Lemma \ref{l:5-10} from above. First, note that 
\[\left(|v_0|^2 - \left(v_0 \cdot \frac w {|w|}\right)^2 + 1\right)^s \lesssim \left(|v_0|^2 - \left(v_0 \cdot \frac w {|w|}\right)^2\right)^{s} + 1,\]
and we have
\[ \bar \Lambda \leq \int_{\R^3} f(\bar v + w) \left(|v_0|^2 - \left(v_0 \cdot \frac w {|w|}\right)^2\right)^{s}\left(\frac {|w|}{|v_0|}\right)^{\gamma+2s} \dd w + \int_{\R^3} f(\bar v + w) \left(\frac {|w|}{|v_0|}\right)^{\gamma+2s} \dd w =: J_1 + J_2.\]
To bound $J_2$, the convolution estimate of Lemma \ref{l:convolution} gives $J_2 \lesssim \|f\|_{L^\infty_q} |\bar v|^{\gamma+2s} |v_0|^{-\gamma-2s} \lesssim \|f\|_{L^\infty_q}$, since $|\bar v| \approx |v_0|$ and $q > 3+\gamma+2s$.

For $J_1$, letting $w = \alpha v_0/|v_0| + b$ with $b\cdot v_0 = 0$, one has
\[ \left(\frac{|w|}{|v_0|}\right)^{\gamma+2s}\left( |v_0|^2 - \left( v_0 \cdot \frac w {|w|}\right)^2\right)^{s} = |v_0|^{-\gamma} |b|^{2s}|w|^\gamma.\]
Noting that $|b|\leq |v_0+w|$, we have 
\[ 
\begin{split}
\int_{\R^3} f(\bar v+w) |v_0|^{-\gamma} |b|^{2s} |w|^\gamma \dd w &\leq |v_0|^{-\gamma} \|f\|_{L^\infty_q(\R^3)} \int_{\R^3} \langle \bar v+w\rangle^{-q} |v_0+w|^{2s} |w|^{\gamma} \dd w\\
&\lesssim   |v_0|^{-\gamma} \|f\|_{L^\infty_q(\R^3)} \int_{\R^3} \langle v_0+w\rangle^{-q+2s} |w|^{\gamma} \dd w \lesssim \|f\|_{L^\infty_q(\R^3)},
\end{split} 
\]
using $\langle \bar v+w\rangle \approx \langle v_0+w\rangle$ and $q> 3+\gamma+2s$.  
This establishes the upper bound $\bar \Lambda \lesssim \|f\|_{L^\infty_q(\R^3)}$. Combining this with Lemma \ref{l:5-10} concludes the proof of the first boundedness condition \eqref{e.first-boundedness}.

To establish \eqref{e.second-boundedness}, we assume as usual that $|v_0|>2$. For any $v'\in B_2$ and $r>0$, changing variables with $\bar v = v_0 + T_0 v$, we have
\[
	\begin{split}
		\int_{\R^3\setminus B_r(v')} \bar K_f(v,v') \dd v
			&= |v_0|^{2+\frac{3\gamma}{2s}} \int_{\R^3\setminus B_r(v')} K_f(\bar v, \bar v') \dd v
			= \int_{\R^3\setminus E_r(\bar v')} K_f(\bar v, \bar v') \dd \bar v. 
	\end{split}
\]
The last integral is estimated in the proof of \cite[Lemma 5.13]{imbert2020smooth}, up to choosing a different value of $r$. More specifically, our $E_r$ would be $E_{\bar r}$ with $\bar r = r|v_0|^{\frac{\gamma+2s}{2s}}$ in the notation of \cite{imbert2020smooth}. Therefore, their calculation (which does not depend on the sign of $\gamma+2s$) implies
\[
	\begin{split}
		& \int_{\R^3\setminus B_r(v')} \bar K_f(v,v') \dd v\\
		& \leq \int_{\R^3} f(\bar v'+u)|u|^\gamma
		\Big( (r|v_0|^\frac{\gamma+2s}{2s})^{-2s} |u|^{2s}
		\Big( 1+ |v_0|^2 - \frac{(v_0\cdot u)^2}{|u|^2} \Big)^s
		+ (r|v_0|^\frac{\gamma+2s}{2s})^{-s} 
		|u|^s
		\Big| \frac u {|u|}\cdot v_0\Big|^s \Big) \dd u \\
		&\leq I_1 + I_2,
	\end{split}
\]
%
where 
\[
	\begin{split}
		I_1
		&= r^{-2s} |v_0|^{-\gamma-2s} \int_{\R^3} f(\bar v' + u) |u|^{\gamma+2s} \Big(1+|v_0|^2 - \frac{(v_0\cdot u)^2}{|u|^2}\Big)^s \dd u,\\
		I_2
		&= r^{-s} |v_0|^{-\gamma/2} \int_{\R^3} f(\bar v' + u) |u|^{\gamma+s} \dd u.
\end{split}
\]
The term $I_1$ is bounded by a constant times $\|f\|_{L^\infty_q(\R^3)}r^{-2s}$, by our estimate of $\bar\Lambda$ in the beginning of the current proof. For $I_2$, we have
\[
	I_2
		\leq r^{-s} |v_0|^{-\gamma/2} \|f\|_{L^\infty_q(\R^3)} \int_{\R^3} \langle \bar v' + u\rangle^{-q} |u|^{\gamma+s} \dd u
		\lesssim r^{-s} |v_0|^{-\gamma/2} \|f\|_{L^\infty_q(\R^3)} \langle \bar v'\rangle^{\gamma+s},
\]
since $q >3+3+\gamma+2s>3+\gamma+s$. By assumption, $v' \in B_2$, which implies $|\bar v'| = |v_0 + T_0v'|\lesssim |v_0|$. Since $\gamma+2s<0$, the factor $|v_0|^{(\gamma+2s)/2}$ is bounded by 1, and we conclude $I_2 \lesssim r^{-s} \|f\|_{L^\infty_q(\R^3)}$. Note that values of $r$ greater than $2$ are irrelevant for \eqref{e.second-boundedness}, because the kernel $\bar K_f(v,v')$ is only defined for $v \in B_1$. For large $r$, the domain of integration in \eqref{e.second-boundedness} is empty. Therefore, we have $r^{-s} \lesssim r^{-2s}$, and the proof is complete.
\end{proof}

We remark that the bound of $I_2$ in the proof of the previous lemma is the only place where our definition \eqref{e.cov-} of the change of variables would not easily generalize to the case $\gamma+2s\geq 0$. 

We also have the following alternative characterization of the upper bounds for $\bar K_f$, which is needed as one of the hypotheses of Theorem \ref{t:schauder}. It follows from \eqref{e.first-boundedness} in Lemma \ref{l:boundedness} in the same way that Lemma \ref{l:K-upper-bound-2} above follows from Lemma \ref{l:K-upper-bound}:

\begin{corollary}\label{c:w-squared}
Let $f\in L^\infty_q([0,T]\times\R^6)$ for some $q>3 + 2s$. For $|v_0|>2$ and $z= (t,x,v)\in Q_1$, let $\bar K_{f,z}(w) = \bar K_f(t, x,v, v+w)$. Then for any $r>0$, there holds
\[ \int_{B_r} \bar K_{f,z} (w) |w|^2 \dd w \leq C \|f(\bar t,\bar x,\cdot)\|_{L^\infty_q(\R^3)} r^{2-2s},\]
The constant $C$ depends only on $\gamma$ and $s$.
\end{corollary}

\subsection{Cancellation conditions} 

Next, we establish two cancellation conditions for $\bar K_f$, which say that $\bar K_f(v,v')$ is not too far from being symmetric, on average.

As a technical tool in proving these lemmas, one needs the following ``modified principal value'' result, which allows one to change variables according to $v' \mapsto \bar v'$ without altering the cancellation involved in defining principal value integrals. This lemma is proven in \cite[Lemmas 5.14 and 5.16]{imbert2020smooth}, with an argument that does not use the sign of $\gamma+2s$. Therefore, the lemma remains valid in our context.
\begin{lemma}\label{l:modified}
Let $\gamma+2s<0$ and $\rho>0$, and let $T_0$ be defined as above.
\begin{enumerate}
\item[(a)] If $f:\R^3\to \R$ is such that $\vv^{\gamma+2s} D^2 f \in L^1(\R^3)$, then
\[ \lim_{\rho\to 0+} 	\int_{B_\rho \setminus T_0(B_\rho)} (K_f(\bar v,\bar v+w) - K_f(\bar v+w,\bar v)) \dd w = 0.\]
\item[(b)] If $f:\R^3\to \R$ is such that $\vv^{\gamma+2s}\nabla f \in L^1(\R^3)$, then
\[ \lim_{\rho\to 0+} \int_{B_\rho\setminus T_0(B_\rho)} w K_f(\bar v+w, \bar v) \dd w = 0.\]
\end{enumerate}
\end{lemma}
This lemma is proven for $T_0^+$ rather than $T_0$ in \cite{imbert2020smooth}. However, $T_0$ and $T_0^+$ can easily be interchanged here, by rescaling the parameter $\rho$.

Next, we prove the first cancellation condition, following the strategy of \cite[Lemma 5.15]{imbert2020smooth}:
\begin{lemma}[First cancellation condition]\label{l:first-cancellation}
Fix $q>3+\gamma$.  Suppose that $f\in L^\infty_q(\R^3)$ and $\vv^{\gamma+2s} D_v^2 f \in L^1(\R^3)$. Then the kernel $\bar K_f$ satisfies
\[
\begin{split}
 \left| {\rm p.v.} \int_{\R^3} (\bar K_f(v,v') - \bar K_f(v',v)) \dd v'\right| &\leq C \|f\|_{L^\infty_q(\R^3)},\end{split}
\]
where $C>0$ is universal. 
\end{lemma}
\begin{proof}
If $|v_0|\leq 2$, then the conclusion follows from the classical Cancellation Lemma, stated for example in \cite[Lemma 3.6]{imbert2016weak}. 

If $|v_0|>2$, then we write for any $v\in B_2$,
\[
\begin{split}
	{\rm p.v.}\int_{\R^3}& (\bar K_f(v,v') - \bar K_f(v',v)) \dd v'\\
	 &= |v_0|^{2+\frac{3\gamma}{2s}} {\rm p.v.}\int_{\R^3} (K_f(\bar v, \bar v') - K_f(\bar v',\bar v)) \dd v'\\
	 &=  |v_0|^{2+\frac{3\gamma}{2s}} \lim_{R\to 0+} \int_{\R^3\setminus B_{R}} (K_f(\bar v, \bar v +  T_0 v') - K_f(\bar v + T_0 v', \bar v)) \dd v'\\
	 &= \lim_{R\to 0+} \int_{\R^3\setminus T_0(B_R)} (K_f(\bar v, \bar v + w)  - K_f(\bar v+ w, \bar v)) \dd  w\\
	 &= \lim_{R\to 0+} \int_{\R^3\setminus B_{R}} (K_f(\bar v, \bar v + w)  - K_f(\bar v+ w, \bar v)) \dd w\\
	 &= {\rm p.v.} \int_{\R^3} (K_f(\bar v, \bar v + w)  - K_f(\bar v+ w, \bar v)) \dd w,
\end{split}
\]
where we have used Lemma \ref{l:modified}(a) with $\rho = R|v_0|^{\gamma/(2s)+1}$ in the next-to-last equality. Next, we apply \cite[Lemma 3.6]{imbert2016weak} and obtain, for $q> \gamma+3$
\[ \left| {\rm p.v.}\int_{\R^3} (\bar K_f(v,v') - \bar K_f(v',v)) \dd v' \right| \leq  C\left(\int_{\R^3} f(z) |\bar v - z|^\gamma \dd z\right)  \leq C \|f\|_{L^\infty_q(\R^3)}, \]
as desired.
\end{proof}

\begin{lemma}[Second cancellation condition]
Fix $s\in [\frac 1 2, 1)$.  Suppose that $f\in L^\infty_q(\R^3)$ with $q>3+2s$ and $\vv^{\gamma+2s} \nabla_v f \in L^1(\R^3)$. Then for all $r\in [0,\frac 1 4]$ and $v\in B_{7/4}$, there holds
\[ \left| {\rm p.v.} \int_{B_r(v)} (\bar K_f(v,v') - \bar K_f(v',v)) (v'-v) \dd v'\right| \leq \Lambda (1+r^{1-2s}),\]
with $\Lambda$ depending on $\|f\|_{L^\infty_q(\R^3)}$. 
\end{lemma}

\begin{proof}
This proof is similar to the proof of \cite[Lemma 5.18]{imbert2020smooth}. Here, we give a sketch of the argument and discuss the changes needed for our setting.

First, we claim that for $|v_0|\geq 2$, $u\in B_1(v_0)$, and $r\in (0,1)$, there holds 
\begin{equation}\label{e.cancellation-claim}
	\int_{\R^3} f(u+z) |z|^{\gamma+1} \min(1, r^{2-2s} |z|^{2s-2}) \dd z
		\lesssim \|f\|_{L^\infty_q(\R^3)} |v_0|^{\gamma+2s-1} r^{1-2s}.
\end{equation}
To show this, we divide the integral into $B_r$ and $\R^3\setminus B_r$, and write
\[ \int_{B_r} f(u+z) |z|^{\gamma+1} \dd z \leq \|f\|_{L^\infty_q(\R^3)} \int_{B_r} \langle u+ z\rangle^{-q} |z|^{\gamma+1} \dd z \lesssim \|f\|_{L^\infty_q(\R^3)} |v_0|^{-q} r^{\gamma+4}.\]
Since $s\in [\frac 1 2, 1)$ and $r\in (0,1)$, we have $r^{\gamma+4}\leq 1\leq r^{1-2s}$. Also, $q> 3 + 2s> \gamma+2s-1$, so $|v_0|^{-q} < |v_0|^{\gamma+2s-1}$.

Next, since $q> 3 + 2s > \gamma+2s+2$, 
\[
\begin{split}
 r^{2-2s}\int_{\R^3\setminus B_r} f(u+z) |z|^{\gamma+2s-1} \dd z &\leq r^{2-2s} \|f\|_{L^\infty_q(\R^3)} \int_{\R^3\setminus B_r} \langle u+z\rangle^{-q} |z|^{\gamma+2s-1} \dd z\\
 & \leq r^{2-2s} \|f\|_{L^\infty_q(\R^3)} \langle u\rangle^{\gamma+2s-1} \leq r^{1-2s} |v_0|^{\gamma+2s-1} \|f\|_{L^\infty_q(\R^3)}, 
 \end{split}
 \]
using $|u|\approx |v_0|$, $\gamma + 2s \leq 0$, and $1 \leq r^{-1}$. This establishes \eqref{e.cancellation-claim}.

Now, to prove the lemma, we may focus on the case $|v_0|>2$, by \cite[Lemma 3.7]{imbert2016weak}. By the symmetry property $K_f(\bar v, \bar v + \bar w) = K_f(\bar v, \bar v-\bar w)$ of $K_f$, one can easily show ${\rm p.v.} \int_{B_r(v)}  \bar K_f(v',v) (v'-v) \dd v' = 0$. Therefore, it suffices to bound the remaining term
\[
	{\rm p.v.} \int_{B_r(v)} \bar K_f(v,v') (v'-v) \dd v'.
\]
Using the definition \eqref{e.barKf} of $\bar K_f$ and changing variables according to $\bar w = T_0(v-v')$ (which is compatible with the principal value integral, by Lemma \ref{l:modified}(b)), this term equals
\begin{equation}\label{e.Er-int}
	{\rm p.v.} \int_{E_r} (T_0^{-1} \bar w) K_f(\bar v - \bar w, \bar v) \dd \bar w.
\end{equation}
With $\bar r = |v_0|^{\frac{\gamma+2s}{2s}} r$ as above, we decompose this integral as follows, using $E_r = T_0(B_r) = T_0^+(B_{\bar r})$:
\be
	\begin{split}
	{\rm p.v.} \int_{E_r} (T_0^{-1} \bar w) K_f(\bar v - \bar w, \bar v) \dd \bar w
		&\leq \Big|{\rm p.v.} \int_{B_{\bar r}} (T_0^{-1} \bar w) K_f(\bar v - \bar w, \bar v) \dd \bar w\Big|
		\\&\qquad + \Big| \int_{T_0^+(B_{\bar r}) \setminus B_{\bar r}} (T_0^{-1} \bar w) K_f(\bar v - \bar w, \bar v) \dd \bar w\Big|
		=: I_1 + I_2.
	\end{split}
\ee
For $I_2$, we use the following inequality from the proof of \cite[Lemma 5.18]{imbert2020smooth}, with $\bar r$ replacing $r$:
\[
\begin{split}
	 \Big| \int_{T_0^+(B_{\bar r}) \setminus B_{\bar r}} ((T_0^+)^{-1} \bar w) K_f(\bar v - \bar w, \bar v) \dd \bar w\Big|
	&\lesssim |v_0| \int_{\R^3} f(\bar v' + z) |z|^{\gamma+1} \min(1, \bar r^{2-2s}|z|^{2s-2}) \dd w \\
	&\quad + \int_{\R^3}f(\bar v' + z) \bar r^{1-2s} |z|^{\gamma +2s} (1+|v_0|^2 - (v_0\cdot z)^2/|z|^2 )^s \dd w.
\end{split}
\]
 The first term on the right is estimated using \eqref{e.cancellation-claim} with $r = \bar r$. The second term is estimated using our upper bound $\bar \Lambda \lesssim \|f\|_{L^\infty_q(\R^3)}$  from the proof of Lemma \ref{l:boundedness} with $\bar v'$ replacing $\bar v$, which is valid because $|\bar v|\approx |\bar v'|\approx |v_0|$. In all, we have
\[
	\left| \int_{T_0^+(B_{\bar r}) \setminus B_{\bar r}} ((T_0^+)^{-1} \bar w) K_f(\bar v - \bar w, \bar v) \dd \bar w\right|
		\lesssim  {\bar r}^{1-2s}|v_0|^{\gamma+2s}
		= |v_0|^{\frac{\gamma+2s}{2s}} r^{1-2s}.
\]
Since $T_0^{-1}w = |v_0|^{-\frac{\gamma+2s}{2s}} (T_0^+)^{-1} \bar w$, this implies $I_2 \lesssim r^{1-2s}$. For $I_1$, we use another calculation quoted from the proof of \cite[Lemma 5.18]{imbert2020smooth}, where $\bar r$ once again plays the role of $r$:
\[\begin{split}
	\left| \int_{B_{\bar r}} ((T_0^+)^{-1} \bar w) K_f(\bar v - \bar w, \bar v) \dd \bar w\right| 
		&\lesssim |v_0| \int_{\R^3} f(\bar v' + z) |z|^{\gamma+1} \min(1, \bar r^{2-2s}|z|^{2s-2}) \dd z ,
\end{split}
\]
and using \eqref{e.cancellation-claim} again, we conclude
\be
	I_1
		\lesssim |v_0|^{- \frac{\gamma+2s}{2s} + \gamma+2s}\bar r^{1-2s}
		= r^{1-2s},
\ee
which concludes the proof.
\end{proof}

\subsection{H\"older continuity}

In this subsection, we establish the H\"older continuity of the kernel $\bar K_f$. First, we have a lemma on the kinetic H\"older spaces and their relationship to the change of variables \eqref{e.cov-}.

\begin{lemma}\label{l:holder-cov}
Given $z_0\in \R^{7}$ and $F:\mathcal E_R(z_0) \to \R$, define $\bar F :Q_R \to \R$ by $\bar F(z) = F(\mathcal T_0(z))$. Then,
\[ \|\bar F\|_{C_\ell^\beta(Q_R)}  \lesssim \|F\|_{C_\ell^\beta(\mathcal E_R(z_0))} \leq |v_0|^{\bar c \beta} \|\bar F\|_{C^\beta_\ell(Q_R)}, \]
with $\bar c :=-\gamma/(2s)>0$.  
\end{lemma}
\begin{proof}
The argument is the same as \cite[Lemma 5.19]{imbert2020smooth}, but the powers of $|v_0|$ are different because of the different definition of $\mathcal T_0$ (recall~\eqref{e.cov-}). 

We claim that for all $|v_0|>2$ and $z, z_1 \in \R^7$,
\begin{align}
d_\ell(\mathcal T^{-1}(z), \mathcal T^{-1}(z_1)) &\leq |v_0|^{-\frac \gamma {2s}} d_\ell(z,z_1),\label{e.1}\\
d_\ell(\mathcal T(z), \mathcal T(z_1)) &\leq d_\ell(z,z_1).\label{e.2}
\end{align}
For \eqref{e.1}, we use $\|T_0^{-1}\| = |v_0|^{1 - \frac{\gamma+2s}{2s}}$ and the definition \eqref{e.dl} of $d_\ell$ to write
\[\begin{split}
d_\ell(&\mathcal T^{-1}(z), \mathcal T^{-1}(z_1))
	\\&= \min_{w\in \R^3} \max\Big( |t-t_1|^\frac{1}{2s}, |T_0^{-1}(x- x_1 - (t-t_1) w)|^\frac{1}{1+2s},
	|T_0^{-1} (v-w)|,  |T_0^{-1} (v_1-w)| \Big)\\
	&\leq |v_0|^{-\frac{\gamma}{2s}}\min_{w\in \R^3} \max\left(|t-t_1|^\frac{1}{2s}, |x-x_1 - (t-t_1)w|^\frac{1}{1+2s},
		 |v-w|, |v_1-w|\right)\\
	&\leq |v_0|^{-\frac{\gamma}{2s}} d_\ell(z,z_1).
\end{split}\]
The proof of \eqref{e.2} is similar, so we omit it. With \eqref{e.1} and \eqref{e.2}, the left invariance of $d_\ell$ implies that for $\bar z = \mathcal T_0 z$ and $\bar z_1 = \mathcal T_0 z_1$,
\begin{equation}\label{e.3}
 d_\ell(\bar z, \bar z_1) \leq d_\ell(z,z_1) \lesssim |v_0|^{-\frac{\gamma}{2s}} d_\ell(\bar z, \bar z_1).\end{equation}
To conclude the proof, fix $z, z_1 \in Q_R$, and let $\bar z = \mathcal T_0 z$ and $\bar z_1 = \mathcal T_0 z_1$. Let $p$ be the polynomial expansion of $\bar F$ at the point $z_1$ of degree $\deg_k p < \beta$, such that $|\bar F(z) - p(z)|\leq [\bar F]_{C_\ell^\beta} d_\ell(z,z_1)^\beta$. Since $p\circ \mathcal T_0^{-1}$ is a polynomial of the same degree as $p$, there holds
\[ |F(\bar z) - (p\circ \mathcal T_0^{-1})(\bar z)| = |\bar F(z) - p(z)|\leq [\bar F]_{C_\ell^\beta} d_\ell(z,z_1)^\beta \leq [\bar F]_{C_\ell^\beta} d_\ell(\bar z, \bar z_1)^\beta |v_0|^{-\frac{\beta\gamma}{2s}},\]
from the second inequality in \eqref{e.3}. This implies the second inequality in the statement of the lemma. The first inequality in the lemma follows from the first inequality of \eqref{e.3} in a similar way.
\end{proof}

Next, we establish the H\"older regularity of the kernel $\bar K_f$. The following lemma extends \cite[Lemma 5.20]{imbert2020smooth}.

\begin{lemma}\label{l:cov-holder}
Assume $\gamma+2s < 0$. For any $f:[0,T]\times \R^3\times\R^3 \to \R$ such that $f\in C^\alpha_{\ell,q}$ with $q> 3 + \alpha/(1+2s)$, 
and for any $|v_0|>2$ and $r\in(0,1]$, let
\[ \bar K_{f,z}(w) := \bar K_f(t,x,v,v+w), \quad z = (t,x,v) \in Q_{2r}.\]
Then we have
\[
	\int_{B_\rho} |\bar K_{f,z_1}(w) - \bar K_{f,z_2}(w)| |w|^2 \dd w
		\leq \bar A_0 \rho^{2-2s} d_\ell(z_1,z_2)^{\alpha'}, \quad \rho > 0, z_1,z_2\in Q_{2r},
\]
with $\alpha' = \alpha \frac {2s} {1+2s}$ and
\[
	\bar A_0 \leq C |v_0|^{2+\frac{\alpha}{1+2s}} \|f\|_{C^\alpha_{\ell, q}},
\]
where 
 the constant $C$ depends on universal quantities, $\alpha$, $q$, and $T$, but is independent of $|v_0|$.
\end{lemma}
\begin{proof}
   Changing variables according to $w \mapsto T_0^{-1} w$,
\[ \begin{split}
	\int_{B_\rho} |\bar K_{f,z_1} (w) - \bar K_{f,z_2}(w)| |w|^2 \dd w
	&= |v_0|^{2+\frac{3\gamma}{2s}} \int_{B_\rho} |K_{f,\bar z_1}(T_0w) - K_{f,\bar z_2}(T_0w)| |w|^2 \dd w\\
	&= \int_{E_\rho} |K_{f,\bar z_1}(w) - K_{f,\bar z_2}( w)| |T_0^{-1} w|^2 \dd  w\\
	&\leq |v_0|^{-\frac{\gamma}{s}}\int_{B_{\rho|v_0|^{(\gamma+2s)/2s}}} |K_{f,\bar z_1}( w) - K_{f,\bar z_2}( w)| |w|^2 \dd w,
\end{split}\]
using $\|T_0^{-1}\| = |v_0|^{1 - \frac{\gamma+2s}{2s}}$ and $E_\rho \subset B_{\rho |v_0|^{(\gamma+2s)/2s}}$. Next, it can be shown (see the proof of \cite[Lemma 5.20]{imbert2020smooth}) from the definition \eqref{e.kernel} of $K_f$ that
\[ |K_{f,\bar z_1}(w) - K_{f,\bar z_2}(w)| \leq K_{\Delta f, \bar z_1}(w),\]
where, following the notation of~\cite{imbert2020smooth},
\[
	\Delta f(z) = |f(z) - f(\xi\circ z)|
\]
and $\xi = \bar z_2 \circ \bar z_1^{-1}$. Using Lemma \ref{l:K-upper-bound-2}, we now have
\begin{equation}\label{e.BrhoK}
	\begin{split}
	  \int_{B_\rho} |\bar K_{f,z_1} (w) - &\bar K_{f,z_2}(w)| |w|^2 \dd w
	  \lesssim |v_0|^{-\gamma/s} \int_{B_{\rho|v_0|^{(\gamma+2s)/2s}}} K_{\Delta f, \bar z_1}(w) |w|^2 \dd w \\
	  &\lesssim |v_0|^{-\gamma/s} \left(\int_{\R^3} \Delta f(\bar t_1,\bar x_1,\bar v_1+u) |u|^{\gamma+2s} \dd u\right) \left( \rho |v_0|^{\frac{\gamma+2s}{2s}}\right)^{2-2s}\\
	  &\lesssim |v_0|^{ 2 - \gamma - 2s}  \left(\int_{\R^3} \Delta f(\bar t_1,\bar x_1,\bar v_1 +u) |u|^{\gamma+2s} \dd u\right)  \rho^{2-2s}.
	  \end{split}
\end{equation}
To estimate $\Delta f(\bar t_1, \bar x_1, \bar v_1 +u)$ from above, using \eqref{e.right-translation}, one can show (see formula (5.23) in \cite{imbert2020smooth} or the analysis of \eqref{e.Brho} above) that
\[
	|f(\bar z_1 \circ (0,0,u)) - f(\bar z_2\circ (0,0,u))|
		\lesssim \Big(d_\ell(\bar z_1, \bar z_2) + |\bar t_1 - \bar t_2|^\frac{1}{1+2s}|u|^{1/(1+2s)}\Big)^\alpha \langle\bar v_1 +u\rangle^{-q} \|f\|_{C^\alpha_{\ell, q}},
\]
with $q := q' + \alpha/(1+2s)$. 
Therefore,
\[
	\begin{split}
	\Delta f(\bar t_1, \bar x_1, \bar v_1 + u) 
	&\leq \langle \bar v_1 + u\rangle^{-q} \|f\|_{C^\alpha_{\ell, q}} \left( d_\ell(\bar z_1 , \bar z_2) + |\bar t_1 - \bar t_2|^\frac{1}{1+2s} |u|^\frac{1}{1+2s}\right)^\alpha\\
	&\lesssim \langle \bar v_1 + u\rangle^{-q} \|f\|_{C^\alpha_{\ell, q}} \Big(d_\ell(z_1 , z_2)^\alpha + |t_1 - t_2|^{\frac{\alpha}{1+2s} }|u|^\frac{\alpha}{1+2s}\Big)\\
	&\leq  \langle \bar v_1 + u\rangle^{-q} |u|^\frac{\alpha}{1+2s} \|f\|_{C^\alpha_{\ell,\tilde q}} d_\ell(z_1,z_2)^{\alpha'} ,
\end{split}
\]
since $d_\ell(\bar z_1, \bar z_2) \leq d_\ell(z_1,z_2)$ and $|\bar t_1 - \bar t_2| = |t_1-t_2|\leq d_\ell(z_1,z_2)^{2s}$. Returning to \eqref{e.BrhoK}, we have
\[
\begin{split}
	  \int_{B_\rho} |\bar K_{f,z_1} (w) - \bar K_{f,z_2}(w)| |w|^2 \dd w &\lesssim |v_0|^{2-\gamma-2s} \|f\|_{C^\alpha_{\ell,q}} d_\ell(z_1,z_2)^{\alpha'}\left(\int_{\R^3} |u|^{\gamma+2s+\frac{\alpha}{1+2s}} \langle v_1 + u\rangle^{-q} \dd u\right) \rho^{2-2s}\\
	  &\lesssim |v_0|^{2+\frac{\alpha}{1+2s}}\|f\|_{C^\alpha_{\ell,q}} d_\ell(z_1,z_2)^{\alpha'} \rho^{2-2s}
\end{split}
\]
since $q>3 + \alpha/(1+2s)$ (recall $\gamma+2s<0$). We have used the convolution estimate from \Cref{l:convolution} and the fact that $|\bar v_1|\approx |v_0|$.
\end{proof}

\section{Technical lemmas}\label{s:lemmas}

In this appendix, we collect some technical lemmas. First, we have an estimate for convolutions with functions $v\mapsto |v|^p$. We state it without proof.
\begin{lemma}\label{l:convolution}
For any $p>-3$ and $f\in L^\infty_q(\R^3)$ with $q>3 + p$, there holds
\[
	\int_{\R^3} f(v+w) |w|^p \dd w \leq C\|f\|_{L^\infty_q(\R^3)}
		\vv^{p + (3-q)_+}
		\cdot
		\begin{cases}
			1
				&\quad\text{ if } q \neq 3,\\
			1 + \log \vv
				&\quad\text{ if } q = 3.
		\end{cases}
\]
for a constant $C$ depending on $p$ and $q$.
\end{lemma}

The following interpolation lemma allows us to trade regularity for decay. We also omit the proof of this lemma, which is standard.

\begin{lemma}\label{l:moment-interpolation}
	Suppose that $\phi:\R^6\to \R$ is such that $\phi\in L^\infty_{q_1}(\R^6)$ and $\phi\in C^\alpha_{\ell,q_2}(\R^6)$, for some $\alpha\in (0,1)$ and $q_1 \geq q_2 \geq 0$.  If $\beta \in (0,\alpha)$ and $m\in [q_2, q_1]$ are such that
	\[
		m
			\leq q_1 \left(1 - \frac{\beta}{\alpha}\right) + q_2 \frac{\beta}{\alpha},
	\]
	then
	\[
		\|\phi\|_{C^{\beta}_{\ell,m}(\R^6)}
			\lesssim [\phi]_{C^\alpha_{\ell,q_2}(\R^6)}^\frac{\beta}{\alpha} \|\phi\|_{L^{\infty}_{q_1}(\R^6)}^{1-\frac{\beta}{\alpha}}.
	\]
\end{lemma}

Next, we quote an estimate for $Q_{\rm s}(f,g)$ H\"older norms. This lemma is stated in \cite{imbert2020smooth} for the case $\gamma+2s\geq0$, but it is clear from the proof that the same statement holds when $\gamma+2s<0$. 

\begin{lemma}{\cite[Lemma 6.8]{imbert2020smooth}}\label{l:Q1holder}
Let $q> 3 + (\gamma+2s)_+$, $\alpha \in (0,\min(1,2s))$,  and $\tau < T$, and assume that $f\in C^\alpha_{\ell,q}([\tau,T]\times\R^3\times\R^3)$ and $g \in C^{2s+\alpha}_{\ell,q}([\tau,T]\times\R^3\times\R^3)$. Then $Q_1(f,g)\in C^{\alpha'}_{\ell,q-(\gamma+2s)_+ - \alpha/(1+2s)}(\R^3)$ for $\alpha' = \frac {2s}{1+2s}\alpha$, and
\[
	\|Q_1(f,g)\|_{C^{\alpha'}_{q-(\gamma+2s)_+ - \alpha/(1+2s)}}
		\leq C
			\|f\|_{C^\alpha_{\ell,q}}
			\|g\|_{C^{2s+\alpha}_{\ell,q}}
.\]
The constant $C$ depends only on $s$, $\gamma$, and the collision kernel.
\end{lemma}

Finally, we have an estimate for $Q_{\rm ns}(f,g)$ in H\"older norms:
\begin{lemma}{\cite[Lemma 6.2]{imbert2020smooth}}\label{l:Q2holder} 
For $\alpha \in (0,\min(1,2s))$ and $\tau < T$, let $q> 3 + \alpha/(1+2s)$ and $\alpha' = \frac {2s}{1+2s} \alpha$. If $f \in C^\alpha_{\ell,q}([\tau,T]\times\R^3\times\R^3)$ and $g\in C^{\alpha'}_{\ell,q+\alpha/(1+2s)+\gamma}([\tau,T]\times\R^3\times\R^3)$, then
\[
	\|Q_{\rm ns}(f,g) \|_{C^{\alpha'}_{\ell,q}}
		\leq C\|f\|_{C^\alpha_{\ell,q}}
			\|g\|_{C^{\alpha'}_{\ell,q+\alpha/(1+2s) + \gamma}}
,\]
with the constant $C$ depending only on $\gamma$, $s$, and $\alpha$.
\end{lemma}

\bibliographystyle{abbrv}
\bibliography{low-regularity}

\end{document}